\documentclass{SCAE}
\numberwithin{equation}{section}

\usepackage{amssymb}
\usepackage{amsmath}
\usepackage{amsthm}
\usepackage{color}
\usepackage{graphicx}
\usepackage{tikz}
\usepackage{accents}

\def\rr{{\mathbb R}}
\def\rn{{{\rr}^n}}

\def\zz{{\mathbb Z}}
\def\nn{{\mathbb N}}

\def\cc{{\mathbb C}}

\def\cs{{\mathcal S}}

\def\cd{{\mathcal D}}
\def\ce{{\mathcal E}}
\def\cf{{\mathcal F}}
\def\cg{{\mathcal G}}
\def\cl{{\mathcal L}}

\def\cq{{\mathcal Q}}
\def\cp{{\mathcal P}}
\def\cm{{\mathcal M}}
\def\cn{{\mathcal N}}

\def\ca{{\mathcal A}}

\def\fz{\infty}
\def\az{\alpha}

\def\dist{{\mathop\mathrm{\,dist\,}}}

\def\lz{\lambda}

\def\bz{\beta}

\def\gz{{\gamma}}

\def\vz{\varphi}
\def\tz{\Theta}

\def\wz{\widetilde}

\def\hs{\hspace{0.3cm}}

\def\ls{\lesssim}
\def\gs{\gtrsim}

\def\gfz{\genfrac{}{}{0pt}{}}

\def\rn{{{\mathbb R}^n}}
\def\rr{{\mathbb R}}
\def\cc{{\mathbb C}}
\def\zz{{\mathbb Z}}
\def\nn{{\mathbb N}}
\def\cm{{\mathcal M}}

\def\hs{\hspace{0.3cm}}

\def\fz{\infty}
\def\az{\alpha}
\def\supp{{\mathop\mathrm{\,supp\,}}}
\def\dist{{\mathop\mathrm{\,dist\,}}}

\def\lz{\lambda}

\def\bz{\beta}

\def\gz{{\gamma}}
\def\vz{\varphi}
\def\tz{\Theta}

\def\wz{\widetilde}

\def\ls{\lesssim}
\def\gs{\gtrsim}
\def\laz{\langle}
\def\raz{\rangle}

\def\bmo{{{\mathrm {bmo}\,(\rn)}}}
\def\unif{{\rm unif}}

\def\dsum{\displaystyle\sum}

\def\dint{\displaystyle\int}

\def\dsup{\displaystyle\sup}

\def\r{\right}
\def\lf{\left}

\newcommand{\at}{{A}_{p,q}^{s,\tau}(\rn)}
\newcommand{\bt}{{B}_{p,q}^{s,\tau}(\rn)}
\newcommand{\ft}{{F}_{p,q}^{s,\tau}(\rn)}

\newcommand{\sat}{{a}_{p,q}^{s,\tau}(\rn)}
\newcommand{\sbt}{{b}_{p,q}^{s,\tau}(\rn)}
\newcommand{\sft}{{f}_{p,q}^{s,\tau}(\rn)}

\newcommand{\atu}{A^{s,\tau}_{p,q,{\rm unif}}(\rn)}
\newcommand{\ftu}{F^{s,\tau}_{p,q,{\rm unif}}(\rn)}
\newcommand{\btu}{B^{s,\tau}_{p,q,{\rm unif}}(\rn)}
\newcommand{\satu}{a^{s,\tau}_{p,q,{\rm unif}}(\rn)}
\newcommand{\sftu}{f^{s,\tau}_{p,q,{\rm unif}}(\rn)}
\newcommand{\sbtu}{b^{s,\tau}_{p,q,{\rm unif}}(\rn)}

\newcommand{\cfi}{{\mathcal{F}}^{-1}}

\def\hs{\hspace{0.3cm}}

\begin{document}

\Year{2015} %
\Month{June}
\Vol{59} %
\No{1} %
\BeginPage{1} %
\EndPage{XX} %
\AuthorMark{YUAN W {\it et al.}}
\ReceivedDay{March 10, 2015}
\AcceptedDay{June 5, 2015}
\PublishedOnlineDay{; published online June ??, 2015}
\DOI{10.1007/s11425-015-5047-8} 

\title{Interpolation of Morrey-Campanato and  Related Smoothness Spaces}{}


\author[1]{YUAN Wen}{}
\author[2]{SICKEL Winfried}{}
\author[1]{YANG Dachun}{Corresponding author}

\address[{\rm1}]{School of Mathematical Sciences, Beijing Normal University,
Laboratory of Mathematics and Complex Systems, \\
Ministry of
Education, Beijing {\rm 100875}, People's Republic of China;}
\address[{\rm2}]{Mathematisches Institut, Friedrich-Schiller-Universit\"at Jena,
Jena {\rm 07743}, Germany.}
\Emails{wenyuan@bnu.edu.cn,
winfried.sickel@uni-jena.de, dcyang@bnu.edu.cn}\maketitle


 {\begin{center}
\parbox{14.5cm}{\begin{abstract}
 In this article, the authors study the  interpolation of Morrey-Campanato spaces
and  some smoothness spaces based on Morrey spaces, e.\,g.,
Besov-type and Triebel-Lizorkin-type spaces. Various interpolation
methods, including the complex method,  the $\pm$-method
and the Peetre-Gagliardo method, are studied in such a framework.
Special emphasize is given to the quasi-Banach case and to
the interpolation property.
\vspace{-3mm}
\end{abstract}}\end{center}}



\renewcommand{\baselinestretch}{1.2}
\begin{center} \renewcommand{\arraystretch}{1.5}
{\begin{tabular}{lp{0.8\textwidth}} \hline \scriptsize
{\bf Citation:}\!\!\!\!&\scriptsize YUAN W, SICKEL W, YANG D. Interpolation of
Morrey-Campanato and  related smoothness spaces. Sci China Math, 2016,
59, doi: 10.1007/s11425-015-5047-8\vspace{1mm}
\\
\hline
\end{tabular}}\end{center}

\baselineskip 11pt\parindent=10.8pt  \wuhao


\section{Introduction}\label{s1}


In this article, we try to give an overview on the
interpolation of Morrey-Campanato spaces as well as
the interpolation of smoothness spaces built on Morrey spaces
(such as Besov-Morrey spaces, Triebel-Lizorkin-Morrey spaces,
Besov-type spaces and Triebel-Lizorkin-type spaces).
Special attention is paid to the quasi-Banach case and to the interpolation property.

Morrey spaces can be understood as a replacement (or a generalization) of the Lebesgue spaces $L_p (\rr^n)$.
This is immediate in view of their definitions.
Let  $0<p\le u\le\infty$.
Then the \emph{Morrey space $\mathcal{M}^u_{p}(\rr^n)$}
is defined as  the collection of all $p$-locally
Lebesgue-integrable functions $f$ on $\rr^n$ such that
\begin{equation}\label{morrey7}
\|f\|_{\mathcal{M}^u_{p}(\rr^n)} :=  \sup_{B}
|B|^{1/u-1/p}\lf[\int_B |f(x)|^p\,dx\r]^{1/p}<\infty\, ,
\end{equation}
where the supremum is taken over all balls $B$ in $\rr^n$.
Obviously, $\mathcal{M}^p_{p}(\rr^n)=L_p(\rn)$.
It is well known that the Morrey spaces have a lot of
applications in partial differential equations and boundedness of
operators; see, for example,
Taylor \cite{ta92}, Kozono and Yamazaki \cite{KY}, Mazzucato \cite{ma01,ma03} and
Lemari{\'e}-Rieusset  \cite{LR07,LR12,LR, LR2+}. Recently,
some applications of Morrey spaces in harmonic analysis and
potential analysis were presented in a series of papers
by  Adams and Xiao \cite{ax04,ax11,ax12a,ax12b}.
It is well known that the real-variable theory of function spaces,
including Morrey spaces, and its  various
applications in analysis are central topics of harmonic analysis; see, for example,
\cite{t78,t83,t92,t06,t12,t14,ysy,s011,s011a,ly14,yhsy14,yhsy15,flyy,tl15,ns14,xu14}.

The study of interpolation properties of Morrey spaces started with the articles of
Stampacchia \cite{s64} in 1964 and of Campanato and Murthy \cite{cm} in 1965.
They proved that, if $T$ is a linear operator which is bounded from
$L_{q_0}(\rn)$ to the Morrey space $\cm^{u_0}_{p_0}(\rn)$ with operator norm
$M_0$ and from
$L_{q_1}(\rn)$ to the Morrey space $\cm^{u_1}_{p_1}(\rn)$ with operator norm
$M_1$, then $T$ is also bounded from
$L_{q}(\rn)$ to $\cm^{u}_{p}(\rn)$, where
$\tz\in(0,1)$, $p_0,\ p_1,\ u_0,\ u_1, \ q_0, \ q_1 \in [1,\infty)$,
\begin{equation}\label{interpol}
\frac 1q := \frac{1-\tz}{q_0} + \frac{\tz}{q_1}\, , \qquad
\frac 1u := \frac{1-\tz}{u_0} + \frac{\tz}{u_1} \, ,
\qquad \frac 1p := \frac{1-\tz}{p_0} + \frac{\tz}{p_1}
\end{equation}
and the operator norm satisfies
\[
\|T\|_{L_{q}(\rn)\to\cm^{u}_{p}(\rn)}:= \sup_{\|\, f \, \|_{L_q(\rn)} \le 1}\,  \|\, Tf\, \|_{\cm^{u}_{p}(\rn)}
\le M_0^{1-\tz}\, M_1^\tz\, .
\]
Also,  Spanne \cite{sp66} in 1966 and Peetre \cite{p69} in 1969 gave  proofs of these properties
and, in addition, they discussed some generalizations via
replacing  the couple  $(L_{q_0}(\rn),L_{q_1}(\rn))$
by an abstract interpolation couple $(X_0,X_1)$. Implicitly contained is the following assertion:
Letting $F$ be an interpolation functor of exponent $\Theta$ such that
\begin{equation}
\label{functor}
F(L_{p_0}(\rn),L_{p_1}(\rn)) \hookrightarrow  L_{p}(\rn) \, ,
\end{equation}
if $T$ is a linear operator such that  $T$ is bounded from
$X_0$ to  $\cm^{u_0}_{p_0}(\rn)$ with  norm
$M_0$ and from
$X_1$ to  $\cm^{u_1}_{p_1}(\rn)$ with  norm
$M_1$, then $T$ is also bounded from
$F(X_0,X_1) $ to $\cm^{u}_{p}(\rn)$, where $u$ and $p$ are defined as in (\ref{interpol})
and
\[
\|\, T\, \|_{F(X_0,X_1)  \to \cm^{u}_{p}(\rn)}
\le c\,  M_0^{1-\tz}\, M_1^\tz\,
\]
with $c$ being a non-negative constant (depending on \eqref{functor}).

However, many questions have been left open. We mention the following:
\begin{itemize}
 \item What about the converse of the above described property? That is, if
 the linear operator  $T$ is bounded from $\cm^{u_0}_{p_0}(\rn)$ into
$L_{q_0}(\rn)$ and from $\cm^{u_1}_{p_1}(\rn)$ into $L_{q_1}(\rn)$, does it follow that
$T$ is also bounded from
$\cm^{u}_{p}(\rn)$ to $L_p(\rn)$? Here $\tz$, $p_0,\ p_1, \ u_0, \ u_1,  \ q_0, \ q_1,\ p,\ q$ are as in
(\ref{interpol}).
\item Is there any concrete interpolation method
(having the interpolation property) such that
 the application of this method to the couple
 $(\cm^{u_0}_{p_0}(\rn),\cm^{u_1}_{p_1}(\rn))$ yields a Morrey space?
\end{itemize}

In 1995, Ruiz and Vega \cite{rv95}
gave a partial negative answer to both questions.
They showed that, when
$n\in\nn\setminus\{1\}$,  $\tz\in(0,1)$, $u\in(0,n)$, and
\[
1\le p_1<p_2<\frac{n-1}u< p_0 < \fz\, ,
\]
then, for any given positive number $C$, there exists
a positive continuous linear operator
$T$: $\cm_{p_i}^u(\rn)\to L_1(\rn)$, $i\in\{0,1,2\}$, with  operator norm satisfying that
$\|T\|_{\cm_{p_i}^u(\rn)\to L_1(\rn)}\le K_i,\ i\in\{0,1\}$,
but
\[
\|T\|_{\cm_{p_2}^u(\rn)\to L_1(\rn)} > CK_0^{1-\tz}K_1^\tz \qquad \mbox{with}
\qquad
\frac1{p_2}=\frac{1-\tz}{p_0}+\frac\tz{p_1}.
\]
This explicit construction requires dimension $n>1$.
In 1999, Blasco, Ruiz and Vega \cite{brv} considered also the case $n=1$.
For a particular $u$,
satisfying $1<p_0<p_1<u<\fz$, they proved that
there exist $q_0,q_1\in(1,\fz)$ and a positive linear operator $T$,
which is bounded from $\cm_{p_i}^u(\rr)$ to
$L_{q_i}(\rr),\,i\in\{0,1\}$, but not
bounded from $\cm_p^u(\rr)$ to $L_q(\rr)$, where
$\tz$, $p_0,\ p_1,  \ q_0, \ q_1, p,q$ are as in (\ref{interpol}).
These counterexamples are making clear that, in general,
the answer to the above two questions must be negative.

After the articles \cite{rv95} and \cite{brv} had appeared, the believe in positive results
in this area was not very big.
However, the recent articles by Lemari{\'e}-Rieusset \cite{LR,LR2}, Yang et al. \cite{yyz} and Lu et al. \cite{lyy} indicated some essential progress.
Lemari{\'e}-Rieusset \cite{LR} proved that, if
\begin{equation}
 \label{interpol1a}
1 < p_i \le u_i <\infty \, , \quad
i\in\{0,1\}\, , \quad
\frac 1u := \frac{1-\tz}{u_0} + \frac{\tz}{u_1} \, ,
\quad \frac 1p := \frac{1-\tz}{p_0} + \frac{\tz}{p_1} \, ,
\end{equation}
then
\begin{equation}
 \label{interpol1}
 \lf[\cm^{u_0}_{p_0}(\rn), \cm^{u_1}_{p_1}(\rn)\r]_\tz \neq
 \cm^{u}_{p}(\rn)
\end{equation}
whenever
\begin{equation}
 \label{important-n}
p_0 /u_0 \neq p_1/u_1 \, ,
\end{equation}
giving a much better understanding of the negative results in this way.
Furthermore, Lemari{\'e}-Rieusset  \cite{LR2} proved that, if
\eqref{important-n} does not holds true, namely,
if
\begin{equation}
 \label{important}
p_0 /u_0 = p_1/u_1 \, ,
\end{equation}
then \begin{equation}\label{outcom}
 \lf[\cm^{u_0}_{p_0}(\rn), \cm^{u_1}_{p_1}(\rn)\r]^\tz=
 \cm^{u}_{p}(\rn).
\end{equation}
Here $[A_0,A_1]_\tz$ and $[A_0,A_1]^\tz$ denote two different complex
methods of  interpolation theory introduced by Calder\'on \cite{ca64}, respectively.
The restriction \eqref{important} will be always present throughout this article in connection with positive results.
Whenever we are able to prove an interpolation formula for Morrey spaces with different $p$ or different $u$, this restriction \eqref{important} will be used.
In this particular case, we will supplement \eqref{interpol1a} by showing that
\eqref{important} is necessary.
The first positive results in interpolation of Morrey spaces go back to  Yang, Yuan and Zhuo \cite{yyz}, in which they proved that
\begin{eqnarray*}
 \lf[\mathring{\cm}^{u_0}_{p_0}(\rn), \mathring{\cm}^{u_1}_{p_1}(\rn)\r]_\tz
 & = & \lf[\mathring{\cm}^{u_0}_{p_0}(\rn), {\cm}^{u_1}_{p_1}(\rn)\r]_\tz
 \\
 &=&
 \lf[{\cm}^{u_0}_{p_0}(\rn), \mathring{\cm}^{u_1}_{p_1}(\rn)\r]_\tz =
 \mathring{\cm}^{u}_{p}(\rn),
\end{eqnarray*}
if the restrictions in \eqref{interpol1a} and \eqref{important} are satisfied.
Here $\mathring{\cm}^{u}_{p}(\rn) $ denotes the \emph{closure} of the Schwartz functions in
${\cm}^{u}_{p}(\rn) $.
In case $p_0 /u_0 \neq p_1/u_1$, but $p_0$, $p_1$, $u_0$, $u_1$, $p$, $u$ as in \eqref{interpol1a}, one knows at least the following hold true:
\begin{equation*}
 \lf[\cm^{u_0}_{p_0}(\rn), \cm^{u_1}_{p_1}(\rn)\r]_\tz \hookrightarrow
 \lf[\cm^{u_0}_{p_0}(\rn)\r]^{1-\Theta} \, \lf[ \cm^{u_1}_{p_1}(\rn)\r]^\tz
 \hookrightarrow \cm^{u}_{p}(\rn) \, ,
\end{equation*}
where $X_0^{1-\Theta}\, X_1^\tz$ denotes the Calder{\'o}n product of $X_0$ and $X_1$
(see \cite{yyz,lyy}).
Concerning the real interpolation  method
with $u$ and $p$ as in (\ref{interpol1a}), one knows that
\begin{equation*}
\lf(\cm^{u_0}_{p_0}(\rn), \cm^{u_1}_{p_1}(\rn)\r)_{\tz,p}
\hookrightarrow \cm^{u}_{p}(\rn) \,
\end{equation*}
(see \cite{LR}, \cite{ma03}, \cite{s011a}), and
\begin{equation*}
\cm^{u}_{p}(\rn) \hookrightarrow \lf(\cm^{u_0}_{p_0}(\rn), \cm^{u_1}_{p_1}(\rn)\r)_{\tz,\infty}  \qquad
\Longleftrightarrow \qquad  p_0 /u_0 = p_1/u_1\,
\end{equation*}
(see \cite{LR}).
It will be our aim to supplement these assertions.
Going back to the
two questions asked above, we see that they are partially answered only.
We will return to these problems in Subsection \ref{sum}.

Now we turn to Campanato spaces, which are some generalizations of Morrey spaces.
We need a few notation. Let $B(x,r)$ denote the ball in $\rn$ with center in $x$ and radius $r\in(0,\fz)$.
By $\cp_k$, we denote the \emph{class of polynomials in $\rn$ of order at most $k$}. In addition, we put $\cp_{-1} := \{0\}$.
Let $p\in(0,\infty)$, $k\in \{-1,0\} \cup \nn$, $\lambda \in[0,\fz)$ and $\Omega$ be a bounded open subset of $\rn$.
Then the \emph{Campanato space $\cl^{p, \lambda}_k (\Omega)$} is defined as
the set of all $f \in L_p^{\ell oc} (\Omega)$ such that
\[
\| \, f \, \|_{\cl^{p, \lambda}_k (\Omega)}:=
\sup_{x \in \Omega} \sup_{r>0} \, \lf[\frac{1}{|B(x,r) \cap \Omega|^{\lambda/n}} \inf_{P \in \cp_k}\,
\int_{B(x,r)\cap \Omega} |f(y)-P(y)|^p\, dy\r]^{1/p} <\infty\, ;
\]
see Campanato \cite{ca1,ca2,ca3,ca4} or the more recent survey by Rafeiro et al. \cite{RSS}. Here and hereafter, $L_p^{\ell oc} (\Omega)$ denotes the
\emph{set of all locally integrable functions on $\Omega$}.
There exist various possibilities to extend this definition to $\rn$.
We decide to use the following ``locally uniform" point of view (i.\,e., we consider a space defined with respect to balls of volume
$\le 1$), picked up from Triebel \cite[3.1.1]{t12} (but we do not follow his notation).

Let $p\in(0, \infty)$, $\tau\in[0,\fz)$ and suppose, for the integer $k \ge -1$, that $k+1 > n\, (\tau - 1/p)$. Then
$\cl^\tau_p (\rn)$ is defined as
the collection of all $f\in L_p^{\ell oc} (\rn)$ such that
\begin{eqnarray*}
\| \, f \, |\cl^{\tau}_p (\rn)\|_k := &&
\sup_{x \in \rn} \sup_{0 < r \le 1} \, \lf[\int_{B(x,r)} |f(y)|^p\, dy\r]^{1/p}
\\
&& +
\sup_{x \in \rn} \sup_{0 < r \le 1} \, |B(x,r)|^{-\tau} \lf[\inf_{P \in \cp_k}\, \int_{B(x,r)} |f(y)-P(y)|^p\, dy\r]^{1/p} <\infty\, .
\end{eqnarray*}
The space $\cl^\tau_p (\rn)$ is quasi-Banach and independent of
the chosen admissible $k$; see \cite[3.1.2/Theorem 3.4]{t12} (which itself is based on \cite[5.3.3]{t92} and
\cite[2.1]{br09}).
The following assertions are part of the classical theory of Campanato spaces:
\begin{itemize}
 \item[(a)] Let $p\in(0,\infty)$ and $\tau\in[0,1/p)$. Then, with $1/u := 1/p -\tau$, we have
$\cl^{\tau}_p (\rn) = M_{p, \unif}^u (\rn)$ in the sense of equivalent quasi-norms, i.\,e., for all $k \ge -1$,
there exist positive constants $A$ and $B$ such that
\[
A \, \| \, f \, \|_{M_{p, \unif}^u (\rn)} \le \| \, f \, |\cl^{\tau}_p (\rn)\|_k \le B \, \| \, f \, \|_{M_{p, \unif}^u (\rn)}
\]
for all $f\in \cl^{\tau}_p (\rn)$.
Here the definition of $M_{p, \unif}^u (\rn)$ is obtained from \eqref{morrey7} by restricting the supremum to balls with volume
$\le 1$.
\item[(b)]
Let $p\in(0,\infty)$ and $\tau\in(1/p,\fz)$. Then
$\cl^{\tau}_p (\rn) = B^{n(\tau -1/p)}_{\infty,\infty} (\rn)$ in the sense of equivalent quasi-norms
(for all admissible $k$).
\item[(c)] Let $p\in(0,\infty)$ and $\tau =1/p$. Then, with $k=0$,
$\cl^{\tau}_p (\rn)= \bmo$ in the sense of equivalent quasi-norms.
If $p\in[1,\fz)$, then this result is true for all $k \ge 0$.
\end{itemize}
For the first two items, we refer the reader
to  Campanato \cite{ca2}, Kufner et al. \cite[4.3]{KJF}, Pick et al. \cite[5.3,5.7]{PKJF},
Brudnij \cite[2.1]{br09}, and Triebel
 \cite[5.3.3]{t92}, \cite[3.1.2/Theorem~3.4]{t12}.
Concerning the last item, we refer the reader
to John and Nirenberg \cite{JN} ($p\in[1,\fz)$), Long
and Yang \cite{LY} ($p\in(0,1)$) and
Triebel  \cite[5.3.3]{t92}, \cite[3.1.2/Theorem~3.4]{t12}; see also Bourdaud \cite{Bou}.

The above quoted articles of Stampacchia \cite{s64},  Campanato and Murthy \cite{cm},
Spanne \cite{sp66} and Peetre \cite{p69} have already dealt with Campanato spaces.
The assertion described above with Morrey spaces remains true for the more general case of Campanato spaces (always with $p \in[1,\fz)$).
We did not find more recent references for the interpolation of Campanato spaces.

Based on the recent progress in the understanding of the interpolation  of Morrey spaces,  we found it desirable to summarize what is known
today about the
interpolation of Morrey-Campanato spaces and smoothness spaces built on Morrey spaces. In almost all
applications of interpolation theory,
the  associated boundedness problem for pairs of linear operators plays a role.
For this reason, we will take care also of this circle of problems.

In the last two decades, partly due to the study of Navier-Stokes equations,
there is an increasing interest in the construction of smoothness spaces based on Morrey spaces
(in what follows called Morrey-type spaces); see, for example,
\cite{KY,LXY,ma03,tx,t12,t14,yy1,yy2,ysy,hs12,hs13}. For us of certain interest are the following:
\begin{itemize}
\item Besov-Morrey spaces $\cn^s_{u,p,q}(\rn)$ (introduced and studied
by Kozono, Yamazaki \cite{KY} and later by Mazzucato \cite{ma03});
\item Triebel-Lizorkin-Morrey spaces $\ce^s_{u,p,q}(\rn)$
(introduced by Tang and Xu \cite{tx});
\item Besov-type spaces $\bt$ (introduced by El Baraka \cite{el021,el022,el062}; see also \cite{yy1,yy2});
\item Triebel-Lizorkin-type spaces $\ft$ (introduced in \cite{yy1,yy2}).
\item Local function spaces $\cl^{r}A^s_{p,q}(\rn)$ (see the recent monographs of Triebel \cite{t12,t14}).
\end{itemize}
First attempts to a systematic investigation of all these scales are made in \cite{ysy} and \cite{s011,s011a}; see also
\cite{yyz1,lsuyy1}.

\subsection*{Plan of the article}

In this article, we shall consider the interpolation properties of all of these spaces.
It turns out that the $\pm$-method of Gustavsson and Peetre represents the most helpful tool in the
interpolation of Morrey-Camapanato and Morrey-type spaces. The main results with respect to this method are
contained in Theorems \ref{COMI} and Corollaries \ref{cima}, \ref{cimab} (see Subsection \ref{inter1a}).
In addition, we shall study the complex method of  interpolation theory (see Subsection \ref{inter1c}) and
a further interpolation method, due to Gagliardo and Peetre (see Subsection \ref{inter1b}).
All three methods are closely related. However, it is of certain interest
that, for the spaces under consideration in almost all cases, the
$\pm$-method leads to a different result than the other two methods.
In Subsection \ref{inter1b}, we shall show that, in the most interesting cases, namely, the interpolation of Morrey spaces,
one has to introduce new spaces; see Theorem \ref{gp01}. The determination of the interpolation spaces will be always connected with
certain density questions, which will be studied in great detail in this subsection.
In Subsection \ref{inter1c}, we consider the complex method as well as the inner complex method.
We derive the results for the (inner) complex method by tracing it back to corresponding statements
for the Peetre-Gagliardo method; see Corollary \ref{nil2}.
In that way, we also obtain the most complete collection
of results concerning the (inner) complex interpolation of Besov and Triebel-Lizorkin spaces; see Subsection \ref{inter1cd}, in particular
Theorem \ref{gagl6}.
Later we shall summarize the known results concerning the real method (see Subsection \ref{inter1d}).
In Subsection \ref{sum}, we discuss consequences for the interpolation property.
This means, our main  results are described in Section \ref{inter1}.
Interpolation of the spaces $\cl^{r}A^s_{p,q}(\rn)$  is discussed shortly in Section \ref{inter2}.
Proofs are concentrated in Section \ref{proofs} and given in the lexicographic order.
Definitions and some basic properties of all these scales of function spaces  will be given
in Appendix at the end of this article.


\subsection*{Notation}


As usual, $\nn$ denotes the natural numbers $\{1,2,\ldots\}$ and $\zz_+:= \nn \cup \{0\}$.
$\zz$ denotes the integers and $\rr$ the real numbers.
The letter $n$ is always reserved for the dimension in $\zz^n$ and  $\rr^n$.
Let $\mathcal{S}(\rr^n)$ be the \emph{space of all Schwartz functions}
on $\rr^n$ endowed with the classical topology and $\mathcal{S}'(\rr^n)$
its \emph{topological dual space}, namely,
the set of all continuous linear functionals on $\mathcal{S}(\rr^n)$
endowed with the weak-$\ast$ topology.
We also need the \emph{Fourier transform $\widehat{f}$},
which is defined on the space  $\cs' (\rr^n)$ of tempered distributions.
We denote, by $C_c^\fz(\rn)$, the  collection of all complex-valued infinitely differentiable
functions with compact support.

If $X$ and $Y$ are two quasi-Banach spaces, then the symbol  $X \hookrightarrow Y$
indicates that the embedding of $X$ into $Y$ is continuous.
By $\cl (X,Y) $, we denote the collection of all linear and bounded operators $T:\, X \to Y$, equipped with the quasi-norm
\[
\|\, T \, \|_{\cl (X,Y)} := \|\, T \, \|_{X \to Y}:= \sup_{\|\, x\, \|_X\le 1} \, \|\, Tx\, \|_Y \, .
\]

The \emph{symbol  $C$} denotes a  positive
constant that is independent of the main parameters involved but
whose value may differ from line to line, and
the \emph{symbol $C_{(\az,\ldots)}$} denotes a positive constant depending on the parameters $\az,$
$\ldots$.  The \emph{symbol $A \ls B$} means that:
there exists a positive constant $C$ such that
 $A \le C \,B$. Similarly $\gs$ is defined. The symbol
$A \asymp B$ will be used as an abbreviation of
$A \ls B \ls A$.

Let \[\mathcal{Q}:= \{Q_{j,k}:= 2^{-j}([0,1)^n+k):\ j\in\zz,\ k\in\zz^n\}\]
be the collection of all \emph{dyadic cubes} in $\rn$.
We also need to consider the subset  $$\mathcal{Q}^*:= \{Q_{j,k}:\ j\in\zz_+,\ k\in\zz^n\}.$$
For all $Q\in\mathcal{Q}$, let $j_Q:=-\log_2\ell(Q)$,
where $\ell (Q)$ denotes the \emph{side-length} of the cube.

{\bf Convention.} If there is no reason to distinguish between
$\bt$ and $\ft$ (resp. between the corresponding sequence spaces
$\sbt$
and $\sft$), we simply write $\at$ (resp. $\sat$).
Here we always assume $p<\infty$ in case $\at = \ft$ (resp. in case $\sat = \sft$).
The same convention applies with respect to the spaces
$\btu$ and $\ftu$ (resp. $\atu$) as well as to $\sbtu$ and $\sftu$ (resp. $\satu$).


\section{Various interpolation methods}
\label{inter1}


Nowadays interpolation theory represents an important tool in various branches of mathematics.
Consulting the most quoted monographs on interpolation theory
(see \cite{BS,BL,KPS,lun,Pe76,t78}),
one obtains the impression that the real and the complex methods are most important.
In the context of Morrey spaces $\cm_p^u(\rn)$, the situation turns out to be different.
The most useful interpolation method, in case of different $p$ (and/or different $u$), turns out to be the $\pm$-method of  Gustavsson and Peetre \cite{gp77};
see also Ovchinnikov \cite{Ov}.
The real method is the most helpful in those situations where we fix $p$ and $u$ (or $p$ and $\tau$).
Also the complex method as well as a further method, introduced by Peetre \cite{p71},
but based on some earlier work of Gagliardo \cite{Ga} and,
in what follows, called Peetre-Gagliardo interpolation method, will be studied.

The main tool in this article will  be the Calder{\'o}n product.
Our method heavily depends on the articles by Nilsson \cite{n85}
(relations between the Calder{\'o}n product and other interpolation methods
in the abstract setting of quasi-Banach spaces)
and by Yang et al. \cite{yyz} (concrete Calder{\'o}n products).
In all cases, special emphasize is given to the interpolation property.


\subsection{The Calder{\'o}n product}
\label{cpr}

Let $({\mathfrak X}, {\mathcal S}, \mu)$ be a $\sigma$-finite measure space and
let ${\mathfrak M}$ be the class of all complex-valued, $\mu$-measurable functions on
${\mathfrak X}$. Then a quasi-Banach space $X\subset {\mathfrak M}$
is called a {\it quasi-Banach lattice of functions}
if, for every $f\in X$ and $g\in{\mathfrak M}$ with $|g(x)|\le |f(x)|$ for $\mu$-a.e. $x\in{\mathfrak X}$,
one has $g\in X$ and $\|g\|_X\le \|f\|_X.$

\begin{definition}
Let
$X_j \subset {\mathfrak M}$, $j\in\{0,1\}$,  be quasi-Banach lattices of functions,
and $\Theta\in(0,1)$. Then the {\it Calder\'on product $X_0^{1-\Theta}X_1^\Theta$}
of $X_0$ and $X_1$ is defined as the collection of all functions $f \in {\mathfrak M}$ such that
the  \emph{quasi-norm}
\begin{eqnarray*}
\|f\|_{X_0^{1-\Theta}X_1^\Theta} & := & \inf\Bigl\{\|f_0\|_{X_0}^{1-\Theta}\|f_1\|_{X_1}^\Theta:\:
|f|\le |f_0|^{1-\Theta}|f_1|^\Theta \quad \mu \mbox{-a.e.},\
f_j\in X_j, \, j\in\{0,1\}\Bigr\}
\end{eqnarray*}
is finite.
\end{definition}

\begin{remark}
Calder{\'on} products were introduced by Calder{\'o}n \cite[13.5]{ca64}.
The usefulness of this method and its limitations  have been perfectly
described by Frazier and Jawerth \cite{fj90}
which we quote now:
{\it Although restricted to the case of a lattice, the Calder{\'o}n product has the advantage of being
defined in the quasi-Banach case, and,
frequently, of being easy to compute. It has the disadvantage that the interpolation property
(i.\,e., the property that a
linear transformation $T$ bounded on $X_0$ and
$X_1$ should be bounded on the spaces in between)
is not clear in general.}
\end{remark}

Calder\'on products have proved to be a very  useful tool for the
study of various  interpolation methods; see, for example, \cite{ca64,l73,n85,km98}.
We collect a few useful properties of Calder{\'o}n products for the
later use; see, for example,  \cite{yyz}.

\begin{lemma}\label{Leb}
Let
$X_j \subset {\mathfrak M}$, $j\in\{0,1\}$,  be quasi-Banach lattices of functions and
let $\Theta\in(0,1)$.

{\rm(i)}
Then the {Calder\'on product $X_0^{1-\Theta}X_1^\Theta$} is a quasi-Banach space.

{\rm (ii)}
Define $\widetilde{X_0^{1-\Theta}X_1^\Theta}$ as the collection of all $f$ such that
there exist a positive real number $\lambda$ and elements $g \in X_0$ and $h \in X_1$
satisfying
\[
|f| \le \lambda |g|^{1-\Theta}\, |h|^\Theta\, \ a.\,e.,
\quad \|g \|_{X_0}\le 1 \quad {\rm and} \quad \| h \|_{X_1}\le 1\, .
\]
Let
\[
\|f  |\widetilde{X_0^{1-\Theta}X_1^\Theta}\| := \inf \Big\{ \lambda >0: \hs
|f| \le \lambda |g|^{1-\Theta}\, |h|^\Theta\, \ a.\,e.,
\hs \|g \|_{X_0}\le 1 \hs {\rm and} \hs \|h\|_{X_1}\le 1\, \Big\}.
\]
Then
$\widetilde{X_0^{1-\Theta}\, X_1^\Theta} = X_0^{1-\Theta}\, X_1^\Theta$ in the
sense of equivalent quasi-norms.
\end{lemma}

Now we turn to the investigation of linear operators and Calder{\'o}n products.
An operator $T$ on a quasi-Banach lattice $X$ is said to be
\emph{positive} if $Tf \ge 0$ whenever $f\ge 0$ is in its domain.
In 1990, Frazier and Jawerth \cite[Proposition 8.1]{fj90}
obtained the following result; see also Shestakov \cite[Theorem 3.1]{sh81}
for the Banach space case.

\begin{proposition}\label{cpi}
Let $\Theta \in (0,1)$.
Let $X_i$ and $Y_i $ be quasi-Banach lattices and let $T$ be a positive linear operator bounded from $X_i$ to $Y_i$, $i\in\{0,1\}$.
Then $T$ is bounded considered as a mapping from the Calder{\'o}n product  $X_0^{1-\Theta}X_1^\Theta$ to the
Calder{\'o}n product  $Y_0^{1-\Theta} Y_1^\Theta$ and
\[
\| \, T\, \|_{X_0^{1-\Theta}X_1^\Theta \to Y_0^{1-\Theta}Y_1^\Theta} \le
\| \, T\, \|_{ X_0 \to Y_0}^{1-\Theta} \, \| \, T\, \|_{X_1\to Y_1}^{\Theta}.
\]
\end{proposition}

Notice that, in Proposition \ref{cpi}, we need the restriction that the operator $T$
is positive. The Calder\'on product is not an interpolation construction in the class
of Banach function lattices. Indeed, an example to show this was given
by Lozanovski\u\i\  \cite{l72}.

We are interested in concrete realizations.
It is easy to see that Morrey spaces are quasi-Banach lattices, but the  smoothness spaces $\bt$, $\ft$
and $\cn^s_{u,p,q} (\rn)$ might not be (at least in general).

\begin{theorem}\label{morrey1}
 Let $\tz\in(0,1)$, $0 <  p_0 \le u_0 < \infty$ and $0 < p_1 \le u_1 < \infty$
 such that
 \[
\frac1p=\frac{1-\tz}{p_0}+\frac\tz{p_1} \qquad  \mbox{and}
\qquad
\frac1u=\frac{1-\tz}{u_0}+\frac\tz{u_1}\, .
\]

{\rm(i)} It holds true that
\begin{equation*}
 \lf[\cm^{u_0}_{p_0}(\rn)\r]^{1-\Theta} \, \lf[ \cm^{u_1}_{p_1}(\rn)\r]^\tz \hookrightarrow  \cm^{u}_{p}(\rn) \, .
\end{equation*}

{\rm(ii)} If $u_0 \, p_1 = u_1\, p_0$, then
\begin{equation*}
\lf[\cm_{p_0}^{u_0}(\rn)\r]^{1-\tz}
\lf[\cm_{p_1}^{u_1}(\rn)\r]^\tz=\cm_{p}^{u}(\rn) \, .
\end{equation*}

{\rm(iii)} If $u_0 \, p_1 \neq u_1\, p_0$, then
\begin{equation*}
\lf[\cm_{p_0}^{u_0}(\rn)\r]^{1-\tz}
\lf[\cm_{p_1}^{u_1}(\rn)\r]^\tz  \subsetneqq \cm_{p}^{u}(\rn) \, .
\end{equation*}
\end{theorem}

For proofs of (i) and (ii) of Theorem \ref{morrey1}, we refer
the reader to Lu et al. \cite{lyy}.
Part (ii) of Theorem \ref{morrey1} is explicitly stated therein, and part (i) of Theorem \ref{morrey1} can be found in \cite[Formula (2.3)]{lyy}.
Part (iii) of Theorem \ref{morrey1} follows from \cite{LR} (see Subsection \ref{proof0} for more details).

The identity in (ii) of Theorem \ref{morrey1} in  case $u_i=p_i$, $i\in\{0,1\}$, given by
\begin{equation*}
\lf[L_{p_0}(\rn)\r]^{1-\tz}\, \lf[L_{p_1}(\rn)\r]^\tz = L_{p}(\rn) \, ,
\end{equation*}
can be found in several places, we refer the reader to
\cite[Exercise 4.3.8]{BK}, \cite[Formula 1.6.1]{KPS}, \cite[P.\, 179, Exercise 3]{Ma}
(Banach case) and \cite{ssv} (general situation). For later use,
we formulate one more elementary
example as follows.
 Let
$s\in\rr$ and $q\in(0,\fz]$. Define  $\ell^{s}_{q}(\zz_+)$
as the collection of all sequences $\{a_j\}_{j\in\zz_+}\subset \cc$ such that
$$\lf\|\{a_j\}_{j\in\zz_+}\r\|_{\ell^{s}_{q}(\zz_+)}:=\lf\{\sum_{j\in\zz_+}
2^{jsq}|a_j|^{q}\r\}^{\frac1q}<\fz.$$

\begin{lemma}\label{kleinlq}
 Let $s_0,s_1\in\rr$, $p_0, p_1 \in(0,\infty]$ and $\Theta \in (0,1)$
 such that $s=(1-\tz)s_0+\tz s_1$ and $\frac1p=\frac{1-\tz}{p_0}+\frac\tz{p_1}.$
 Then
\begin{equation*}
\lf[\ell^{s_1}_{p_0}(\zz_+)\r]^{1-\tz}\, \lf[\ell^{s_0}_{p_1}(\zz_+)\r]^\tz = \ell^s_{p}(\zz_+).
\end{equation*}
\end{lemma}

 The proof of $\lf[L_{p_0}(\rn)\r]^{1-\tz}\, \lf[L_{p_1}(\rn)\r]^\tz = L_{p}(\rn)$, given in \cite{ssv},
carries over to the discrete situation.
Even more interesting are the more complicated sequence spaces $\sbt$, $\sft$
and $n^s_{u,p,q} (\rn)$,  associated to the
scales $\bt$, $\ft$ and $\cn^s_{u,p,q} (\rn)$; see Appendix at the end of this
article for their definitions and properties.
All these sequence spaces are quasi-Banach lattices.
The following results
were  proved in \cite[Propositions 2.6, 2.7 and 2.8]{yyz} and are of basic importance for the remainder of this article.

\begin{proposition}\label{morrey2}
Let $\tz\in(0,1)$, $s,\ s_0,\ s_1\in\rr$, $\tau,\ \tau_0,\ \tau_1\in[0,\fz)$,
$p,\ p_0,\ p_1\in(0,\fz]$ and $q,\ q_0,\ q_1\in(0,\fz]$
such that $s=s_0(1-\tz)+s_1\tz$, $\tau=\tau_0(1-\tz)+\tau_1\tz$,
$\frac1p=\frac{1-\tz}{p_0}+\frac\tz{p_1}$ and
$\frac1q=\frac{1-\tz}{q_0}+\frac\tz{q_1}$.
Then it holds true that
\[
\lf[a_{p_0,q_0}^{s_0,\tau_0}(\rn)\r]^{1-\tz}
\lf[a_{p_1,q_1}^{s_1,\tau_1}(\rn)\r]^\tz \hookrightarrow a_{p,q}^{s,\tau}(\rn) \, ,
\qquad a\in \{f,\ b\}\, .
\]
If, in addition,  $\tau_0 \, p_0 = \tau_1\, p_1$, then
\begin{equation*}
\lf[a_{p_0,q_0}^{s_0,\tau_0}(\rn)\r]^{1-\tz}
\lf[a_{p_1,q_1}^{s_1,\tau_1}(\rn)\r]^\tz=a_{p,q}^{s,\tau}(\rn) \, ,
\qquad a\in \{f,\ b\}\, .
\end{equation*}
\end{proposition}

\begin{proposition}\label{morrey3}
Let $\tz\in(0,1)$, $s,\ s_0,\ s_1\in\rr$, $q,\ q_0,\ q_1\in(0,\fz]$,
$0<p\le u\le\fz$, $0<p_0\le u_0\le\fz$ and $0<p_1\le u_1\le\fz$
such that $\frac1q=\frac{1-\tz}{q_0}+\frac\tz{q_1}$,
$\frac1p=\frac{1-\tz}{p_0}+\frac\tz{p_1}$, $s=s_0(1-\tz)+s_1\tz$
and $\frac1u=\frac{1-\tz}{u_0}+\frac{\tz}{u_1}$.
Then it holds true that
\begin{equation*}
\lf[n_{u_0,p_0,q_0}^{s_0}(\rn)\r]^{1-\tz}\lf[n_{u_1,p_1,q_1}^{s_1}(\rn)\r]^\tz \hookrightarrow
n_{u,p,q}^{s}(\rn).
\end{equation*}
If, in addition, $p_0u_1=p_1u_0,$
then
\begin{equation*}
\lf[n_{u_0,p_0,q_0}^{s_0}(\rn)\r]^{1-\tz}\lf[n_{u_1,p_1,q_1}^{s_1}(\rn)\r]^\tz=
n_{u,p,q}^{s}(\rn).
\end{equation*}
\end{proposition}

The assertions stated in Propositions \ref{morrey2} and \ref{morrey3} are far away
from being trivial.
The prototype is the ingenious  proof of the formula
\[
\lf[f_{p_0,q_0}^{s_0,0}(\rn)\r]^{1-\tz}
\lf[f_{p_1,q_1}^{s_1,0}(\rn)\r]^\tz = f_{p,q}^{s,0}(\rn) \, ,
\]
due to Frazier and Jawerth \cite{fj90}; see also Bownik \cite{Bow2}. Various different proofs of
\[
\lf[b_{p_0,q_0}^{s_0,0}(\rn)\r]^{1-\tz}
\lf[b_{p_1,q_1}^{s_1,0}(\rn)\r]^\tz = b_{p,q}^{s,0}(\rn)
\]
can be found in the literatures, we refer the reader to Mendes and Mitrea \cite{MM},  Kalton et al. \cite{kmm} and Sickel et al. \cite{ssv}.


\subsection{The $\pm$-method of Gustavsson and Peetre}
\label{inter1a}


The next interpolation method, called the $\pm$-method,  was originally introduced by
Gustavsson and Peetre \cite{gp77,g82}. Later it has been  considered also by
Berezhnoi \cite{be80}, Gustavsson \cite{g82}, Nilsson \cite{n85}, Ovchinnikov
\cite{Ov} and Shestakov \cite{sh81}.

Consider a \emph{couple of quasi-Banach spaces}
(for short, a \emph{quasi-Banach couple}), $X_0$ and $X_1$,
which are continuously embedded into a larger Hausdorff topological
vector space $Y$. The {space} $X_0+X_1$ is given by
$$X_0+X_1:=\{h\in Y:\ \exists\ h_i\in X_i,\ i\in\{0,1\},\ {\rm such\ that}\ h=h_0+h_1\},$$
equipped with the  {quasi-norm}
\[
\|h\|_{X_0+X_1}:=\inf\Big\{\|h_0\|_{X_0}+\|h_1\|_{X_1}:\ h=h_0+h_1,
\ h_0\in X_1\ {\rm and}\ h_1\in X_1\Big\}.
\]

\begin{definition}
Let $(X_0,X_1)$ be a  quasi-Banach couple and $\Theta \in (0,1)$.
An $a\in X_0+X_1$ is said to belong to
$\laz X_0, X_1,\Theta\raz$ if there exists a sequence $\{a_i\}_{i\in\zz}
\subset X_0\cap X_1$ such that $a=\sum_{i\in\zz}\, a_i$ with convergence
in $X_0+X_1$
and, for any finite subset $F\subset \zz$ and any bounded sequence $\{\varepsilon_i\}_{i\in\zz}\subset\cc$,
\begin{equation}\label{2.1x}
\lf\|\sum_{i\in F} \varepsilon_i \, 2^{i(j-\Theta)}\, a_i\r\|_{X_j}\le
C \sup_{i\in\zz}|\varepsilon_i|
\end{equation}
for some non-negative constant $C$ independent of $F$
and $j\in\{0,1\}$. The \emph{quasi-norm} of
$a\in\laz X_0, X_1, \Theta\raz$  is defined as
\[
\|a\|_{\laz X_0, X_1,\Theta\raz}:=\inf \lf\{C: \, C\ \ \mbox{satisfies}\ \ \eqref{2.1x}\r\}.
\]
\end{definition}

We recall a few results from \cite[Proposition 6.1]{gp77}.

\begin{proposition}\label{gustav}
Let $(A_0,A_1)$ and  $(B_0,B_1)$ be any two quasi-Banach couples and  $\Theta \in (0,1)$.

{\rm(i)}  It holds true that $\laz A_0, A_1,\Theta\raz$ is   a quasi-Banach space.

{\rm(ii)} If $T \in \cl (A_0,B_0) \cap \cl(A_1,B_1)$, then $T$ maps  $\laz A_0, A_1,\Theta \raz$
continuously into $\laz B_0, B_1,\Theta\raz$. Furthermore,
\[
\| \, T \, \|_{\laz A_0, A_1,\Theta \raz \to \laz B_0, B_1,\Theta \raz} \le \max \lf\{
\| \, T \, \|_{A_0 \to B_0} ,\, \|\, T \, \|_{A_1 \to  B_1}\r\}\, .
\]
\end{proposition}

Next we recall a standard method in the
interpolation theory, namely, the method of retraction.
Let $X$ and $Y$ be two quasi-Banach spaces. Then
$Y$ is called a \emph{retract} of $X$ if there exist two bounded
linear operators $E: ~Y \to X$ and $R: ~X \to Y$
such that $R \circ E = I$, the identity map on $Y$.
Proposition \ref{gustav} allows to apply standard arguments to establish the following property
(we refer the reader to \cite[Theorem  1.2.4]{t78} for those arguments  and,
in addition, one should notice that the closed graph theorem remains
true in the context of quasi-Banach spaces).

\begin{proposition}\label{complexretract-pm}
Let $(X_0,X_1)$ and $(Y_0,Y_1)$ be two interpolation couples of quasi-Banach spaces
such that $Y_j$ is a retract of $X_j$, $j\in\{0,1\}$.
Then, for each $\Theta \in (0,1)$,
\[
\laz Y_0,Y_1,\Theta\raz = R (\laz X_0, X_1,\Theta\raz)\, .
\]
\end{proposition}

As a consequence of Propositions \ref{morrey2} and \ref{morrey3} and a general result
of Nilsson \cite[Theorem 2.1]{n85} (see Proposition \ref{t-n} below),
we obtain the first main result of this article.

\begin{theorem}\label{COMI}
Let $\tz\in(0,1)$, $s_i\in\rr$, $\tau_i\in[0,\fz)$,
$p_i$, $q_i\in(0,\fz]$ and $u_i\in[p_i,\fz]$, $i\in\{0,1\},$
such that $s=(1-\tz)s_0+\tz s_1$, $\tau=(1-\tz)\tau_0+\tz\tau_1$,
\[
\frac1p=\frac{1-\tz}{p_0}+\frac\tz{p_1}\, , \quad \frac1q=\frac{1-\tz}{q_0}+\frac\tz{q_1}\, \quad
\mbox{and} \quad \frac1u=\frac{1-\tz}{u_0}+\frac{\tz}{u_1}\, .
\]

{\rm (i)} If $\tau_0 \, p_0 = \tau_1\, p_1$, then
\begin{equation*}
\lf\laz A_{p_0,q_0}^{s_0,\tau_0}(\rn), A_{p_1,q_1}^{s_1,\tau_1}(\rn),\tz\r\raz=A_{p,q}^{s,\tau}(\rn),\quad A\in\{B,F\}.
\end{equation*}

{\rm(ii)}
If $p_0\, u_1=p_1\, u_0,$
then
$$
\lf\laz \cn_{u_0,p_0,q_0}^{s_0}(\rn), \cn_{u_1,p_1,q_1}^{s_1}(\rn),\tz\r\raz=
\cn_{u,p,q}^{s}(\rn).$$
\end{theorem}

\begin{remark}\label{2.12x}
Here are several observations on Theorem \ref{COMI}.

(i) We comment on the restriction $\tau_0 \, p_0 = \tau_1\, p_1$.
This required identity has serious consequences.
If either $\tau_0 =0$ or $\tau_1=0$, it immediately follows $\tau= \tau_0=\tau_1 =0$
and we are back in the classical situation of Besov and Triebel-Lizorkin spaces.
If $\max \{p_0,p_1\}< \infty$ and either $\tau_0 =1/p_0$ or $\tau_1=1/p_1$, then
we obtain
\[
\tau_0 \, p_0 = \tau_1\, p_1 = \tau \, p =1\, .
\]
 With $A=F$, Theorem \ref{COMI}(i) reads as
\begin{equation*}
\lf\laz F_{p_0,q_0}^{s_0,1/p_0}(\rn), F_{p_1,q_1}^{s_1,1/p_1}(\rn),\tz\r\raz=
\lf\laz F_{\infty,q_0}^{s_0}(\rn), F_{\infty,q_1}^{s_1}(\rn),\tz\r\raz =
F_{\infty,q}^{s}(\rn);
\end{equation*}
see Proposition \ref{basic1}(ii).
This has been known before, we refer the reader to Frazier and Jawerth \cite[Theorem~8.5]{fj90}.
The counterpart with $A=B$ seems to be a novelty.
Finally, we consider the case that
$\max_{i\in\{0,1\}} \{\tau_i-1/p_i\}>0$. Then the above restriction implies
\[
\min \lf\{\tau-1/p, \tau_0-1/p_0, \tau_1-1/p_1\r\} >0 \, .
\]
Taking into account Proposition \ref{basic1}(ii), Theorem  \ref{COMI} reduces to
\begin{eqnarray*}
\lf\laz B_{p_0,q_0}^{s_0,\tau_0}(\rn), B_{p_1,q_1}^{s_1,\tau_1}(\rn),\tz\r\raz & = &
\lf\laz B_{\infty,\infty}^{s_0+n(\tau_0-1/p_0)}(\rn), B_{\infty,\infty}^{s_1+ n(\tau_1-1/p_1)}(\rn),\tz\r\raz
\\
& = &
B_{\infty,\infty}^{s + n(\tau-1/p)}(\rn) \, .\nonumber
\end{eqnarray*}

(ii) We explain the difference in the restrictions in (i) and (ii)
of Theorem \ref{COMI}.
For simplicity, we concentrate on the $F$-case.
Recall that, when $\tau \in [0,1/p)$, the spaces $\ft$ coincide with the
Triebel-Lizorkin-Morrey spaces $\ce^s_{u,p,q}(\rn)$,
that is,
\[
\ce^s_{u,p,q}(\rn) = F^{s,1/p-1/u}_{p,q}(\rn)\, , \qquad 0 < p\le u < \infty\, ;
\]
see Proposition \ref{basic2} in Appendix.
A reformulation of Theorem \ref{COMI}(i) (with $A=F$) reads as follows:
\[
\lf\laz \ce_{u_0, p_0,q_0}^{s_0}(\rn), \ce_{u_1, p_1,q_1}^{s_1}(\rn),\tz\r\raz= \ce_{u,p,q}^{s}(\rn)
\]
holds true under the restrictions in  Theorem \ref{COMI}(ii). The condition $\tau_0 \, p_0 = \tau_1\, p_1$
is obviously equivalent to $p_0\, u_1=p_1\, u_0$.

(iii) In the case $\tau_0 = \tau_1=0$, the formula
\begin{equation}\label{fraz1}
\lf\laz F_{p_0,q_0}^{s_0}(\rn), F_{p_1,q_1}^{s_1}(\rn),\tz\r\raz =
\lf\laz F_{p_0,q_0}^{s_0,0}(\rn), F_{p_1,q_1}^{s_1,0}(\rn),\tz\r\raz = F_{p,q}^{s,0}(\rn)= F_{p,q}^{s}(\rn)
\end{equation}
was proved  by Frazier and Jawerth \cite[Theorem~8.5]{fj90}.
Also the formula
\[
\lf\laz F_{p_0,q_0}^{s_0,1/p_0}(\rn), F_{p_1,q_1}^{s_1,1/p_1}(\rn),\tz\r\raz = F_{p,q}^{s,1/p}(\rn)
= F_{\infty,q}^{s}(\rn)
\]
as well as
\[
\lf\laz F_{p_0,q_0}^{s_0,1/p_0}(\rn), B_{p_1,\infty}^{s_1,1/p_1}(\rn),\tz\r\raz = F_{p,q}^{s,1/p}(\rn)
= F_{\infty,q}^{s}(\rn)
\]
can be found therein. Notice that the idea used in the proof for \eqref{fraz1}
carries over to the general case.

(iv) As we have seen in (i) of this remark, the restriction $\tau_0 \, p_0 = \tau_1\, p_1$ splits the admissible
parameters into four groups:
\begin{enumerate}
\item[(a)] $\tau_0 = \tau_1 =0$;
\item[(b)] $0<\tau_i<1/p_i$, $i\in\{0,1\}$;
\item[(c)] $ \tau_0-1/p_0 = \tau_1-1/p_1=0$;
\item[(d)] $\max_{i\in\{0,1\}} \{\tau_i-1/p_i\}>0$.
\end{enumerate}
Our methods do not apply to other situations. However, Frazier and Jawerth
\cite[Theorem~8.5]{fj90} proved that, if $p_0,\ p_1\in(0,\infty)$ (and with no further restrictions), then
\[
\lf\laz F_{p_0,q_0}^{s_0,0}(\rn), F_{p_1,q_1}^{s_1,1/p_1}(\rn),\tz\r\raz = F_{p,q}^{s,0}(\rn)
\]
and,
if $p_0,\ p_1 \in(0,\infty)$ and $\tau_1 \in(1/p_1,\fz)$, then
\begin{eqnarray*}
\lf\laz F_{p_0,q_0}^{s_0,0}(\rn), F_{p_1,q_1}^{s_1,\tau_1}(\rn),\tz\r\raz & = &
\lf\laz F_{p_0,q_0}^{s_0,0}(\rn), B_{\infty,\infty}^{s_1 + n(\tau_1-1/p_1)}(\rn),\tz\r\raz
\\
& = & F_{p,q}^{s+n(\tau-1/p) + n(1-\Theta)/p_0,0}(\rn).
\end{eqnarray*}
\end{remark}

\subsection*{The $\pm$-method and Morrey-Campanato spaces}

Recall that $F^{0,1/p-1/u}_{p,2}(\rn) = \cm_p^u(\rn)$ if
$1<p\le u<\fz$ (see Mazzucato \cite{ma01} and  Sawano \cite{sawch}).
Then, partially as  a further corollary of Theorem \ref{COMI},
we have the following conclusion  on the interpolation of  Morrey spaces.

\begin{corollary}\label{cima}
Let $\tz\in(0,1)$, $0<p_0\le u_0<\fz$ and $0<p_1\le u_1<\fz$.
Let  $\frac1u:=\frac{1-\tz}{u_0}+\frac\tz{u_1}$ and $\frac1p:=\frac{1-\tz}{p_0}+\frac\tz{p_1}$.

{\rm (i)} If $ p_0\, u_1=p_1\, u_0$,
then
$$\lf\laz\cm^{u_0}_{p_0}(\rn), \cm^{u_1}_{p_1}(\rn), \tz\r\raz=\cm^u_p(\rn).$$

{\rm (ii)} If $\min\{p_0,p_1\} >1$ and  $p_0\, u_1 \neq p_1\, u_0$,
then
$$\lf\laz\cm^{u_0}_{p_0}(\rn), \cm^{u_1}_{p_1}(\rn), \tz\r\raz \neq \cm^u_p(\rn).$$
\end{corollary}

Corollary \ref{cima}(i) has been known before; see \cite[Theorem 2.3]{lyy}. However, our proof, given in Subsection \ref{proof1},
will differ from that one given in \cite[Theorem 2.3]{lyy}.
Corollary \ref{cima} allows us now to consider also Campanato spaces.

\begin{corollary}\label{cimab}
Let $\tz\in(0,1)$, $p_0,\ p_1\in(0,\fz)$ and $\tau_0,\ \tau_1 \in [0, \fz)$.
If  $\tau:=(1-\tz)\, \tau_0 + \tz\, \tau_1$,
$\frac1p:=\frac{1-\tz}{p_0}+\frac\tz{p_1}$ and $
p_0\, \tau_0=p_1\, \tau_1$,
then
$$\lf \laz \cl^{\tau_0}_{p_0}(\rn), \cl^{\tau_1}_{p_1}(\rn), \tz\r\raz=\cl^\tau_p(\rn)$$
(by using the convention that, in case $\tau = 1/p$, either $p \in[1,\fz)$ and $k \ge 0$ or $p\in(0,1)$ and $k=0$).
\end{corollary}


\subsection{The Peetre-Gagliardo interpolation method}
\label{inter1b}


The following interpolation method $\laz\cdot\, , \, \cdot \raz_\Theta$
was introduced  by Peetre \cite{p71} (based on some earlier work of Gagliardo \cite{Ga}).

\begin{definition}
Let $(X_0,X_1)$ be a  quasi-Banach couple  and $\Theta \in (0,1)$.
An $a\in X_0+X_1$ is said to belong to  $\laz X_0, X_1\raz_\Theta$
if there exists a sequence $\{a_i\}_{i\in\zz}
\subset X_0\cap X_1$ such that $a=\sum_{i\in\zz} \, a_i$ with convergence in $X_0+X_1$
and, for any bounded sequence $\{\varepsilon_i\}_{i\in\zz}\subset\cc$,
$$\sum_{i\in\zz} \varepsilon_i \, 2^{i(j-\Theta)} \, a_i$$
converges in $X_j$, $j\in\{0,1\}$, and satisfies
\begin{equation}\label{2.1y}
\lf\|\sum_{i\in\zz} \varepsilon_i \, 2^{i(j-\Theta)} \, a_i\r\|_{X_j}\le
C \, \sup_{i\in\zz}|\varepsilon_i|
\end{equation}
for some non-negative constant $C$ independent of $j\in\{0,1\}$. The \emph{quasi-norm} of
$a\in \laz X_0, X_1\raz_\Theta$ is defined as
\[
\|a\|_{\laz X_0, X_1\raz_\Theta}:=\inf \lf\{C: \, C\ \ \mbox{satisfies}\ \ \eqref{2.1y}\r\}.
\]
\end{definition}

\begin{proposition}\label{interprop}
Let $(A_0,A_1)$ and  $(B_0,B_1)$ be any two quasi-Banach couples and  $\Theta \in (0,1)$.

{\rm(i)}  It holds true that $\laz A_0, A_1\raz_\Theta$ is   a quasi-Banach space.

{\rm(ii)}
If $T \in \cl (A_0,B_0) \cap \cl(A_1,B_1)$, then $T$ maps  $\laz A_0, A_1\raz_\Theta $
continuously into $\laz B_0, B_1\raz_\Theta $. Furthermore
\[
\| \, T \, \|_{\laz A_0, A_1\raz_\Theta  \to \laz B_0, B_1\raz_\Theta}  \le \max \lf\{
\| \, T \, \|_{A_0 \to B_0} ,\, \|\, T \, \|_{A_1 \to  B_1}\r\}\, .
\]
\end{proposition}

By Proposition \ref{interprop}, we also have a counterpart of Proposition \ref{complexretract-pm} as follows.

\begin{proposition}\label{complexretract-pg}
Let $(X_0,X_1)$ and $(Y_0,Y_1)$ be two interpolation couples of quasi-Banach spaces
such that $Y_j$ is a retract of $X_j$, $j\in\{0,1\}$.
Then, for each $\Theta \in (0,1)$,
\[
\laz Y_0,Y_1 \raz_\Theta = R (\laz X_0, X_1\raz_\Theta)\, .
\]
\end{proposition}

Of course, there is only a minimal difference between
$\laz\cdot\, , \, \cdot \raz_\Theta$ and   $\lf\laz \cdot \, , \, \cdot \, , \tz\r\raz$.
As an immediate conclusion of their definitions, we obtain
$\laz X_0, X_1\raz_\Theta \hookrightarrow \laz X_0, X_1\, , \Theta\raz$.
In some case, one can say more about the relation between
$\laz X_0, X_1\raz_\Theta $ and $\laz X_0, X_1\, , \Theta\raz$.
To this end, we need a few more notation.

A quasi-Banach space $X$ is called an \emph{intermediate space} with respect
to $X_0+X_1$ if
\[
X_0\cap X_1 \hookrightarrow  X\hookrightarrow  X_0+X_1.
\]

\begin{definition}
Let $X_0,\ X_1$ be a couple of quasi-Banach spaces and $X$ an intermediate space
with respect to $X_0+X_1$.
Define
\[
X^\# := (X_0,X_1,X,\#):= {\overline{X_0\cap X_1}}^{\|\, \cdot \, \|_X}\, ,
\]
i.\,e., $X^\#$ is the closure of $X_0\cap X_1$ in $X$.
\end{definition}

Many times  the space ${\overline{X_0\cap X_1}}^{\|\, \cdot \, \|_X}$
 is  denoted by $X^\circ$; see, e.\,g., \cite{n85}.
However, in this article,  the closure of the test functions in $X$,
 denoted by $\mathring{X}$, will play a certain role. By using this different
 notation, we  hope to avoid confusion.
For us of certain interest is the following relation; see Nilsson \cite[(1.5)]{n85} and
also Janson \cite[Theorem 1.8]{jan}.

\begin{proposition}\label{nil}
Let  $(X_0,X_1)$ be a  couple of  quasi-Banach spaces
 and $\Theta \in (0,1)$.
Then
\[
\laz X_0, X_1\raz_\Theta = (X_0, X_1, \laz X_0, X_1\, , \Theta\raz, \#)\, , \quad \tz\in(0,1).
\]
\end{proposition}

Combining Theorem \ref{COMI} and its sequence space version Theorem \ref{comi},
Proposition \ref{nil}, and Propositions \ref{wave1} and \ref{wave2} in Appendix below,
we immediately obtain the next interesting conclusions.

\begin{proposition}\label{C-CIS}
Let all parameters
be as in Theorem \ref{COMI}.

{\rm(i)} If  $\tau_0 \, p_0 = \tau_1 \, p_1$, then
\begin{equation}\label{eq-231}
\lf\laz A_{p_0,q_0}^{s_0,\tau_0}(\rn), A_{p_1,q_1}^{s_1,\tau_1}(\rn)\r\raz_\tz
= (A_{p_0,q_0}^{s_0,\tau_0}(\rn),  A_{p_1,q_1}^{s_1,\tau_1}(\rn), A_{p,q}^{s,\tau}(\rn),\#)\, .
\end{equation}

{\rm(ii)} If $p_0 \, u_1=p_1\, u_0$, then
\begin{equation}\label{eq-232}
\lf\laz \cn_{u_0,p_0,q_0}^{s_0}(\rn), \cn_{u_1,p_1,q_1}^{s_1}(\rn)\r\raz_\tz
= (\cn_{u_0, p_0,q_0}^{s_0}(\rn), \cn_{u_1, p_1,q_1}^{s_1}(\rn), \cn_{u,p,q}^{s}(\rn), \#)\, .
\end{equation}
\end{proposition}

\begin{remark}
The homogeneous counterparts  of Theorems \ref{COMI} and \ref{C-CIS}  also hold true
(more exactly, Theorem \ref{COMI} and Proposition \ref{C-CIS} with the inhomogeneous spaces
$\at$ and $\cn_{u,p,q}^{s}(\rn)$ replaced, respectively,
by their homogeneous counterparts
$\dot{A}_{p,q}^{s, \, \tau}(\rn)$ and $\dot{\cn}_{u,p,q}^{s}(\rn)$ remain valid).
However, to limit the length of this article, we omit the details and refer the reader to
\cite{yyz} for some results in this direction.
\end{remark}

Hence, we are left with the problem to calculate the right-hand sides in (\ref{eq-231}) and
(\ref{eq-232}). But this seems to be a difficult problem.
A first hint in this direction is given by the following observation. Therefore we need new spaces.

\begin{definition}\label{d-cm}
Let $X$ be a quasi-Banach space of distributions or functions.

{\rm (i)} By
$\accentset{\diamond}{X}$ we denote the closure in $X$ of the set of all infinitely
differentiable functions $f$ such that
$D^\alpha f \in X$ for all $\alpha \in (\zz_+)^n$.

{\rm (ii)}
Let $C_c^\infty (\rn) \hookrightarrow X$.
Then we denote by $\mathring{X}$ the closure of $C_c^\infty (\rn)$ in $X$.
\end{definition}

\begin{remark}\label{r2.24}
Recall that, in Section \ref{s1} of this article, we define $\mathring{\cm}^{u}_{p}(\rn) $ as the
{closure} of the Schwartz functions in the Morrey space ${\cm}^{u}_{p}(\rn)$.
This definition coincides with that in Definition \ref{d-cm}(ii) with $X={\cm}^{u}_{p}(\rn)$. Indeed, it is obvious that the closure of Schwartz functions in  ${\cm}^{u}_{p}(\rn)$ contains the closure of $C_c^\fz(\rn)$, while the inverse inclusion
follows from the well-known fact that a Schwartz function can be approximated by
$C_c^\infty(\rn)$ functions in Schwartz norms and hence in Morrey spaces (due to
the continuous embedding from the Schwartz space into Morrey spaces).
\end{remark}

Clearly, $\mathring A_{p,q}^{s,\tau}(\rn) \hookrightarrow \accentset{\diamond}{A}^{s,\tau}_{p,q} (\rn)$ and
$\mathring \cn_{u,p,q}^{s}(\rn) \hookrightarrow \accentset{\diamond}{\cn}_{u,p,q}^{s}(\rn)$.
First we clarify the relation of these two scales to each other.

\begin{lemma}\label{diamond1}
Let  $A \in \{B,F\}$ and $s\in\rr$. Then

{\rm (i)}  $\mathring A_{p,q}^{s,\tau}(\rn) = \accentset{\diamond}{A}^{s,\tau}_{p,q} (\rn)$
if and only if $\tau =0$, $p\in(0, \infty)$ and $q \in(0,\infty]$.

{\rm (ii)} $\mathring \cn_{u,p,q}^{s}(\rn) = \accentset{\diamond}{\cn}^{s}_{u,p,q} (\rn)$
if and only if $u=p\in(0,\infty)$ and $q\in(0,\infty]$.
\end{lemma}

Now we are interested in the relation of
$\accentset{\diamond}{A}^{s,\tau}_{p,q} (\rn)$ to ${A}^{s,\tau}_{p,q} (\rn)$.

\begin{lemma}\label{diamond2}
Let  $A \in \{B,F\}$ and $s\in\rr$. Then

 {\rm (i)}
\[
\accentset{\diamond}{A}^{s,\tau}_{p,q} (\rn) =  {A}^{s,\tau}_{p,q} (\rn)\qquad \mbox{if and only if}\qquad  \tau =0\ \ \mbox{and}\ \  q\in(0,\infty).
\]

{\rm (ii)}
\[
\accentset{\diamond}{\cn}^{s}_{u,p,q} (\rn) =  {\cn}^{s}_{u,p,q} (\rn) \qquad \mbox{if and only if}
\qquad q\in(0,\infty).
\]
\end{lemma}

Lemma \ref{diamond2} shows the essential difference between the scales ${A}^{s,\tau}_{p,q} (\rn)$ and
${\cn}^{s}_{u,p,q} (\rn)$.
Under the conditions of Theorem \ref{COMI}, we know that
\begin{equation}\label{ws-69}
\lf\laz A_{p_0,q_0}^{s_0,\tau_0}(\rn), A_{p_1,q_1}^{s_1,\tau_1}(\rn)\r\raz_\tz
\hookrightarrow \lf\laz A_{p_0,q_0}^{s_0,\tau_0}(\rn), A_{p_1,q_1}^{s_1,\tau_1}(\rn), \tz \r\raz = A_{p,q}^{s,\tau}(\rn)\, .
\end{equation}
But this can be easily improved.

\begin{lemma}\label{densel}
Let $s_0 \neq s_1$.

{\rm (i)} Under the conditions of Theorem \ref{COMI}(i), it holds true that
\begin{equation*}
\mathring A_{p,q}^{s,\tau}(\rn) \hookrightarrow
\lf\laz A_{p_0,q_0}^{s_0,\tau_0}(\rn), A_{p_1,q_1}^{s_1,\tau_1}(\rn)\r\raz_\tz
\hookrightarrow \accentset{\diamond}{B}^{s + n\tau-n/p}_{\infty,\infty} (\rn)\cap A_{p,q}^{s,\tau}(\rn) \, .
\end{equation*}

{\rm (ii)} Under the conditions of Theorem \ref{COMI}(ii), it holds true that
\begin{equation*}
\mathring \cn_{u,p,q}^{s}(\rn) \hookrightarrow
\lf\laz \cn_{u_0, p_0,q_0}^{s_0}(\rn),  \cn_{u_1, p_1,q_1}^{s_1}(\rn)\r\raz_\tz
\hookrightarrow \accentset{\diamond}{B}^{s + n\tau-n/p}_{\infty,\infty} (\rn) \cap \cn_{u, p,q}^{s}(\rn) \, .
\end{equation*}

{\rm (iii)} Let $\tau\in(0,\frac1p)$. Then, under the conditions of Theorem \ref{COMI}(i),
it holds true  that
$$\lf\laz A_{p_0,q_0}^{s_0,\tau_0}(\rn), A_{p_1,q_1}^{s_1,\tau_1}(\rn)\r\raz_\tz \subsetneqq
A_{p,q}^{s,\tau}(\rn).$$
\end{lemma}

Lemma \ref{densel} yields the following problem: under which conditions, we have
\begin{equation}\label{ws-76}
\lf\laz A_{p_0,q_0}^{s_0,\tau_0}(\rn), A_{p_1,q_1}^{s_1,\tau_1}(\rn)\r\raz_\tz =
\accentset{\diamond}{A}^{s,\tau}_{p,q} (\rn)
\end{equation}
and
\begin{equation}\label{ws-77}
\lf\laz \cn_{u_0, p_0,q_0}^{s_0}(\rn),  \cn_{u_1, p_1,q_1}^{s_1}(\rn)\r\raz_\tz =
\accentset{\diamond}{\cn}^{s}_{u,p,q} (\rn) \, ,
\end{equation}
respectively? The answer is not always yes, but sometimes.
The situation is better understood in case $\tau_0 =\tau_1 = 0$.

\begin{theorem}\label{gagl1}
Let $\tz\in(0,1)$, $s_i\in\rr$, $\tau_i\in[0,\fz)$,
$p_i, q_i\in(0,\fz]$ and $u_i\in[p_i,\fz]$, $i\in\{0,1\},$
such that $s=(1-\tz)s_0+\tz s_1$, $\tau=(1-\tz)\tau_0+\tz\tau_1$,
\[
\frac1p=\frac{1-\tz}{p_0}+\frac\tz{p_1}\, , \quad \frac1q=\frac{1-\tz}{q_0}+\frac\tz{q_1}\, \quad
\mbox{and} \quad \frac1u=\frac{1-\tz}{u_0}+\frac{\tz}{u_1}\, .
\]
Let $A \in \{B,F\}$.

{\rm (i)} If $\min \{p_0+ q_0, p_1+q_1\}<\infty$, then
\begin{equation*}
\lf\laz A_{p_0,q_0}^{s_0}(\rn), A_{p_1,q_1}^{s_1}(\rn)\r\raz_\tz
=  \mathring A_{p,q}^{s}(\rn) = \accentset{\diamond}{A}^{s}_{p,q} (\rn) = {A}^{s}_{p,q} (\rn)\, .
\end{equation*}

{\rm (ii)} If either  $s_0 \neq s_1$ or  $s_0 = s_1$ and $q_0 \neq q_1$, then
\begin{equation*}
\lf\laz B_{\infty,q_0}^{s_0}(\rn), B_{\infty,q_1}^{s_1}(\rn)\r\raz_\tz
= \accentset{\diamond}{B}^{s}_{\infty,q} (\rn) \, .
\end{equation*}

{\rm (iii)} If $p_0 = p_1 = p<\infty$, $s_0 \neq s_1$ and $q_0 = q_1 = \infty$, then
\begin{equation*}
\lf\laz A_{p,\infty}^{s_0}(\rn), A_{p,\infty}^{s_1}(\rn)\r\raz_\tz
 = \accentset{\diamond}{A}^{s}_{p,\infty} (\rn)  \subsetneqq
{A}^{s}_{p,\infty} (\rn) \, .
\end{equation*}

{\rm (iv)} Let $0 < p_0 <  p_1 \le \infty$ and $q_0 = q_1 = \infty$.
If $s_0-n/p_0 > s_1-n/p_1$, then
\begin{equation*}
\lf\laz B_{p_0,\infty}^{s_0}(\rn), B_{p_1,\infty}^{s_1}(\rn)\r\raz_\tz
 = \mathring{B}^{s}_{p,\infty} (\rn) = \accentset{\diamond}{B}^{s}_{p,\infty}(\rn) \subsetneqq {B}^{s}_{p,\infty}(\rn)\, .
\end{equation*}

{\rm (v)} Let $0 < p_0 <  p_1 < \infty$ and $q_0 = q_1 = \infty$.
If $s_0-n/p_0 \ge s_1-n/p_1$, then
\begin{equation*}
\lf\laz F_{p_0,\infty}^{s_0}(\rn), F_{p_1,\infty}^{s_1}(\rn)\r\raz_\tz
 = \mathring{F}^{s}_{p,\infty} (\rn) = \accentset{\diamond}{F}^{s}_{p,\infty}(\rn)\subsetneqq {F}^{s}_{p,\infty}(\rn) \, .
\end{equation*}

{\rm (vi)} Let $0 < p_0 <  p_1 \le \infty$ and $q_0 = q_1 = \infty$.
If $s_0-n/p_0 \le  s_1-n/p_1$, then
\begin{equation*}
\accentset{\diamond}{B}^{s}_{p,\infty}(\rn)  =
\mathring{B}^{s}_{p,\infty} (\rn) \hookrightarrow
\lf\laz B_{p_0,\infty}^{s_0}(\rn), B_{p_1,\infty}^{s_1}(\rn)\r\raz_\tz
\subsetneqq  {B}^{s}_{p,\infty}(\rn) \, .
\end{equation*}

{\rm (vii)} Let $0 < p_0 <  p_1 = \infty$, $q_0 =  \infty$ and $q_1\in(0, \infty)$.
If $s_0-n/p_0 >  s_1$, then
\begin{equation*}
\lf\laz B_{p_0,\infty}^{s_0}(\rn), B_{\infty,q_1}^{s_1}(\rn)\r\raz_\tz
 = \mathring{B}^s_{p,q} (\rn) =  \accentset{\diamond}{B}^{s}_{p,q}(\rn)= {B}^{s}_{p,q}(\rn)  \, .
\end{equation*}
\end{theorem}

The cases (ii) and (iv) through (vi) of Theorem \ref{gagl1} are representing examples for
$$\laz X_0,X_1\raz_\tz \neq \laz X_0,X_1,\tz\raz;$$
see Theorem \ref{COMI}.

We supplement Theorem \ref{gagl1} by two results for $\tau\in(0,\fz)$.
In the first case, we fix $u$ and $p$.

\begin{theorem}\label{gagl7}
Let  $0< p\le u \le \infty $, $\tau \in [0,\infty)$, $s_0,\ s_1 \in \rr$ and  $q_0,\ q_1 \in (0,\infty]$.
Let  $s:=(1-\tz)s_0+\tz s_1$ and
$\frac1q:=\frac{1-\tz}{q_0}+\frac\tz{q_1}$.
Assume either  $s_0 \neq s_1$  or $s_0 = s_1$ and $q_0 \neq q_1$. Then
\[
\lf\laz A_{p,q_0}^{s_0,\tau}(\rn), A_{p,q_1}^{s_1,\tau}(\rn)\r\raz_\tz =
\accentset{\diamond}{A}^{s,\tau}_{p,q} (\rn),\quad A \in \{B,F\}.
\]
\end{theorem}

In the second case, we consider large $\tau$.

\begin{theorem}\label{gagl2}
Suppose
 either $\tau_i \in(1/p_i,\fz)$ and  $q_i \in (0,\infty]$ or
$\tau_i = 1/p_i>0$ and $q_i= \infty$, $i\in\{0,1\}$.
If  $s_0 \neq s_1$, then
\[
\lf\laz A_{p_0,q_0}^{s_0,\tau_0}(\rn), \ca_{p_1,q_1}^{s_1,\tau_1}(\rn)\r\raz_\tz = \accentset{\diamond}{A}^{s,\tau}_{p,q} (\rn)
\]
for any pair $A, \ca \in \{B,F\}$.
\end{theorem}

It is still an open problem whether Theorems \ref{gagl7}
 and \ref{gagl2} can be extended to a greater range of parameters or not.
Definitely it is not true for all parameter constellations; see \eqref{ws-49} with $q=2$ and compare with
Lemma \ref{diamond2}(i).

\begin{remark}
 Again we have to mention the fundamental article of Frazier and Jawerth \cite{fj90}
 for a number of further results.
There the following formulas are proved:
\[
\lf\laz F_{p_0,q_0}^{s_0}(\rn), F_{p_1,q_1}^{s_1}(\rn)\r\raz_\tz = F_{p,q}^{s}(\rn)\, ,
\]
\begin{equation}\label{ws-49}
\lf\laz F_{p_0,q_0}^{s_0,1/p_0}(\rn), F_{p_1,q_1}^{s_1,1/p_1}(\rn)\r\raz_\tz = F_{p,q}^{s,1/p}(\rn)
= F_{\infty,q}^{s}(\rn),
\end{equation}
\[
\lf\laz F_{p_0,q_0}^{s_0,1/p_0}(\rn), B_{\infty,\infty}^{s_1,0}(\rn)\r\raz_\tz = F_{p,q}^{s,1/p}(\rn)
= F_{\infty,q}^{s}(\rn)
\]
as well as
\[
\lf\laz F_{p_0,q_0}^{s_0,0}(\rn), F_{p_1,q_1}^{s_1,1/p_1}(\rn)\r\raz_\tz = F_{p,q}^{s,0}(\rn)
\]
and
\begin{eqnarray*}
\lf\laz F_{p_0,q_0}^{s_0,0}(\rn), F_{p_1,q_1}^{s_1,\tau_1}(\rn)\r\raz_\tz & = &
\lf\laz F_{p_0,q_0}^{s_0,0}(\rn), B_{\infty,\infty}^{s_1 + n(\tau_1-1/p_1)}(\rn)\r\raz_\tz
\\
& = & F_{p,q}^{s+n(\tau-1/p) + n(1-\Theta)/p_0,0}(\rn)
\end{eqnarray*}
if $p_0,\ p_1\in(0,\fz)$ and $\tau_1\in(1/p_1,\fz)$ (and with no further restrictions);
see \cite[Corollary~8.4]{fj90}.
\end{remark}

There is one more case where we can calculate the associated interpolation space.

\begin{theorem}\label{gagl4}
Suppose  $0 < p \le u <\infty$, $s_0,s_1 \in \rr$ and   $q_i \in (0,\infty)$, $i\in\{0,1\}$.
Let $s:= (1-\Theta)\, s_0 + \Theta\, s_1$ and $\frac 1q : =\frac{1-\tz}{q_0}+ \frac \tz{q_1}$.
Then
\[
\lf\laz \cn_{u,p,q_0}^{s_0}(\rn), \cn_{u,p,q_1}^{s_1}(\rn)\r\raz_\tz = \accentset{\diamond}{\cn}^{s}_{u,p,q} (\rn)
\]
follows for all $\tz \in (0,1)$.
\end{theorem}

Up to now,  Theorem \ref{gagl7} (resp. Theorem \ref{gagl4}) is the only answer we have to the question in
\eqref{ws-76} (resp. \eqref{ws-77}) in case $0 < \tau < 1/p$ (resp. $0 < p< u< \infty$).
Let us now have a closer look onto this problem  for $p_0 < p_1$ in the most simple situations of
Morrey spaces itself.


\subsection*{The Peetre-Gagliardo method and Morrey-Campanato spaces}


Now we are interested in $\lf\laz\cm^{u_0}_{p_0}(\rn), \cm^{u_1}_{p_1}(\rn)\r\raz_\tz $
and its relation to $ \mathring{\cm}^u_p(\rn)$, $\accentset{\diamond}{\cm}^u_p(\rn)$ and
${\cm}^u_p(\rn)$.
There is one more space of certain interest in the framework of Morrey spaces.
We define the \emph{space $\accentset{*}{\cm}^u_p(\rn)$} as
the closure in $\cm^u_p(\rn) $ of the set
of compactly supported functions.
The next lemma gives explicit descriptions
of  $\mathring{\cm}^{u}_{p}(\rn)$, $\accentset{*}{\cm}^{u}_{p}(\rn)$ and
$\accentset{\diamond}{\cm}^{u}_{p}(\rn)$ very much in the spirit of the original definition of  Morrey spaces.
This is of interest for its own.

\begin{lemma}\label{morrey43}
Let $1 \le p < u < \infty$. Then

{\rm (i)}  $\mathring{\cm}^{u}_{p}(\rn)$ is equal to the collection of all $f \in {\cm}^{u}_{p}(\rn)$ having the
following three properties:
\begin{eqnarray}
&& \qquad \lim_{r \downarrow 0}
|B(y,r)|^{1/u-1/p}\lf[\int_{B(y,r)} |f(x)|^p\,dx\r]^{1/p}   =  0 \, ,
\label{ws-51}
\\
&&
\nonumber
\\
\label{ws-50}
&& \qquad
\lim_{r \to \infty}
|B(y,r)|^{1/u-1/p}\lf[\int_{B(y,r)} |f(x)|^p\,dx\r]^{1/p}  =  0\, ,
\end{eqnarray}
both uniformly in  $y \in \rr^n$, and
\begin{equation}
 \label{ws-74}
 \lim_{|y|\to \infty}
|B(y,r)|^{1/u-1/p}\lf[\int_{B(y,r)} |f(x)|^p\,dx\r]^{1/p}  =  0
\end{equation}
uniformly in $r\in(0,\fz)$.

{\rm (ii)} $\accentset{*}{\cm}^{u}_{p}(\rn)$ is equal to the collection of all $f \in {\cm}^{u}_{p}(\rn)$
such that \eqref{ws-50} (uniformly in  $y \in \rr^n$) and \eqref{ws-74} (uniformly in  $r\in(0,\fz)$)  hold true.

{\rm (iii)} $\accentset{\diamond}{\cm}^{u}_{p}(\rn)$ is equal to the collection of all $f \in {\cm}^{u}_{p}(\rn)$
such that \eqref{ws-51} holds true
uniformly in  $y \in \rr^n$.
\end{lemma}

\begin{remark}
The restriction $p \ge 1$ of Lemma \ref{morrey43} comes into play with the construction of smooth approximations.
In addition, this condition is not needed for the proofs that
 $\mathring{\cm}^{u}_{p}(\rn)$, $\accentset{*}{\cm}^{u}_{p}(\rn)$ and
$\accentset{\diamond}{\cm}^{u}_{p}(\rn)$  have the claimed properties.
\end{remark}

There are simple, but important examples of functions explaining the difference between these spaces.
Let $\psi$ be a function as in \eqref{eq-05} in Appendix.
For $\alpha \in(0,\fz)$ and $x\in\rn\setminus\{0\}$, let
\begin{eqnarray}
 f_\alpha (x) & := & |x|^{-\alpha}\, ,
 \label{falpha}
 \\
  g_\alpha (x) & := & \psi (x)\, |x|^{-\alpha}\, ,
  \label{galpha}
  \\
   h_\alpha (x) & := & (1-\psi (x))\, |x|^{-\alpha}\, .
   \label{halpha}
\end{eqnarray}
Elementary calculations yield that
\[
f_\alpha \in \cm_p^u (\rn) \qquad \Longleftrightarrow \qquad \alpha = \frac nu \qquad\mbox{and} \qquad  \alpha < \frac np\, .
\]
Similarly
\[
 g_\alpha \in \cm_p^u (\rn) \qquad \Longleftrightarrow \qquad \alpha \le  \frac nu \qquad\mbox{and} \qquad \alpha < \frac np\, ,
\]
and
\[
 h_\alpha \in \cm_p^u (\rn) \qquad \Longleftrightarrow \qquad   \frac nu \le  \alpha \, .
\]
In the limiting situation, we find that there exists a positive constant $C_{(p,u)}$, depending on
$p$ and $u$, such that, for all $r\in(0,\fz)$,
\[
|B(0,r)|^{1/u-1/p}\lf[\int_{B(0,r)} |x|^{-np/u}\,dx\r]^{1/p} = C_{(p,u)}.
\]
Hence $f_{n/u} \not\in \accentset{\diamond}{\cm}^{u}_{p}(\rn)$ and
$f_{n/u} \not\in \accentset{*}{\cm}^{u}_{p}(\rn)$.

\begin{lemma}\label{morrey410}
Let $0 < p < u < \infty$.

{\rm (i)}  $\accentset{*}{\cm}^{u}_{p}(\rn)$ and
$\accentset{\diamond}{\cm}^{u}_{p}(\rn)$
are proper subspaces of  ${\cm}^{u}_{p}(\rn)$.

{\rm (ii)} It holds true that
\[
\mathring{\cm}^{u}_{p}(\rn) \hookrightarrow \accentset{*}{\cm}^{u}_{p}(\rn) \cap
\accentset{\diamond}{\cm}^{u}_{p}(\rn)\, .
\]

{\rm (iii)} It holds true  that
\[
\accentset{*}{\cm}^{u}_{p}(\rn) \not\subset
\accentset{\diamond}{\cm}^{u}_{p}(\rn)
\qquad\mbox{and}\qquad
\accentset{\diamond}{\cm}^{u}_{p}(\rn)\not \subset
\accentset{*}{\cm}^{u}_{p}(\rn) \, .
\]
\end{lemma}

\begin{remark}\label{Sawano}
Let $0<p\le u<\fz$. Sawano and Tanaka in \cite{st15}
considered another subspace of Morrey spaces, $S\accentset{*}{\cm}^{u}_{p}(\rn)$. This space $S\accentset{*}{\cm}^{u}_{p}(\rn)$ is defined to be
the collection of
all $f \in \cm_p^u (\rn)$ which can be approximated by finite sums
of multiples of characteristic functions of  sets with finite Lebesgue measures
in $\rn$. It was proved in \cite{st15}
that $S\accentset{*}{\cm}^{u}_{p}(\rn)$ is a proper subspace of ${\cm}^{u}_{p}(\rn)$
whenever $1<p<u<\fz$.

Obviously, when $p=u\in(0,\fz)$,  $\accentset{*}{\cm}^{u}_{p}(\rn)$ coincides with $S\accentset{*}{\cm}^{u}_{p}(\rn)$.
For the case that $p<u$, we have the following embeddings:
\begin{equation}\label{emb1}
\lf(\accentset{*}{\cm}^{u}_{p}(\rn)\cap L^\infty(\rn)\r)
\subset \lf(S\accentset{*}{\cm}^{u}_{p}(\rn)\cap L^\infty(\rn)\r),\quad 0<p<u<\fz
\end{equation}
and
\begin{equation}\label{emb2}
S\accentset{*}{\cm}^{u}_{p}(\rn)\subset\accentset{*}{\cm}^{u}_{p}(\rn),\quad 1\le p< u<\fz.
\end{equation}

To see \eqref{emb1}, let $f\in L^\infty(\rn)$ be a compactly supported function in ${\cm}^{u}_{p}(\rn)$ with $\supp f\subset K$,
where $K$ is a compact set.
Then it is well known that there exists a sequence $\{g\}_{k\in\nn}$
of simple functions (i.\,e., a complex function whose range consists of only finitely
many points) such that
$\|f-g_k\|_{L^\infty(\rn)}\to 0$ as $k\to \fz$. Define $f_k:=g_k\chi_K$ for all $k\in \nn$. Then each $f_k$ is a finite
sum of multiples of characteristic functions of  sets with finite Lebesgue measures.
Since $\supp f\subset K$, we further see that
$$\|f-f_k\|_{L^\fz(\rn)}=\|(f-g_k)\chi_K\|_{L^\fz(\rn)}\le \|f-g_k\|_{L^\infty(\rn)}\to 0$$
as $k\to\fz$, which, together with $0<p<u<\fz$, implies that
\begin{eqnarray*}
\|f-f_k\|_{\cm^u_p(\rn)}&&\le \|f-g_k\|_{L^\fz(\rn)} \|\chi_K\|_{\cm^u_p(\rn)}\\
&&=\|f-g_k\|_{L^\fz(\rn)}\sup_{{\rm balls}\ B} |B|^{\frac1u}\lf(\frac{|B\cap K|}{|B|}\r)^{\frac1p}\\
&&\ls \|f-g_k\|_{L^\fz(\rn)}|K|^{\frac1u} \to 0
\end{eqnarray*}
as $k\to\fz$. This proves the above embedding \eqref{emb1}.

To show \eqref{emb2}, by Lemma
\ref{morrey43}(ii), it suffices to show that, if $f=\chi_E$ with $|E|<\fz$,
then $f$ satisfies \eqref{ws-50} and \eqref{ws-74}. Indeed,
notice that, in this case,
\begin{equation}\label{XXX}
|B(y,r)|^{1/u-1/p}\lf[\int_{B(y,r)} |f(x)|^p\,dx\r]^{1/p}=|B(y,r)|^{1/u}\frac{|E\cap B(y,r)|^{1/p}}{|B(y,r)|^{1/p}}.
\end{equation}
Thus, by $p<u$ and $|E|<\fz$, we see that
$$\lim_{r\to\fz}|B(y,r)|^{1/u-1/p}\lf[\int_{B(y,r)} |f(x)|^p\,dx\r]^{1/p}\le w_n^{1/u-1/p}|E|^{1/p}\lim_{r\to\fz} r^{n(1/u-1/p)}=0,$$
where $w_n$ denotes the volume of the unit sphere.
This shows that $f$ satisfies \eqref{ws-50}.
Next we show that $f$ satisfies \eqref{ws-74}. By the above proved
conclusion, we know that, for any $\varepsilon\in(0,1)$,
there exists $R_\varepsilon\in(0,\fz)$ such that, if $r>R_\varepsilon$,
then, for all $y\in \rn$,
$$|B(y,r)|^{1/u-1/p}\lf[\int_{B(y,r)} |f(x)|^p\,dx\r]^{1/p}<\varepsilon.$$
On the other hand, by \eqref{XXX}, there exists $r_\varepsilon:=w_n^{-1/u}\varepsilon^{u/n}>0$
such that, if  $r<r_\varepsilon$, then, for all $y\in \rn$,
$$|B(y,r)|^{1/u-1/p}\lf[\int_{B(y,r)} |f(x)|^p\,dx\r]^{1/p}\le |B(y,r)|^{1/u}<\varepsilon.$$
It remains to consider the case $r_\varepsilon\le r\le R_\varepsilon$.
Since $f\in L^p(\rn)$, it follows that there exists $L_\varepsilon\in(0,\fz)$ such that
$$\int_{\rn\setminus B(0,L_\varepsilon)}|f(x)|^p\,dx< w_n^{1-p/u}r_\varepsilon^{n(1-u/p)}\varepsilon^p.$$
Thus, if $|y|>L_\varepsilon+R_\varepsilon$, we then have
$$B(y,r)\subset B(y,R_\varepsilon)\subset \rn\setminus B(0,L_\varepsilon)$$
and hence $$|B(y,r)|^{\frac1u-\frac1p}\lf[\int_{B(y,r)} |f(x)|^p\,dx\r]^{\frac1p}\le w_n^{\frac1u-\frac1p}
r_\varepsilon^{n(\frac1u-\frac1p)}\lf[\int_{\rn\setminus B(0,L_\varepsilon)} |f(x)|^p\,dx\r]^{\frac1p}<\varepsilon.$$
Combining the above estimates, we then know that $f$ satisfies \eqref{ws-74}.
This shows that, if $1\le p<u<\fz$, then $S\accentset{*}{\cm}^{u}_{p}(\rn)\subset\accentset{*}{\cm}^{u}_{p}(\rn)$
and hence proves \eqref{emb2}.
\end{remark}

In view of Proposition \ref{C-CIS}, we need to study intersections of Morrey spaces.

\begin{lemma}\label{morrey41}
Let $0< p_0 \le u_0 < \infty$ and $0< p_1 \le   u_1 < \infty$ such that $p_0 \le  p_1$.
Let $\tz \in (0,1)$,
\[
\frac1p:=\frac{1-\tz}{p_0}+\frac\tz{p_1} \qquad  \mbox{and}
\qquad
\frac1u:=\frac{1-\tz}{u_0}+\frac\tz{u_1}\, .
\]
Assume that $u_1 >u$ and $p< u$.
Then $\cm^{u_0}_{p_0}(\rn) \cap \cm^{u_1}_{p_1}(\rn)$ is not dense in $\cm^{u}_{p}(\rn)$.
If, in addition, $p\in [1,\fz)$, then
$$\cm^{u_0}_{p_0}(\rn) \cap \cm^{u_1}_{p_1}(\rn) \hookrightarrow
\accentset{\diamond}{\cm}^{u}_{p}(\rn).$$
\end{lemma}

\begin{corollary}\label{morrey42}
Let $0< p_i \le u_i < \infty$, $i\in\{0,1\}$, and  $\tz \in (0,1)$.
Let
\[
\frac1p:=\frac{1-\tz}{p_0}+\frac\tz{p_1} \qquad  \mbox{and}
\qquad
\frac1u:=\frac{1-\tz}{u_0}+\frac\tz{u_1}\, .
\]

{\rm (i)} It holds true that
\[
\lf\laz\cm^{u_0}_{p_0}(\rn), \cm^{u_1}_{p_1}(\rn)\r\raz_\tz \hookrightarrow  \cm^{u}_{p}(\rn)
\]
and the embedding is always proper except the trivial cases consisting in
\begin{enumerate}
\item[{\rm(a)}] $p_0 = p_1$ and $u_0 = u_1$, or
\item[{\rm(b)}] $p_0 = u_0$ and $p_1 = u_1$.
\end{enumerate}

{\rm (ii)}  In addition, assume $p_0 < p_1$ and $p_0\,  u_1 = p_1 \, u_0$.
Then
$\mathring{\cm}^{u}_{p}(\rn)$ is a proper subspace of
$$\lf\laz\cm^{u_0}_{p_0}(\rn), \cm^{u_1}_{p_1}(\rn)\r\raz_\tz.$$

{\rm (iii)} Assume $p_0 < p_1$, $p_0\,  u_1 = p_1 \, u_0$ and  $p \in[1,\fz)$. Then
$$\lf\laz\cm^{u_0}_{p_0}(\rn), \cm^{u_1}_{p_1}(\rn)\r\raz_\tz
\hookrightarrow \accentset{\diamond}{\cm}^{u}_{p}(\rn)$$
and this embedding is proper.
\end{corollary}

By Corollary \ref{morrey42}(iii), we know that we have to introduce new spaces
to obtain an explicit description of $\lf\laz\cm^{u_0}_{p_0}(\rn), \cm^{u_1}_{p_1}(\rn)\r\raz_\tz$.

\begin{definition}\label{d2.38}
 Let $0 <p_i < u_i< \infty$, $i\in\{0,1\}$, $p_0 \le p_1$ and $\tz \in (0,1)$.
Define
\[
 \frac1p:=\frac{1-\tz}{p_0}+\frac\tz{p_1} \qquad \mbox{and}\qquad    \frac1u:=\frac{1-\tz}{u_0}+\frac\tz{u_1}\, .
\]
Then the \emph{space}
$\cm^{u_0,u_1,\tz}_{p_0,p_1} (\rn)$ is
defined as the collection of all $f \in L_{p_1}^{\ell oc} (\rn) $
such that
\begin{equation}\label{ws-95a}
I_1 (f):= \sup_{y\in \rn} \, \sup_{0< r < 1}\,
|B(y,r)|^{1/u-1/p}\lf[\int_{B(y,r)} |f(x)|^p\,dx\r]^{1/p}  < \infty\, ,
\end{equation}
\begin{equation}\label{ws-95b}
\lim_{r \downarrow 0}
|B(y,r)|^{1/u-1/p}\lf[\int_{B(y,r)} |f(x)|^p\,dx\r]^{1/p}  =  0
\end{equation}
uniformly in  $y \in \rr^n$,
\begin{equation}\label{ws-96}
I_2 (f):= \sup_{y\in \rn} \sup_{r \ge 1}\,
|B(y,r)|^{\frac{1}{u_0}-\frac{1}{p_0}}\lf[\int_{B(y,r)} |f(x)|^{p_0}\,dx\r]^{1/p_0}  <\infty
\end{equation}
and
\begin{equation}\label{ws-97}
I_3 (f):= \sup_{y\in \rn} \sup_{r \ge 1}\,
|B(y,r)|^{\frac{1}{u_1}-\frac{1}{p_1}}\lf[\int_{B(y,r)} |f(x)|^{p_1}\,dx\r]^{1/p_1}  <\infty\, .
\end{equation}
Define
\[
\| \, f \, \|_{\cm^{u_0,u_1,\tz}_{p_0,p_1} (\rn)} := I_1 (f) + I_2(f) + I_3(f)\, .
\]
\end{definition}

By means of these new spaces $\cm^{u_0,u_1,\tz}_{p_0,p_1}(\rn)$, we obtain now the first main result
with respect to the Peetre-Gagliardo method applied to Morrey spaces.

\begin{theorem}\label{gp01}
Let $\tz\in(0,1)$, $0<p_i < u_i<\fz$, $i\in\{0,1\}$,  $1 \le p_0 < p_1$
and $ p_0\, u_1=p_1\, u_0$. Define
\[
 \frac1p:=\frac{1-\tz}{p_0}+\frac\tz{p_1} \qquad \mbox{and}\qquad    \frac1u:=\frac{1-\tz}{u_0}+\frac\tz{u_1}\, .
\]
Then
$$
\lf\laz\cm^{u_0}_{p_0}(\rn), \cm^{u_1}_{p_1}(\rn)\r\raz_\tz =
\cm^{u_0,u_1,\tz}_{p_0,p_1} (\rn)\, .
$$
\end{theorem}

\begin{remark}
(i)  Notice that, by Corollary \ref{morrey42}, we know that
\[
\mathring{\cm}_p^u (\rn) \hookrightarrow \cm^{u_0,u_1,\tz}_{p_0,p_1} (\rn)\hookrightarrow
 \accentset{\diamond}{\cm}^{u}_{p}(\rn)
\]
and all embeddings are proper.

(ii) Based on Theorem \ref{gp01}, we expect
that the description of $\lf\laz A_{p_0,q_0}^{s_0,\tau_0}(\rn), A_{p_1,q_1}^{s_1,\tau_1}(\rn)\r\raz_\tz$
in case $\tau _0\in(0,1/p_0)$ and $\tau _1\in(0, 1/p_1)$
requires some new spaces. Observe that
\[
\lf\laz F_{p_0,2}^{0,\frac{1}{p_0} - \tau_0}(\rn), F_{p_1,2}^{{0,\frac{1}{p_1} - \tau_1}}(\rn)\r\raz_\tz
= \lf\laz\cm^{u_0}_{p_0}(\rn), \cm^{u_1}_{p_1}(\rn)\r\raz_\tz =
\cm^{u_0,u_1,\tz}_{p_0,p_1} (\rn)\,
\]
with $p_0$, $p_1\in(1,\infty)$ and $\tau_i=1/u_i$, $i\in\{0,1\}$,
by means of the  Littlewood-Paley characterization of the Morrey spaces, namely,
$F_{p,2}^{0,\frac{1}{p} - \frac1u}(\rn)=\cm^u_p(\rn)$
for all $1<p\le u<\fz$ (see, for example, \cite[Corollary 3.3]{ysy}).
\end{remark}

There is an important difference between the interpolation of Morrey spaces on
unbounded domains and bounded ones. First, we need to recall the definition of
Morrey spaces on domains.

\begin{definition}
Let $0 < p \le u < \infty$ and let $\Omega \subset \rn$ be bounded.
 Then the \emph{Morrey space $\cm^{u}_p (\Omega)$} is defined as the set of all
 $f \in L_p^{\ell oc} (\Omega)$ such that
\[
\| \, f \, \|_{\cm^{u}_p (\Omega)}:=
\sup_{x \in \Omega} \sup_{r\in(0,\fz)} \, |B(x,r) \cap \Omega|^{\frac 1u - \frac 1p} \lf[
\int_{B(x,r)\cap \Omega} |f(y)|^p\, dy\r]^{1/p} <\infty\, .
\]
\end{definition}

For easier reference, we will concentrate on $\Omega = (0,1)^n$.
We need to recall an embedding result of Dchumakaeva \cite{dch}:
\[
 W^m (\cm^u_p) ((0,1)^n)   \hookrightarrow  C ((0,1)^n)\, , \qquad 1 \le p< \infty, \quad m > \frac np\, .
\]
Here $W^m (\cm^u_p) ((0,1)^n) $ denotes the Sobolev space built on the Morrey space
$\cm_p^u ((0,1)^n) $ and, as usual, $C ((0,1)^n)$ denotes the space of all
continuous functions
on $(0,1)^n$ equipped with the supremum norm.
By means of this embedding, we can derive the following
conclusion,
the details being omitted.

\begin{lemma}\label{gp02}
Let $1 \le p \le u <\infty$. Then
\begin{eqnarray*}
\Big\{f \in \cm_p^u ((0,1)^n):&& \: D^\alpha f \in \cm_p^u ((0,1)^n)
\quad \mbox{for all}\quad \alpha \in \zz_+
\Big\}
\\
& = &
\Big\{f \in C^\infty ((0,1)^n): \quad D^\alpha f \in L_\infty ((0,1)^n)
\quad \mbox{for all}\quad \alpha \in \zz_+ \Big\}\, .
\end{eqnarray*}
\end{lemma}

Based on this simple lemma, it is now possible to show that there is no need for new spaces in case of bounded domains.

\begin{theorem}\label{gp03}
Let $\tz\in(0,1)$, $1\le p_0\le u_0<\fz$ and $1\le p_1\le u_1<\fz$.
If  $\frac1u:=\frac{1-\tz}{u_0}+\frac\tz{u_1}$,
$\frac1p:=\frac{1-\tz}{p_0}+\frac\tz{p_1}$ and $
p_0\, u_1=p_1\, u_0$,
then
$$
\lf\laz\cm^{u_0}_{p_0}((0,1)^n), \cm^{u_1}_{p_1}((0,1)^n)\r\raz_\tz =
\accentset{\diamond}{\cm}^u_p((0,1)^n).
$$
\end{theorem}

We point out that, by a proof similar to that of Theorem \ref{gp03},
we see that the conclusion of Theorem \ref{gp03} still holds true,
if $(0,1)^n$ is replaced by a bounded domain
$\Omega$.

The corresponding result for Besov-Morrey spaces on some domains
also holds true; see Appendix, Subsection \ref{domain}, for a definition.
Here we concentrate us on Lipschitz domains $\Omega\subset \rn$.
By a Lipschitz domain, we mean either a special or a bounded Lipschitz
domain. Recall that a \emph{special Lipschitz domain} is an open set $\Omega\subset
\rn$ lying above the graph of a Lipschitz function $w:\ \rr^{n-1}\to\rr$, namely,
$$\Omega:=\{(x',x_n)\in\rn:\ x_n>w(x')\},$$
where $w$ satisfies that, for all $x',\ y'\in \rr^{n-1}$,
$$|w(x')-w(y')|\le A |x'-y'|$$
with a positive constant $A$ independent of $x'$ and $y'$.
A \emph{bounded  Lipschitz domain} is a bounded domain $\Omega\subset \rn$
whose boundary $\partial \Omega$ can be cover by a finite number of open balls $B_k$
such that, for each $k$, after a suitable rotation, $\partial\Omega\cap B_k$
is a part of the graph of a Lipschitz function.

\begin{theorem}\label{gagl5}
Let $\Omega \subset \rn$ be an interval if $n=1$ or a Lipschitz domain if $n\ge 2$.
Assume that  $0 < p_i \le u_i <\infty$, $s_0,\,s_1 \in \rr$ and   $q_i \in (0,\infty)$, $i\in\{0,1\}$.
Let $s:= (1-\Theta)\, s_0 + \Theta\, s_1$, $\frac 1p : =\frac{1-\tz}{p_0}+ \frac \tz{p_1}$ and $\frac 1q : =\frac{1-\tz}{q_0}+ \frac \tz{q_1}$.
If $u_0 p_1 = u_1 p_0$, then
\[
\lf\laz \cn_{u_0,p_0,q_0}^{s_0}(\Omega), \cn_{u_1,p_1,q_1}^{s_1}(\Omega)\r\raz_\tz = \accentset{\diamond}{\cn}^{s}_{u,p,q} (\Omega)
\]
holds true for all $\tz \in (0,1)$.
\end{theorem}


\subsection{The complex method of interpolation for quasi-Banach spaces}
\label{inter1c}


The complex method in case of interpolation couples of Banach spaces is a
well-studied subject;
see, e.\,g., \cite{ca64,BL,t78,km98,kmm}.
Here we are interested in the complex method in case of interpolation couples
of certain quasi-Banach spaces.


\subsubsection{Basics}
\label{inter1ca}


We begin with some basic notation taken  from \cite{km98,kmm,MM}.
We always assume that the quasi-Banach space $X$ is equipped with a continuous quasi-norm
$\| \cdot  \|_X$ (this is always possible).

\begin{definition}
 \label{Aconvex}
A quasi-Banach space $(X,\|\cdot\|_{X})$ is said to be {\it analytically convex} if there
is a positive constant $C$ such that, for every polynomial $P:\ \cc \to X$,
$$\|P(0)\|_{X}\le C\max_{|z|=1}\|P(z)\|_{X}.$$
\end{definition}

In the framework of  analytically convex quasi-Banach spaces, one of the key properties of Banach spaces,
the maximum modulus principle, still  holds true.
To recall this, let
\[
S_0:= \{z \in \cc : \: 0<\Re e\, z<1\}\qquad \mbox{and} \qquad
S:=\{z\in\cc :\: 0\le \Re e\, z\le 1\}\, ,
\]
here and hereafter, $\Re e\, z$ for any $z\in\cc$ denotes the \emph{real part} of $z$.
The following result can be found in \cite[Theorem 7.4]{kmm}.

\begin{proposition}\label{ac}
For a quasi-Banach space $(X,\|\cdot\|_{X})$, the following
conditions are equivalent:

{\rm(i)} $X$ is analytically convex.

{\rm(ii)} There exists a positive constant  $C$ such that
$$
\max\{\|f(z)\|_{X}:\ z\in S_0\}\: \le\:  C\, \max\{\|f(z)\|_{X}:\ z\in S\setminus S_0\}
$$
for any function $f:S\to X$ which is analytic on $S_0$, continuous and bounded on
$S$.
\end{proposition}

Here $f$ being \emph{analytic} in the open set $U$ means that,
for given $z_0 \in U$, there exists some positive number $\eta$ such that there is a power series
expansion
\[
f(z)= \sum_{j=0}^\infty x_j \, (z-z_0)^j\, , \qquad x_j\in X\, ,\quad
\mbox{uniformly convergent for} \quad |z-z_0| <\eta\, .
\]
The theory of analytic functions with values in quasi-Banach spaces has been developed in
\cite{tu,k86,k86b}.
In \cite{k86}, one can find the following result.

\begin{proposition}\label{acc}
Let $U$ be an open subset of the complex plane and let $X$ be a quasi-Banach space.
Let $f_n: ~U \to X$ be a sequence of analytic functions.
If $\lim_{n\to \infty} f_n (z)= f(z)$ uniformly on compacta, then $f$ is also analytic.
\end{proposition}

For quasi-Banach lattices, one knows a simple criterion for being
analytically convex; see \cite{MM} and \cite[Theorem~7.8]{kmm}.

\begin{proposition}\label{lattice}
For a quasi-Banach lattice $(X,\|\cdot\|_{X})$ of functions, the following
conditions are equivalent:

{\rm(i)} $X$ is analytically convex.

{\rm(ii)} $X$ is lattice $r$-convex for some $r\in(0,\fz)$, i.\,e.,
$$
\lf\| \lf(\sum_{j=1}^m |f_j|^r\r)^{1/r}\, \r\|_X  \le
 \lf(\sum_{j=1}^m \|\, f_j\, \|^r_X\r)^{1/r}
$$
for any finite family $\{f_j\}_{j=1}^m$ of functions from $X$.
\end{proposition}

Based on the notion of the
analytical convexity, the following definition makes sense.

\begin{definition}\label{dci}
Let $(X_0,X_1)$ be an interpolation couple of quasi-Banach spaces,
i.\,e., $X_0$ and $X_1$ are continuously embedded into a larger topological vector space $Y$.
In addition, let $X_0 + X_1$ be analytically convex. Let $\Theta\in(0,1)$.

{\rm (i)} Let ${\mathcal A}:=\ca(X_0,X_1)$ be the set of all bounded and analytic functions $f: S_0 \to X_0 + X_1$, which
extend continuously to the closure $S$ of the strip $S_0$
such that  the traces $t \mapsto f(j+it)$ are bounded continuous
functions into $X_j$, $j\in\{0,1\}$. We endow $\ca$ with the quasi-norm
\[
\|\, f \, \|_{\ca} := \max \lf\{\sup_{t\in \rr}\, \|\, f(it)\,\|_{X_0},\
\,  \sup_{t\in \rr}\, \|\, f(1+it)\,\|_{X_1}\r\}\, .
\]
The \emph{complex interpolation space}
$[X_0,X_1]_\Theta$ is defined as the set of
all $x \in \ca (\Theta):= \{f(\Theta): \ f \in \ca \}$ and,
for all $x \in \ca (\Theta)$, let
\[
\|\, x \, \|_{[X_0,X_1]_\Theta}:=
\inf \Big\{\|\, f\, \|_{\ca}: \: f\in\ca,\ f(\Theta) = x\Big\}\, .
\]

{\rm (ii)} Let $\ca_0:=\ca_0(X_0,X_1)$ be the
closure of all functions $f\in \ca$ such that
$f(z)\in X_0\cap X_1$ for all $z\in S_0$. Then the \emph{inner complex interpolation space}
$[X_0,X_1]_\Theta^i$ is defined in the same way as $[X_0,X_1]_\Theta$ with $\ca$
replaced by $\ca_0$.
\end{definition}

\begin{remark}\label{r-com}
(i) There are different definitions of the complex
method in case of quasi-Banach spaces in the literature,
e.\,g., in Kalton et al. \cite{kmm},
 the condition
 \[
 X_0 \cap X_1 \qquad \mbox{is dense in} \quad X_j\, , \quad j\in\{0,1\},
 \]
 is added to the above restrictions.
 We found this condition inconvenient in the cases we want to apply the complex method
 to smoothness spaces built on Morrey spaces.
 So we have tried to avoid it.

(ii) Any Banach space is analytically convex. Hence, if
$(X_0,X_1)$ is an interpolation couple of Banach spaces,
the interpolation space $[X_0,X_1]_\Theta$ for $\tz\in(0,1)$
in Definition \ref{dci} reduces to the standard definition of
$[X_0,X_1]_\Theta$; see Calder{\'o}n \cite{ca64}.
This requires an additional comment, for which we follow \cite[p.~49]{lun}.
Calder{\'o}n \cite{ca64} (see also \cite[4.1]{BL}) in addition assumed
\[
\lim_{|t| \to \infty} f(it) = \lim_{|t| \to \infty} f(1+it)=0 \, .
\]
We denote the subclass of analytic functions $f\in \ca(X_0,X_1)$ with this additional property by
$\wz\ca(X_0,X_1)$.
Let $f \in \ca (X_0,X_1)$ and $\tz\in(0,1)$.
Then $f_\delta (z) := e^{\delta (z-\Theta)^2}\, f(z)$ for all $z\in S_0$ belongs to
$\ca (X_0,X_1)$ as well and
\[
\|\, f_\delta \, \|_{\ca} \le  \max \Big\{e^{\delta \, \Theta^2},\
e^{\delta \, (1-\Theta)^2}\Big\}\, \|\, f \, \|_{\ca} \, .
\]
Furthermore, letting $\delta \to 0$, we obtain
\[
\inf_{f \in \ca (X_0,X_1), \, f(\Theta) = x} \|\, f \, \|_{\ca}  =
\inf_{f \in \wz\ca (X_0,X_1), \, f(\Theta) = x} \|\, f \, \|_{\ca}.
\]
Hence, restricting the set of admissible functions $f$ to the subset $\wz\ca (X_0,X_1)$
is  changing neither  the space $[X_0,X_1]_\Theta$ nor the quasi-norm.

(iii) Let $(X_0,X_1)$ be an interpolation couple of Banach spaces. Then it is well known that $$[X_0,X_1]_\Theta=[X_0,X_1]_\Theta^i;$$
see, e.\,g., \cite[1.9.2]{t78}.
However, it is unclear whether this is still
true for general quasi-Banach cases or not.
\end{remark}

The next three properties of $[X_0,X_1]_\Theta$ are essentially taken from
\cite{km98,kmm}.

\begin{proposition}\label{complexfinal}
Let $(X_0,X_1)$ be an interpolation couple of quasi-Banach spaces.

{\rm (i)} Then
$ \ca (X_0,X_1)$ is a quasi-Banach space  continuously embedded into $C_b (S,X_0 + X_1)$,
the set of all bounded continuous functions from $S$ to $X_0+X_1$.

{\rm (ii)} Also $[X_0,X_1]_\Theta$ is a quasi-Banach space.
\end{proposition}

Recall that an interpolation functor is said
to \emph{be of exponent $\Theta \in (0,1)$}
if there exists a positive constant $C$ such that, for all admissible  interpolation couples of quasi-Banach spaces,
$(X_0,X_1)$ and $(Y_0,Y_1)$, the inequality
\[
\| \, T\, \|_{F(X_0,X_1) \to F(Y_0,Y_1)} \le C\, \| \, T\, \|_{X_0 \to Y_0}^{1-\Theta} \, \| \, T\, \|_{X_1 \to Y_1}^\Theta
\]
holds true for all linear operators $T \in \cl (X_0,Y_0) \cap \cl (X_1,Y_1)$.

\begin{proposition}\label{complexinterpol}
Let $(X_0,X_1)$ and $(Y_0,Y_1)$ be  interpolation couples of quasi-Banach spaces.
Assume that  $X_0+X_1$ and $Y_0 + Y_1$ are analytically convex.
Let $T: ~X_j \to Y_j$, $j\in\{0,1\}$, be a linear and bounded operator.
Then, for each $\Theta \in (0,1)$,
$T$ is a linear and  bounded operator which maps $[X_0,X_1]_\Theta$
into $[Y_0,Y_1]_\Theta$. In addition,
\[
\|\, T \|_{[X_0,X_1]_\Theta\to [Y_0,Y_1]_\Theta}
\le \|\, T \|_{X_0 \to Y_0}^\Theta
\|\, T \|_{X_1 \to  Y_1}^{1-\Theta} \, ,
\]
i.\,e., the complex method represents an exact interpolation functor of exponent $\Theta$
also in the framework of quasi-Banach spaces.
\end{proposition}

We also have the following conclusion on the retraction and the coretraction.

\begin{proposition}\label{complexretract}
Let $(X_0,X_1)$ and $(Y_0,Y_1)$ be two interpolation couples of quasi-Banach spaces
such that $Y_j$ is a retract of $X_j$, $j\in\{0,1\}$.
Assume that  $X_0+X_1$ and $Y_0 + Y_1$ are analytically convex.
Then, for each $\Theta \in (0,1)$,
\[
[Y_0,Y_1]_\Theta = R ([X_0, X_1]_\Theta)\, .
\]
\end{proposition}

We mention that the method of the retraction and
the coretraction in interpolation theory can be
found in many places. We refer the reader to, e.\,g., \cite[6.4]{BL} and \cite[1.2.4,~2.4.1,~2.4.2]{t78}.

\begin{remark}\label{inner1}
 For later use, we mention that
Propositions \ref{complexfinal},  \ref{complexinterpol} and \ref{complexretract}
remain true for the inner complex method.
\end{remark}


\subsubsection{Complex interpolation of Morrey-Campanato and related spaces. I}
\label{inter1cb}


In this subsection, we deal with consequences of
Subsection \ref{cpr}  for the complex interpolation of
Morrey-Campanato and related spaces.

First we  quote a result of Yang et al.  \cite{yyz} (with forerunners in case   $\tau =0$ in
Mendez and Mitrea \cite{MM}, Kalton et al. \cite[Proposition~7.7]{kmm}, and  Sickel et al.
\cite{ssv}).

\begin{lemma}[\cite{yyz}]\label{convex}
Let $q\in(0,\infty]$, $s \in \rr$ and $\tau\in[0,\fz)$.

\begin{enumerate}
\item[{\rm(i)}] Let $p\in(0,\fz)$. Then
$\ft $ and $\sft$ are analytically convex.

\item[{\rm(ii)}] Let $p\in(0,\fz]$. Then
$\bt$ and $\sbt$ are analytically convex.

\item[{\rm(iii)}] Let $0<p\le u\le \fz$. Then
$\cn^{s}_{u,p,q}(\rn)$ and $n^{s}_{u,p,q}(\rn)$ are analytically convex.
\end{enumerate}
\end{lemma}

\begin{lemma}\label{ring}
Let $s\in\rr$, $p,\ q\in (0,\fz]$ ($p<\fz$ for $F$-spaces), $u\in[p,\fz]$
and $\tau\in[0,\fz)$.

\begin{enumerate}
\item[{\rm(i)}] It holds true that
$\mathring B_{p,q}^{s,\tau}(\rn)$, $\mathring F_{p,q}^{s,\tau}(\rn)$ and
$\mathring{\cn}^{s}_{u,p,q}(\rn)$  are  analytically convex.

\item[{\rm (ii)}] It holds true that the spaces $A_{p,q}^{s,\tau}(\rn)$ when $\tau\in (0,\fz)$, and
$\cn_{u,p,q}^{s}(\rn)$ when $p<u$,  are nonseparable.

\item[{\rm(iii)}] If $\tau \in(0,\fz)$, then  $\mathring A_{p,q}^{s,\tau}(\rn)$ is a proper subspace of
${A}_{p,q}^{s,\tau}(\rn)$, $A \in \{B,F\}$.

\item[{\rm(iv)}] If $0<p< u\le\fz$, then $\mathring \cn_{u,p,q}^{s}(\rn)$ is a proper subspace of
${\cn}_{u,p,q}^{s}(\rn)$.
\end{enumerate}
\end{lemma}

Whereas part (i) of Lemma \ref{ring} is an immediate consequence of Lemma \ref{convex}, the remainder of Lemma \ref{ring}
is a little bit more complicated to prove, we refer the reader
to \cite{yyz} for details.

It is well known that there are nice  connections between
the complex interpolation spaces and
the corresponding Calder\'on products as follows (see the original article of Calder{\'o}n \cite{ca64}, \cite[Theorem 3.4]{km98} or  \cite[Theorem 7.9]{kmm}).

\begin{proposition}\label{Thm4}
Let $({\mathfrak X}, {\mathcal S},\mu)$ be a complete separable metric space, $\mu$ a $\sigma$-finite Borel measure on ${\mathfrak X}$, and $X_0,X_1$ a pair of
quasi-Banach lattices of functions on $({\mathfrak X},\mu)$.
If both $X_0$ and $X_1$ are analytically convex and separable, then
$X_0+X_1$ is also analytically convex and
\begin{equation*}
[X_0,X_1]_\Theta =[X_0,X_1]_\tz^i= X_0^{1-\Theta}\, X_1^\Theta \, ,\qquad  \Theta\in(0,1).
\end{equation*}
\end{proposition}

Because of the separability conditions in Proposition \ref{Thm4},
we can not apply this proposition to Morrey-type spaces $\at$ and to Besov-Morrey spaces $\cn_{u,p,q}^{s}(\rn)$.
However, it can be applied to subspaces obtained as the closure of the test functions.
Before we turn to results of such a type, we comment on sequences spaces.
For sequence spaces, one knows a little bit more. We quote Remark in front of  \cite[Theorem 7.10]{kmm}; see also
\cite{MM}.

\begin{lemma}\label{sl1}
Let $X_0,X_1$ be a pair of quasi-Banach sequence lattices.
If both $X_0$ and $X_1$ are analytically convex and at least one is separable, then
$X_0+X_1$ is also analytically convex and
\begin{equation*}
[X_0,X_1]_\Theta = [X_0,X_1]_\tz^i=X_0^{1-\Theta}\, X_1^\Theta \, ,\qquad  \Theta\in(0,1).
\end{equation*}
\end{lemma}

\begin{remark}\label{rem-dense}
We need to go back to the problem described in Remark \ref{r-com}.
Neither in \cite{km98} nor in \cite{kmm}, the additional condition
that $ X_0 \cap X_1$ is dense in $X_j$, $j\in\{0,1\}$, is used in the proofs of \cite[Theorem 3.4]{km98} and
\cite[Theorem 7.9]{kmm}. Hence, we can avoid the use
of this condition in Proposition \ref{Thm4} and Lemma \ref{sl1}.
\end{remark}

Essentially as a consequence of Propositions \ref{morrey2} and \ref{morrey3},
the wavelet characterization of the spaces
under consideration (see Propositions \ref{wave1} and \ref{wave2} in
Appendix), Proposition \ref{Thm4} and Lemma \ref{sl1}, one obtain
the following result (see  \cite{yyz} for all details).

\begin{proposition}\label{ci}
Let all parameters
be as in Theorem \ref{COMI}.

{\rm(i)} If  $\tau_0 \, p_0 = \tau_1 \, p_1$, then
\begin{eqnarray*}
\mathring A_{p,q}^{s,\tau}(\rn)
&&=\lf[\mathring A_{p_0,q_0}^{s_0,\tau_0}(\rn),
\mathring A_{p_1,q_1}^{s_1,\tau_1}(\rn)\r]_\tz
=\lf[A_{p_0,q_0}^{s_0,\tau_0}(\rn),
\mathring A_{p_1,q_1}^{s_1,\tau_1}(\rn)\r]_\tz
\\
&&=\lf[\mathring A_{p_0,q_0}^{s_0,\tau_0}(\rn),
A_{p_1,q_1}^{s_1,\tau_1}(\rn)\r]_\tz\, .
\end{eqnarray*}

{\rm(ii)} If $p_0 \, u_1=p_1\, u_0$, then
\begin{eqnarray*}
\mathring \cn_{u,p,q}^{s}(\rn)
&&=\lf[\mathring \cn_{u_0,p_0,q_0}^{s_0}(\rn),
\mathring \cn_{u_1,p_1,q_1}^{s_1}(\rn)\r]_\tz
=\lf[\cn_{u_0,p_0,q_0}^{s_0}(\rn),
\mathring \cn_{u_1,p_1,q_1}^{s_1}(\rn)\r]_\tz
\\
&&=\lf[\mathring \cn_{u_0,p_0,q_0}^{s_0}(\rn),
\cn_{u_1,p_1,q_1}^{s_1}(\rn)\r]_\tz .
\end{eqnarray*}
\end{proposition}

\begin{remark}\label{complex1a}
Proposition \ref{ci} covers almost all cases for which
the complex interpolation of Besov and Triebel-Lizorkin spaces
is known.
In particular, we obtain the formulas
\[
B_{p,q}^{s}(\rn)
= \lf[B_{p_0,q_0}^{s_0}(\rn), B_{p_1,q_1}^{s_1}(\rn)\r]_\tz \quad\mbox{with} \quad \max \{p_0,q_0\} <\infty\, ,
\]
and
\[
F_{p,q}^{s}(\rn)
= \lf[F_{p_0,q_0}^{s_0}(\rn), F_{p_1,q_1}^{s_1}(\rn)\r]_\tz\quad\mbox{with}
\quad  \max \{p_0,p_1\} + \min \{q_0,q_1\} <\infty\, ,
\]
where
\[
s=(1-\Theta)\, s_0 + \Theta\, s_1\, \qquad
\frac 1q := \frac{1-\tz}{q_0} + \frac{\tz}{q_1}\qquad
\mbox{and}
\qquad \frac 1p := \frac{1-\tz}{p_0} + \frac{\tz}{p_1}\, .
\]
Here we have used the fact that
\[
\mathring{A}_{p,q}^{s}(\rn) = A_{p,q}^{s}(\rn) \qquad \Longleftrightarrow \qquad \max \{p,q\}<\infty, \qquad A\in\{B,F\};
\]
see \cite[Theorem 2.3.3]{t83}.
These interpolation formulas have been known before, we refer the reader to Bergh, L\"ofstr\"om
\cite[Theorem 6.4.5]{BL}, Triebel \cite[2.4.1/2]{t78}, \cite{t81}, Frazier, Jawerth \cite{fj90},
Mendez, Mitrea \cite{MM} and
Kalton et al. \cite{kmm}.
In the next subsection, we shall continue this discussion by considering those situations
where both spaces are non-separable.
\end{remark}


\subsubsection{Complex interpolation of Morrey-Campanato and related spaces. II}
\label{inter1cd}


Notice that, in Proposition \ref{ci}, at least one of
the interpolated spaces should be the closure of test functions.
The natural question here is, can we remove this restriction and
calculate $\lf[A_{p_0,q_0}^{s_0,\tau_0}(\rn),
A_{p_1,q_1}^{s_1,\tau_1}(\rn)\r]_\tz$ and/or  $\lf[\cn_{u_0,p_0,q_0}^{s_0}(\rn),
\cn_{u_1,p_1,q_1}^{s_1}(\rn)\r]_\tz$? In such a situation, we can not use
Proposition \ref{Thm4} because of the non-separability of the spaces involved.
Instead one can apply an argument of Shestakov \cite{s74,s74m}.

\begin{proposition}\label{shest}
Let $(X_0,X_1)$ be a couple of Banach lattices and $\Theta \in (0,1)$. Then
\[
[X_0,X_1]_\Theta =
\lf[X_0, X_1, X_0^{1-\Theta} \, X_1^\Theta, \#\r]\, .
\]
\end{proposition}

Using the inner complex method instead of the usual complex method, the following
generalization of Proposition \ref{shest} was obtained in \cite{y14}.

\begin{proposition}\label{shest2}
Let $(X_0,X_1)$ be a couple of analytically convex
quasi-Banach lattices and $\Theta \in (0,1)$. Then
\[
[X_0,X_1]_\Theta^i =
\lf[X_0, X_1, X_0^{1-\Theta} \, X_1^\Theta, \#\r]\, .
\]
\end{proposition}

In view of Proposition \ref{nil}, we have the following obvious conclusion which will be our main tool in this subsection.

\begin{corollary}
 \label{nil2}
Let $(X_0,X_1)$ be a couple of analytically convex
quasi-Banach lattices and $\Theta \in (0,1)$.
If
\[
X_0^{1-\Theta} \, X_1^\Theta = \laz X_0, \, X_1, \Theta\raz,
\]
then
\[
[X_0,X_1]_\Theta^i = \laz X_0, \, X_1\raz_\tz
\]
follows.
\end{corollary}

Arguing first on the level of sequence spaces and then transferring the result to
function spaces by means of the method of the retraction and the coretraction, we obtain the following.

\begin{theorem}\label{gagl6}
 Theorems \ref{gagl1},  \ref{gagl7}, \ref{gagl2} and \ref{gagl4} remain true with
 $\laz \, \cdot, \, \cdot \, \raz_\tz$ replaced by
$[ \, \cdot, \, \cdot \, ]_\tz^i$.
 \end{theorem}


\subsection*{Complex interpolation of Besov and Triebel-Lizorkin spaces}


Theorem \ref{gagl6} has some very interesting consequences for the complex interpolation of Besov and Triebel-Lizorkin spaces.
On the one side, it covers all the cases discussed in Remark \ref{complex1a} (see Theorem \ref{gagl1}(i)),
on the other hand,
it provides the most complete collection of results concerning the complex interpolation of Besov and Triebel-Lizorkin
spaces where both spaces are non-separable.
In particular, we have
\[
\accentset{\diamond}{B}^{s}_{p,q} (\rn)
= \lf[B_{p_0,q_0}^{s_0}(\rn), B_{p_1,q_1}^{s_1}(\rn)\r]_\tz^i\,,
\]
if the standard assumptions
\begin{equation} \label{standard}
s=(1-\Theta)\, s_0 + \Theta\, s_1\,, \qquad
\frac 1q := \frac{1-\tz}{q_0} + \frac{\tz}{q_1}\qquad
\mbox{and}
\qquad \frac 1p := \frac{1-\tz}{p_0} + \frac{\tz}{p_1}\,
\end{equation}
are fulfilled and, in addition, one of the further sets of the following restrictions holds true:
\begin{enumerate}
\item[(a)] $s_0 \neq s_1$ and $p_0 = p_1 = \infty$;
\item[(b)] $s_0 = s_1$, $p_0 = p_1 = \infty$ and $q_0 < q_1$;
\item[(c)] $s_0 \neq s_1$, $0 < p_0 = p_1 < \infty$ and $q_0 = q_1 = \infty$;
\item[(d)] $0 < p_0 < p_1 < \infty$, $s_0 -n/p_0 > s_1 - n/p_1$ and $q_0 = q_1 = \infty$.
\end{enumerate}

Similarly,
\[
\accentset{\diamond}{F}^{s}_{p,q} (\rn)
= \lf[F_{p_0,q_0}^{s_0}(\rn), F_{p_1,q_1}^{s_1}(\rn)\r]_\tz^i\,,
\]
if the standard assumptions  \eqref{standard}
are fulfilled and, in addition, one of the further sets of the following
restrictions holds true:
\begin{enumerate}
\item[(e)] $s_0 \neq s_1$, $0 < p_0 = p_1 < \infty$ and $q_0 = q_1 = \infty$;
\item[(f)] $0 < p_0 < p_1 < \infty$, $s_0 -n/p_0 \ge s_1 - n/p_1$ and $q_0 = q_1 = \infty$.
\end{enumerate}

By a closer look onto these different sets of restrictions, we find
that the space
$$\lf[B_{p_0,q_0}^{s_0}(\rn), B_{p_1,q_1}^{s_1}(\rn)\r]_\tz^i$$
is not known  if either
\[
0 < p_0 < p_1  \le  \infty, \qquad s_0 -n/p_0 \le  s_1 - n/p_1 \qquad \mbox{and} \qquad  q_0 = q_1 = \infty
\]
or
\[
0 < p_0 < p_1  =  \infty, \qquad s_0 -n/p_0 \le  s_1\, , \qquad q_0 = \infty  \qquad \mbox{and} \qquad  q_1\in(0, \infty).
\]
Similarly,  the space $\lf[F_{p_0,q_0}^{s_0}(\rn), F_{p_1,q_1}^{s_1}(\rn)\r]_\tz^i$ is not known  if
\[
0 < p_0 < p_1 < \infty, \qquad s_0 -n/p_0 \le  s_1 - n/p_1 \qquad \mbox{and} \qquad  q_0 = q_1 = \infty\, .
\]
We add a comment to a formula stated in \cite[Theorem 6.4.5]{BL},
which claims that
\[
B_{p,q}^{s}(\rn)
= \lf[B_{p_0,q_0}^{s_0}(\rn), B_{p_1,q_1}^{s_1}(\rn)\r]_\tz\,
\]
for all $s_0,\ s_1 \in \rr$ and $p_0,\ p_1,\ q_0,\ q_1\in[1,\infty]$.
By our previous remarks, this can not be true in this generality. There are plenty of counterexamples if
$p_0 + q_0 = p_1 + q_1 = \infty$.

\begin{remark}
We shall make some comments to the literature.
The formula
$$
\lf[B_{p,\fz}^{s_0}(\rn), B_{p,\fz}^{s_1}(\rn)\r]_\tz = \mathring{B}_{p,\fz}^{s}(\rn)\subsetneqq B_{p,\fz}^{s}(\rn)
$$
has been proved by Triebel \cite[Theorem 2.4.1]{t78} under the restrictions $s_0\neq s_1$ and $p\in(1,\fz)$.
The counterparts of part (iv), (v) and (vi) of Theorem \ref{gagl1} for the  complex method
have been proved before in Sickel et al. \cite{ssv} (see also \cite{ssvb}).
Also, Sawano and Tanaka \cite{sat} have studied
the complex interpolation of
Besov-Morrey and Triebel-Lizorkin-Morrey spaces with fixed $s$ and fixed $p$
(Banach cases), in which they proved
\[
\lf[\cn^{s}_{u,p,q_0}(\rn), \cn^{s}_{u, p,q_1}(\rn)\r]_\tz   = \cn^{s}_{u,p,q}(\rn)
\]
and
\[
\lf[F^{s,\tau}_{p,q_0}(\rn), F^{s,\tau}_{p,q_1}(\rn)\r]_\tz   = F^{s,\tau}_{p,q}(\rn)\,,
\]
under the restrictions $s \in \rr$, $1 < p \le u < \infty$, $q_0,\ q_1\in(1,\infty]$ and
$\frac 1q := \frac{1-\tz}{q_0} + \frac{\tz}{q_1}$, altogether being special cases of
Theorem \ref{gagl6}.
\end{remark}

Now we turn to Morrey spaces. By Corollary \ref{nil2}, it is not a big surprise
that the space $$[\cm^{u_0}_{p_0}(\rn), \cm^{u_1}_{p_1}(\rn)]^i_\tz $$ behaves as $ \laz\cm^{u_0}_{p_0}(\rn), \cm^{u_1}_{p_1}(\rn)\raz_\tz $;
see Corollary \ref{morrey42}.

\begin{theorem}\label{morrey4}
Let $\Theta\in(0,1)$, $0 < p_0 \le u_0 <\infty$,
\[
0 < p_1 \le u_1 <\infty \, , \quad
\frac 1u := \frac{1-\tz}{u_0} + \frac{\tz}{u_1} \quad \mbox{and}
\quad  \frac 1p := \frac{1-\tz}{p_0} + \frac{\tz}{p_1} \, .
\]

{\rm (i)}
It holds true that
\[
[\cm^{u_0}_{p_0}(\rn), \cm^{u_1}_{p_1}(\rn)]^i_\tz \hookrightarrow  \cm^{u}_{p}(\rn)
\]
and the embedding is always proper except the trivial cases consisting in
\begin{enumerate}
\item[{\rm(a)}] $p_0 = p_1$ and $u_0 = u_1$, or
\item[{\rm(b)}] $p_0 = u_0$ and $p_1 = u_1$.
\end{enumerate}

{\rm (ii)} Suppose, in addition, $1 \le  p_0 < p_1$, $p_0<u_0$ and $u_0 p_1 = u_1 p_0$. Then
\[
[\cm^{u_0}_{p_0}(\rn), \cm^{u_1}_{p_1}(\rn)]^i_\tz =  \cm^{u_0,u_1,\tz}_{p_0,p_1}(\rn)\, .
\]
\end{theorem}

Theorem \ref{morrey4}(i) extends and supplements the negative result \eqref{interpol1} of Lemari{\'e}-Rieusset \cite[Theorem 3]{LR},
already mentioned in Section \ref{s1} of this article, to the case of quasi-Banach spaces.

Concerning the description of $\lf[A^{s_0,\tau_0}_{p_0,q_0}(\rn), A^{s_1,\tau_1}_{p_1,q_1}(\rn)\r]_\tz^i$, $p_0 \neq p_1$,
$\tau_i\in(0,1/p_i)$, $i\in\{0,1\}$,
we have very little to say.

\begin{proposition}\label{morreyx}
Let $\tz\in(0,1)$,  $s_i \in \rr$, $p_i\in(0,\fz)$, $q_i\in(0,\fz]$, $\tau_i\in[0,1/p_i)$, $i\in\{0,1\}$,
$s  := (1-\tz) \, s_0 + \tz \, s_1$,
\[
\frac 1p := \frac{1-\tz}{p_0} + \frac{\tz}{p_1},
\quad  \frac 1q := \frac{1-\tz}{q_0} + \frac{\tz}{q_1} \quad\mbox{and}\quad
\tau := (1-\tz)\tau_0+\tz\tau_1.
\]
Then
\[
 \lf[F^{s_0,\tau_0}_{p_0,q_0}(\rn), F^{s_1,\tau_1}_{p_1,q_1}(\rn)\r]_\tz^i   \hookrightarrow
\lf[F^{s_0,\tau_0}_{p_0,q_0}(\rn), F^{s_1,\tau_1}_{p_1,q_1}(\rn)\r]_\tz   \hookrightarrow
F^{s,\tau}_{p,q}(\rn)
\]
and
\[
 \lf[B^{s_0,\tau_0}_{p_0,q_0}(\rn), B^{s_1,\tau_1}_{p_1,q_1}(\rn)\r]_\tz^i   \hookrightarrow
\lf[B^{s_0,\tau_0}_{p_0,q_0}(\rn), B^{s_1,\tau_1}_{p_1,q_1}(\rn)\r]_\tz   \hookrightarrow
B^{s,\tau}_{p,q}(\rn)\, .
\]
\end{proposition}

\begin{remark}
Also for the complex method,  the article of Frazier and Jawerth \cite{fj90}
is an important source and contains  a number of further results.
There the following formulas were proved:
\[
[ F_{p_0,q_0}^{s_0}(\rn), F_{p_1,q_1}^{s_1}(\rn)]_\tz =
[F_{p_0,q_0}^{s_0,0}(\rn), F_{p_1,q_1}^{s_1,0}(\rn)]_\tz =
F_{p,q}^{s,0}(\rn)= F_{p,q}^{s}(\rn)\, ,
\]
\[
[F_{p_0,q_0}^{s_0,1/p_0}(\rn), F_{p_1,q_1}^{s_1,1/p_1}(\rn)]_\tz = F_{p,q}^{s,1/p}(\rn)
= F_{\infty,q}^{s}(\rn),
\]
\[
[F_{p_0,q_0}^{s_0,1/p_0}(\rn), B_{p_1,\infty}^{s_1,1/p_1}(\rn)]_\tz = F_{p,q}^{s,1/p}(\rn)
= F_{\infty,q}^{s}(\rn)
\]
as well as
\[
[F_{p_0,q_0}^{s_0,0}(\rn), F_{p_1,q_1}^{s_1,1/p_1}(\rn)]_\tz = F_{p,q}^{s,0}(\rn)
\]
and
\begin{eqnarray*}
[ F_{p_0,q_0}^{s_0,0}(\rn), F_{p_1,q_1}^{s_1,\tau_1}(\rn)]_\tz & = &
[F_{p_0,q_0}^{s_0,0}(\rn), B_{\infty,\infty}^{s_1 + n(\tau_1-1/p_1)}(\rn)]_\tz
\\
& = & F_{p,q}^{s+n(\tau-1/p) + n(1-\Theta)/p_0,0}(\rn)
\end{eqnarray*}
if $p_0,\ p_1\in[1, \infty)$, $q_0,\ q_1\in[1,\infty]$, $\min\{q_0, \ q_1\} < \infty$,
$s_0,\ s_1\in\rr$,
$\tau_1\in(1/p_1,\fz)$, $\tz\in(0,1)$ and $s,\ p,\ q$ are as in Proposition
\ref{morreyx} (and with no further restrictions);
see \cite[Corollary~8.3]{fj90}.
\end{remark}

We turn to Besov-Morrey spaces. A counterpart to Proposition \ref{morreyx} holds true
as follows.

\begin{proposition}\label{morreyn}
Let $\tz\in(0,1)$,  $s_i \in \rr$, $u_i \in (0,\infty)$,
$p_i\in(0,u_i)$, $q_i\in(0,\fz]$,  $i\in\{0,1\}$,
\[
s  := (1-\tz) \, s_0 + \tz \, s_1,\quad\frac 1p := \frac{1-\tz}{p_0} + \frac{\tz}{p_1}
 \, , \quad
\frac 1u := \frac{1-\tz}{u_0} + \frac{\tz}{u_1}\,
\quad\mbox{and}\quad  \frac 1q := \frac{1-\tz}{q_0} + \frac{\tz}{q_1}\, .
\]
Then
\[
 \lf[\cn^{s_0}_{u_0, p_0,q_0}(\rn), \cn^{s_1}_{u_1,p_1,q_1}(\rn)\r]_\tz^i   \hookrightarrow
\lf[\cn^{s_0}_{u_0,p_0,q_0}(\rn), \cn^{s_1}_{u_1, p_1,q_1}(\rn)\r]_\tz   \hookrightarrow
\cn^{s}_{u,p,q}(\rn) \, .
\]
\end{proposition}

At the end of this subsection,
we consider function spaces on the unit open cube $(0,1)^n$.
Parallel to the conclusions of
Theorem \ref{gp03}
for $\laz \cm^{u_0}_{p_0}((0,1)^n), \cm^{u_1}_{p_1}((0,1)^n)\raz_\tz$,
and Theorem \ref{gagl5} for $$\laz\cn_{u_0,p_0,q_0}^{s_0}((0,1)^n), \cn_{u_1,p_1,q_1}^{s_1}((0,1)^n) \raz_\tz,$$
we obtain the following conclusion.

\begin{theorem}\label{gp03n}
Let $\tz\in(0,1)$, $1\le p_0\le u_0<\fz$ and $1\le p_1\le u_1<\fz$.
If  $\frac1u:=\frac{1-\tz}{u_0}+\frac\tz{u_1}$,
$\frac1p:=\frac{1-\tz}{p_0}+\frac\tz{p_1}$ and $
p_0\, u_1=p_1\, u_0$,
then
$$
[\cm^{u_0}_{p_0}((0,1)^n), \cm^{u_1}_{p_1}((0,1)^n)]_\tz^i =
\accentset{\diamond}{\cm}^u_p((0,1)^n).
$$
\end{theorem}


\subsection{The second complex method of interpolation}
\label{inter1f}


For the convenience of the reader, we also recall the second complex interpolation method of Calder{\'o}n; see \cite{ca64}
or \cite[4.1]{BL}.

\begin{definition}\label{dci2}
Let $(X_0,X_1)$ be an interpolation couple of Banach spaces,
i.\,e., $X_0$ and $X_1$ are continuously embedded into a larger topological vector space $Y$.
Let $\Theta\in(0,1)$.

Let ${\mathcal G}:=\cg(X_0,X_1)$ be the set of all functions $f: S \to X_0 + X_1$ such that
\begin{itemize}
 \item[(a)] $\frac {f(\cdot)}{1+|\cdot|}$ is continuous and bounded on $S$;
\item[(b)] $f$ is analytic in $S_0$;
\item[(c)] $f(j+it_1)-f(j+it_2)$ has values in $X_j$ for all $(t_1,t_2) \in \rr^2$,  $j\in\{0,1\}$;
\item[(d)] the quantity
\[
\|\, f \, \|_{\cg} := \max \left\{\sup_{t_1 \neq t_2}\, \lf\|\, \frac{f(it_2) - f(it_1)}{t_2 - t_1}\,\r\|_{X_0},\
\,  \sup_{t_1 \neq t_2}\, \lf\|\, \frac{f(1+it_2) - f(1+it_1)}{t_2 - t_1}\,\r\|_{X_1}\right\}
\]
is finite.
\end{itemize}

The \emph{complex interpolation space}
$[X_0,X_1]^\Theta$ is defined as the set of
all $x \in \cg (\Theta):= \{f(\Theta): \, f \in \cg \}$ and, for all $x \in \cg (\Theta)$,
\[
\|\, x \, \|_{[X_0,X_1]^\Theta}:= \inf \Big\{\|\, f\, \|_{\cg}: \: f\in\cg,\ f(\Theta) = x\Big\}\, .
\]
\end{definition}

Some basic properties of this interpolation method are summarized in the next proposition; see \cite[Theorem~4.1.4]{BL}.

\begin{proposition}
Let $(X_0,X_1)$ be an interpolation couple of Banach spaces and $\Theta\in(0, 1)$.
The space $[X_0,X_1]^\Theta$ is a Banach space and the functor $[\, \cdot \, , \, \cdot \, ]^{\Theta}$
is an exact interpolation functor of exponent $\Theta$.
\end{proposition}

The relations of the two complex interpolation methods $[X_0,X_1]_\Theta$ and $[X_0,X_1]^\Theta$ are well
understood; see \cite[Theorem~4.3.1]{BL}.

\begin{proposition}
Let $(X_0,X_1)$ be an interpolation couple of Banach spaces and $\Theta\in(0, 1)$.
Then
\[
[X_0,X_1]_\Theta \hookrightarrow [X_0,X_1]^\Theta \, .
\]
If, at least, one of the two spaces, $X_0$ and $X_1$, is reflexive, then
\[
[X_0,X_1]_\Theta = [X_0,X_1]^\Theta \, .
\]
\end{proposition}

Finally, we quote the result of Lemari{\'e}-Rieusset  \cite{LR2},
which is the reason why we recalled this interpolation
method here.

\begin{theorem}\label{lr1}
Let $\tz\in(0,1)$, $1< p_0\le u_0<\fz$ and $1< p_1\le u_1<\fz$.
If  $\frac1u:=\frac{1-\tz}{u_0}+\frac\tz{u_1}$,
$\frac1p:=\frac{1-\tz}{p_0}+\frac\tz{p_1}$ and $
p_0\, u_1=p_1\, u_0$,
then
$$
[\cm^{u_0}_{p_0}(\rn), \cm^{u_1}_{p_1}(\rn)]^\tz = {\cm}^u_p(\rn).
$$
\end{theorem}

\begin{remark}
(i) Theorem \ref{lr1} shows that this second complex interpolation method has the potential to become as useful as the
$\pm$-method in the context of Morrey and Morrey-type spaces.
However, the disadvantage of this type of the complex interpolation
for us consists in the limitation to Banach spaces.

(ii)
In \cite[pp.~90]{BL}, Bergh and L\"ofstr\"om wrote: ``We shall
consider the space $[X_0, X_1]^\tz$ more or less as a technical tool".
Probably the needs of the interpolation theory of Morrey spaces will lead to a new evaluation of this method.
\end{remark}


\subsection{The real method of interpolation}
\label{inter1d}


For the convenience of the reader,  we recall some basic notation.
First we recall Peetre's $K$-functional. Let $(X_0,X_1)$ be a quasi-Banach couple.
Then, for any $t\in(0,\fz)$ and any $x \in X_0 + X_1$,  define
\[
K(t,x,X_0,X_1):= \inf_{\gfz{x=x_0 + x_1}{x_0 \in X_0,\, x_1 \in X_1}} \, \lf\{\|x_0\|_{X_0} + t\, \| \, x_1\, \|_{X_1}\,\r\}.
\]

\begin{definition}
Let $\tz \in (0,1)$ and $q\in(0,\infty]$.
The real interpolation space
$(X_0,X_1)_{\tz,q}$ is defined as
the collection of all $x\in X_0 + X_1$ such that
\[
\|\, x \, \|_{(X_0,X_1)_{\tz,q}} := \lf\{\int_0^\infty [t^{-\tz} \, K(t,x,X_0,X_1)]^q \, \frac{dt}{t}\r\}^{1/q}<\infty\, .
\]
\end{definition}

Concerning the relation between the real and the complex methods, we refer the reader, for example, to \cite{cpp}.
We also recall some basic properties of the real interpolation; see, e.\,g., \cite{BL} (Banach case) or \cite[2.4.1]{t83} (quasi-Banach case).

\begin{proposition}\label{realbasic}
Let $(X_0,X_1)$ and  $(Y_0,Y_1)$ be any two quasi-Banach couples and $\Theta \in (0,1)$.

{\rm(i)}  It holds true that $(A_0, A_1)_{\Theta,q}$ is  a quasi-Banach space, where $A\in\{X,\ Y\}$.

{\rm(ii)} If $T \in \cl (X_0,Y_0) \cap \cl(X_1,Y_1)$, then $T$ maps  $(X_0, X_1)_{\Theta,q}$
continuously into $(Y_0, Y_1)_{\Theta}$. Furthermore,
\[
\| \, T \, \|_{(X_0, X_1)_{\Theta,q}\to (Y_0, Y_1)_{\Theta,q}}\le  \| \, T \, \|_{X_0\to Y_0}^{1-\tz} \,
\|\, T \, \|_{X_1 \to  Y_1}^\tz\, ,
\]
i.\,e., all
functors $(\, \cdot\, ,\, \cdot\, )_{\Theta,q}$ are exact and of exponent $\Theta$.
\end{proposition}


\subsubsection{Real interpolation with fixed $p$ and $\tau$ (or $u$)}


It is a well-known fact that the
real interpolation of Besov and Triebel-Lizorkin spaces  is helpful for fixed $p$.
For different $p$, in general, Lorentz spaces instead of Lebesgue spaces come into play.
This continues to be true for the spaces under consideration here. In addition, one has to fix either $\tau$ or $u$.
Our first result concerns the real interpolation of the classes $\at$.

\begin{theorem}\label{approx5}
Let  $\Theta \in(0, 1)$, $s_0,s_1 \in \rr$, $s_0 < s_1$,
$p \in(0, \infty)$, $\tau \in[0, 1/p)$, $ q,\ q_0,\ q_1 \in(0, \infty]$ and
$s:= (1-\Theta)\, s_0 + \Theta \, s_1$.
Let $A,\ \ca \in \{B,\ F\}$.
Then
\[
\cn^{s}_{u,p,q}(\rn) =
(A^{s_0,\tau}_{p,q_0}(\rn),\ca^{s_1,\tau}_{p,q_1}(\rn))_{\Theta,q} \, , \qquad \frac 1u := \frac 1p -\tau\, ,
\]
in the sense of equivalent quasi-norms.
\end{theorem}

\begin{remark}
(i) It is a little bit surprising that the Besov-type spaces $\bt$ do not
form a scale of interpolation spaces for the real method.
However, in case $\tau =0$, we get back the following well-known formulas
that, for $\tz\in(0,1)$, $s_0,s_1\in\rr$, $s_0<s_1$, $p\in(0,\fz)$,
$q,q_0,q_1\in(0,\fz]$ and $s:=(1-\tz)s_0+\tz s_1$,
\begin{eqnarray*}
B^{s}_{p,q} (\rn)= \cn^{s}_{p,p,q} (\rn) & = &
(B^{s_0,0}_{p,q_0}(\rn),B^{s_1,0}_{p,q_1}(\rn))_{\Theta,q} = (F^{s_0,0}_{p,q_0}(\rn),B^{s_1,0}_{p,q_1}(\rn))_{\Theta,q}
\\
& = & (B^{s_0,0}_{p,q_0}(\rn),F^{s_1,0}_{p,q_1}(\rn))_{\Theta,q} =
(F^{s_0,0}_{p,q_0}(\rn),F^{s_1,0}_{p,q_1}(\rn))_{\Theta,q}\, ;
\end{eqnarray*}
see \cite[Theorem 2.4.2]{t83}.

(ii) A proof of Theorem \ref{approx5} was given in \cite{s011a}.
\end{remark}

Now we turn to the real interpolation of Besov-Morrey  and Triebel-Lizorkin-Morrey spaces.

\begin{theorem}\label{approx7}
Let  $\Theta \in(0, 1)$, $s_0,s_1 \in \rr$, $s_0 < s_1$, $0 < p \le u \le \infty$
and $q,q_0,q_1\in(0,\infty]$.
Let $s:= (1-\Theta)\, s_0 + \Theta \, s_1$,
$A,\ca \in \{\ce,\cn\}$ ($u < \infty$ if either $A=\ce$ or $\ca= \ce$).
Then
\[
\cn^{s}_{u,p,q} (\rn) =
(A^{s_0}_{u,p,q_0}(\rn),\ca^{s_1}_{u,p,q_1}(\rn))_{\Theta,q}
\]
in the sense of equivalent quasi-norms.
\end{theorem}

\begin{remark}
(i) Kozono and Yamazaki \cite{KY} already
considered the real interpolation of Besov-Morrey spaces. They
proved that
\[
(\cn^{s_0}_{u,p,q_0}(\rn), \cn^{s_1}_{u,p,q_1}(\rn))_{\tz,q}=\cn^s_{u,p,q}(\rn)
\]
under the restrictions  $\tz\in(0,1)$, $1 < p\le u\le\fz$, $s_0,\ s_1\in\rr$, $s_0<s_1$, $s=(1-\tz)s_0+\tz s_1$ and
$q,\ q_0,\ q_1\in[1,\fz]$; see also  Sawano, Tanaka \cite{sat}. This has been supplemented
by Mazzucato in \cite[Proposition 2.7]{ma03} that, for $s\in\rr$, $\tz\in(0,1)$, $1<p\le u<\fz$ and
$q,\ q_0,\ q_1\in[1,\fz]$,
\[
(\cn^{s}_{u,p,q_0}(\rn), \cn^{s}_{u,p,q_1}(\rn))_{\tz,q}=\cn^s_{u,p,q}(\rn)
\]
if, in addition, $1/q=(1-\tz)/q_0 +\tz/q_1$.

(ii) Mazzucato in \cite[Proposition 4.12]{ma03} proved that
\[
\cn^{s}_{u,p,q} (\rn) =
(\ce^{s_0}_{u,p,q}(\rn),\ce^{s_1}_{u,p,q}(\rn))_{\Theta,q}
\]
with $\Theta \in(0, 1)$, $s_0,s_1 \in \rr$, $s_0 < s_1$, $s=(1-\tz)s_0+\tz s_1$, $1< p \le u < \infty$
and $q\in[1,\infty]$.

(iii) A proof of Theorem \ref{approx7} was also given in \cite{s011a}.
\end{remark}

There is one special case of particular interest. We recall that $\cn^{s}_{u,p,\infty} (\rn) = B^{s,\tau}_{p,\infty}(\rn)$,
$\tau:= \frac 1p - \frac 1u$ (see \cite[Corollary 3.3]{ysy}).

\begin{corollary}\label{approx6}
Let  $\Theta \in(0, 1)$, $s_0,s_1 \in \rr$, $s_0 < s_1$,
$p \in(0,\fz]$, $\tau \in[0, 1/p)$ and $ q,q_0,q_1 \in(0,\fz]$.
Let $s:= (1-\Theta)\, s_0 + \Theta \, s_1$ and
$A,\ca \in \{B,F\}$.
Then
\[
B^{s,\tau}_{p,\infty}(\rn) =
(A^{s_0,\tau}_{p,q_0}(\rn),\ca^{s_1,\tau}_{p,q_1}(\rn))_{\theta,\infty}
\]
in the sense of equivalent quasi-norms.
\end{corollary}

For completeness, we also treat the case $\tau \in [1/p,\fz)$.

\begin{corollary}\label{approx10}
Let  $\Theta \in(0, 1)$, $s_0,s_1 \in \rr$, $s_0 < s_1$,
$p\in(0,\fz]$,  and $q,q_0,q_1 \in(0,\fz]$.
Let either $q_i\in(0,\fz)$ and $\tau_i\in(1/p,\fz)$ or
$q_i=\fz$ and $\tau_i\in[1/p,\fz)$, $i\in\{0,1\}$.
Let
\[
s:= (1-\Theta)\, s_0 + \Theta \, s_1,
\, \quad \tau := (1-\Theta)\, \tau_0 + \Theta \, \tau_1
\qquad \mbox{and}\qquad
\frac 1p := \frac{1-\tz}{p_0} + \frac{\tz}{p_1}\, ,
\]
and $A,\ca \in \{B,F\}$.
Then
\[
B^{s+n\tau-n/p}_{\infty,q}(\rn) =
(A^{s_0,\tau_0}_{p,q_0}(\rn),\ca^{s_1,\tau_1}_{p,q_1}(\rn))_{\Theta,q}
\]
in the sense of equivalent quasi-norms.
\end{corollary}


\subsubsection{Real interpolation with different $p$}


We summarize some known results, due to Triebel \cite[Theorem~2.4.2/1]{t78}
($1 <p_0 < p_1<\fz$, $q \in (1,\infty]$), and Frazier, Jawerth \cite[Corollary~6.7]{fj90}.

\begin{proposition}
Let   $0 <p_0 < p_1<\fz$, $s \in \rr$, $q \in (0,\infty]$  and  $1/p:=(1-\Theta)/p_0+\Theta/p_1$.
Then
\begin{equation*}
F^{s}_{p,q}(\rn) =
(F^{s}_{p_0,q}(\rn),F^{s}_{p_1,q}(\rn))_{\Theta,p}
\end{equation*}
and
\begin{equation*}
F^{s,0}_{p_0/(1-\Theta),q}(\rn) =
(F^{s,0}_{p_0,q}(\rn), F^{s,1/p_1}_{p_1,q}(\rn))_{\Theta,p} \, .
\end{equation*}
\end{proposition}

Finally we turn to the real interpolation of Morrey spaces themselves. Here we only consider embeddings.

\begin{lemma}\label{morreyreal}
Let $0 <p_i\le u_i<\fz$, $i\in\{0,1\}$, $1/p:=(1-\Theta)/p_0+\Theta/p_1$ and
$1/u:=(1-\Theta)/u_0+\Theta/u_1$.

{\rm (i)} It always holds true that
\begin{equation*}
\lf(\cm^{u_0}_{p_0}(\rn), \cm^{u_1}_{p_1}(\rn)\r)_{\tz,p}
\hookrightarrow \cm^{u}_{p}(\rn) \, .
\end{equation*}

{\rm (ii)} Let $\min\{p_0,p_1\}>1$.  Then
\begin{equation*}
\cm^{u}_{p}(\rn) \hookrightarrow \lf(\cm^{u_0}_{p_0}(\rn), \cm^{u_1}_{p_1}(\rn)\r)_{\tz,\infty}
\end{equation*}
holds true if and only if  $p_0 /u_0 = p_1/u_1$.

{\rm (iii)} If $p_0 = p_1$, then
\begin{equation}\label{w-52}
\lf(\cm^{u_0}_{p_0}(\rn), \cm^{u_1}_{p_1}(\rn)\r)_{\tz,\infty}\hookrightarrow \cm^{u}_{p}(\rn) \, .
\end{equation}
\end{lemma}

The embedding in Lemma  \ref{morreyreal}(i)
with $\min\{p_0,p_1\}\ge 1$ has been known for some time, we refer the reader to
 Mazzucato \cite{ma03}, Lemari{\'e}-Rieussiet \cite{LR} and Sickel \cite{s011a}.
Lemma  \ref{morreyreal}(ii) is taken from Lemari{\'e}-Rieussiet \cite{LR}. Concerning Lemma  \ref{morreyreal}(iii), we wish to mention that,
in case $\min\{p_0,p_1\}\ge 1$, Lemari{\'e}-Rieussiet \cite{LR} has proved a  sharper version, in which the necessity of  $p_0 = p_1$  for the validity of the embedding \eqref{w-52} has been shown.
Whereas Lemma \ref{morreyreal} is useful, the next statement is instructive
to what concerns the limitations
of the real method in our context.

\begin{theorem}\label{morreyreal2}
Let $1\le p_i \le  u_i<\fz$, $i\in\{0,1\}$,  $1/p:=(1-\Theta)/p_0+\Theta/p_1$ and
$1/u:=(1-\Theta)/u_0+\Theta/u_1$.
Then, for all $q\in(0,\infty]$,
\begin{equation*}
\lf(\cm^{u_0}_{p_0}(\rn), \cm^{u_1}_{p_1}(\rn)\r)_{\tz,q} \neq  \cm^{u}_{p}(\rn) \, ,
\end{equation*}
except the trivial cases consisting in
\begin{enumerate}
\item[{\rm(a)}] $p_0 = p_1$ and $u_0 = u_1$, or
\item[{\rm(b)}] $p_0 = u_0$, $p_1 = u_1$ and $q=p$.
\end{enumerate}
\end{theorem}

We now turn to Besov-Morrey spaces. Here we are going to use  the following conclusion,
whose Banach version was proved in \cite[Theorem 5.6.2]{BL}
(see also \cite[1.18.1, 1.18.2]{t78}).

\begin{lemma}\label{lem-real}
 Assume that $s_0,s_1\in\rr$, $p_0,p_1\in(0,\fz)$, $s:=(1-\tz)s_0+\tz s_1$
and $1/p:=(1-\tz)/p_0+\tz/p_1$. Let $(X_0,X_1)$ be an interpolation couple of quasi-Banach spaces. Then
$$(\ell^{s_0}_{p_0}(X_0), \ell^{s_1}_{p_1}(X_1))_{\tz,p}=\ell^s_p((X_0,X_1)_{\tz,p}).$$
\end{lemma}

This lemma, applied with $X_0:=\cm^{u_0}_{p_0}(\rn)$ and $X_1:=\cm^{u_1}_{p_1}(\rn)$,
taking into account  Lemma \ref{morreyreal}(i) and \ref{lem-real}, yields the following statement.

\begin{proposition}\label{p-real}
Let $s_i\in\rr$, $0< p_i\le u_i<\fz$, $q_i \in (0,\infty]$, $i\in\{0,1\}$,  $s:=(1-\tz)s_0+\tz s_1$, $1/p:=(1-\Theta)/p_0+\Theta/p_1$,
$1/q:=(1-\Theta)/q_0+\Theta/q_1$ and
$1/u:=(1-\Theta)/u_0+\Theta/u_1$.
Then
\begin{equation*}
\lf(\cn^{s_0}_{u_0,p_0,p_0}(\rn), \cn^{s_1}_{u_1,p_1,p_1}(\rn)\r)_{\tz,p}
\hookrightarrow \cn^{s}_{u,p,p}(\rn).
\end{equation*}
\end{proposition}


\subsection{The interpolation property}
\label{sum}


The  aim of this subsection consists in a collection of the consequences of the
previously obtained  results for the interpolation property of linear operators.

Before turning to these results, we
would like to give a comment on the positive interpolation results obtained so far.
As we have seen in Subsections \ref{cpr} through \ref{inter1d},
positive results were always connected with the restriction
$p_0 /u_0 = p_1/u_1$ (or $\tau_0 \, p_0 = \tau_1 \, p_1$).
There is an  explanation which we learned from Lemari{\'e}-Rieussiet
(a personal communication with the second author)
as follows. The condition $p_0 /u_0 = p_1/u_1$
characterizes those pairs of Morrey spaces  $(\cm^{u_0}_{p_0}(\rn),
\cm^{u_1}_{p_1}(\rn))$ which are connected by
a bijection. More precisely, the mapping
\[
T_\delta:\ f \mapsto (\arg f) \, |f|^\delta
\]
is a bijection from $\cm^{u_0}_{p_0}(\rn)$ onto $\cm^{u_1}_{p_1}(\rn)$ if $\delta = p_0 /p_1= u_0 /u_1$.
Hence, positive interpolation results were only obtained
within a scale of images $\{T_\delta (\cm^{u_0}_{p_0}(\rn))\}_{\delta >0}$ of a fixed Morrey space $\cm^{u_0}_{p_0}(\rn)$.
The second Morrey space $\cm^{u_1}_{p_1}(\rn)$ has to belong to this scale
and, as a result, the interpolation space will belong to as well (for some methods, not all).
Without this bijectivity, we do not know any positive results.

We continue by recalling the most prominent statement concerning the interpolation property in the framework
of Morrey spaces.

\begin{lemma}\label{help}
Let $\tz \in (0,1)$, $0 < p_0 \le u_0 < \infty$, $0 < p_1\le u_1  <\infty$ and define
\[
\frac 1p := \frac{1-\tz}{p_0} + \frac{\tz}{p_1}
\qquad \mbox{and}\qquad \frac 1u := \frac{1-\tz}{u_0} + \frac{\tz}{u_1} \, .
\]
Let $X_0,X_1$ be an interpolation couple of quasi-Banach spaces and
$F$  an interpolation functor of exponent $\tz$ such that
\begin{equation}\label{morrey20}
F(L_{p_0}(\rn),L_{p_1}(\rn) ) \hookrightarrow L_{p}(\rn)\, .
\end{equation}
If $T$ is a linear operator which is bounded from
$X_0$ to the Morrey space $\cm^{u_0}_{p_0}(\rn)$ with  norm
$M_0$ and from
$X_1$ to the Morrey space $\cm^{u_1}_{p_1}(\rn)$ with  norm
$M_1$, then $T$ is also bounded from
$F(X_0,X_1) $ to $\cm^{u}_{p}(\rn)$
and
\begin{equation*}
\|\, T\, \|_{F(X_0,X_1)  \to \cm^{u}_{p}(\rn)}
\le c\,  M_0^{1-\tz}\, M_1^\tz\, ,
\end{equation*}
where $c$ denotes a positive constant independent of $T$, $M_0$ and $M_1$.
\end{lemma}

\begin{remark}\label{morrey24}
(i) Specializing
\[
X_0 := \cm^{u_0}_{p_0}(\rn) \qquad \mbox{and} \qquad X_1 := \cm^{u_1}_{p_1}(\rn)
\]
and choosing $T$ to be the identity I, we see that
\[
\|I\|_{F(\cm^{u_0}_{p_0}(\rn), \cm^{u_1}_{p_1}(\rn)) \to \cm^{u}_{p}(\rn)} \le c < \infty\, .
\]
In other words, for any functor of exponent $\tz$ such that \eqref{morrey20} holds
true, we have the continuous embedding
\begin{equation*}
 F(\cm^{u_0}_{p_0}(\rn), \cm^{u_1}_{p_1}(\rn)) \hookrightarrow \cm^{u}_{p}(\rn)\, .
\end{equation*}
Since $(\, \cdot \, , \, \cdot \,)_{\tz,p}$ and $[\, \cdot \, , \, \cdot \,]_{\tz}$
are functors satisfying \eqref{morrey20}, it follows that
\begin{equation*}
(\cm^{u_0}_{p_0}(\rn), \, \cm^{u_1}_{p_1}(\rn))_{\tz,p} \hookrightarrow \cm^{u}_{p}(\rn)
\ \ \mbox{and}\ \
[\cm^{u_0}_{p_0}(\rn), \, \cm^{u_1}_{p_1}(\rn)]_{\tz} \hookrightarrow \cm^{u}_{p}(\rn)
\end{equation*}
under the assumptions of Lemma \ref{help}.
Notice that  $\cm^{u_0}_{p_0}(\rn) + \cm^{u_1}_{p_1}(\rn)$ is lattice $r$-convex
for any  $r \in (0, \min\{1,p_0,p_1\}]$,
since $\cm^{u_0}_{p_0}(\rn)$ is lattice $r$-convex with $r\in (0, \min\{1,p_0\}]$  and  $\cm^{u_1}_{p_1}(\rn)$
 is lattice $r$-convex with $r\in (0, \min\{1,p_1\}]$.

(ii) Lemma \ref{help} is implicitly contained in Spanne \cite{sp66} and Peetre \cite{Pe76}.
An extension of this lemma to more general situations (such as Besov-type or Triebel-Lizorkin-type spaces)
would be highly desirable.
\end{remark}

Let $\tz\in(0,1)$, $s_i\in\rr$, $\tau_i\in[0,\fz)$,
$p_i, q_i\in(0,\fz]$ and $u_i\in[p_i,\fz]$, $i\in\{0,1\},$
such that $s:=(1-\tz)s_0+\tz s_1$, $\tau:=(1-\tz)\tau_0+\tz\tau_1$,
\[
\frac1p:=\frac{1-\tz}{p_0}+\frac\tz{p_1}\, , \quad \frac1q:=\frac{1-\tz}{q_0}+\frac\tz{q_1}\, \quad
\mbox{and} \quad \frac1u:=\frac{1-\tz}{u_0}+\frac{\tz}{u_1}\, .
\]
Let $(X_0,X_1)$ and $(Y_0,Y_1)$ be  quasi-Banach couples.
In what follows, $T: X\to Y$ means
that  $T$ is a linear bounded operator from $X$ and $Y$.
The following interpolation properties of linear operators on
smoothness function spaces built on Morrey spaces are obtained in
this article:
\begin{enumerate}
 \item[(a)] Besov-type and Triebel-Lizorkin-type spaces.
In addition, we assume $\tau_0 \, p_0 = \tau_1\, p_1$.
Then, from Proposition \ref{gustav} and Theorem \ref{COMI},
we deduce that, with $A \in \{B,F\}$,
\begin{eqnarray*}
&&T:\left\{
\begin{array}
    {l@{}l}
A_{p_0,q_0}^{s_0,\tau_0}(\rn)\longrightarrow  Y_0\\
A_{p_1,q_1}^{s_1,\tau_1}(\rn)\longrightarrow Y_1
\end{array}
\right.
\quad\Longrightarrow\quad T : A_{p,q}^{s,\tau}(\rn)\longrightarrow \laz Y_0,Y_1, \tz\raz;\\
&&T:\left\{
\begin{array}
    {l@{}l}
 X_0\longrightarrow A_{p_0,q_0}^{s_0,\tau_0}(\rn)\\
X_1\longrightarrow A_{p_1,q_1}^{s_1,\tau_1}(\rn)
\end{array}
\right.
\quad\Longrightarrow\quad T : \laz X_0,X_1, \tz\raz\longrightarrow A_{p,q}^{s,\tau}(\rn).
\end{eqnarray*}

\item[(b)] Besov-type and Triebel-Lizorkin-type spaces. This time we allow $\tau_0 \, p_0 \neq \tau_1\, p_1$ but require
$\tau_i\in[0,1/p_i)$, $i\in\{0,1\}$.
Let $X_0 + X_1$ be analytically convex.
Then, by Propositions \ref{morreyx} and \ref{complexinterpol}, and Remark \ref{inner1}, we find that, with  $A \in \{B,F\}$,
\begin{eqnarray*}
&&T:\left\{
\begin{array}
    {l@{}l}
 X_0\longrightarrow A_{p_0,q_0}^{s_0,\tau_0}(\rn)\\
X_1\longrightarrow A_{p_1,q_1}^{s_1,\tau_1}(\rn)
\end{array}
\right.
\quad\Longrightarrow\quad T : [X_0,X_1]_\tz^i \longrightarrow A_{p,q}^{s,\tau}(\rn).
\end{eqnarray*}

\item[(c)] Besov-Morrey and Triebel-Lizorkin-Morrey spaces.
In addition, we assume $p_0\, u_1=p_1\, u_0$.
Then, from Proposition \ref{gustav} and Theorem \ref{COMI},
we deduce the following, with $\ca \in \{\cn,\ce\}$,
\begin{eqnarray*}
&&T:\left\{
\begin{array}
    {l@{}l}
\ca_{u_0,p_0,q_0}^{s_0}(\rn)\longrightarrow  Y_0\\
\ca_{u_1,p_1,q_1}^{s_1}(\rn)\longrightarrow Y_1
\end{array}
\right.
\quad\Longrightarrow\quad T : \ca_{u,p,q}^{s}(\rn)\longrightarrow \laz Y_0,Y_1, \tz\raz;\\
&&T:\left\{
\begin{array}
    {l@{}l}
 X_0\longrightarrow \ca_{u_0,p_0,q_0}^{s_0}(\rn)\\
X_1\longrightarrow \ca_{u_1,p_1,q_1}^{s_1}(\rn)
\end{array}
\right.
\quad\Longrightarrow\quad T : \laz X_0,X_1, \tz\raz\longrightarrow \ca_{u,p,q}^{s}(\rn).
\end{eqnarray*}

\item[(d)] Besov-Morrey spaces. We do not require  $p_0\, u_1=p_1\, u_0$ in this case.
Let $X_0 + X_1$ be analytically convex.
Then, by Proposition \ref{morreyn}, Proposition \ref{complexinterpol} and Remark \ref{inner1},
we know that
\begin{eqnarray*}
&&T:\left\{
\begin{array}
    {l@{}l}
 X_0\longrightarrow \cn_{u_0,p_0,q_0}^{s_0}(\rn)\\
X_1\longrightarrow \cn_{u_1,p_1,q_1}^{s_1}(\rn)
\end{array}
\right.
\quad\Longrightarrow\quad T : [X_0,X_1]_\tz^i \longrightarrow \cn_{u,p,q}^{s}(\rn).
\end{eqnarray*}

\item[(e)] Besov-Morrey spaces,  Besov-type and Triebel-Lizorkin-type spaces.
Suppose $0 < \tau := \frac 1p - \frac 1u< 1/p$.
Then, from Theorem \ref{approx5} and Proposition \ref{realbasic},
 it follows that, in case $s_0 \neq s_1$, it holds true that
\begin{eqnarray*}
&&T:\left\{
\begin{array}
    {l@{}l}
A_{p,q_0}^{s_0,\tau_0}(\rn)\longrightarrow  Y_0\\
\ca_{p,q_1}^{s_1,\tau_1}(\rn)\longrightarrow Y_1
\end{array}
\right.
\quad\Longrightarrow\quad T : \cn_{u,p,r}^{s}(\rn)\longrightarrow  (Y_0,Y_1)_{\tz,r}
\end{eqnarray*}
and
\begin{eqnarray*}
&&T:\left\{
\begin{array}
    {l@{}l}
 X_0\longrightarrow A_{p,q_0}^{s_0,\tau_0}(\rn)\\
X_1\longrightarrow \ca_{p,q_1}^{s_1,\tau_1}(\rn)
\end{array}
\right.
\quad\Longrightarrow\quad T : (X_0,X_1)_{\tz,r} \longrightarrow \cn_{u,p,r}^{s}(\rn)
\end{eqnarray*}
with arbitrary $r\in (0,\infty]$ and $A,\ca \in \{B,F\}$.

\item[(f)] Besov-Morrey spaces,  Besov-type and Triebel-Lizorkin-type spaces.
Then, by Theorem \ref{approx7} and Proposition \ref{realbasic}, we know that,
in case $s_0 \neq s_1$, it holds true that
\begin{eqnarray*}
&&T:\left\{
\begin{array}
    {l@{}l}
A_{u,p,q_0}^{s_0}(\rn)\longrightarrow  Y_0\\
\ca_{u,p,q_1}^{s_1}(\rn)\longrightarrow Y_1
\end{array}
\right.
\quad\Longrightarrow\quad T : \cn_{u,p,r}^{s}(\rn)\longrightarrow  (Y_0,Y_1)_{\tz,r}\end{eqnarray*}
and
\begin{eqnarray*}
&&T:\left\{
\begin{array}
    {l@{}l}
 X_0\longrightarrow A_{u,p,q_0}^{s_0}(\rn)\\
X_1\longrightarrow \ca_{u,p,q_1}^{s_1}(\rn)
\end{array}
\right.
\quad\Longrightarrow\quad T : (X_0,X_1)_{\tz,r} \longrightarrow \cn_{u,p,r}^{s}(\rn)
\end{eqnarray*}
with arbitrary $r\in (0,\infty]$ and $A,\ca \in \{\cn, \ce\}$.

\item[(g)] Morrey spaces.
It follows, from
Proposition \ref{realbasic} and Lemma \ref{morreyreal}, that
\begin{itemize}
\item $T:\left\{
\begin{array}
    {l@{}l}
X_0\longrightarrow  \cm_{p_0}^{u_0}(\rn)\\
X_1\longrightarrow \cm_{p_1}^{u_1}(\rn)
\end{array}
\right.
\quad\Longrightarrow\quad T : (X_0,X_1)_{\Theta,p}\longrightarrow \cm_{p}^{u}(\rn)$;
\item
If $p_0=p_1=p$, then
\begin{eqnarray*}
T:\left\{
\begin{array}
    {l@{}l}
X_0\longrightarrow  \cm_{p}^{u_0}(\rn)\\
X_1\longrightarrow \cm_{p}^{u_1}(\rn)
\end{array}
\right.
\quad\Longrightarrow\quad T : (X_0,X_1)_{\Theta,\fz}\longrightarrow \cm_{p}^{u}(\rn);
\end{eqnarray*}
\item
Let $X_0 + X_1$ be analytically convex. Then
\begin{eqnarray*}
T:\left\{
\begin{array}
    {l@{}l}
X_0 \longrightarrow  \cm_{p_0}^{u_0}(\rn) \\
X_1 \longrightarrow  \cm_{p_1}^{u_1}(\rn)
\end{array}
\right.
\quad\Longrightarrow\quad T : [X_0,X_1]_{\Theta}^i \longrightarrow \cm_{p}^{u}(\rn)  ;
\end{eqnarray*}
see Theorem  \ref{morrey4}, Proposition \ref{complexinterpol} and Remark \ref{inner1}.
\end{itemize}
\end{enumerate}


\section{Interpolation of local spaces}
\label{inter2}


Very recently, Triebel in \cite{t12} (see also \cite{t14}) systematically  introduced and studied
two new scales of function spaces, $\cl^r B^s_{p,q} (\rn)$ and $\cl^r F^s_{p,q} (\rn)$,
which were called \emph{local} (or \emph{Morreyfied}) \emph{spaces}.
The original definition of these spaces
relies on the
appropriate wavelet decomposition of the distributions under consideration; see \cite[1.3.1]{t12}.
Later on in \cite{ysy2}, it was proved that the local spaces $\cl^r B^s_{p,q} (\rn)$ and $\cl^r F^s_{p,q} (\rn)$ coincide
with the uniform spaces of the scales
$\bt$ and $\ft$, respectively. By this reason, we skip
the original definition of the local spaces here and
deal with the equivalent description of these classes
as localized variants of $\bt$ and $\ft$.

\begin{definition}\label{dunif}
Let $\Psi$ be a non-negative smooth function in $\rn$ with compact support such that $\Psi (0) >0$. Let $A \in \{B,F\}$,
$s\in\rr$, $\tau\in[0,\fz)$ and $p,q\in(0,\fz]$ ($p\in(0,\fz)$ if $A=F$).
The \emph{uniform space} $A^{s,\tau}_{p,q,{\rm unif}}(\rn)$ is defined as the space of
all $f\in\cs'(\rn)$ such that
$$\|f\|_{A^{s,\tau}_{p,q,{\rm unif}}(\rn)}:=\sup_{\ell\in\zz^n} \|\Psi(\cdot-\ell) f(\cdot)\|_{\at}<\fz.$$
\end{definition}

It was proved in \cite{ysy,ysy2} that the spaces
$A^{s,\tau}_{p,q,{\rm unif}}(\rn)$ are quasi-Banach
spaces independent of the choice of $\Psi$
(in the sense of equivalent quasi-norms).
Obviously, $\at \hookrightarrow \atu$.
In addition, one knows, from \cite{ysy2}, that
\[\at=A^{s,\tau}_{p,q,{\rm unif}}(\rn) \qquad \Longleftrightarrow \qquad
\tau\in[1/p,\fz) \, .
\]

The main result of \cite{ysy2} is the following identification.

\begin{proposition}
Let $s\in\rr$, $\tau\in[0,\fz)$ and $p,q\in(0,\fz]$ ($p\in(0,\fz)$ if $A=F$). Then
\[
\atu=\cl^{n(\tau-1/p)}A^s_{p,q}(\rn)
 \]
in the sense of  equivalent quasi-norms.
\end{proposition}

Parallel to the nonlocal situation  one can prove the following.

\begin{theorem}\label{COMI-u}
Let $\Theta,\, s,\, s_0,\, s_1,\, \tau,\, \tau_0,\,\tau_1,\, p,\, p_0,\, p_1,\, q,\, q_0,\, q_1,\, u,\, u_0,\, u_1$
be as in Theorem \ref{COMI} ($p,\, p_0,\, p_1\in(0,\fz)$ if $A=F$).
If $\tau_0 \, p_0 = \tau_1\, p_1$, then
\begin{equation*}
\lf\laz A_{p_0,q_0,{\rm unif}}^{s_0,\tau_0}(\rn), A_{p_1,q_1,{\rm unif}}^{s_1,\tau_1}(\rn),\tz\r\raz=A_{p,q,{\rm unif}}^{s,\tau}(\rn).
\end{equation*}
\end{theorem}

\begin{remark}
 One can expect that a certain part of the theory, developed in Section 2, carries over to these local spaces,
 which will not be dealt with in this article.
For brevity, we concentrate on the most important results here.
\end{remark}


\section{Proofs}
\label{proofs}


In this section, we give proofs of the results stated in
Sections 2 and 3.


\subsection{Proofs of results in Subsection \ref{cpr}}
\label{proof0}


\begin{proof}[Proof of Theorem \ref{morrey1}(iii)]
Under the restrictions $1 < p_i\le u_i < \infty$, $i\in\{0,1\}$, and  $u_0 p_1 \neq u_1 p_0$,
Lemari{\'e}-Rieusset \cite{LR} constructed a family of
fractal sets, $K_m^\beta$, with $\bz\in[0,n)$ and $m\in \nn$, such that the associated
family of positive linear operators $T_m: \cm_{p_i}^{u_i} (\rn) \to \rr$,
$i\in\{0,1\}$,
\[
T_m f := \int_{K_m^\beta} f(x)\, dx\, ,
\]
has the property
\begin{equation}\label{ws-99}
\sup_{m \in \nn} \frac{\|\,  T_m\, \|_{\cm_{p}^u (\rn) \to \rr}}{\|\,  T_m\, \|_{\cm_{p_0}^{u_0} (\rn) \to \rr}^{1-\Theta}\,
 \|\,  T_m\, \|_{\cm_{p_1}^{u_1} (\rn) \to \rr}^\Theta} = \infty
\end{equation}
under some conditions on $\beta$.
By Proposition \ref{cpi}, this implies that
\[
\lf[\cm_{p_0}^{u_0}(\rn)\r]^{1-\tz}
\lf[\cm_{p_1}^{u_1}(\rn)\r]^\tz \subsetneqq \cm_{p}^{u}(\rn)
\]
under the extra condition $\min \{p_0,p_1\}>1$.

We claim that the above restriction $\min \{p_0,p_1\}>1$ can be removed by studying the mapping $f \mapsto |f|^\delta$, $\delta \in(0,\fz)$.
To see this, we argue by contradiction.  Our assumption consists in
\[
\lf[\cm_{p_0}^{u_0}(\rn)\r]^{1-\tz}
\lf[\cm_{p_1}^{u_1}(\rn)\r]^\tz = \cm_{p}^{u}(\rn)
\]
for some $p_0,p_1$ such that  $\min \{p_0,p_1\}=p_0 \le 1$.
We are going to use the following observations:\\
(a) A function $f$ belongs to $\cm_p^u (\rn)$ if and only if $|f|$ belongs to $\cm_p^u (\rn)$.
\\
(b) If $f$ belongs to $\cm_p^u (\rn)$, then $|f|^\delta \in \cm_{p/\delta}^{u/\delta}(\rn)$ and
\[
\|\, |f|^\delta \,\|_{\cm_{p/\delta}^{u/\delta}(\rn)} = \|\, f \,\|_{\cm_{p}^{u}(\rn)} \, .
\]
Indeed, by (a) and (b), we see that
$T: \, f \mapsto (\arg f) \,|f|^\delta$ is a bijection with respect to the pair
$(\cm_{p}^{u}(\rn), \cm_{p/\delta}^{u/\delta}(\rn))$.
Now, we choose $\delta <p_0$.
Then it is easy to see that
\[
T \lf(\lf[\cm_{p_0}^{u_0}(\rn)\r]^{1-\tz}\lf[\cm_{p_1}^{u_1}(\rn)\r]^\tz\r)
= \lf[\cm_{p_0/\delta}^{u_0/\delta}(\rn)\r]^{1-\tz} \lf[\cm_{p_1/\delta}^{u_1/\delta}(\rn)\r]^\tz \, .
\]
Since $T(\cm_{p}^{u}(\rn)) =  \cm_{p/\delta}^{u/\delta}(\rn)$,
 we obtain
$$\cm_{p/\delta}^{u/\delta}(\rn)
= [\cm_{p_0/\delta}^{u_0/\delta}(\rn)]^{1-\tz} [\cm_{p_1/\delta}^{u_1/\delta}(\rn)]^\tz,$$
but this is in conflict with the above known result for
$\min \{p_0,p_1\}>1$, which completes the proof of Theorem \ref{morrey1}(iii).
\end{proof}


\subsection{Proofs of results in Subsection \ref{inter1a}}
\label{proof1}


First, we need to recall a few more notions; see, for example, \cite{n85}.

\begin{definition}
(i) Let $X$ be a quasi-Banach lattice and $p\in[1,\fz]$.
The \emph{$p$-convexification} of $X$, denoted by $X^{(p)}$,
is defined as follows: $x\in X^{(p)}$ if and only if
$|x|^p\in X$. For all $x\in X^{(p)}$, define
$$\|x\|_{X^{(p)}}:=\||x|^p\|_X^{1/p}.$$

(ii)
A quasi-Banach lattice $X$ is said to \emph{be of type
$\mathfrak{E}$ } if there exists an equivalent quasi-norm $|||\cdot|||_X$ on $X$
such that, for some $p\in[1,\fz]$,
$X^{(p)}$ is a Banach lattice in the norm $\|\, \cdot \, \|_{X^{(p)}}:=||||\cdot|^p|||_X^{1/p}$.
\end{definition}

Let $\delta\in(0,\min\{1,p,q\}]$.
Then it is easy to see that
$$\|t\|_{[\sat]^{(1/\delta)}}=\||t|^{1/\delta}\|_{\sat}^\delta= \|t\|_{a^{\delta(s+n/2)-n/2, \tau \delta}_{p/\delta,q/\delta}(\rn)}.$$
Since $p/\delta\ge1$ and $q/\delta\ge1$, we know that $[\sat]^{(1/\delta)}$ is a Banach lattice and
hence $\sat$ is of  type $\mathfrak{E}$. Similarly, the space $n^s_{u,p,q}(\rn)$ is
 also of  type $\mathfrak{E}$.

\begin{definition}
Let $X_0,X_1$ be a couple of quasi-Banach spaces and let $X$ be  an intermediate space
with respect to $X_0+X_1$. Then the \emph{Gagliardo closure} of $X$  with respect
to $X_0+X_1$, denoted by $X^{\sim}$, is defined as
the collection of all  $a\in X_0 + X_1$
such that there exists a sequence $\{a_i\}_{i\in\zz_+}\subset X$
satisfying $a_i\to a$ as $i\to\fz$ in $X_0+X_1$
and $\|a_i\|_X\le \lz$ for some $\lz<\fz$ and all $i\in\zz_+$. For all $a\in X^\sim$,
define
$\|a\|_{X^{\sim}}:=\inf \lz$.
\end{definition}

Our argument will be based on the following result of Nilsson \cite[Theorem 2.1]{n85}.

\begin{proposition}\label{t-n}
Let $X_0$ and $X_1$ be two quasi-Banach lattices
of type $\mathfrak{E}$ and $\Theta \in (0,1)$. Then
$$\laz X_0, X_1\raz_\Theta = (X_0^{1-\Theta}\, X_1^\Theta)^\#$$
and
\[
X_0^{1-\Theta}\, X_1^\Theta \hookrightarrow  \laz X_0,X_1,\Theta\raz
\hookrightarrow (X_0^{1-\Theta}\, X_1^\Theta)^{\sim}.
\]
\end{proposition}

As a preparation, we need the following result on sequence spaces which is of interest for its own.

\begin{theorem}\label{comi}
Let $\tz, \ s,\ s_0,\ s_1, \ \tau,\ \tau_0,\ \tau_1, \ p,\ p_0,\ p_1, \ q,\ q_0,\ q_1$
be as in Theorem \ref{COMI}.
If $\tau_0 \, p_0 = \tau_1 \, p_1$, then
\begin{eqnarray*}
\lf\laz a_{p_0,q_0}^{s_0,\tau_0}(\rn), a_{p_1,q_1}^{s_1,\tau_1}(\rn)\r\raz_\tz & = &
\lf(\lf[a_{p_0,q_0}^{s_0,\tau_0}(\rn)\r]^{1-\tz}
\lf[a_{p_1,q_1}^{s_1,\tau_1}(\rn)\r]^\tz\r)^\#
\\
& = & (a_{p_0,q_0}^{s_0,\tau_0}(\rn), \, a_{p_1,q_1}^{s_1,\tau_1}(\rn), \, a_{p,q}^{s,\tau}(\rn), \#)
\end{eqnarray*}
and
$$
\lf\laz a_{p_0,q_0}^{s_0,\tau_0}(\rn), a_{p_1,q_1}^{s_1,\tau_1}(\rn),\tz\r\raz=
\lf[a_{p_0,q_0}^{s_0,\tau_0}(\rn)\r]^{1-\tz}
\lf[a_{p_1,q_1}^{s_1,\tau_1}(\rn)\r]^\tz=a_{p,q}^{s,\tau}(\rn).$$
\end{theorem}

The first formula in Theorem \ref{comi} is a direct consequence of Propositions \ref{morrey2}
and \ref{t-n}.
To prove the second formula, by Proposition \ref{morrey2},  it suffices to show that
$[a_{p,q}^{s,\tau}(\rn)]^\sim = a_{p,q}^{s,\tau}(\rn)$,
which is the conclusion of the following lemma.

\begin{lemma}\label{gagliardo}
Under the same assumptions as in Theorem  \ref{COMI}(i),
the Gagliardo closure $(a_{p,q}^{s,\tau}(\rn))^\sim$ of  $a_{p,q}^{s,\tau}(\rn)$
with respect to
$a_{p_0,q_0}^{s_0,\tau_0}(\rn) + a_{p_1,q_1}^{s_1,\tau_1}(\rn)$
is given by  $a_{p,q}^{s,\tau}(\rn)$.
\end{lemma}

\begin{proof}
Clearly, under the given conditions, $a_{p,q}^{s,\tau}(\rn)$ is an intermediate space
with respect to the pair $(a_{p_0,q_0}^{s_0,\tau_0}(\rn), a_{p_1,q_1}^{s_1,\tau_1}(\rn))$; see
Proposition \ref{morrey2}.

Let $t\in [a_{p,q}^{s,\tau}(\rn)]^\sim$. Then
there exists a sequence $\{t_i\}_{i\in\zz_+}\subset \sat$
such that $t_i\to t$ as $i\to\fz$ in $a_{p_0,q_0}^{s_0,\tau_0}(\rn)+ a_{p_1,q_1}^{s_1,\tau_1}(\rn)$ and
$\|t_i\|_{\sat}\ls \|t\|_{(\sat)^\sim}$ for all $i\in\zz_+$.
Therefore, there exist $t_i^0\in a_{p_0,q_0}^{s_0,\tau_0}(\rn)$
and $t^1_i\in a_{p_1,q_1}^{s_1,\tau_1}(\rn)$ such that
$$t-t_i= t_i^0+t_i^1\, , \qquad i \in \zz_+\, , $$
and $$\|t_i^0\|_{a_{p_0,q_0}^{s_0,\tau_0}(\rn)}
+\|t_i^1\|_{a_{p_1,q_1}^{s_1,\tau_1}(\rn)}\ls
\|t-t_i\|_{a_{p_0,q_0}^{s_0,\tau_0}(\rn)+ a_{p_1,q_1}^{s_1,\tau_1}(\rn)}
\to0$$
as $i\to\fz$. Notice that, for all $Q\in\cq^*$,
$$|(t^0_i)_Q|\ls |Q|^{\frac{s_0}n+\frac12-\frac 1{p_0}+\tau_0}
\|t^0_i\|_{a_{p_0,q_0}^{s_0,\tau_0}(\rn)}\quad {\rm and}\quad
|(t^1_i)_Q|\ls |Q|^{\frac{s_1}n+\frac12-\frac 1{p_1}+\tau_1}
\|t^0_i\|_{a_{p_1,q_1}^{s_1,\tau_1}(\rn)}.$$
We then know that
$(t^0_i)_Q\to0$ and $(t^1_i)_Q\to0$ as $i\to\fz$. Hence
$(t_i)_Q\to t_Q$ as $i\to\fz$.
By the Fatou Lemma, we find that
\begin{eqnarray*}
\|t\|_{\sat}=\lf\|\lf\{\lim_{i\to\fz} (t_i)_Q \r\}_{Q\in\cq^*}\r\|_{\sat}\le
\liminf_{i\to\fz} \|t_i\|_{\sat}\ls \|t\|_{(\sat)^\sim}\, ,
\end{eqnarray*}
which implies that
$[a_{p,q}^{s,\tau}(\rn)]^\sim \hookrightarrow  a_{p,q}^{s,\tau}(\rn)$.
This, combined with Proposition \ref{morrey2}, further shows that
$[a_{p,q}^{s,\tau}(\rn)]^\sim= a_{p,q}^{s,\tau}(\rn)$ in the sense of equivalent
quasi-norms, which completes the proof of Lemma \ref{gagliardo}.
\end{proof}

Replacing Proposition \ref{morrey2} by
Proposition \ref{morrey3}, via an  argument similar to the proof of Theorem \ref{comi},
we obtain the following result on the sequence space  $n^s_{u,p,q}(\rn)$,
the details being omitted.

\begin{theorem}\label{comi-bm}
Let $\tz, \ s,\ s_0,\ s_1,  \ p,\ p_0,\ p_1, \ q,\ q_0,\ q_1, \ u, \ u_0, \ u_1$
be as in Theorem \ref{COMI}.
If $p_0u_1=p_1u_0,$
then
\begin{eqnarray*}
\lf\laz n_{u_0,p_0,q_0}^{s_0}(\rn), n_{u_1,p_1,q_1}^{s_1}(\rn)\r\raz_\tz & = &
\lf(\lf[n_{u_0,p_0,q_0}^{s_0}(\rn)\r]^{1-\tz}
\lf[n_{u_1,p_1,q_1}^{s_1}(\rn)\r]^\tz\r)^\#
\\
& = & (n_{u_0, p_0,q_0}^{s_0}(\rn), \, n_{u_1,p_1,q_1}^{s_1}(\rn), \, n_{u,p,q}^{s}(\rn), \, \#)
\end{eqnarray*}
and
$$
\lf\laz n_{u_0,p_0,q_0}^{s_0}(\rn), n_{u_1,p_1,q_1}^{s_1}(\rn),\tz\r\raz=
\lf[n_{u_0,p_0,q_0}^{s_0}(\rn)\r]^{1-\tz}
\lf[n_{u_1,p_1,q_1}^{s_1}(\rn)\r]^\tz=n_{u,p,q}^{s}(\rn).$$
\end{theorem}

Applying Theorems \ref{comi} and \ref{comi-bm}, together with
Propositions \ref{shest} and \ref{shest2}, we have the following conclusion,
the details being omitted.

\begin{corollary}\label{c-cis}
Let $\tz, \ s,\ s_0,\ s_1, \ \tau,\ \tau_0,\ \tau_1, \ p,\ p_0,\ p_1, \ q,\ q_0,\ q_1, \ u, \ u_0, \ u_1$
be as in Theorem \ref{COMI}.

{\rm (i)} If   $\tau_0\, p_0 = \tau_1\, p_1$, then
\begin{eqnarray*}
\lf\laz a_{p_0,q_0}^{s_0,\tau_0}(\rn), a_{p_1,q_1}^{s_1,\tau_1}(\rn)\r\raz_\tz
&=& \lf[a_{p_0,q_0}^{s_0,\tau_0}(\rn), a_{p_1,q_1}^{s_1,\tau_1}(\rn)\r]^i_\tz
\\
&=& (a_{p_0,q_0}^{s_0,\tau_0}(\rn), \, a_{p_1,q_1}^{s_1,\tau_1}(\rn), \, a_{p,q}^{s,\tau}(\rn), \, \#)\,.
\end{eqnarray*}

{\rm (ii)}
If $\min \{p_0,p_1,q_0,q_1\}\ge 1$
and  $p_0\, u_1 = p_1\, u_0,$ then
\begin{eqnarray*}
\lf\laz n_{u_0,p_0,q_0}^{s_0}(\rn), n_{u_1,p_1,q_1}^{s_1}(\rn)\r\raz_\tz
& = & \lf[n_{u_0,p_0,q_0}^{s_0}(\rn), n_{u_1,p_1,q_1}^{s_1}(\rn)\r]^i_\tz
\\
& = & (n_{u_0, p_0,q_0}^{s_0}(\rn), \, n_{u_1,p_1,q_1}^{s_1}(\rn), \, n_{u,p,q}^{s}(\rn), \, \#).
\end{eqnarray*}
\end{corollary}


\subsection*{Proof of Theorem \ref{COMI}}


Theorem \ref{COMI}{\rm(i)} follows from Theorem \ref{comi} in combination with
Proposition \ref{wave1}. Indeed, by  Proposition \ref{wave1}, we know the existence of a homeomorphism $R: ~\at \to \sat$
(for all parameters $s,\ \tau,\ p,\ q$).
Then, for any $f\in \at$, we obtain $R(f)\in \sat$ and
$$\|R(f)\|_{\sat}\ls \|f\|_{\at}.$$
Moreover, by Theorem \ref{comi}, there  exists a sequence $\{t_i\}_{i\in\zz}
\subset a^{s_0,\tau_0}_{p_0,q_0}(\rn)\cap a^{s_1,\tau_1}_{p_1,q_1}(\rn)$ such that $R(f)=\sum_{i\in\zz}\, t_i$ with convergence
in $a^{s_0,\tau_0}_{p_0,q_0}(\rn)+a^{s_1,\tau_1}_{p_1,q_1}(\rn)$
and, for any finite subset $F\subset \zz$ and any bounded sequence $\{\varepsilon_i\}_{i\in\zz}\subset\cc$,
$$
\lf\|\sum_{i\in F} \varepsilon_i \, 2^{i(j-\Theta)}\, t_i\r\|_{a^{s_j,\tau_j}_{p_j,q_j}(\rn)}\ls
\|f\|_{\at} \sup_{i\in\zz}|\varepsilon_i|,\quad j\in\{0,1\}.
$$
Notice that $f_i:=R^{-1}(t_i)\in A^{s_0,\tau_0}_{p_0,q_0}(\rn)\cap A^{s_1,\tau_1}_{p_1,q_1}(\rn)$ for all $i\in\zz$ and $f=\sum_{i\in\zz} R^{-1}(t_i)$
in $A^{s_0,\tau_0}_{p_0,q_0}(\rn)+A^{s_1,\tau_1}_{p_1,q_1}(\rn)$.
Moreover, by Proposition \ref{gustav}, we know that
$$\lf\|\sum_{i\in F} \varepsilon_i \, 2^{i(j-\Theta)}\, R^{-1}(t_i)\r\|_{A^{s_j,\tau_j}_{p_j,q_j}(\rn)}\ls \|f\|_{\at} \sup_{i\in\zz}|\varepsilon_i|,\quad j\in\{0,1\}.$$
This implies   that
$\at \hookrightarrow\laz A^{s_0,\tau_0}_{p_0,q_0}(\rn), A^{s_1,\tau_1}_{p_1,q_1}(\rn),\tz\raz.$
The reverse embedding $$\laz A^{s_0,\tau_0}_{p_0,q_0}(\rn), A^{s_1,\tau_1}_{p_1,q_1}(\rn),\tz\raz\hookrightarrow \at$$
follows from a similar argument to above, the details being omitted.

Similarly,
Theorem \ref{COMI}{\rm(ii)} is a consequence of Theorem \ref{comi-bm} in combination with
Proposition \ref{wave2}, the details being omitted.


\subsection*{Proof of Corollary \ref{cima}}


{\em Step 1.} Proof of (i).
Recall that $F^{0,1/p-1/u}_{p,2}(\rn) = \cm_p^u(\rn)$ if
$1<p\le u<\fz$ (see Mazzucato \cite{ma01} and  Sawano \cite{sawch}).
Hence, if $p_0, p_1\in(1,\infty)$,
then Corollary \ref{cima}(i) is a direct consequence of Theorem  \ref{COMI}(i) with $A=F$.
Furthermore, by the proof of Theorem  \ref{COMI},
we know that, in this case, the Gagliardo
closure of $\cm_p^u(\rn)$ in
$\cm_{p_0}^{u_0}(\rn)+\cm_{p_1}^{u_1}(\rn)$ is $\cm_p^u(\rn)$ itself.

Let us turn to the case  $\min \{p_0,p_1\}\le 1$.
We claim that, also in this case, the Gagliardo closure of
$\cm_p^u(\rn)$ coincides with $\cm_p^u(\rn)$.
Indeed, let $f$ belong to the Gagliardo
closure of $\cm_p^u(\rn)$ in
$\cm_{p_0}^{u_0}(\rn)+\cm_{p_1}^{u_1}(\rn)$.
Since $f\in \cm^u_p(\rn)$ if and only if $|f|\in \cm^u_p(\rn)$, we only need to prove that $|f|\in \cm^u_p(\rn)$.
We know that there exists a sequence $\{f_j\}_j\subset \cm_{p_0}^{u_0}(\rn) + \cm_{p_1}^{u_1}(\rn)$ such that
\[
\lim_{j\to \infty} \|\, f - f_j\, \|_{\cm_{p_0}^{u_0}(\rn) + \cm_{p_1}^{u_1}(\rn)} =0
\qquad \mbox{and} \qquad \sup_{j \in \nn} \, \|\, f_j\, \|_{\cm_{p}^{u}(\rn)} <\infty\, .
\]
Since $||f|-|f_j||\le |f-f_j|$, we conclude that
\[
\lim_{j\to \infty} \|\, |f| - |f_j|\, \|_{\cm_{p_0}^{u_0}(\rn) + \cm_{p_1}^{u_1}(\rn)} =0 \, .
\]
Hence $|f|$ also belongs to the Gagliardo
closure of $\cm_p^u(\rn)$ in
$\cm_{p_0}^{u_0}(\rn)+\cm_{p_1}^{u_1}(\rn)$
and it can be approximated by $\{|f_j|\}_j$ in the quasi-norm $\|\cdot\|_{\cm_{p_0}^{u_0}(\rn) + \cm_{p_1}^{u_1}(\rn)}$.

Choose $\delta < \min \{p_0,p_1\}\le 1$. Notice that, for any $g$,
\[
\|\, g\, \|_{\cm_p^u(\rn)}^\delta = \lf\| \, |g|^\delta \, \r\|_{\cm_{p/\delta}^{u/\delta}(\rn)}.
\]
This yields $|f_j|^\delta\in \cm_{p/\delta}^{u/\delta}(\rn)$.
Since $||f|^\delta-|f_j|^\delta|\le |f-f_j|^\delta$, we find that
\begin{eqnarray*}
\lf\|\, |f|^\delta - |f_j|^\delta\, \r\|_{\cm_{p_0/\delta}^{u_0/\delta}(\rn) + \cm_{p_1/\delta}^{u_1/\delta}(\rn)}
&&\le \lf\|\, |f-f_j|^\delta\, \r\|_{\cm_{p_0/\delta}^{u_0/\delta}(\rn) + \cm_{p_1/\delta}^{u_1/\delta}(\rn)}\\
&&\le \|\,f - f_j\,\|_{\cm_{p_0}^{u_0}(\rn) + \cm_{p_1}^{u_1}(\rn)}^\delta \to 0
\end{eqnarray*}
as $j\to\fz.$ Hence, $|f|^\delta$
belongs to the Gagliardo
closure of $\cm^{u/\delta}_{p/\delta}(\rn)$ in
$\cm_{p_0/\delta}^{u_0/\delta}(\rn) + \cm_{p_1/\delta}^{u_1/\delta}(\rn)$.
Then, by the choice of $\delta$ and the above
known results for $\min\{p_0,p_1\}>1$,  we
conclude that $|f|^\delta\in \cm^{u/\delta}_{p/\delta}(\rn)$
and hence $|f|\in \cm^{u}_{p}(\rn)$. This proves that
the Gagliardo
closure of $\cm^{u/\delta}_{p/\delta}(\rn)$ in
$\cm_{p_0/\delta}^{u_0/\delta}(\rn) + \cm_{p_1/\delta}^{u_1/\delta}(\rn)$ is a subspace of $\cm^u_p(\rn).$
Now Corollary \ref{cima}(i) follows from Proposition \ref{t-n} in combination with Theorem  \ref{morrey1}(ii).

\noindent{\em Step 2.} Proof of (ii).
We may argue as in the proof of Theorem \ref{morrey1}(iii).
First, we need to add a comment.
By using the notation as in the proof of Theorem \ref{morrey1}(iii), we claim that
\begin{equation}\label{ws-100}
\sup_{m \in \nn} \frac{\|\,  T_m\, \|_{\cm_{p}^u (\rn) \to \rr}}{\max \{\|\,  T_m\, \|_{\cm_{p_0}^{u_0} (\rn) \to \rr}\, ,
 \|\,  T_m\, \|_{\cm_{p_1}^{u_1} (\rn) \to \rr}\}} = \infty
\end{equation}
under some conditions on $\beta\in[0,n)$ (see \eqref{ws-99}).
Without loss of generality, we may assume $p_0/u_0 < p_1/u_1$.
Then, by choosing $\beta = n(1-p/u)$ as Lemari{\'e}-Rieusset \cite{LR} did,
we find that
\begin{eqnarray*}
\|\,T_m\,\|_{\cm_{p}^{u} (\rn) \to \rr}& \ge& C\, 2^{-m(n-\beta)(1-1/p)}\,,\\
 \|\,  T_m\, \|_{\cm_{p_0}^{u_0} (\rn) \to \rr} & \le &  C \, 2^{-m(n-\beta)(1-1/p_0)}\, 2^{m(\beta/p_0 + n/u_0-n/p_0)}\, ,
\\
 \|\,  T_m\, \|_{\cm_{p_1}^{u_1} (\rn) \to \rr} &  \le &  C \, 2^{-m(n-\beta)(1-1/p_1)}
\end{eqnarray*}
for some positive constant $C$ independent of $m$; see \cite[Section 6]{LR}.
In case $p< p_1$, we have
\begin{equation}\label{ws-102}
\lim_{m \to \infty} \frac{\|\,  T_m\, \|_{\cm_{p}^u (\rn) \to \rr}}{\|\,  T_m\, \|_{\cm_{p_1}^{u_1} (\rn) \to \rr}} = \infty\, ;
\end{equation}
whereas, in case $u < u_0$, we find that
\begin{equation}\label{ws-103}
\lim_{m\to \infty} \frac{\|\,  T_m\, \|_{\cm_{p}^u (\rn) \to \rr}}{\|\,  T_m\, \|_{\cm_{p_0}^{u_0} (\rn) \to \rr} } = \infty\, .
\end{equation}
Now we distinguish our considerations into three cases.
\begin{enumerate}
\item[(a)] $p_0 < p_1$. In this case, $p<p_1$ and the claim in \eqref{ws-100} follows from \eqref{ws-102}.
\item[(b)] $p_1< p_0$. In this case, our  assumption $p_0/u_0 < p_1/u_1$ implies $u_0 > u_1$ and hence $u_0 >u$. Thus, the claim in \eqref{ws-100} follows from \eqref{ws-103}.
\item[(c)]
$p_0 =p_1$. In this case,
our  assumption $p_0/u_0 < p_1/u_1$ again implies $u_0 >u$ and we can argue as in (b) to show the claim in \eqref{ws-100} holds true.
\end{enumerate}

The final step of the proof is now done by applying Proposition \ref{gustav}(ii),
which completes the proof of  Corollary \ref{cima}.


\subsection*{Proof of Corollary \ref{cimab}}


As a preparation we need the following technical lemma.

\begin{lemma}\label{morrey1b}
Let $\tz\in(0,1)$, $0 <  p_0 \le u_0 < \infty$ and $0 < p_1 \le u_1 < \infty$
such that
 \[
\frac1p=\frac{1-\tz}{p_0}+\frac\tz{p_1} \qquad  \mbox{and}
\qquad
\frac1u=\frac{1-\tz}{u_0}+\frac\tz{u_1}\, .
\]
If $u_0 \, p_1 = u_1\, p_0$, then
\begin{equation*}
\lf[\cm_{p_0, \unif}^{u_0}(\rn)\r]^{1-\tz}
\lf[\cm_{p_1, \unif}^{u_1}(\rn)\r]^\tz=\cm_{p, \unif}^{u}(\rn) \, .
\end{equation*}
\end{lemma}

\begin{proof}
The arguments used in the proof of \cite[Proposition 2.1]{lyy}
carry over to this locally uniform situation, the details being omitted.
\end{proof}

Now we turn to the proof of Corollary \ref{cimab}.

\begin{proof}[Proof of Corollary \ref{cimab}]
We have to distinguish four cases, comparing with the discussion in Remark \ref{2.12x}.

\noindent {\em Step 1.}  $\min\{\tau_0, \tau_1\} =0$. In this case, according to our restriction $\tau_0 \, p_0= \tau_1 \, p_1$,
we obtain $\tau = \tau_0= \tau_1=0$.
This implies
$$\cl^0_p (\rn) = \cm_{p, \unif}^p (\rn) = L_{p, \unif} (\rn);$$
see (a) of Section \ref{s1} of this article.
By the same arguments as those  used in
 the proof of Corollary \ref{cima}, we conclude that the Gagliardo closure of
$\cm_{p, \unif}^p (\rn)$ with respect to
$\cm_{p_0, \unif}^{p_0} (\rn) + \cm_{p_1, \unif}^{p_1} (\rn)$
is just $\cm_{p, \unif}^p (\rn)$ itself. Now the
desired conclusion of Corollary \ref{cimab} follows from Proposition \ref{t-n}.

\noindent {\em Step 2.} $\min\{\tau_0, \tau_1\}>0$ and $\min\{\tau_0-1/p_0, \tau_1-1/p_1\}<0$. In this case,
our restriction $\tau_0 \, p_0= \tau_1 \, p_1$ implies
\[
\max  \lf\{\tau-1/p, \tau_0-1/p_0, \tau_0-1/p_1\r\} < 0 \, .
\]
In this situation, by (a) of Section \ref{s1} of this article,
we know that $\cl^\tau_p (\rn)=\cm_{p,\unif}^u(\rn)$ with $1/u=1/p-\tau$, and thus we
can argue as in Step 1.

\noindent {\em Step 3.} Either $\tau_0 = 1/p_0$ or $\tau_1=1/p_1$.
In this case, our restriction $\tau_0 \, p_0= \tau_1 \, p_1$ implies $\tau \, p =1$.
Using (c) of Section \ref{s1} of this article, we find that
\[
\cl_{p_0}^{1/p_0}(\rn) = \cl_{p_1}^{1/p_1}(\rn)  = \cl_{p}^{1/p}(\rn)= \bmo \, .
\]
Since $\langle X, X , \Theta\rangle = X$ for any quasi-Banach space,
the desired conclusion of Corollary \ref{cimab} follows also in this case.

\noindent {\em Step 4.} $\max \{\tau_0 - 1/p_0, \tau_1-1/p_1\}>0$.
In this case, by (b) of Section \ref{s1} of this article, we know that
\[
\cl^{\tau_0}_{p_0} (\rn) = B^{n(\tau_0 -1/p_0)}_{\infty,\infty} (\rn)
\, , \quad \cl^{\tau_1}_{p_1} (\rn) = B^{n(\tau_1 -1/p_1)}_{\infty,\infty} (\rn)\, ,
\quad  \cl^{\tau}_p (\rn) = B^{n(\tau -1/p)}_{\infty,\infty} (\rn)\, .
\]
Applying Theorem  \ref{COMI} and Proposition \ref{basic1}(iii), we find that
\begin{eqnarray*}
\lf\laz B_{\infty,\infty}^{n(\tau_0-1/p_0)}(\rn), B_{\infty,\infty}^{n(\tau_1-1/p_1)}(\rn),\tz\r\raz & = &\lf\laz B_{p_0,q_0}^{0,\tau_0}(\rn), B_{p_1,q_1}^{0,\tau_1}(\rn),\tz\r\raz
\nonumber
\\
& = & B^{0,\tau}_{p,q}(\rn)=
B_{\infty,\infty}^{n(\tau-1/p)}(\rn) = \cl^{\tau}_p (\rn),
\end{eqnarray*}
where $q,\ q_0,\ q_1\in(0,\fz]$ satisfy $1/q=1/q_0+1/q_1$. This finishes the proof of Corollary \ref{cimab}.
\end{proof}


\subsection{Proofs of results in Subsection \ref{inter1b}}
\label{proof5}



\subsection*{Proof of Proposition \ref{interprop}}


First we deal with the completeness of $\laz A_0,A_1\raz_\tz$ in (i).
Let $\{a^\ell\}_{\ell\in\nn}$ be a fundamental sequence in $\laz A_0, A_1\raz_\Theta$.
Then, for any $\varepsilon\in(0,\fz)$, there exists
$N:=N_{(\varepsilon)}\in\zz_+$, depending on $\varepsilon$,
such that, for any $\ell,\,m\ge N$,
$$\|a^m-a^\ell\|_{\laz A_0, A_1\raz_\tz}<\varepsilon/4.$$
Hence, there exists $\{b^{m,\ell}_i\}_{i\in\zz}\subset A_0\cap A_1$ such that
\[
a^m-a^\ell= \sum_{i=-\infty}^\infty b_i^{m,\ell} \qquad \, \mbox{converges in} \quad A_0 + A_1\,  ,
\]
$\sum_{i=-\infty}^\infty \varepsilon_i \, 2^{i(j-\Theta)}\, b_i^{m,\ell}$ converges in $A_j$, $j\in\{0,1\}$,
and
\begin{equation}\label{ws-101}
\lf\|\, \sum_{i=-\infty}^\infty \varepsilon_i \, 2^{i(j-\Theta)}\, b_i^{m,\ell} \,  \r\|_{A_j} \le
\frac{\varepsilon}2 \, \sup_{i \in \zz}\, |\varepsilon_i| \,
\end{equation}
for any bounded sequence $\{\varepsilon_i\}_{i\in\zz}$.
By choosing
$\varepsilon_i := \delta_{i,k}$, $i\in \zz$, for fixed $k$,
where $\delta_{i,k}:=1$ if $i=k$, otherwise $\delta_{i,k}:=0$, we know
that, for any $i\in \zz$, $\|b_i^{m,\ell}\|_{A_j}<2^{-i(j-\tz)}\varepsilon/2$, $j\in\{0,1\}$.

On the other hand, since $a^N\in \laz A_0,A_1\raz_\tz$,
it follows that there exists $\{a_i^N\}_{i\in\zz}
\subset A_0\cap A_1$ such that
\[
a^N = \sum_{i=-\infty}^\infty a_i^N \qquad \, \mbox{converges in} \quad A_0 + A_1\,  ,
\]
$\sum_{i=-\infty}^\infty \varepsilon_i \, 2^{i(j-\Theta)}\, a_i^N$ converges in $A_j$, $j\in\{0,1\}$,
and
\begin{equation}\label{ws-101-2}
\lf\|\, \sum_{i=-\infty}^\infty \varepsilon_i \, 2^{i(j-\Theta)}\, a_i^N \,  \r\|_{X_j} \le
(\|a^N\|_{\laz A_0,A_1\raz_\tz}+\delta) \, \sup_{i \in \zz}\, |\varepsilon_i| \,
\end{equation}
for any bounded sequence $\{\varepsilon_i\}_{i\in\zz}$, where $\delta\in(0,\fz)$ can be chosen as small as we want.
Thus, for any $m\ge N$,
\[
a^{m}=a^N+\sum_{i=-\infty}^\infty b_i^{m,N}=\sum_{i=-\infty}^\infty (a_i^N+b_i^{m,N})=: \sum_{i=-\infty}^\infty \wz{a}_i^m \qquad \mbox{converge in} \quad A_0 + A_1\,  ,
\]
where $\{\wz{a}_i^m\}_{i\in\zz}\subset A_0\cap A_1$. Moreover, by
\eqref{ws-101} and \eqref{ws-101-2}, we see that
$\sum_{i=-\infty}^\infty \varepsilon_i \, 2^{i(j-\Theta)}\, \wz{a}_i^m$ converges in $A_j$, $j\in\{0,1\}$,
and
\begin{equation}\label{ws-101-3}
\lf\|\, \sum_{i=-\infty}^\infty \varepsilon_i \, 2^{i(j-\Theta)}\, \wz{a}_i^m \,  \r\|_{A_j} \le
(\|a^N\|_{\laz A_0,A_1\raz_\tz}+\delta+\varepsilon/2) \, \sup_{i \in \zz}\, |\varepsilon_i| \,
\end{equation}
for any bounded sequence $\{\varepsilon_i\}_{i\in\zz}$.
Since, for any $i\in\zz$ and $m,\ \ell\ge N$,
$$\|\wz{a}_i^m-\wz{a}_i^\ell\|_{X_j}\le \|b_i^{m,N}\|_{A_j}+
\|b_i^{\ell,N}\|_{A_j}<2^{-i(j-\tz)}\varepsilon,\quad j\in\{0,1\},$$
we know that $\{\wz{a}_i^m\}_{m}$ is a Cauchy sequence in $A_0\cap A_1$.
By the completeness of $A_0$ and $A_1$,  we obtain the existence
of some $a_{i} := \lim_{m \to \infty} \wz{a}_{i}^m$ in $A_0\cap A_1$.

Since $a^{m}=\sum_{i=-\infty}^\infty \wz{a}_i^m$ in $A_0+A_1$,
for any $\varepsilon\in(0,\fz)$, there exists $L\in\zz_+$ such that, for any $k>L$,
$\|\sum_{L\le|i|\le k}\wz{a}_i^m\|_{A_0+A_1}<\varepsilon,$
which, together with the fact that $\wz{a}_{i}^m\to a_i$ in $A_0\cap A_1$, implies that
$\|\sum_{L\le|i|\le k}{a}_i\|_{A_0+A_1}\le \varepsilon.$
Thus,
$a := \sum_{i=-\infty}^\infty a_i$ converges in $A_0 + A_1$.
Similarly, by \eqref{ws-101-3}, we find that
 $\sum_{i=-\infty}^\infty \varepsilon_i \, 2^{i(j-\Theta)}\, a_i $ converges in $A_j$, $j\in\{0,1\}$,
as well as
\[
\lf\|\, \sum_{i=-\infty}^\infty \varepsilon_i \, 2^{i(j-\Theta)}\, a_i \,  \r\|_{A_j} \ls
\, \sup_{i \in \zz}\, |\varepsilon_i|
\]
for any bounded sequence $\{\varepsilon_i\}_{i\in\zz}$. Moreover, by the definition
of $a_i$ and \eqref{ws-101}, we know that $$\|a-a^m\|_{\laz A_0,A_1\raz_\tz}<\varepsilon$$
whenever $m\ge N$, namely, $a^m$ converges to $a$ in $\laz A_0,A_1\raz_\tz$
as $m\to\fz$. This finishes the proof for the completeness of $\laz A_0,A_1\raz_\tz$
in (i).

Next we show (ii). Notice that, if the summation $\sum_{i=-\infty}^\infty \varepsilon_i \, 2^{i(j-\Theta)}\,  a_i $ converges in $A_j$, then
$$\sum_{i=-\infty}^\infty \varepsilon_i \, 2^{i(j-\Theta)}\, Ta_i$$
converges in $B_j$, $j\in\{0,1\}$, due to the boundedness of
$T:\ A_j\to B_j$.
Then, from
\begin{eqnarray*}
\lf\|\, \sum_{i=-\infty}^\infty \varepsilon_i \, 2^{i(j-\Theta)}\, T a_i \,  \r\|_{B_j}
&= & \lf\|\, T \lf(\sum_{i=-\infty}^\infty \varepsilon_i \, 2^{i(j-\Theta)}\,  a_i \r)\,  \r\|_{B_j}
\\
& \le &  \max \{\|T\|_{A_0 \to B_0}, \|T\|_{A_1 \to B_1} \} \,
\lf\|\,\sum_{i=-\infty}^\infty \varepsilon_i \, 2^{i(j-\Theta)}\,  a_i \,  \r\|_{A_j}
\\
& \ls & \max \{\|T\|_{A_0 \to B_0}, \|T\|_{A_1 \to B_1} \} \,
 \sup_{i \in \zz}\, |\varepsilon_i| \, ,
\end{eqnarray*}
we deduce the desired conclusion in (ii), which completes
the proof of Proposition  \ref{interprop}.


\subsection*{Proof of Lemma \ref{diamond1}}


{\em Step 1.} Proof of (i). {\em Substep 1.1.}
It is known that
\[
\mathring{A}^{s, \tau}_{p,q} (\rn) = {A}^{s,\tau}_{p,q} (\rn) \qquad \Longleftrightarrow
\qquad \tau =0 \quad \mbox{and}\quad \max \{p,q\} <\infty\, ,
\]
$A \in \{B,F\}$; see \cite[Theorem~2.3.3]{t83} for $\tau =0$ and \cite{yyz} for $\tau \in(0,\fz)$.
This implies $\mathring{A}^s_{p,q} (\rn) = \accentset{\diamond}{A}^s_{p,q} (\rn) $ if
$\max \{p,q\} < \infty$.\\
{\em Substep 1.2.} We prove $\mathring{A}^{s,\tau}_{p,q} (\rn)\subsetneqq \accentset{\diamond}{A}^{s,\tau}_{p,q} (\rn) $ if $\tau\in [1/p,\fz)$ and $p\in(0,\fz)$.

We study properties of the function $g(x)\equiv 1$ for all $x\in\rn$.
Let $\{\varphi_j\}_{j\in\zz_+}$ be the smooth decomposition of unity
as defined in (\ref{eq-05}) and (\ref{eq-06}).
By the basic properties of both the Fourier transform and our smooth decomposition of unity, we see that, for all $x\in\rn$,
\[
\cfi(\varphi_j\, \cf g)(x) = \left\{
\begin{array}{lll}
1 & \qquad & \mbox{if}\quad j=0\, ,
\\
0 && \mbox{otherwise}\, .
\end{array}
\right.
\]
This implies that
\[
\|g\|_{\bt}= \sup_{P\in\mathcal{Q}, \, |P|\ge 1} \frac{|P|^{1/p}}{|P|^{\tau}}
\]
for any $\tau,\ p,\ q$ and  $s$ as in Lemma \ref{diamond1}(i).
Hence, $g \in \bt$ if and only if  $\tau \in [1/p,\fz)$.
Since all derivatives of $g$ equal to zero, we further know that
$g\in \accentset{\diamond}{B}^{s, \tau}_{p,q} (\rn)$.

Now let $f \in C_c^\infty (\rn)$. We may assume $\supp f \subset [-2^N, 2^N]^n$
for some $N \in \nn$.
Let $\phi$ denote the scaling function in Proposition \ref{wave1}.
Now we choose a cube $P$ such that $P=Q_{0,m}$ and
\[
\dist (Q_{0,m},[-2^N, 2^N]^n ) > N_2 \, ,
 \]
 where $N_2$ is as in \eqref{4.19}.
Then Proposition \ref{wave1} yields
\[
\| \, g-f\, \|_{\bt}\ge |\laz g-f, \phi_{0,m}\raz| = \lf|\int_\rn
\phi (x)\, dx  \r|>0 \, .
\]
Hence, the function $g \equiv 1$ belongs to $\accentset{\diamond}{B}^{s, \tau}_{p,q} (\rn)$
but not to $\mathring{B}^{s, \tau}_{p,q} (\rn)$ if $\tau \in[1/p,\fz)$.
It is easy to see  that these statements remain true if we replace $\accentset{\diamond}{B}^{s, \tau}_{p,q} (\rn)$ and
$\mathring{B}^{s, \tau}_{p,q} (\rn)$, respectively, by $\accentset{\diamond}{F}^{s, \tau}_{p,q} (\rn)$ and $\mathring{F}^{s, \tau}_{p,q} (\rn)$.
\\
{\em Substep 1.3.}
Now we prove  $\mathring{A}^s_{p,\fz}(\rn) = \accentset{\diamond}{A}^s_{p,\fz}(\rn)$ when $p\in(0,\fz)$. Obviously, we have
$\mathring{A}^s_{p,\fz}(\rn) \hookrightarrow \accentset{\diamond}{A}^s_{p,\fz}(\rn)$.

To see the converse, let $\{\varphi_j\}_{j\in\zz_+}$ be the smooth decomposition of unity as defined in (\ref{eq-05}) and (\ref{eq-06}).
By the Paley-Wiener theorem, for any $f\in \cs'(\rn)$ and any $j \in \zz_+$,
the convolution $\cfi(\varphi_j\, \cf f)$
is a smooth function, i.\,e., an infinitely differentiable function.
Hence, also the function $S_Nf$, defined by
\begin{equation}\label{ws-66}
S_N f (x) := \sum_{j=0}^N \cfi(\varphi_j\, \cf f)(x)\, , \qquad x \in \rn\, , \quad N \in \zz_+\, ,
\end{equation}
is a smooth function.
We claim that
\begin{enumerate}
\item[(a)] $\accentset{\diamond}{B}^s_{p,\infty} (\rn)$ is the collection of all $f \in {B}^s_{p,\infty} (\rn)$
such that
\[
\lim_{j \to \infty}\,  2^{js}\, \|\, \cfi(\varphi_j\, \cf f)\, \|_{L_p (\rn)} =0\, ;
\]
\item[(b)] $\accentset{\diamond}{F}^s_{p,\infty} (\rn)$ is the collection of all $f \in {F}^s_{p,\infty} (\rn)$
such that
\[
\lim_{N \to \infty} \, \lf\| \sup_{j \ge N}\,  2^{js}\, |\cfi(\varphi_j\, \cf f)| \, \r\|_{L_p (\rn)} =0\, .
\]
\end{enumerate}
To prove these claims, we argue as follows. For simplicity,
we concentrate on the $F$-case.
Using some standard Fourier multiplier assertions (see \cite[2.3.7]{t83}),
 embeddings (see
\cite[2.7.1]{t83}) and the lifting properties (see \cite[2.3.8]{t83}), it is easily seen that
$f \in F^s_{p,\infty} (\rn)$ implies that $S_N f \in F^\sigma_{p,q} (\rn)$ for all $N \in \zz_+$, all $\sigma \in \rr$
and all $q \in (0,\infty]$.
For $f \in \accentset{\diamond}{F}^s_{p,\infty} (\rn)$, let $\{f_\ell\}_{\ell\in\nn}$
be  a sequence such that
$D^\alpha f_\ell \in F^s_{p,\infty} (\rn)$ for all $\alpha \in (\zz_+)^n$ and $\ell\in\nn$,  and
\[
\lim_{\ell \to \infty} \, \| \, f-f_\ell\, \|_{F^s_{p,\infty} (\rn)}=0 \, .
\]
Then, for any $\varepsilon\in(0,\fz)$, there exists $N\in\nn$, depending on $\varepsilon$, such that $\|f-f_{\ell}
\|_{F^s_{p,\infty} (\rn)}<\varepsilon$ whenever $\ell\ge N$.
Observe that $f_\ell \in F^\sigma_{p,q}(\rn) $ for all $\sigma \in \rr$, all $q$ and all $\ell$ (see \cite[2.3.8]{t83}).
On the other hand, it holds true that
\begin{eqnarray*}
\| \, f_\ell -S_Nf_\ell\, \|_{F^s_{p,\infty} (\rn)}  & = &
\left\| \, \sum_{j=N}^\infty \, \cfi(\varphi_j\, \cf f_\ell) \,  \right\|_{F^s_{p,\infty} (\rn)}
\lesssim
\lf\| \, \sup_{j \ge N} \, 2^{js}\, |\cfi(\varphi_j\, \cf f_\ell)| \,  \r\|_{L_{p} (\rn)}
\\
& \lesssim & 2^{-N(\sigma -s)}\,
\lf\| \, \sup_{j \ge N} \, 2^{j\sigma}\, |\cfi(\varphi_j\, \cf f_\ell)| \,  \r\|_{L_{p} (\rn)}
\\
& \lesssim & 2^{-N(\sigma -s)} \,
\| \, f_\ell \,  \|_{F^s_{p,\infty} (\rn)} \,
\end{eqnarray*}
for all $\sigma >s$.
The implicit positive
constants in these inequalities are independent of $\ell$.
Therefore, by this and the boundedness in
$F^s_{p,\fz}(\rn)$ of $S_N$ uniformly in $N\in\nn$, with $\kappa:= \min \{1,p\}$, we have
\begin{eqnarray*}
 \| \, f- S_N f\, \|_{F^s_{p,\infty} (\rn)}^\kappa
& \le &
\| \, f-f_\ell\, \|_{F^s_{p,\infty} (\rn)}^\kappa + \| \, f_\ell - S_N f_\ell\, \|_{F^s_{p,\infty} (\rn)}^\kappa
+ \| \, S_N f_\ell - S_N f\, \|_{F^s_{p,\infty} (\rn)}^\kappa
\\
& \ls &
\| \, f-f_\ell\, \|_{F^s_{p,\infty} (\rn)}^\kappa + \| \, f_\ell - S_N f_\ell\, \|_{F^s_{p,\infty} (\rn)}^\kappa\\
&\ls& \| \, f-f_\ell\, \|_{F^s_{p,\infty} (\rn)}^\kappa + 2^{-N\kappa(\sigma-s)} \| \, f_\ell \,  \|_{F^s_{p,\infty} (\rn)}^\kappa,
\end{eqnarray*}
which tends to $0$ as $N\to\fz$,
since
$ \{\| \, f_\ell\,  \|_{F^s_{p,\infty} (\rn)} \}_\ell$ is bounded.
Hence $f = \lim_{N \to \infty} S_N f$ in $F^s_{p,\infty} (\rn)$.
The claim (b) then follows from the observation that
\[
\| \, f- S_N f\, \|_{F^s_{p,\infty} (\rn)} \ls
\lf\| \, \sup_{j \ge N} \, 2^{js}\, |\cfi(\varphi_j\, \cf f)| \,  \r\|_{L_{p} (\rn)}
\ls \| \, f- S_{N-1} f\, \|_{F^s_{p,\infty} (\rn)}
\]
with the implicit positive constants independent of $N$.

The proof of claim (a) is similar, the details being omitted.

Now we are ready to prove
$\mathring{A}^s_{p,\infty} (\rn) = \accentset{\diamond}{A}^s_{p,\infty} (\rn) $.
Let $f \in \accentset{\diamond}{A}^s_{p,\infty} (\rn)$. By the above claims,  it
 suffices to approximate
$S_N f$ by functions from $C_c^\infty (\rn)$.
Indeed, let $\psi$ be as in \eqref{eq-05}. Then
\[
\lim_{M \to \infty} \psi (2^{-M}x)\, S_N f (x) = S_N f (x)
\]
with convergence in ${A}^s_{p,\infty} (\rn)$, as desired.
\\
{\em Substep 1.4.}
We now show that  $\mathring{A}^{s,\tau}_{p,q}(\rn)
\subsetneqq\accentset{\diamond}{A}^{s,\tau}_{p,q}(\rn)$ when $\tau\in(0, 1/p)$.
We use functions defined by a wavelet series (see Appendix and, particularly,
Subsection \ref{sequence}).
Let
\[
f := \sum_{\ell=1}^\infty \phi_{0,(2^\ell, 0, \ldots,0)} \, ,
\]
where $\phi$ is the scaling function in Proposition \ref{wave1}.
Now we show $f\in A^{s, \tau}_{p,q} (\rn)$ by using Proposition \ref{wave1}.
First we see  that
\[
\sum_{i=1}^{2^n-1} \|\, \{ \laz f,\,\psi_{i,j,k}\raz\}_{j,k}\, \|_{\sat} =0\, .
\]
It remains to estimate
\[
\sup_{P\in\cq,\ |P|\ge 1 } \frac{1}{|P|^\tau}\, \lf(\sum_{Q_{0,m}\subset P} |\laz f,\,\phi_{0,m}\raz|^p\r)^{1/p} \, .
\]
An inspection of the supports of the functions $\phi_{0,(2^\ell, 0, \ldots,0)}$ makes clear
that it suffices to consider either cubes $P\in\cq$ with volume $1$ or cubes $P\in\cq$ with
$P=P_M:=[0,2^M]^n$, $M\in \nn$.
Concentrating on the second case, we obtain
\[
\sup_{M \in \nn} \frac{1}{2^{\tau Mn}}\,
\lf(\sum_{0 \le m_j<2^M, \, j\in\{1, \ldots , n\}} |\laz f,\,\phi_{0,m}\raz|^p\r)^{1/p}
\lesssim \sup_{M \in \nn} \frac{1}{2^{\tau Mn}}\, M^{1/p} \lesssim 1\, ,
\]
because $\tau \in(0,\fz)$.
If $P\in\cq$ with volume $1$, then $P=Q_{0,m}$ for some $m\in\zz^n$ and,
in this case,
$$\frac{1}{|P|^\tau}\lf(\sum_{Q_{0,m}\subset P}|\laz f,\phi_{0,m}\raz|^p\r)^{1/p} \asymp1.$$
Thus, $f\in A^{s,\tau}_{p,q}(\rn)$.
Since the support of $f$ is not compact, it is easily shown by using Proposition \ref{wave1}
that $f \not\in \mathring{A}^{s,\tau}_{p,q}(\rn)$; see also Substep 1.2
of this proof.

Now we construct an approximation of $f$ by smooth functions.
The function $\phi$ can be approximated by its Sobolev mollification, denoted by
$\phi^{(\varepsilon)}$, in the norm of $C^{N_1-1}(\rn)$.
To explain the notation, let $\omega$ be an infinitely differentiable function such that $\supp \omega \subset B(0,1)$, $\omega \ge 0$
and $\int_\rn \omega (x)\, dx = 1$.
Then, for $\varepsilon\in(0,\fz)$, we put
\begin{equation}\label{ws-73}
\phi^{(\varepsilon)} (x):= \varepsilon^{-n} \int_{\rn} \omega \lf(\frac{x-y}{\varepsilon}\r) \phi (y)\, dy\, , \qquad x \in \rn\, ,
\end{equation}
as well as
\[
f_\varepsilon := \sum_{\ell=1}^\infty \phi^{(\varepsilon)}_{0,(2^\ell, 0, \ldots,0)} \, .
\]
Clearly, $f_\varepsilon \in C^\infty (\rn)$.
For all $\alpha \in (\zz_+)^n$ with $|\alpha |< N_1 $, it follows
that, for all $x\in\rn$,
\begin{eqnarray}\label{ws-61}
|D^\alpha f (x)-D^\alpha f_\varepsilon (x)| & \le &
\varepsilon^{-n} \int_{\rn} \omega \lf(\frac{x-y}{\varepsilon}\r) \Big| D^\alpha \phi (y)-D^\alpha \phi (x)\Big| \, dy
\le \sqrt{n}
\, \|\, \phi \, \|_{C^{N_1}(\rn)} \, \varepsilon,
\end{eqnarray}
where
\[
\|\, \phi \, \|_{C^{N}(\rn)}:= \max_{|\alpha|\le N} \sup_{x\in \rn} \, |D^\alpha \phi (x)|\, .
\]

We claim $f_\varepsilon \in \bt$ and
\begin{equation}\label{4.10x}
\lim_{\varepsilon \downarrow 0} \, \| \, f - f_\varepsilon\, \|_{\bt} =0.
\end{equation}
Due to the definitions of $f$ and $f_\varepsilon$, and the lifting property
(see \cite[Proposition 5.1]{ysy}), it suffices to prove this claim for $s$ large and $q= \infty$.
In such a situation, we may use the characterization of $\bt$
by differences as proved in
\cite[4.3.2]{ysy}.

Let $p\in [1, \fz]$, $q\in(0,\fz]$ and
\[
0 < s \le \max \{s,\, s+n\tau -n/p\} < M
\]
with $M \in \nn$. We define
\begin{equation*}
\|f\|^\clubsuit_{B^{s, \tau}_{p, \infty}(\rn)} := \dsup_{P\in\mathcal{
Q}}\frac1{|P|^\tau} \, \dsup_{0 < t   < 2\min\{\ell(P),1\}} \, t^{-s} \lf(\dint_P
[a_t(f,x)]^p\,dx\r)^{1/p} \, ,
\end{equation*}
where $\ell(P)$ denotes the side-length of the cube $P$,
\begin{equation*}
a_t(f,x):=  t^{-n}\dint_{t/2\le|h|<t}|\Delta^M_h f(x)|\,dh\,,\quad x\in\rn,
\end{equation*}
and $\Delta^M_h f$ denotes the $M$-th difference of $f$.
In addition, we put
\[
\| \, f\,\|_{L_p^\tau (\rn)} := \sup_{P\in\mathcal{Q},~|P|\ge 1} \, \frac1{|P|^\tau}\,
\lf[\int_P |f(x)|^p dx \r]^{1/p}\, .
\]
Then, from \cite[4.3.2]{ysy}, we deduce that $f\in B^{s, \tau}_{p, \infty}(\rn)$ if and only
if $f\in L^p_\tau(\rn)$ and $\|f\|^\clubsuit_{B^{s, \tau}_{p, \infty}(\rn)}<\fz$. Furthermore,
$\|f\|_{L^p_\tau(\rn)}+\|f\|^\clubsuit_{B^{s, \tau}_{p, \infty}(\rn)}$
and $\|f\|_{B^{s, \tau}_{p, \infty}(\rn)}$ are
equivalent.

As a consequence of \eqref{ws-61}, we find that
\[
\| \, f-f_\varepsilon \,\|_{L_p^\tau (\rn)} \le \sqrt{n} \,
\|\, \phi \, \|_{C^{N_1}(\rn)} \,  \varepsilon \, .
\]
Now we investigate $\|\, f - f_\varepsilon\,
\|^\clubsuit_{B^{s, \tau}_{p, \infty}(\rn)}$. Let $M < N_1$, where $N_1$ is as in \eqref{4.19}.
Since, for all $x\in\rn$,
\[
t^{-n}\dint_{t/2\le|h|<t}|\Delta^M_h (f-f_\varepsilon)(x)|\,dh
\lesssim    \varepsilon \, t^M,
\]
we conclude, for small cubes $P$ (i.\,e., $|P|\le 1$), that
\begin{eqnarray*}
\dsup_{P\in\mathcal{Q},\ |P| \le 1} \frac1{|P|^\tau}
\dsup_{0< t < 2\min\{\ell(P),1\}} \, t^{-s} \lf(\dint_P
[\varepsilon \, t^M]^p\,dx\r)^{1/p} & \lesssim & \varepsilon
\dsup_{P\in\mathcal{Q},\ |P| \le 1} \frac{|P|^{1/p} \, [l(P)]^{M-s}}{|P|^\tau}
\lesssim \varepsilon \, .
\end{eqnarray*}
Within the large cubes $P$, it suffices to consider
$P_L=[0,2^L]^n$, $L\in \nn$.
Then, by the support condition of $f$ and $f_\varepsilon$, we see that
\begin{eqnarray}\label{ws-63}
\dsup_{L \in \nn} \frac{1}{2^{Ln\tau}}
\dsup_{0< t < 2} \, t^{-s} \lf(\dsum_{\ell=1}^{L} \int_{Q_{0,(2^\ell, 0 , \ldots \, ,0)}}
[\varepsilon \, t^M]^p\,dx\r)^{1/p} & \lesssim & \varepsilon
\dsup_{L \in \nn} \frac{L^{1/p}}{2^{Ln\tau}}
 \lesssim  \varepsilon \, .
\end{eqnarray}
Combining these two inequalities, we show the above claim \eqref{4.10x}
in case $p\in[1,\fz]$.

For $p\in(0,1)$, we can argue in principal as above.
However, a few modifications are necessary, since, in this case, the characterization
by differences looks a bit different.
Let $p\in (0,1)$, $q\in(0,\fz]$ and
\[
n \, \lf(\frac 1p -1 \r)  < s \le \max \{s,\, s+n\tau -n/p\} < M
\]
with $M \in \nn$.  Let $s_0$ be chosen such that
$n \, ( 1/p -1 )  < s_0 < s $.
Then $f\in {B}^{s, \tau}_{p,\fz} (\rn)$ if and only
if $f\in L^p_\tau(\rn)$,  $\|f\|^\clubsuit_{{B}^{s, \tau}_{p,\fz} (\rn)}<\fz$ and
\begin{equation}\label{ws-64}
\sup_{P \in \mathcal{Q}, \, |P|\ge 1}
\frac{\|f \|_{B^{s_0}_{p,\infty} (2P)}}{|P|^\tau} < \infty\,,
\end{equation}
where $B^{s_0}_{p,\infty} (2P)$ is defined as in Definition \ref{d5.12} below.
Furthermore,
\[
\|f\|_{L^p_\tau(\rn)}+\|f\|^\clubsuit_{{B}^{s, \tau}_{p,\fz} (\rn)}
+ \sup_{P \in \mathcal{Q}, \, |P|\ge 1}
\frac{\|f \|_{B^{s_0}_{p,\infty} (2P)}}{|P|^\tau}
\qquad \mbox{and} \qquad  \|f\|_{{B}^{s, \tau}_{p,\fz} (\rn)}
\]
are equivalent.
The additional term in \eqref{ws-64} can be treated as in \eqref{ws-63},
the details being omitted.
Thus,  \eqref{4.10x} also holds true in this case.
By this and $f\in {B}^{s,\tau}_{p,q}(\rn)$, we further conclude that $f\in \accentset{\diamond}{B}^{s,\tau}_{p,q}(\rn)$.
Recall that it is proved, in Substep 1.4, that  $f\notin\mathring{B}^{s,\tau}_{p,q}(\rn)$.
Thus, we have $\mathring{B}^{s,\tau}_{p,q}(\rn)
\subsetneqq\accentset{\diamond}{B}^{s,\tau}_{p,q}(\rn)$ when $\tau\in(0, 1/p)$.

The $F$-case can be derived from
${B}^{s,\tau}_{p,\min(p,q)}(\rn)\hookrightarrow  \ft \hookrightarrow {B}^{s,\tau}_{p,\max(p,q)}(\rn)$, the details being omitted.
\\
{\em Step 2.} Proof of (ii).
In case $u=p$, we have $\cn^s_{p,p,q}(\rn) = B^s_{p,q}(\rn)$. With $p\in(0,\infty)$, Step 1 implies that
\[
\mathring{\cn}^s_{p,p,q}(\rn) = \mathring{B}^s_{p,q}(\rn)=
\accentset{\diamond}{B}^s_{p,q} (\rn)= \accentset{\diamond}{\cn}^s_{p,p,q} (\rn)\, .
\]
The non-coincidence when $0<p<u<\fz$
can be proved by an argument parallel to that used in Step 1 for the spaces $\at$. We refer
the reader to Proposition \ref{wave2}
for the wavelet characterization of ${\cn}^s_{p,p,q}(\rn)$ and to
\cite[4.5.2]{ysy} for an appropriate  characterization of ${\cn}^s_{p,p,q}(\rn)$ by differences,
the details being omitted.
This finishes the proof of Lemma \ref{diamond1}.


\subsection*{Proof of Lemma \ref{diamond2}}


{\em Step 1.} Proof of (i). We prove this by considering four cases.
\\
{\em Substep 1.1.} $\tau\in(1/p,\fz)$.
In this case,
\[
A_{p,q}^{s,\tau}(\rn) = B_{\fz,\fz}^{s+n(\tau-1/p)}(\rn) \, ,  \qquad A \in \{B,F\}\, ;
\]
see Proposition \ref{basic1}(iii) below.
Hence, in this case, to show Lemma \ref{diamond2}(i), it suffices to prove that $\accentset{\diamond}{B}_{\fz,\fz}^{s}(\rn)$
is a proper subspace of ${B}_{\fz,\fz}^{s}(\rn)$.
Temporarily we assume $s\in(0, 1)$.
Let $\psi \in C_c^\infty (\rn)$ such that $\psi (x)=1$ if $|x|\le 1$.
Then we define
\begin{equation*}
f_\alpha (x):= \psi (x) \, |x|^\alpha\,, \quad x\in\rn.
\end{equation*}
Clearly, $f_\alpha \in {B}_{\fz,\fz}^{s}(\rn)$ if and only if
$\alpha \ge s$; see, e.\,g., \cite[Lemma 2.3.1/1]{rs96}.
We claim that $f_s \in {B}_{\fz,\fz}^{s}(\rn) \setminus \accentset{\diamond}{B}_{\fz,\fz}^{s}(\rn)$.
When $s\in(0,1)$, since $B^s_{\fz,\fz}(\rn)$ is just the Lipschitz space
of order $s$ (see, for example, \cite[2.3.5]{t83}),
it follows that
\[
\| \, f \, \|_{{B}_{\fz,\fz}^{s}(\rn)} \ge  \sup_{x \neq y} \frac{|f(x)-f(y)|}{|x-y|^s} \, .
\]
Notice that,
for all smooth functions $g$ satisfying that $D^\az g\in B^{s}_{\fz,\fz}(\rn)$ for all $\az\in\zz_+^n$, by the lifting property of $B^s_{\fz,\fz}(\rn)$ (see
\cite[Theorem 2.3.8]{t83}), we conclude that
$g\in B^\sigma_{\fz,\fz}(\rn)$ for all $\sigma\in(s,\fz)$ and hence
\[
\lim_{y \to 0} \frac{|g(0)-g(y)|}{|y|^s}\le \lim_{y \to 0} |y|^{\sigma-s}\|g\|_{B^\sigma_{\fz,\fz}(\rn)} = 0.
\]
Since this kind of functions is dense in $\accentset{\diamond}{B}_{\fz,\fz}^{s}(\rn)$,
the above assentation holds true also for all $g\in \accentset{\diamond}{B}_{\fz,\fz}^{s}(\rn)$.
On the other hand,
\[
\frac{|f_s(0)-f_s(y)|}{|y|^s} = 1 \qquad \mbox{for all} \quad y \neq 0\, , \quad |y|\le 1\, .
\]
This implies
$\| \, f_s - g \, \|_{{B}_{\fz,\fz}^{s}(\rn)}\ge 1$ for all $g \in \accentset{\diamond}{B}_{\fz,\fz}^{s}(\rn)$, and proves the above claim.

Now we remove the restriction $ s\in(0, 1)$ by assuming $s\in\rr$.
In this general case,  we apply the lifting operator
\begin{equation*}
I_\sigma : ~ f \mapsto \cfi\lf((1+ |\cdot|^2)^{\sigma/2} \cf f (\cdot)\r)\, , \qquad f \in \cs' (\rn).
\end{equation*}
Here $\sigma$ is a real number.
It is well known that $I_\sigma$ is an isomorphism which maps
${B}_{\fz,\fz}^{s}(\rn)$ onto ${B}_{\fz,\fz}^{s-\sigma}(\rn)$; see \cite[2.3.8]{t83}.
In addition, $I_\sigma$ maps $\accentset{\diamond}{B}_{\fz,\fz}^{s}(\rn)$ onto
$\accentset{\diamond}{B}_{\fz,\fz}^{s-\sigma}(\rn)$.
By using this, we can transfer the problem to the case $s\in(0,1)$.
This shows that the above claim also holds true for all $s \in \rr$.
\\
{\em Substep 1.2.} $\tau = 1/p >0$. In this case,
we have to prove that $\accentset{\diamond}{A}^{s,\tau}_{p,q} (\rn)$ is a proper subspace of
$\at$.
This time we use wavelet representations of $\at$ (see Proposition \ref{wave1} in Appendix).
Let
\[
f_s := \sum_{j=1}^\infty 2^{-j(s+n/2)} \psi_{1,j,(0, \ldots \, ,0)}\, .
\]
Proposition \ref{wave1} yields
\[
\|\, f_s\, \|_{\bt}  \asymp
\sup_{P\subset Q_{0,0}}  \frac1{|P|^{\tau}}
\left\{ \sum_{j=j_P}^\fz 2^{j(s+\frac n2)q} \left[ \int_P  2^{-j(s+n/2)p}\, \chi_{Q_{j,0}}(x)\,  dx\right]^{\frac qp}
\right\}^{\frac 1q}\, .
\]
Observe that $\sup_{P\subset Q_{0,0}} \cdots
= \sup_{\ell\in\zz_+, ~P=Q_{\ell,0}}\cdots$. This implies,
in case $p< \infty$, that
\[
\|\, f_s\, \|_{\bt}   \asymp
\sup_{\ell \in \zz_+}  2^{\ell n/p} \,
\left\{ \sum_{j=\ell}^\fz 2^{-jnq/p}\right\}^{\frac 1q} \lesssim  1 \, .
\]
Hence,  $f_s \in A^{s,\tau}_{p,q}(\rn)$ for all $q \in (0,\infty]$ and all $p \in (0,\infty)$, $A \in \{B,F\}$
(the $F$-case follows from
${B}^{s,\tau}_{p,\min(p,q)}(\rn)\hookrightarrow  \ft \hookrightarrow {B}^{s,\tau}_{p,\max(p,q)}(\rn)$).

We claim that $f_s \not \in \accentset{\diamond}{A}^{s,\tau}_{p,q} (\rn)$.
The proof makes use of the continuous embedding
\[
{A}^{s,\tau}_{p,q} (\rn) \hookrightarrow {B}^{s + n\tau - n/p}_{\infty,\infty} (\rn)\, ;
\]
see \cite[Proposition~2.6]{ysy}.
From this embedding, we immediately derive that
\[
\accentset{\diamond}{A}^{s,\tau}_{p,q} (\rn) \hookrightarrow \accentset{\diamond}{B}^{s + n\tau - n/p}_{\infty,\infty} (\rn) =
\accentset{\diamond}{B}^{s}_{\infty,\infty} (\rn)\, ,
\]
since $\tau = 1/p$.

Recall that the wavelet characterization of $\accentset{\diamond}{B}^{s}_{\infty,\infty} (\rn)$
looks as follows (see Remark \ref{rbinf}):
Let $s\in\rr$ and $N_1 >s$.  Then $f \in \accentset{\diamond}{B}^{s}_{\infty,\infty} (\rn)$
if and only if $f$ can be represented as in \eqref{4.22} (with
convergence in $\cs'(\rn)$),
$\| \, \Phi (f)\, \|^*_{b^s_{\infty,\infty}(\rn)} < \infty$
and
\[
 \lim_{j\to \infty} 2^{j(s+n/2)} \, \max_{i\in\{1, \ldots , 2^n-1\}} \, \sup_{k\in \zz^n}\, |\laz f,\,\psi_{i,j,k}\raz|=0\, .
\]
Form the definition of $f_s$, it follows obviously that $f_s \not\in \accentset{\diamond}{B}^{s}_{\infty,\infty} (\rn)$.
Consequently, $f_s \not\in \accentset{\diamond}{A}^{s, \tau}_{p,q} (\rn)$
and hence $\accentset{\diamond}{A}^{s, \tau}_{p,q} (\rn)\subsetneqq
{A}^{s, \tau}_{p,q}(\rn)$. \\
{\em Substep 1.3.}  $\tau\in(0,1/p)$ and $q\in(0, \infty]$. In this case,
we argue as in Substep 1.2.
Let
\begin{equation}\label{ws-67}
f_{s,p,\tau} := \sum_{j=1}^\infty 2^{-j(s+ n(\tau -1/p) + n/2)} \psi_{1,j,(0, \ldots \, ,0)}\, .
\end{equation}
Proposition \ref{wave1} yields
\[
\|\, f_{s,p,\tau}\, \|_{\bt}  \asymp
\sup_{P\subset Q_{0,0}}  \frac1{|P|^{\tau}}
\left\{ \sum_{j=j_P}^\fz 2^{j(s+\frac n2)q} \left[ \int_P  2^{-j(s+n(\tau -1/p)+n/2)p}\, \chi_{Q_{j,0}}(x)\,  dx\right]^{\frac qp}
\right\}^{\frac 1q}\, .
\]
As above, this implies,
in case $\tau \in(0,\fz)$, that
\[
\|\, f_{s,p,\tau}\, \|_{\bt}   \asymp
\sup_{\ell \in \zz_+}  2^{\ell n \tau} \,
\left\{ \sum_{j=\ell}^\fz 2^{-jn\tau q}\right\}^{\frac 1q} \lesssim  1 \, .
\]
Hence,  $f_{s,p,\tau} \in A^{s,\tau}_{p,q}(\rn)$ for all $q \in (0,\infty]$, $A \in \{B,F\}$
(again the $F$-case follows from ${B}^{s,\tau}_{p,\min(p,q)}(\rn)\hookrightarrow  \ft \hookrightarrow {B}^{s,\tau}_{p,\max(p,q)}(\rn)$).
Since $f_{s,p,\tau} \not \in \accentset{\diamond}{B}^{s + n\tau-n/p}_{\infty,\infty} (\rn)$ (see Substep 1.2.),
we conclude $f_{s,p,\tau} \not\in \accentset{\diamond}{A}^{s,\tau}_{p,q} (\rn)$
and hence $\accentset{\diamond}{A}^{s, \tau}_{p,q} (\rn)\subsetneqq
{A}^{s, \tau}_{p,q}(\rn)$.
\\
{\em Substep 1.4.}  $\tau =0$. In this case, it is known that
\[
\mathring{A}^s_{p,q} (\rn) = {A}^s_{p,q} (\rn) \qquad \Longleftrightarrow \qquad
\max\{p,q\} <\infty\, ;
\]
see \cite[Theorem~2.3.3]{t83}.
Hence, by Lemma \ref{diamond1}, we know that $\accentset{\diamond}{A}^{s}_{p,q} (\rn) = {A}^{s}_{p,q} (\rn)$ if $\max \{p,q\}<\infty$.

Let $p< \infty$ and $q= \infty$. Again, by Lemma \ref{diamond1}, we have $\accentset{\diamond}{A}^{s}_{p,\infty} (\rn) =
\mathring{A}^{s}_{p,\infty} (\rn)$,
and the latter is known to be a proper subspace of ${A}^{s}_{p,\infty} (\rn)$.

Let $p = \infty$ and $q < \infty$.
Then $\accentset{\diamond}{B}^{s}_{\infty,q} (\rn) = {B}^{s}_{\infty,q} (\rn)$, since
\[
\lim_{N \to \infty} \| \, f- S_N f \, \|_{{B}^{s}_{\infty,q} (\rn)} \asymp
\lim_{N \to \infty} \lf\{ \sum_{j=N}^\infty 2^{jsq}\,
\| \, \cfi(\varphi_j\, \cf f)\, \|_{L_{\infty} (\rn)}^q\r\}^{1/q} =0\, .
\]

Finally, $\accentset{\diamond}{B}^{s}_{\infty,\infty} (\rn)$ is a proper subset of
${B}^{s}_{\infty,\infty} (\rn)$,
since $f_s$ in Substep 1.2 belongs to ${B}^{s}_{\infty,\infty} (\rn)$ and it does not belong to $\accentset{\diamond}{B}^{s}_{\infty,\infty} (\rn)$.

This finishes the proof of (i).
\\
{\em Step 2.} Proof of (ii). To show this, we consider two cases.\\
{\em Substep 2.1.}
Let $0 < p \le u < \infty$ and $q\in(0,\infty)$.
Let  $S_N f$ be defined as in \eqref{ws-66}.
Using the  Fourier multiplier theorem of Tang and Xu \cite{tx},
it is not difficult to prove
\[
\| \, f - S_N f \, \|_{\cn^s_{u,p,q}(\rn)} \ls \lf( \sum_{j=N}^\infty 2^{jsq}\, \|\, \cfi(\varphi_j \cf f)\|^q_{\cm_p^u (\rn)}
\r)^{1/q} \longrightarrow 0
\]
if  $N$ tends to infinity. This implies $\accentset{\diamond}{\cn}^s_{u,p,q}(\rn) = \cn^s_{u,p,q}(\rn)$ for $q\in(0,\infty)$.
\\
{\em Substep 2.2.}
Let $0 < p \le u < \infty$ and $q= \infty$.
Because of ${\cn}^s_{u,p,\infty}(\rn) = {B}^{s, \frac 1p - \frac 1u}_{p,\infty}(\rn)$,
the desired conclusion follows from Step 1. This finishes the proof of
Lemma \ref{diamond2}.

\begin{remark}
It is of certain interest to point out
that Proposition \ref{wave2} implies that  the function $f_{s,p,\tau}$, defined in \eqref{ws-67},
does not belong to  any of the spaces ${\cn}^s_{u,p,q}(\rn)$, $1/u := 1/p - \tau$, $ \tau\in(0,1/p)$ and $q\in(0,\infty)$.
This implies that
\[
{B}^{s, \frac 1p - \frac 1u}_{p,q_0}(\rn)  \hookrightarrow  {\cn}^s_{u,p,q_1}(\rn)  \qquad \Longleftrightarrow
\qquad q_1 = \infty
\]
and
\[
{F}^{s, \frac 1p - \frac 1u}_{p,q_0}(\rn)  \hookrightarrow  {\cn}^s_{u,p,q_1}(\rn)  \qquad \Longrightarrow
\qquad q_1 = \infty\, ,
\]
which have been known before, and we refer the reader
to Sawano \cite{sa0}.
\end{remark}


\subsection*{Proof of Lemma \ref{densel}}


{\em Step 1.} Proof of (i). We need the following modification of Proposition \ref{morrey2} (see \cite{yyz}), here and hereafter,
$\mathring{a}_{p,q}^{s,\tau}(\rn)$ denotes the closure of finite sequences in
${a}_{p,q}^{s,\tau}(\rn)$.

\begin{proposition}\label{morrey2b}
Let $\tz\in(0,1)$, $s,\ s_0,\ s_1\in\rr$, $\tau,\ \tau_0,\ \tau_1\in[0,\fz)$,
$p,\ p_0,\ p_1\in(0,\fz]$ and $q,\ q_0,\ q_1\in(0,\fz]$
such that $s=s_0(1-\tz)+s_1\tz$, $\tau=\tau_0(1-\tz)+\tau_1\tz$,
$\frac1p=\frac{1-\tz}{p_0}+\frac\tz{p_1}$ and
$\frac1q=\frac{1-\tz}{q_0}+\frac\tz{q_1}$.
If $\tau_0 \, p_0 = \tau_1\, p_1$, then
\begin{equation*}
\lf[\mathring{a}_{p_0,q_0}^{s_0,\tau_0}(\rn)\r]^{1-\tz}
\lf[\mathring{a}_{p_1,q_1}^{s_1,\tau_1}(\rn)\r]^\tz= \mathring{a}_{p,q}^{s,\tau}(\rn) \, ,
\qquad a\in \{f,b\}\, .
\end{equation*}
\end{proposition}

Thanks to Proposition \ref{t-n}, we obtain
\[
\lf\laz \mathring{a}_{p_0,q_0}^{s_0,\tau_0}(\rn), \mathring{a}_{p_1,q_1}^{s_1,\tau_1}(\rn)\r\raz_\tz
=  \overline{\mathring{a}_{p_0,q_0}^{s_0,\tau_0}(\rn) \cap \mathring{a}_{p_1,q_1}^{s_1,\tau_1}(\rn)
}^{\| \, \cdot \, \|_{{a}_{p,q}^{s,\tau}(\rn)}}
\, .
\]
Since all finite sequences belong to $\mathring{a}_{p_0,q_0}^{s_0,\tau_0}(\rn) \cap \mathring{a}_{p_1,q_1}^{s_1,\tau_1}(\rn)$,
it is clear that the closure must be $\mathring{a}_{p,q}^{s,\tau}(\rn)$.
By means of Propositions \ref{wave1} and  \ref{interprop}, this carries over to the function spaces.
Consequently, it holds true that
\[
\mathring{A}_{p,q}^{s,\tau}(\rn) = \lf\laz \mathring{A}_{p_0,q_0}^{s_0,\tau_0}(\rn), \mathring{A}_{p_1,q_1}^{s_1,\tau_1}(\rn)\r\raz_\tz
\hookrightarrow \lf\laz {A}_{p_0,q_0}^{s_0,\tau_0}(\rn), {A}_{p_1,q_1}^{s_1,\tau_1}(\rn)\r\raz_\tz \, .
\]
Concerning the remaining embedding, we recall
\[
{A}^{s,\tau}_{p,q} (\rn) \hookrightarrow {B}^{s + n\tau - n/p}_{\infty,\infty} (\rn)\,
\]
(see \cite[Proposition~2.6]{ysy}), and therefore
\[
\lf\laz {A}_{p_0,q_0}^{s_0,\tau_0}(\rn), {A}_{p_1,q_1}^{s_1,\tau_1}(\rn)\r\raz_\tz \hookrightarrow
\lf\laz {B}_{\infty,\infty}^{s_0 + n \tau_0 -n/p_0}(\rn), {B}_{\infty,\infty}^{s_1 + n \tau_1-n/p_1}(\rn)\r\raz_\tz \, .
\]
To calculate
$\laz {B}_{\infty,\infty}^{s_0 + n \tau_0 -n/p_0}(\rn), {B}_{\infty,\infty}^{s_1 + n \tau_1-n/p_1}(\rn)\raz_\tz $,
we use the identity
\begin{eqnarray*}
&& \hspace{-0.7cm}
\lf\laz {B}_{\infty,\infty}^{s_0 + n \tau_0 -n/p_0}(\rn), {B}_{\infty,\infty}^{s_1 + n \tau_1-n/p_1}(\rn)\r\raz_\tz
\\
& = &
\lf( {B}_{\infty,\infty}^{s_0 + n \tau_0 -n/p_0}(\rn), {B}_{\infty,\infty}^{s_1 + n \tau_1-n/p_1}(\rn),
{B}_{\infty,\infty}^{s + n \tau-n/p}(\rn) , \#\r)\,
\end{eqnarray*}
in Proposition \ref{C-CIS}(i).
Then,  it suffices to show
\begin{equation}\label{ws-68}
\lf( {B}_{\infty,\infty}^{s_0}(\rn), {B}_{\infty,\infty}^{s_1}(\rn), {B}_{\infty,\infty}^{s}(\rn) , \#\r)
 =
\accentset{\diamond}{B}_{\infty,\infty}^{s}(\rn) \, ,
\end{equation}
if $s_0 \neq s_1$, $s:= s_0\,  (1-\Theta) + s_1 \, \Theta$.
To see this, without loss of generality, we may assume $s_0 < s_1$. Then ${B}_{\infty,\infty}^{s_1}(\rn) \hookrightarrow {B}_{\infty,\infty}^{s_0}(\rn)$
and
\[
{B}_{\infty,\infty}^{s_1}(\rn) \cap {B}_{\infty,\infty}^{s_0}(\rn) = {B}_{\infty,\infty}^{s_1}(\rn)
\]
follows.
Applying a simple lifting argument (see Theorem 2.3.8 in \cite{t83}), we find
that
\begin{eqnarray*}
\{ f \in {B}_{\infty,\infty}^{s_1}(\rn):&&  D^\alpha f \in {B}_{\infty,\infty}^{s_1}(\rn) \quad \mbox{for all}\: \alpha \in \zz_+\}
\\
& = & \{ f \in {B}_{\infty,\infty}^{s}(\rn):\quad  D^\alpha f \in {B}_{\infty,\infty}^{s}(\rn) \quad \mbox{for all}\: \alpha \in \zz_+\}
\, .
\end{eqnarray*}
Hence
\[
\accentset{\diamond}{B}_{\infty,\infty}^{s}(\rn) \hookrightarrow
\overline{{B}_{\infty,\infty}^{s_1}(\rn)}^{\|\, \cdot \, \|_{{B}_{\infty,\infty}^{s}(\rn)}}  \, .
\]
To prove the converse, let $f\in {B}_{\infty,\infty}^{s_1}(\rn)$.
Then
\begin{eqnarray*}
\| \, f- S_N f \, \|_{{B}_{\infty,\infty}^{s}(\rn)} & \lesssim &
\sup_{j\ge N} \, 2^{js} \, \| \, \cfi(\varphi_j\, \cf f) \, \|_{L_{\infty}(\rn)}
\\
& \lesssim &
2^{-N(s_1 - s)}\, \sup_{j\ge N} \, 2^{js_1} \, \| \, \cfi(\varphi_j\, \cf f) \, \|_{L_{\infty}(\rn)}
\\
& \lesssim &
2^{-N(s_1 - s)}\, \| \, f\, \|_{B_{\infty,\infty}^{s_1}(\rn)}\, \to0
\end{eqnarray*}
as  $N\to\fz$.
This proves $f \in \accentset{\diamond}{B}_{\infty,\infty}^{s}(\rn)$ and therefore, the claim \eqref{ws-68} is established.
\\
{\em Step 2.} Proof of (ii). This time we need the following modification of Proposition \ref{morrey3}  (see \cite{yyz}).

\begin{proposition}\label{morrey3b}
Let $\tz\in(0,1)$, $s,\ s_0,\ s_1\in\rr$, $q,\ q_0,\ q_1\in(0,\fz]$,
$0<p\le u\le\fz$, $0<p_0\le u_0\le\fz$ and $0<p_1\le u_1\le\fz$
such that $\frac1q=\frac{1-\tz}{q_0}+\frac\tz{q_1}$,
$\frac1p=\frac{1-\tz}{p_0}+\frac\tz{p_1}$, $s=s_0(1-\tz)+s_1\tz$
and $\frac1u=\frac{1-\tz}{u_0}+\frac{\tz}{u_1}$.
If $p_0u_1=p_1u_0,$
then
\begin{equation*}
\lf[\mathring{n}_{u_0,p_0,q_0}^{s_0}(\rn)\r]^{1-\tz}\lf[\mathring{n}_{u_1,p_1,q_1}^{s_1}(\rn)\r]^\tz=
\mathring{n}_{u,p,q}^{s}(\rn).
\end{equation*}
\end{proposition}

Now we can proceed as in Step 1 since, with $\tau:= \frac 1p - \frac 1u$,
\[
{\cn}_{u,p,q}^{s}(\rn) \hookrightarrow {\cn}_{u,p,\infty}^{s}(\rn) =
{B}^{s,\tau}_{p,\infty} (\rn) \hookrightarrow {B}^{s + n\tau - n/p}_{\infty,\infty} (\rn)\, .
\]
{\em Step 3.} Proof of (iii). Clearly, by \eqref{ws-69}, we know that
$\lf\laz {A}_{p_0,q_0}^{s_0,\tau_0}(\rn), {A}_{p_1,q_1}^{s_1,\tau_1}(\rn)\r\raz_\tz$
is a subspace of ${A}_{p,q}^{s,\tau}(\rn)$.
If $\lf\laz {A}_{p_0,q_0}^{s_0,\tau_0}(\rn), {A}_{p_1,q_1}^{s_1,\tau_1}(\rn)\r\raz_\tz$ would coincide with
${A}_{p,q}^{s,\tau}(\rn)$, Step 1 would imply
that
\[
{A}_{p,q}^{s,\tau}(\rn) \hookrightarrow \accentset{\diamond}{B}_{\infty,\infty}^{s+n\tau -n/p}(\rn)\, .
\]
But, in \eqref{ws-67}, we have found a function $f_{s,p,\tau}$
such that $f_{s,p,\tau} \in {A}_{p,q}^{s,\tau}(\rn)  \setminus
\accentset{\diamond}{B}_{\infty,\infty}^{s+n\tau -n/p}(\rn)$ if $\tau\in(0,1/p)$.
Thus, $$\lf\laz {A}_{p_0,q_0}^{s_0,\tau_0}(\rn), {A}_{p_1,q_1}^{s_1,\tau_1}(\rn)\r\raz_\tz
\subsetneqq {A}_{p,q}^{s,\tau}(\rn),$$
which completes the proof of Lemma \ref{densel}.


\subsection*{Proof of Theorem \ref{gagl1}}


{\em Step 1.} Proof of (i). Without loss of generality, we may
assume $p_0 + q_0 <\infty$.
By Propositions \ref{morrey2} and \ref{t-n}, we see that
\[
\lf\laz {a}_{p_0,q_0}^{s_0}(\rn), {a}_{p_1,q_1}^{s_1}(\rn)\r\raz_\tz =
\overline{{a}_{p_0,q_0}^{s_0}(\rn) \cap {a}_{p_1,q_1}^{s_1}(\rn)}^{\|\, \cdot \, \|_{a_{p,q}^{s}(\rn)}}\, .
\]
Since finite sequences are contained in
${a}_{p_0,q_0}^{s_0}(\rn) \cap {a}_{p_1,q_1}^{s_1}(\rn)$,
and dense in $a_{p,q}^{s}(\rn)$ due to $p + q <\infty$,
it follows that
\[
\overline{{a}_{p_0,q_0}^{s_0}(\rn) \cap {a}_{p_1,q_1}^{s_1}(\rn)}^{\|\, \cdot \, \|_{a_{p,q}^{s}(\rn)}} = a_{p,q}^{s}(\rn)\,.
\]
Based on Propositions \ref{interprop} and \ref{wave1}, this equality can be transferred to the related function spaces. The equivalence with
$\accentset{\diamond}{A}^s_{p,q}(\rn)$ and $\mathring{A}^s_{p,q}(\rn)$
follows from Lemmas \ref{diamond1} and \ref{diamond2}.
\\
{\em Step 2.} Proof of (ii). Let $s_0 \neq s_1$. Then the case $q_0 = q_1 =q = \infty$ has been treated above (see \eqref{ws-68}).
The general case $q_0,q_1 \in (0, \infty]$ can be handled in the same way.

Now we turn to the case $s=s_0 = s_1$ and $q_0 < q < q_1$.
Observe that
\[
{B}_{\infty,q_0}^{s}(\rn) \cap {B}_{\infty,q_1}^{s}(\rn) =
{B}_{\infty,q_0}^{s}(\rn) \, .
\]
Hence, we need to calculate the closure of ${B}_{\infty,q_0}^{s}(\rn)$ in ${B}_{\infty,q}^{s}(\rn)$.
Clearly, for $f \in {B}_{\infty,q}^{s}(\rn)$, we have
$S_N f \in {B}_{\infty,q_0}^{s}(\rn)$ and
\begin{eqnarray*}
\| \, f - S_N f \, \|_{{B}_{\infty,q}^{s}(\rn)} & \lesssim &  \lf[
\sum_{j=N}^\infty 2^{jsq}\, \|\, \cfi(\varphi_j\, \cf f) \, \|_{L_\infty (\rn)}^q\r]^{1/q}
\\
& \lesssim &  \lf[
\sum_{j=N}^\infty 2^{jsq_0}\, \|\, \cfi(\varphi_j\, \cf f)\, \|_{L_\infty (\rn)}^{q_0}
\r]^{1/q_0}\, .
\end{eqnarray*}
Since the right-hand side of the above inequalities tends to $0$ as $N \to \infty$, we conclude that
\[
\laz {B}_{\infty,q_0}^{s}(\rn), {B}_{\infty,q_1}^{s}(\rn) \raz_\tz=\overline{{B}_{\infty,q_0}^{s}(\rn)}^{\| \, \cdot \, \|_{{B}_{\infty,q}^{s}(\rn)}} =
 {B}_{\infty,q}^{s}(\rn) \, .
 \]
{\em Step 3.} Proof of (iii). We follow the proof of the claim in \eqref{ws-68}.
Without loss of generality, we may assume $s_0 < s_1$. Then ${A}_{p,\infty}^{s_1}(\rn) \hookrightarrow {A}_{p,\infty}^{s_0}(\rn)$
and
\[
{A}_{p,\infty}^{s_1}(\rn) \cap {A}_{p,\infty}^{s_0}(\rn) =
{A}_{p,\infty}^{s_1}(\rn)
\]
follows.
Again  a simple lifting argument (see Theorem 2.3.8 in \cite{t83}) yields
\begin{eqnarray*}
\{ f \in {A}_{p,\infty}^{s_1}(\rn)  :  &&  D^\alpha f \in {A}_{p,\infty}^{s_1}(\rn)  \quad \mbox{for all}\: \alpha \in \zz_+\}
\\
& = & \{ {A}_{p,\infty}^{s}(\rn) :\quad  D^\alpha f \in {A}_{p,\infty}^{s}(\rn)  \quad \mbox{for all}\: \alpha \in \zz_+\}
\, .
\end{eqnarray*}
Hence
\[
\accentset{\diamond}{A}_{p,\infty}^{s}(\rn) \hookrightarrow
\overline{{A}_{p,\infty}^{s_1}(\rn)}^{\|\, \cdot \, \|_{{A}_{p,\infty}^{s}(\rn)}}  \, .
\]
To prove the converse, let $f\in {A}_{p,\infty}^{s_1}(\rn)$.
Then
\[
\| \, f- S_N f \, \|_{{A}_{p,\infty}^{s}(\rn)}  \lesssim
2^{-N(s_1 - s)}\, \| \, f\, \|_{A_{p,\infty}^{s_1}(\rn)}\, .
\]
This implies $f \in \accentset{\diamond}{A}_{p,\infty}^{s}(\rn)$, and hence
\[
\overline{{A}_{p,\infty}^{s_0}(\rn) \cap {A}_{p,\infty}^{s_1}(\rn)}^{\|\, \cdot \, \|_{A_{p,\infty}^{s}(\rn)}} =
\accentset{\diamond}{A}_{p,\infty}^{s}(\rn)\,.
\]
{\em Step 4.} Proofs of (iv) and (v). We  follow some arguments taken over from
\cite{ssv} where a similar situation for the complex method has been treated.
By similarity, we concentrate on the $F$-case. Notice that
our assumptions are guaranteeing
\[
F_{p_0,\infty}^{s_0}(\rn) \hookrightarrow F_{p,1}^{s}(\rn) \hookrightarrow F_{p_1,\infty}^{s_1}(\rn)\, ;
\]
see \cite[Theorem~2.7.1]{t83}.
This implies that
\[
F_{p_0,\infty}^{s_0}(\rn) \cap  F_{p_1,\infty}^{s_1}(\rn) = F_{p_0,\infty}^{s_0}(\rn) \, .
\]
Clearly,
\[
\overline{C_c^\infty (\rn)}^{\|\, \cdot \, \|_{F_{p,\infty}^{s}(\rn) }}
\hookrightarrow \overline{F_{p_0,\infty}^{s_0}(\rn)}^{\|\, \cdot \, \|_{F_{p,\infty}^{s}(\rn) }}
\hookrightarrow \overline{F_{p,1}^{s}(\rn)}^{\|\, \cdot \, \|_{F_{p,\infty}^{s}(\rn) }}
\]
Since $C_c^\infty (\rn)$ is dense in $F_{p,1}^{s}(\rn)$, we know that
all three spaces coincide with $\mathring{F}_{p,\infty}^{s}(\rn)$.
But, in Lemma \ref{diamond1}, we prove, in this situation,
$\mathring{F}_{p,\infty}^{s}(\rn) = \accentset{\diamond}{F}_{p,\infty}^{s}(\rn)$.
\\
{\em Step 5.} Proof of (vi).
Again, we follow \cite{ssv}, where a similar situation for the complex method is treated.
It suffices to argue on the level of sequence spaces due to Proposition \ref{wave1}.
For $j\in\zz_+$, let  $K_j$ be a subset of $\zz^n$ with cardinality
\[
|K_j|= \lceil 2^{-j\{(s_1-s_0)\cdot\frac{1}{1/p_1-1/p_0}-d\}}\rceil\, ,
\]
where $\lceil t\rceil$
denotes the smallest integer larger than or equal to $t\in \rr.$
We define a sequence $\lambda :=\{\lambda_{j,k}\}_{j,k}$ by
\[
\lambda_{j,k}:=\begin{cases}\displaystyle 2^{j\cdot\frac{p_1s_1-p_0s_0}{p_0-p_1}}& \qquad \text{if}\quad  k\in K_j,\ j\in\zz_+,\\
0
& \qquad \text{otherwise}.\end{cases}
\]
It is easily checked that
\[
\lambda\in b^{s,0}_{p,\infty}(\rn)\setminus \mathring{b}^{s,0}_{p,\infty}(\rn) \qquad \mbox{and}\qquad
\lambda\in b^{s_0,0}_{p_0,\infty}(\rn)\cap b^{s_1,0}_{p_1,\infty}(\rn) \, .
\]
The counterpart of Proposition \ref{C-CIS} on the sequence space level yields
\[
\lambda\in \overline{b^{s_0,0}_{p_0,\infty}(\rn) \cap b^{s_1,0}_{p_1,\infty}(\rn)}^{b^s_{p,\infty}(\rn)}
=\laz b^{s_0,0}_{p_0,\infty}(\rn),b^{s_1,0}_{p_1,\infty}(\rn) \raz_\Theta.
\]
Hence,  the embedding $\mathring{b}^{s,0}_{p,\infty}(\rn) \hookrightarrow \laz b^{s_0,0}_{p_0,\infty}(\rn),
\, b^{s_1,0}_{p_1,\infty}(\rn) \raz_\Theta$, as well as its function space version,
is strict in this case.
Moreover, $\laz B^{s_0}_{p_0,\infty}(\rn), \, B^{s_1}_{p_1,\infty}(\rn) \raz_\Theta \subsetneqq B^{s}_{p,\infty}(\rn) $
is a consequence of Lemma \ref{densel}(i).
\\
{\em Step 6.} Proof of (vii).
We proceed as in Step 4. Observe that this time $\max \{p,q\}< \infty$.
Thanks to the embedding
$B^{s_0}_{p_0,\infty} (\rn) \hookrightarrow B^{s_1}_{\infty,q_1} (\rn)$
(see \cite[2.7.1]{t83}),
we know that
\[C_c^\fz(\rn)\hookrightarrow
B^{s_0}_{p_0,\infty} (\rn) \cap B^{s_1}_{\infty,q_1} (\rn) = B^{s_0}_{p_0,\infty} (\rn)\, .
\]
On the other hand, using $B^{s_0}_{p_0,\infty} (\rn) \hookrightarrow B^{s}_{p,q} (\rn)$ (see \cite[2.7.1]{t83}), we find that
\[
\overline{C_c^\infty (\rn)}^{B^{s}_{p,q} (\rn)} \hookrightarrow
\overline{B^{s_0}_{p_0,\infty} (\rn)}^{B^{s}_{p,q} (\rn)} \hookrightarrow
\overline{B^{s}_{p,q} (\rn)}^{B^{s}_{p,q} (\rn)}  = B^{s}_{p,q} (\rn)\, .
\]
Lemma \ref{diamond1}(i) shows that the space on the left-hand side
of the above formula coincides with $B^{s}_{p,q} (\rn)$.
The proof of Theorem \ref{gagl1} is then complete.


\subsection*{Proof of Theorem \ref{gagl7}}


By Proposition \ref{nil}, we have to calculate
\[
\overline{{A}_{p,q_0}^{s_0, \tau}(\rn) \cap {A}_{p,q_1}^{s_1,\tau}(\rn)}^{
\|\, \cdot \, \|_{{A}_{p,q}^{s, \tau}(\rn)}}\, .
\]
{\em Step 1.}
First, we assume $A=B$ and $s_0 >s_1$. Clearly,
${B}_{p,q_0}^{s_0, \tau}(\rn) \cap {B}_{p,q_1}^{s_1,\tau}(\rn) =
{B}_{p,q_0}^{s_0, \tau}(\rn)$.
We claim
\[
\overline{{B}_{p,q_0}^{s_0, \tau}(\rn)}^{\|\, \cdot \, \|_{{B}_{p,q}^{s, \tau}(\rn)}}
= \accentset{\diamond}{B}_{p,q}^{s, \tau}(\rn) \, .
\]
But this follows immediately from
\begin{eqnarray}\label{ws-105}
\Big\{f \in {B}_{p,q_0}^{s_0, \tau}(\rn): && \ D^\alpha f \in {B}_{p,q_0}^{s_0, \tau}(\rn) \quad
\mbox{for all}\quad \alpha \in \zz_+^n\Big\}
\nonumber\\
&& = \Big\{f \in {B}_{p,q}^{s, \tau}(\rn): \quad D^\alpha f \in {B}_{p,q}^{s, \tau}(\rn) \quad
\mbox{for all}\quad \alpha \in \zz_+^n\Big\} \, .
\end{eqnarray}
To prove this identity,  we need the following result:
Let $m  \in \nn$. Then $A^{s,\tau}_{p,q} (\rn)$ is the
collection of all $f\in \cs' (\rn)$ such that
\[
\sum_{|\alpha|=m} \, \| \, D^\alpha f\, \|_{A^{s-m,\tau}_{p,q} (\rn)} < \infty
\]
in the sense of equivalent quasi-norms.
If  $A=F$, by Propositions \ref{basic1} and \ref{basic2},
this property of $\ft$ when $\tau=0$ or $\tau\in[1/p,\fz)$ can be found in \cite[2.3.8]{t83}, while when $\tau\in(0,1/p)$ was proved by Tang and Xu \cite{tx};
If $A=B$,  by Proposition \ref{basic1} again,  this property of $\bt$ when $\tau=0$ or $\tau\in(1/p,\fz)$ was proved in \cite[2.3.8]{t83}, while
when $\tau\in(0,1/p]$ can be proved by an argument similar to that used
in the proof of \cite[Theorem 2.15(ii)]{tx}.

In addition, we mention the embedding
\[
A^{s,\tau}_{p,q} (\rn) \hookrightarrow C_{ub}(\rn) \qquad \mbox{if}\quad s+n\tau -n/p >0
\]
(see  \cite[Proposition~2.6]{ysy}), where
$C_{ub}(\rn)$ denotes the space of uniformly continuous and bounded functions on $\rn$.
This, together with the above property of $\at$,
shows that both sets in \eqref{ws-105} contain only $C^\infty(\rn)$ functions
and hence they are equal. Thus, the above claim
$\overline{{B}_{p,q_0}^{s_0, \tau}(\rn)}^{\|\, \cdot \, \|_{{B}_{p,q}^{s, \tau}(\rn)}}
= \accentset{\diamond}{B}_{p,q}^{s, \tau}(\rn)$ holds true.
This proves the desired conclusion of Theorem \ref{gagl7}
in this case.
\\
{\em Step 2.}
This time we consider the case $A=B$, $s=s_0 =s_1$ and $q_0 < q <  q_1$. Clearly,
this time
$${B}_{p,q_0}^{s, \tau}(\rn) \cap {B}_{p,q_1}^{s,\tau}(\rn) =
{B}_{p,q_0}^{s, \tau}(\rn)$$ and we can argue as in Step 1.
\\
{\em Step 3.} Also, if $A=F$, we can argue as in Step 1, the details being omitted.
This finishes the proof of Theorem \ref{gagl7}.


\subsection*{Proof of Theorem \ref{gagl2}}


Under the assumptions of Theorem \ref{gagl2}, we see that
\begin{eqnarray*}
\lf\laz {A}_{p_0,q_0}^{s_0, \tau_0}(\rn), {\ca}_{p_1,q_1}^{s_1,\tau_1}(\rn)\r\raz_\tz & = &
\lf\laz {B}_{\infty,\infty}^{s_0+n(\tau_0-1/p_0)}(\rn), {B}_{\infty,\infty}^{s_1+n(\tau_1-1/p_1)}(\rn)\r\raz_\tz
\\
& = & \accentset{\diamond}{B}_{\infty,\infty}^{s+n(\tau-1/p)}(\rn) = \accentset{\diamond}{A}_{p,q}^{s, \tau}(\rn)\, ;
\end{eqnarray*}
see Proposition \ref{basic1} in  Appendix and \eqref{ws-68}.
This finishes the proof of Theorem \ref{gagl2}.


\subsection*{Proof of Theorem \ref{gagl4}}


Clearly, it follows from Proposition \ref{C-CIS} that
\[
\lf\laz \cn_{u,p,q_0}^{s_0}(\rn), \cn_{u,p,q_1}^{s_1}(\rn)\r\raz_\tz = \overline{\cn_{u,p,q_0}^{s_0}(\rn)\cap \cn_{u,p,q_1}^{s_1}(\rn)
}^{\|\, \cdot \, \|_{\cn_{u,p,q}^{s}(\rn)}}.
\]
For fixed $u$ and $p$, we put
\[
N_q^s(\rn):=
\Big\{ f \in C^\infty (\rn): \quad D^\alpha f \in \cn_{u,p,q}^{s}(\rn)\quad \mbox{for all}\quad \alpha \in \zz_+^n \Big\}\, .
\]
If $s_0 >s_1$, then the obvious embedding ${\cn}_{u,p,q_0}^{s_0}(\rn) \hookrightarrow {\cn}_{u,p,q_1}^{s_1}(\rn)$
implies that
\[
N_{q_0}^{s_0}(\rn) = N_{q_1}^{s_1}(\rn) = N_{q}^{s}(\rn)\, .
\]
However, also in case $s_0 = s_1$, we have the
coincidence of these spaces independent of $q_0,q_1$, due to an argument similar to
that
used in the proof for
\eqref{ws-105}.
On the other hand, Lemma \ref{diamond2} yields $\accentset{\diamond}{\cn}_{u,p,q_i}^{s_i}(\rn) = {\cn}_{u,p,q_i}^{s_i}(\rn)$, $i\in\{0,1\}$.
All above observations, together with Proposition \ref{C-CIS}, further implies
that
\begin{eqnarray*}
\lf\laz \cn_{u,p,q_0}^{s_0}(\rn), \cn_{u,p,q_1}^{s_1}(\rn)\r\raz_\tz
&&\hookleftarrow \overline{N_{q_0}^{s_0}(\rn)\cap N_{q_1}^{s_1}(\rn)}^{\|\, \cdot \, \|_{\cn_{u,p,q}^{s}(\rn)}}= \overline{N_{q}^{s}(\rn)}^{\|\, \cdot \, \|_{\cn_{u,p,q}^{s}(\rn)}}\\
&& = \accentset{\diamond}{\cn}_{u,p,q}^{s}(\rn)= {\cn}_{u,p,q}^{s}(\rn)
\hookleftarrow \lf\laz \cn_{u,p,q_0}^{s_0}(\rn), \cn_{u,p,q_1}^{s_1}(\rn)\r\raz_\tz,
\end{eqnarray*}
which implies the desired conclusion of Theorem \ref{gagl4} and hence completes its proof.


\subsection*{Proof of Lemma \ref{morrey43}}


{\em Step 1.} Proof of (i).
Let $M(\rn)$ denote the space of all functions  $g \in {\cm}_p^u (\rn)$ having the properties
\eqref{ws-51}, \eqref{ws-50} (uniformly in $y\in \rn$) and \eqref{ws-74} (uniformly in $r\in(0,\fz)$).

We first show  $\mathring{\cm}_p^u (\rn) \hookrightarrow M(\rn)$.
Obviously, all smooth compactly supported functions satisfy
the conditions  \eqref{ws-51} through
\eqref{ws-74}.
Let now $g \in \mathring{\cm}_p^u (\rn)$ and $\{f_\ell\}_\ell \in C_c^\infty(\rn)$
be an approximating sequence of $g$ in ${\cm}_p^u (\rn)$.
Then
\begin{eqnarray}\label{ws-70}
|B(y,r)|^{1/u-1/p} \lf[\int_{B(y,r)} |g(x)|^p\,dx\r]^{1/p}\nonumber
&&\le   |B(y,r)|^{1/u-1/p}\lf[\int_{B(y,r)} |g(x)-f_\ell(x)|^p\,dx\r]^{1/p}\nonumber\\
&& \quad +
|B(y,r)|^{1/u-1/p}\lf[\int_{B(y,r)} |f_\ell(x)|^p\,dx\r]^{1/p}\, .
\end{eqnarray}
For any given $\varepsilon\in(0,1)$, since $f_\ell\to g$
in $\cm^u_p(\rn)$ as $\ell\to\fz$,
we choose $\ell$ so large such that
\[
|B(y,r)|^{1/u-1/p}\lf[\int_{B(y,r)} |g(x)-f_\ell(x)|^p\,dx\r]^{1/p} \le \frac{\varepsilon}{2}
\]
(simultaneously for all $r$ and all $y$).
Next, by $f_\ell$ satisfying \eqref{ws-50}, we choose $r$, depending on the already chosen $\ell$ but independent of
$y$,   so large such that
\[
|B(y,r)|^{1/u-1/p}\lf[\int_{B(y,r)} |f_\ell(x)|^p\,dx\r]^{1/p} \le \frac{\varepsilon}{2}\, .
\]
Consequently,
\[
|B(y,r)|^{1/u-1/p}  \lf[\int_{B(y,r)} |g(x)|^p\,dx\r]^{1/p} \le  \varepsilon,
\]
if $r$ is large enough and  independent of $y$. This proves that $g$ has
the property \eqref{ws-50}.

Now we turn to show that $g$ satisfies  \eqref{ws-51}.
Again we make use of \eqref{ws-70}. Then, instead of choosing $r$ large, we have to choose $r$ small enough. From \eqref{ws-70} and $\{f_\ell\}_{\ell\in\nn}$ satisfying
\eqref{ws-51}, together with an argument similar to the above, it follows that
\[
|B(y,r)|^{1/u-1/p}  \lf[\int_{B(y,r)} |g(x)|^p\,dx\r]^{1/p} \le \varepsilon,
\]
if $r$ is small enough and independent of $y$. This means, $g$ also satisfies \eqref{ws-51}.

It remains to prove that $g$ satisfies \eqref{ws-74}. By \eqref{ws-70} again,
together with $\{f_\ell\}_{\ell\in\nn}$ satisfying
\eqref{ws-74} and an argument similar to above, we see that
\[
|B(y,r)|^{1/u-1/p}  \lf[\int_{B(y,r)} |g(x)|^p\,dx\r]^{1/p} \le \varepsilon,
\]
if $|y|\ge N:=N_{(\varepsilon)}$ is large enough. Thus, $g\in M(\rn)$.
Altogether this proves that  $\mathring{\cm}_p^u (\rn) \hookrightarrow M(\rn)$.

To show the converse, we suppose $g \in M(\rn)$.
For such $g$, we have to prove the existence of a sequence of smooth compactly supported
functions  approximating $g$ in the norm of ${\cm}_p^u (\rn)$.
We prove this by three steps.
\\
{\em Substep 1.1.} {For $g\in M(\rn)$, we wish to approximate $g$
by uniformly continuous functions}.
In this first step, we will find it convenient
to replace the balls in the definition of $\|\cdot\|_{\cm^u_p(\rn)}$ by dyadic cubes. This results in an equivalent norm, and the
convergence is not influenced.

Let
\[
N(t):= \left\{ \begin{array}{lll}
t & \qquad &                \mbox{if}\quad 0\le t \le 1,
\\
2-t &&  \mbox{if}\quad 1\le t \le 2,
\\
0 &&  \mbox{otherwise}
               \end{array}\right.
\]
be the standard hat function ($B$-spline of order 2).
Its tensor product is denoted by
\[
\overline{N}(x):= \prod_{i=1}^n N(x_i)\,  \qquad \mbox{for all}\ \ x:= (x_1, \ldots \, , x_n) \, .
\]
Then the integer shifts of $\overline{N}$ form a decomposition of unity, i.\,e.,
\[
\sum_{k \in \zz^n} \overline{N}(x-k) = 1 \qquad \mbox{for all}\quad x \in \rn\, .
\]
For $j\in\zz$ and $k\in\zz^n$, let
\[
Q_{j,k}^* := \{x \in \rn: \quad 2^{-j}k_i \le x_i < 2^{-j}(k_i +2)\, , \: i\in\{1, \ldots , n\}\}\, .
\]
Observe, $\overline{Q_{j,k}^*}= \supp \overline{N}(2^j \, \cdot \, -k)$.
For $f \in L_1^{\ell oc} (\rn)$ and $x\in\rn$, we put
\[
T_j f(x):= \sum_{k \in \zz^n}  2^{n(j-1)}\, \lf[ \int_{Q_{j,k}^*} f(y)\, dy\r]\, \overline{N}(2^jx-k)\, , \qquad j \in \zz_+ \, .
\]

Many times, we use the following $L_p(\rn)$-stability of the integer translates of $\overline{N}$.
For any dyadic cube $Q \subset \rn $, it holds true that
\begin{eqnarray*}
\int_{Q} |T_j f(x)|^p dx & = & \int_Q  \lf|\sum_{k: \: Q_{j,k}^* \cap Q \neq \emptyset}  2^{n(j-1)}\,
\lf[\int_{Q_{j,k}^*} f(y)\, dy\r]\, \overline{N}(2^jx-k)\r|^p\, dx
\\
& \asymp & \sum_{k: \: Q_{j,k}^* \cap Q \neq \emptyset}  \int_{Q_{j,k}^*}  \lf|2^{n(j-1)}\,  \int_{Q_{j,k}^*} f(y)\, dy\r|^p dx
\asymp \sum_{k: \: Q_{j,k}^* \cap Q \neq \emptyset} 2^{jn(p-1)}  \lf|\int_{Q_{j,k}^*} f(y)\, dy\r|^p.
\end{eqnarray*}
Next we employ the H\"older inequality and find that
\begin{eqnarray*}
\lf[\int_{Q} |T_j f(x)|^p dx\r]^{1/p} & \lesssim & \lf[
\sum_{k: \: Q_{j,k}^* \cap Q \neq \emptyset} 2^{jn(p-1)} \, 2^{-jnp/p'} \int_{Q_{j,k}^*} |f(y)|^p\, dy  \r]^{1/p}
\lesssim  \lf[\int_{Q} |f(y)|^p\, dy  \r]^{1/p}\, .
\end{eqnarray*}
Here the implicit positive constants are independent  of $f,\ Q$ and $j$.
Replacing the dyadic cube by a cube of  type, $\wz Q:=\Pi_{i=1}^n[y_i-2^N,y_i+2^N]$, $y \in \rn$, we find that
a similar estimate is true:
\begin{equation}\label{ws-71}
\lf[\int_{\wz Q} |T_j f(x)|^p dx\r]^{1/p}  \lesssim  \lf[\int_{2\wz Q} |f(y)|^p\, dy  \r]^{1/p}\, ,
\end{equation}
where $2\wz Q$ denotes the cube with the same center as $\wz Q$, sides parallel to the sides of $\wz Q$ and side-length
twice larger.
The inequality \eqref{ws-71} further implies that,
if $g \in M(\rn)$, also $T_j g$ belongs to $M(\rn)$.

Now we prove that $T_j g$ converges to $g$ in $\cm^u_p(\rn)$.
We consider two cases: \\
(a) Let $Q$ be a dyadic cube such that $|Q|\ge 2^{-jn}$.
For brevity, we put
\[
a_{j,k} := 2^{n(j-1)}\,  \int_{Q_{j,k}^*} g(y)\, dy \, .
\]
This implies that
\begin{eqnarray*}
\int_Q |g(x)- T_j g(x)|^p \, dx & = &  \int_Q \lf| \sum_{k\in\zz^n: ~Q_{j,k}^* \cap Q \neq \emptyset} [g(x)-a_{j,k}] \, \overline{N}(2^jx-k)\r|^p \, dx
\\
& \lesssim & \sum_{k\in\zz^n: ~Q_{j,k}^* \cap Q \neq \emptyset} \int_{Q_{j,k}^*} |g(x)-a_{j,k}|^p dx.
\end{eqnarray*}
By the H\"older inequality, we see that
\begin{eqnarray*}
&&\sum_{k\in\zz^n: ~Q_{j,k}^* \cap Q \neq \emptyset} \int_{Q_{j,k}^*} |g(x)-a_{j,k}|^p dx \\
 &&\quad =
\sum_{k\in\zz^n: ~Q_{j,k}^* \cap Q \neq \emptyset} \int_{Q_{j,k}^*} 2^{n(j-1)p}\,  \lf|\int_{Q_{j,k}^*} [g(x) - g(y)]\, dy \r|^p dx
\\
&&\quad \le2^{jn}
\sum_{k\in\zz^n: ~Q_{j,k}^* \cap Q \neq \emptyset}  \, \int_{Q_{j,k}^*} \int_{Q_{j,k}^*} |g(x) - g(y)|^p\, dy dx
\\
&&\quad =2^{jn} \int_{Q_{j,0}^*}
\sum_{k\in\zz^n: ~Q_{j,k}^* \cap Q \neq \emptyset}  \, \int_{Q_{j,k}^*} |g(x+ 2^{-j}k) - g(y)|^p\, dy dx
\\
&&\quad \lesssim \sup_{|h|\le \sqrt{n} 2^{-j+1}} \sum_{k\in\zz^n: ~Q_{j,k}^* \cap Q \neq \emptyset} \,  \,
\int_{Q_{j,k}^*} |g(y+h) - g(y)|^p\, dy
\\
 &&\quad \lesssim \sup_{|h|\le \sqrt{n} 2^{-j+1}}  \,  \int_{Q} |g(y+h) - g(y)|^p\, dy.
\end{eqnarray*}
As above, we switch again to cubes of  type,
$\wz Q:=\Pi_{i=1}^n[y_i-2^N,y_i+2^N]$, $y \in \rn$, and
it follows that
\begin{equation}\label{ws-72}
\lf[\int_{\wz Q} |g(x)- T_j g(x)|^p \, dx \r]^{1/p} \lesssim
\sup_{|h|\le \sqrt{n} 2^{-j+1}}  \,  \lf[\int_{2\wz Q} |g(y+h) - g(y)|^p\, dy\r]^{1/p}.
\end{equation}
(b) Let $\wz Q$ be a  cube of type, $\Pi_{i=1}^n[y_i-2^N,y_i+2^N]$, $y \in \rn$, such that $|\wz Q| < 2^{-jn}$.
By obvious modifications of the above argument, we find, also in this case, that \eqref{ws-72} holds true.

Now we use \eqref{ws-72}
to prove $\|\, g - T_j g \, \|_{\cm_p^u (\rn)}\to 0$ as $j\to\fz$.
Let $\tau_h g (x):= g(x+h)$, $x \in \rn$.
Continuity of  translations is well known in the context of $L_p(\rn)$-spaces.
This will be also applied here. However, we need one more preparation.
Since $g\in M(\rn)$, it follows that, for any $\varepsilon \in(0,\fz)$, there exists $N_0\in \nn$, depending on $\varepsilon$,  such that
\begin{equation}\label{4.18x1}
 2^{Nn(\frac 1u - \frac 1p)}\, \lf[\int_{\Pi_{i=1}^n[y_i-2^N,\,y_i+2^N]} |g(x)|^p dx \r]^{1/p} < \varepsilon
\qquad \mbox{for all} \quad N \ge N_0\, ,
\end{equation}
uniformly in $y$.
Similarly, there exists $N_1\in\nn$, depending
on $\varepsilon$, such that
\begin{equation}\label{4.18x2}
 2^{-Nn(\frac 1u - \frac 1p)}\, \lf[\int_{\Pi_{i=1}^n[y_i-2^{-N},\,y_i+2^{-N}]} |g(x)|^p dx \r]^{1/p} < \varepsilon
\qquad \mbox{for all} \quad N \ge N_1\, ,
\end{equation}
uniformly in $y$.
Using \eqref{ws-71},  it suffices to deal with those cubes $\wz Q$
such that $2^{-N_1n} \le |\wz Q| \le 2^{N_0n}$.
Then \eqref{ws-72} leads to
\begin{eqnarray}\label{ws-80}
&&\sup_{-N_1 \le N \le N_0} 2^{Nn(\frac 1u - \frac 1p)}\, \lf[\int_{\Pi_{i=1}^n[y_i-2^N,\,y_i+2^N]} |g(x)-T_j g (x)|^p dx \r]^{1/p}
\nonumber\\
&&  \quad \lesssim 2^{-N_1n(\frac 1u - \frac 1p)}\, \lf[\int_{\Pi_{i=1}^n[y_i-2^{N_0},\,y_i+2^{N_0}]} |g(x)-T_j g (x)|^p dx \r]^{1/p}
\nonumber
\\
&&\quad  \lesssim  2^{-N_1n(\frac 1u - \frac 1p)}\, \sup_{|h|\le \sqrt{n} 2^{-j+1}}  \,
\lf[\int_{\Pi_{i=1}^n[y_i-2^{N_0},\,y_i+2^{N_0}]} |g(x)- g(x+h)|^p dx \r]^{1/p}\, .
\end{eqnarray}
Since
\[
\lim_{|h|\to 0}\, \| \, f (\, \cdot\, ) - f (\, \cdot\, +h)\, \|_{L_p (\rn)} = 0 \qquad \mbox{for all} \quad f\in L_p (\rn),
\]
the right-hand side in \eqref{ws-80} tends to zero for $j$ tending to $\infty$.
By the translation invariance of the $L_p(\rn)$-norm, this estimate is independent of $y$. Combining \eqref{4.18x1}, \eqref{4.18x2} and \eqref{ws-80},
we see that, for all $g \in M(\rn)$,
\begin{equation}\label{4.19x}
\|\, g - T_j g \, \|_{\cm_p^u (\rn)} \lesssim \varepsilon
\end{equation}
if $j$ is large enough. Observe that, for each $j\in\zz_+$, $T_j g$ is an uniformly continuous function on $\rn$,
since
\begin{eqnarray*}
2^{n(j-1)}\, \lf| \int_{Q_{j,k}^*} g(y)\, dy\r| & \ls &  2^{jn/p} \lf[\int_{Q_{j,k}^*} |g(y)|^p\, dy\r]^{1/p}
\\
& \lesssim_j &  |Q_{j,k}^*|^{\frac 1u - \frac 1p}\,   \lf[\int_{Q_{j,k}^*} |g(y)|^p\, dy\r]^{1/p}
 \lesssim_j  \| \, g \, \|_{\cm_p^u (\rn)}\, ,
\end{eqnarray*}
where $\ls_j$ denotes the implicit positive constants depending on $j$. Thus,
this shows that $g\in M(\rn)$ can be approximated by uniformly continuous functions
in $\cm^u_p(\rn)$.
\\
{\em Substep 1.2.} For $g \in M(\rn)$, we wish to approximate $T_j g$   by $C^\infty$-functions.
To this end,  we use the Sobolev mollification (see \eqref{ws-73}).
To simplify notation, we denote $T_jg$ just by $u$.
Consequently, for any given $\wz \varepsilon\in(0,\fz)$, there exists $\delta\in(0,\fz)$,
depending on $\wz \varepsilon$, such that,  if $|x-y|<\delta$, then
$|u(x)-u(y)|< \wz \varepsilon$. Thus,
for given $\varepsilon \in(0,\fz)$, we see that
\begin{eqnarray*}
&& \hspace*{-0.7cm}
\sup_{-N_1 \le N \le N_0}\,  2^{Nn(\frac 1u - \frac 1p)}\, \lf[\int_{\Pi_{i=1}^n[y_i-2^{N},\,y_i+2^{N}]}
|u(x)-u^{(\delta)} (x)|^p dx \r]^{1/p}
\\
& \lesssim &   2^{-N_1 n(\frac 1u - \frac 1p)}\, \lf[\int_{\Pi_{i=1}^n[y_i-2^{N_0},\,y_i+2^{N_0}]}
\lf|\delta^{-n} \int_{|x-y|<\delta} \omega
\lf(\frac{x-y}{\delta}\r)\, [u(x)-u(y)] \, dy  \r|^p dx \r]^{1/p}\\
& \lesssim & \wz\varepsilon \,  2^{-N_1n(\frac 1u - \frac 1p)}\, 2^{N_0n/p}\lesssim\varepsilon\, ,
\end{eqnarray*}
if $\delta$ is chosen small enough,
where $u^{(\delta)}$ is defined as in \eqref{ws-73}
with $\phi$ and $\varepsilon$ replaced, respectively, by $u$ and $\delta$.
This, together with \eqref{4.18x1}, \eqref{4.18x2}, \eqref{4.19x} and the definition of $u^{(\delta)}$, further implies that
\begin{equation}\label{4.19xx}
\|\, u - u^{(\delta)} \, \|_{\cm_p^u (\rn)} \lesssim \varepsilon
\end{equation}
if $\delta$ is sufficiently small. Since $u^{(\delta)}$ is smooth,
we obtain the desired conclusion in Substep 1.2.
 \\
{\em Substep 1.3.} The final step consists now in approximating
$u^{(\delta)}$ by compactly supported smooth functions. Let $\psi$ be as in \eqref{eq-05}.
We define
\[
u^{(\delta)}_\ell (x) := \psi (x/\ell) \, u^{(\delta)} (x)  \, , \qquad x \in \rn\, , \quad \ell \in \nn\, .
\]
Of course, $u^{(\delta)}$ and $u^{(\delta)}_\ell$ have the properties
\eqref{ws-51}, \eqref{ws-50} and \eqref{ws-74}, since $0 \le \psi (x)\le 1$ for all $x$.
Similar to the above estimates \eqref{4.18x1}
and \eqref{4.18x2} for $g$, we conclude that, for given $\varepsilon\in(0,\fz)$,
there exists $\varepsilon_1\in(0,\fz)$ such that
\begin{equation}\label{4.19x3}
\sup_{|B| < \varepsilon_1}
|B|^{\frac 1u - \frac 1p}\,   \lf[\int_{B}
|u^{(\delta)} (x)|^p\, dx\r]^{1/p} \le  \varepsilon
\end{equation}
and
\begin{equation}\label{4.19x4}
\sup_{1/\varepsilon_1< |B|}
|B|^{\frac 1u - \frac 1p}\,   \lf[\int_{B}
|u^{(\delta)} (x)|^p\, dx\r]^{1/p}  \le  \varepsilon \, .
\end{equation}
In addition, since $u^{(\delta)}$
satisfies \eqref{ws-74}, we know that there exists $\varepsilon_2\in(0,\fz)$ such that
\begin{equation}\label{4.19x5}
\sup_{|y| >1/ \varepsilon_2}
|B(y,r)|^{\frac 1u - \frac 1p}\,   \lf[\int_{B(y,r)}
|u^{(\delta)} (x)|^p\, dx\r]^{1/p}  \le  \varepsilon \, .
\end{equation}
All three  inequalities  above remain true with $u^{(\delta)} $ replaced by $u^{(\delta)}_\ell$ since $0 \le \psi \le 1$. From these observations
\eqref{4.19x3}, \eqref{4.19x4} and \eqref{4.19x5},
 we deduce that
\begin{eqnarray}\label{ws-82}
&&\ \ \|\, u^{(\delta)}  - u^{(\delta)}_\ell \|_{\cm_p^u(\rn)}\nonumber\\
 &&\ \ \quad\le \sup_{|y|\le 1/\varepsilon_2} \sup_{\varepsilon_1 < |B(y,r)|< 1/\varepsilon_1} |B(y,r)|^{\frac 1u - \frac 1p}\,   \lf[\int_{B(y,r)}
|u^{(\delta)} (x)  - u^{(\delta)}_\ell(x)|^p\, dx\r]^{1/p}\nonumber\\
&&\ \ \quad\quad +\sup_{\gfz{|y|>1/\varepsilon_2}{r\in(0,\fz)}}\cdots
+\sup_{\gfz{y\in\rn}{ |B(y,r)|\ge 1/\varepsilon_1}}\cdots+\sup_{\gfz{y\in\rn}{ |B(y,r)|\le\varepsilon_1}}\cdots\,\nonumber\\
&&\ \ \quad\ls \sup_{|y|\le 1/\varepsilon_2} \sup_{\varepsilon_1 < |B(y,r)|< 1/\varepsilon_1}|B(y,r)|^{\frac 1u - \frac 1p}\,   \lf[\int_{B(y,r)}
|u^{(\delta)} (x)  - u^{(\delta)}_\ell(x)|^p\, dx\r]^{1/p} +6 \varepsilon\, .
\end{eqnarray}
Choosing
\begin{equation}\label{ws-83}
\ell \ge \frac{1}{\varepsilon_2} + \lf(\frac{1}{\varepsilon_1}\r)^{1/n},
\end{equation}
we find that the first term on the right-hand side of
\eqref{ws-82} vanishes, due to the
definition of $u^{(\delta)}_\ell$.
Combining this observation, \eqref{4.19x} and \eqref{4.19xx}, we know that
$\|g-u^{(\delta)}_\ell\|_{\cm_p^u(\rn)}\ls \varepsilon$ if $\ell$
is large enough and $\delta$ is small enough, namely,  $g$
can be approximated by a sequence of smooth compactly supported functions
in  ${\cm}_p^u (\rn)$.
Hence $M(\rn)= \mathring{\cm}_p^u (\rn)$.
 This proves (i).
\\
{\em Step 2.} Proof of (ii).
We proceed as in Step 1. This time  $M(\rn)$ denotes the space of all functions  $g \in {\cm}_p^u (\rn)$ having the properties
\eqref{ws-50} (uniformly in $y\in \rn$) and \eqref{ws-74} (uniformly in $r$).

First we prove $\accentset{*}{\cm}_p^u (\rn) \hookrightarrow M(\rn)$.
It is easy to see that, if $f$ is a compactly supported function,
then $f$ satisfies \eqref{ws-50} and \eqref{ws-74}, due to
 the compactness of its support.
Moreover, the limits of compactly supported functions in $\cm_p^u (\rn)$
also satisfy  these two properties \eqref{ws-50} and \eqref{ws-74}. This proves $\accentset{*}{\cm}_p^u (\rn) \hookrightarrow M(\rn)$.

Next we show $M(\rn) \hookrightarrow \accentset{*}{\cm}_p^u (\rn)$.
Let $f \in M(\rn)$. We need to find a sequence of compactly supported
functions which converges to $f$ in $\cm_p^u (\rn)$. Indeed,
the desired approximating sequence  of $f$ in $\cm_p^u (\rn)$
is simply given by
\[
f_\ell (x) := f(x)\, \chi_{B(0,\ell)} (x)\, , \qquad x \in \rn\, , \quad \ell \in \nn\, .
\]
To see this, since $f$ and $\{f_\ell\}_{\ell\in\zz_+}$ are all elements of
$M(\rn)$, we conclude that \eqref{4.19x4} and \eqref{4.19x5}
remain true for $f$ and $f_\ell$ with $\ell\in\zz_+$. Then,
 similar to the proof of \eqref{ws-82}, we see that
\begin{eqnarray*}
\|\, f^{(\delta)}  - f_\ell^{(\delta)} \|_{\cm_p^u(\rn)}
\ls  \varepsilon
+ \sup_{|y|\le 1/\varepsilon_2} \sup_{|B(y,r)|< 1/\varepsilon_1}\,
|B(y,r)|^{\frac 1u - \frac 1p}\,   \lf[\int_{B(y,r)}
|f^{(\delta)} (x)  - f^{(\delta)}_\ell(x)|^p\, dx\r]^{1/p} \, .
\end{eqnarray*}
Choosing $\ell$ as in \eqref{ws-83}, we find that
the second term on the right-hand side
of the above inequality vanishes.
On the other hand, similar to the arguments used in Substep 1.2, we know that
$$\|f-f^{(\delta)}\|_{\cm_p^u (\rn)}+\|f_\ell-f_\ell^{(\delta)}\|_{\cm_p^u (\rn)}\ls
\varepsilon$$
if $\delta$ is small enough.  Altogether we find that
$f$ can be approximated by $f_\ell$ in $\cm_p^u (\rn)$. This proves $M(\rn)= \accentset{*}{\cm}_p^u (\rn)$.
\\
{\em Step 3.} Proof of (iii). This time we let $M(\rn)$ denote
the collection of all $f \in {\cm}^{u}_{p}(\rn)$
such that \eqref{ws-51} holds true uniformly in  $y \in \rr^n$.

First we show $\accentset{\diamond}{\cm}_p^u (\rn) \hookrightarrow M(\rn)$.
Let $f$ be a $C^\infty(\rn)$ function such that $f$ and all its
derivatives $D^\alpha f$ belong to
${\cm}^{u}_{p}(\rn)$. From this and Proposition \ref{basic1}(iv),
it follows that $f\in W^m {\cm}^{u}_{p}(\rn)$ for any $m\in\nn$, as well as all of
its derivatives, here $W^m {\cm}^{u}_{p}(\rn)$ denotes the Morrey-Sobolev
space of order $m$.
Since $W^m {\cm}^{u}_{p}(\rn) \hookrightarrow L_\infty (\rn)$ for sufficiently large $m$,
we know that $f$ and  its derivatives $D^\alpha f$ are all bounded.
Consequently, we conclude that
\begin{eqnarray*}
|B(y,r)|^{\frac 1u - \frac 1p}\,   \lf[\int_{B(y,r)}
|f(x)|^p\, dx\r]^{1/p}
&& \le   |B(y,r)|^{\frac 1u - \frac 1p}\,   \lf[\int_{B(y,r)}
|f(y)|^p\, dx\r]^{1/p} \\
&&\quad+ |B(y,r)|^{\frac 1u - \frac 1p}\,   \lf[\int_{B(y,r)}
|f(x)-f(y)|^p\, dx\r]^{1/p}
\\
&& \lesssim  |f(y)|\, |B(y,r)|^{\frac 1u}\, +
\max_{|\alpha |=1}\, \| \, D^\alpha f\, \|_{L_\infty (\rn)} \, |B(y,r)|^{\frac 1u + \frac 1n}
\\
&& \lesssim  \max_{|\alpha | \le 1}\, \| \, D^\alpha f\, \|_{L_\infty (\rn)} \, |B(y,r)|^{\frac 1u}\,\max\{r,1\}.
\end{eqnarray*}
Clearly, the right-hand side of the above inequalities
tends to $0$ if $r \downarrow 0$ (uniformly in $y$).
This property carries over to the limits in $\cm^u_p(\rn)$ of such kind of functions.
Hence $\accentset{\diamond}{\cm}_p^u (\rn) \hookrightarrow M(\rn)$.

It remains to prove that any $f \in M(\rn)$ can be approximated by functions  which, together
with all its derivatives, belong to ${\cm}^{u}_{p}(\rn)$.
We shall work with the Sobolev mollification $f^{(\delta)} $ of $f$.
By the definition of the Sobolev mollification and the generalized Minkowski inequality,  we have
\begin{eqnarray}\label{ws-84}
\lf[\int_{B(y,r)}
|D^\alpha f^{(\delta)} (x)|^p\, dx\r]^{1/p}
&&=
\lf[\int_{B(y,r)}
\lf|\delta^{-n-|\alpha|} \int_\rn (D^\alpha \omega) \lf(\frac{x-y}{\delta}\r)f(y)\, dy\r|^p\,
dx\r]^{1/p}\nonumber
\\
&& =
\lf[\int_{B(y,r)}
\lf|\delta^{-n-|\alpha|} \int_\rn (D^\alpha \omega) \lf(\frac{z}{\delta}\r)f(x-z)\, dz\r|^p\,
dx\r]^{1/p}
\nonumber
\\
&&\le
\delta^{-n-|\alpha|} \int_\rn \lf|(D^\alpha \omega) \lf(\frac{z}{\delta}\r)\r|
\lf[\int_{B(y,r)} |f(x-z)|^p dx \r]^{1/p}\, dz
\nonumber
\\
&& \le  \delta^{-|\alpha|}
\sup_{z \in \rn}\lf[\int_{B(z,r)} |f(x)|^p dx \r]^{1/p} \, .
\end{eqnarray}
This shows that also $f^{(\delta)}$ and all its
derivatives $D^\alpha f^{(\delta)} $
belong to $\cm_p^u (\rn)$.
Since $f$ satisfies \eqref{ws-51}, we see that,
for given $\varepsilon\in(0,\fz)$, there exists
$\varepsilon_1\in(0,\fz)$ such that
\[
\sup_{|B| < \varepsilon_1}
|B|^{\frac 1u - \frac 1p}\,   \lf[\int_{B}
|f (x)|^p\, dx\r]^{1/p} \le  \varepsilon \, .
\]
Hence, employing \eqref{ws-84} with $\alpha =0$ and the Minkowski inequality, we find that
\begin{eqnarray*}
\|\, f^{(\delta)}  -  f\,  \|_{\cm_p^u(\rn)}
 & \le &  2 \varepsilon
+
\sup_{|B| \ge  \varepsilon_1}
|B|^{\frac 1u - \frac 1p}\,   \lf[\int_{B}
|f(x)- f^{(\delta)} (x)|^p\, dx\r]^{1/p}
\\
& \le &  2 \varepsilon
+ \delta^{-n}\,
\sup_{|B| \ge  \varepsilon_1}
|B|^{\frac 1u - \frac 1p}\,   \lf\{\int_{B}
\lf| \int_\rn \omega \lf(\frac{x-y}{\delta} \r)[f(x)- f(y)]\, dy
\r|^p\, dx\r\}^{1/p}
\\
& \le &  2 \varepsilon
+ \sup_{|B| \ge  \varepsilon_1}
|B|^{\frac 1u - \frac 1p}\, \sup_{|h|\le \delta}  \lf[\int_{B}
|f(x)- f(x+h)|^p\, dx\r]^{1/p}
\, .
\end{eqnarray*}
By the definition of the supremum, there exists a sequence $\{(y_j,r_j)\}_{j\in\nn}$
such that
\begin{eqnarray*}
\sup_{|B| \ge  \varepsilon_1}  \sup_{|h|\le \delta} &&
|B|^{\frac 1u - \frac 1p}\lf[\int_{B} |f(x)  -  f(x+h)|^p dx\r]^{1/p}
\nonumber
\\
& < &
\frac{1}{j} +  \sup_{|h|\le \delta} |B(y_j,r_j)|^{\frac 1u - \frac 1p}\, \lf[\int_{B(y_j,r_j)}
|f(x)- f(x+h)|^p\, dx\r]^{1/p}
\le 2 \varepsilon,
\end{eqnarray*}
if $j$ is large enough and if $\delta $ is small enough
(since, for a fixed $j$,
we can apply the $L_p(\rn)$-continuity of the translation).
Inserting this inequality into the previous one, we are done,
which completes the proof of Lemma \ref{morrey43}.


\subsection*{Proof of Lemma  \ref{morrey410}}


Part (i) is already proved by using the example \eqref{falpha}. Part (ii) follows from
Lemma \ref{morrey43}. We focus on (iii).
Obviously, the function
$g_{n/u}$ in \eqref{galpha} belongs to $ \accentset{*}{\cm}^{u}_{p}(\rn)$.
However, it does not belong to  $\accentset{\diamond}{\cm}^{u}_{p}(\rn)$, since,
 according to Lemma \ref{morrey43}(iii),
functions from this space satisfy \eqref{ws-51} but
\[
\lim_{r \downarrow 0} |B(0,r)|^{\frac 1u - \frac 1p} \lf[\int_{B(0,r)} |x|^{np/u}\, dx\r] >0\,,
\]
which implies that $g_{n/u}$ does not satisfies \eqref{ws-51}.
Moreover, the function $h_{n/u}$ in  \eqref{halpha} belongs to $ \accentset{\diamond}{\cm}^{u}_{p}(\rn)$,
since $h_{n/u}$ is a $C^\infty(\rn)$ function such that
all derivatives also belong to $\cm_p^u (\rn)$.
It does not belong to  $\accentset{*}{\cm}^{u}_{p}(\rn)$, since,
according to Lemma \ref{morrey43}(iii),
functions from this space satisfy \eqref{ws-50} but
\[
\lim_{r \downarrow \infty} |B(0,r)|^{\frac 1u - \frac 1p} \lf[\int_{B(0,r)} |x|^{np/u}\, dx\r] >0\,
\]
which implies that $h_{n/u}$ does not satisfies \eqref{ws-50}.
This proves (iii) and hence finishes the proof of Lemma  \ref{morrey410}.


\subsection*{Proof of Lemma \ref{morrey41}}


{\em Step 1.} Preliminaries.
For any ball $B\subset \rn$, the H\"older inequality yields
\[
\lf[\int_B |f(x)|^p dx\r]^{1/p} \le
\lf[\int_B |f(x)|^{p_0} dx\r]^{(1-\Theta)/p_0} \, \lf[\int_B |f(x)|^{p_1} dx\r]^{\Theta/p_1} \, ,
\]
which implies that
\[
|B|^{\frac 1u - \frac 1p} \lf[\int_B |f(x)|^p dx\r]^{1/p}  \le
\|\, f\, \|_{\cm_{p_0}^{u_0}(\rn)}^{1-\Theta} \, \|\, f\, \|_{\cm_{p_1}^{u_1}(\rn)}^\Theta\, .
\]
In other words, $$\cm_{p_0}^{u_0}(\rn) \cap \cm_{p_1}^{u_1}(\rn) \hookrightarrow \cm_{p}^{u}(\rn).$$
\\
{\em Step 2.} Suppose $f \in \cm_{p_0}^{u_0}(\rn) \cap \cm_{p_1}^{u_1}(\rn)$. By
$p_0 \le p \le p_1$ and the H\"older inequality, we know that
\begin{eqnarray}\label{2.36x1}
 |B(y,r)|^{\frac 1u - \frac 1p}\lf[\int_{B(y,r)} |f(x)|^p dx\r]^{1/p}
&&\le  |B(y,r)|^{\frac 1u - \frac{1}{p_1}} \lf[\int_{B(y,r)} |f(x)|^{p_1} dx\r]^{1/p_1}\nonumber
\\
 && = |B(y,r)|^{\frac 1u - \frac{1}{u_1}}
|B(y,r)|^{\frac{1}{u_1} - \frac{1}{p_1}}
\lf[\int_{B(y,r)} |f(x)|^{p_1} dx\r]^{1/p_1}\nonumber
\\
 && \le  |B(y,r)|^{\frac 1u - \frac{1}{u_1}}
\|\, f \, \|_{\cm_{p_1}^{u_1}(\rn)} \, ,
\end{eqnarray}
which converges to $0$ as $r\to0$, due to $u_1 >u$. By Lemma \ref{morrey43}(iii),
this proves $f \in \accentset{\diamond}{\cm}^{u}_{p}(\rn)$ if $p \in [1,\fz)$.
Thus, if  $p \in [1,\fz)$, $$\cm_{p_0}^{u_0}(\rn) \cap \cm_{p_1}^{u_1}(\rn)\hookrightarrow
\accentset{\diamond}{\cm}^{u}_{p}(\rn).$$

Now we argue by using our test function $g_\alpha$ in \eqref{galpha} with $\alpha = n/u$
to show
that $\cm_{p_0}^{u_0}(\rn) \cap \cm_{p_1}^{u_1}(\rn)$ is not
dense in $\cm^u_p(\rn)$.
It is known that this time $g_\az\in \cm^u_p(\rn)$.
From $p< u$ and \eqref{2.36x1},  it follows
that
\begin{eqnarray*}
&&\lim_{r \downarrow 0} \, |B(0,r)|^{\frac 1u - \frac 1p}
\lf[\int_{B(0,r)} |g_\alpha (x) - f(x)|^p dx\r]^{1/p}
=  \lim_{r \downarrow 0} \, |B(0,r)|^{\frac 1u - \frac 1p}
\lf[\int_{B(0,r)} |g_\alpha (x)|^p dx\r]^{1/p} >0
\end{eqnarray*}
for all $f \in \cm_{p_0}^{u_0}(\rn) \cap \cm_{p_1}^{u_1}(\rn)$.
This means that there exist functions in $\cm^u_p(\rn)$ that can not
be approximated by functions from $\cm_{p_0}^{u_0}(\rn) \cap \cm_{p_1}^{u_1}(\rn)$,
namely, $\cm_{p_0}^{u_0}(\rn) \cap \cm_{p_1}^{u_1}(\rn)$ is not dense in $\cm^u_p(\rn)$,
which completes the proof of Lemma \ref{morrey41}.

\begin{remark}
In the above proof of Lemma \ref{morrey41}, we prove more than stated.
Indeed, we show
\begin{enumerate}
\item[{\rm (i)}] $\accentset{*}{\cm}^{u}_{p}(\rn) \not\subset
\overline{\cm_{p_0}^{u_0}(\rn) \cap
\cm_{p_1}^{u_1}(\rn)}^{\|\, \cdot \, \|_{\cm_p^u(\rn)}} \, ,$
\end{enumerate}
but
\begin{enumerate}
\item[{\rm (ii)}] $\overline{
\cm_{p_0}^{u_0}(\rn) \cap \cm_{p_1}^{u_1}(\rn)}^{\|\, \cdot \, \|_{\cm_p^u(\rn)}}
\hookrightarrow \accentset{\diamond}{\cm}^{u}_{p}(\rn).$
\end{enumerate}
\end{remark}


\subsection*{Proof of Corollary \ref{morrey42}}


{\em Step 1}. Proof of (i). Theorem \ref{morrey1}(i) and  Proposition \ref{t-n}
yield
\[
\lf\laz\cm^{u_0}_{p_0}(\rn), \cm^{u_1}_{p_1}(\rn)\r\raz_\tz \hookrightarrow
\lf[\cm^{u_0}_{p_0}(\rn)\r]^{1-\Theta} \, \lf[ \cm^{u_1}_{p_1}(\rn)\r]^\tz \hookrightarrow  \cm^{u}_{p}(\rn)
\, .
\]
Theorem \ref{morrey1}(iii) implies that this embedding is proper
if $u_0 p_1 \neq u_1 p_0$.

It remains to consider the case $u_0 p_1 = u_1 p_0$.
Without loss of generality, we may assume $p_0 \le p_1$.
Under these conditions, by Proposition \ref{nil} and
Corollary \ref{cima}(i), we see that
\begin{eqnarray}\label{2.37x1}
\lf\laz\cm^{u_0}_{p_0}(\rn), \cm^{u_1}_{p_1}(\rn)\r\raz_\tz =
\overline{\cm^{u_0}_{p_0}(\rn) \cap  \cm^{u_1}_{p_1}(\rn)}^{\|\, \cdot \, \|_{\cm_p^u(\rn)}}\, .
\end{eqnarray}
If $p_0=u_0$ and $p_1=u_1$, then
we have
\[
\lf\laz\cm^{p_0}_{p_0}(\rn), \cm^{p_1}_{p_1}(\rn)\r\raz_\tz =
\overline{L_{p_0}(\rn) \cap  L_{p_1}(\rn)}^{\|\, \cdot \, \|_{L_p(\rn)}}=L_p(\rn)=\cm^p_p(\rn).
\]
If  $p_0 =  p_1$, we must have  $u_0 = u_1 $ and therefore
\[
\lf\laz\cm^{u_0}_{p_0}(\rn), \cm^{u_1}_{p_1}(\rn)\r\raz_\tz = \lf\laz\cm^{u_0}_{p_0}(\rn), \cm^{u_0}_{p_0}(\rn)\r\raz_\tz
= \cm^{u_0}_{p_0}(\rn) \, .
\]
If $p_0\neq u_0$ and $p_0 <  p_1 $, then $u_0 < u_1 $ follows  and therefore $u < u_1$.
In this case,  Lemma  \ref{morrey41} yields the desired conclusion.
This proves (i).

\noindent{\em Step 2}. Proof of (ii).
By Proposition \ref{nil} and Corollary \ref{cima}(i), it is known that
\[
\mathring{\cm}^{u}_{p}(\rn) \hookrightarrow
\lf\laz \mathring{\cm}^{u_0}_{p_0}(\rn), \mathring{\cm}^{u_1}_{p_1}(\rn)\r\raz_\tz
\hookrightarrow \lf\laz {\cm}^{u_0}_{p_0}(\rn), {\cm}^{u_1}_{p_1}(\rn)\r\raz_\tz
\, .
\]
On the other hand, it is easy to show that
the test function $g_\alpha$ in \eqref{galpha} with $\alpha = n/u_1$
belongs to $\cm^{u_0}_{p_0}(\rn) \cap  \cm^{u_1}_{p_1}(\rn)$, and hence belongs to
$\lf\laz {\cm}^{u_0}_{p_0}(\rn), {\cm}^{u_1}_{p_1}(\rn)\r\raz_\tz$.
However, Step 2 of the proof of Lemma \ref{morrey41} implies that $g_\az$
can not be approximated by $C_c^\fz(\rn)$ functions in
the norm $\|\cdot\|_{\cm^u_p(\rn)}$,
due to Lemma \ref{morrey43}(i). Namely, $g_\az$ does not belong to
$\mathring{\cm}^{u}_{p}(\rn)$. This proves (ii).

\noindent{\em Step 3}. Proof of (iii). Since $p_0 < p_1$, it follows $u_0 < u_1$.
Hence, we apply Lemma \ref{morrey41} in case $p \in [1,\fz)$ and conclude
$$\lf\laz\cm^{u_0}_{p_0}(\rn), \cm^{u_1}_{p_1}(\rn)\r\raz_\tz \hookrightarrow \accentset{\diamond}{\cm}^{u}_{p}(\rn).$$
It remains to show that these spaces do not coincide.
This time we argue with our test function $h_{n/u}$ in \eqref{halpha}.
Clearly, $h_{n/u} \in \accentset{\diamond}{\cm}^{u}_{p}(\rn)$. We now show that
$$h_{n/u}\notin \laz\cm^{u_0}_{p_0}(\rn), \cm^{u_1}_{p_1}(\rn)\raz_\tz.$$
By \eqref{2.37x1}, we only need to
show that $$h_{n/u}\notin
\overline{\cm^{u_0}_{p_0}(\rn) \cap  \cm^{u_1}_{p_1}(\rn)}^{\|\, \cdot \, \|_{\cm_p^u(\rn)}}.$$
Observe that, by an elementary calculation, we know that, for all $t\in (0,\infty)$,
\[
2^{jn (\frac 1u - \frac 1t)} \, \lf[\int_{2^j \le |x|\le 2^{j+1}} |h_{n/u}(x)|^t \, dx\r]^{1/t} = C_{(n,t,u)} >0\,, \qquad
j\in \nn \, ,
\]
where $C_{(n,t,u)}$ denotes a positive constant depending on $n$, $t$ and $u$.
Now, for any given $\varepsilon\in(0,\fz)$, assume that there exists $f_\varepsilon \in  \cm^{u_0}_{p_0}(\rn) \cap  \cm^{u_1}_{p_1}(\rn)$ such that
$\| \, h_{n/u} - f_\varepsilon\, \|_{\cm_p^u(\rn)} < \varepsilon $.
Without loss of generality, we may assume $|f_\varepsilon(x)|\le |x|^{-n/u}$ for all
$x\in\rn\setminus\{(0,\ldots,0)\}$
(otherwise, we switch to $\min\{|f_\varepsilon(x)|, |x|^{-n/u}\}$).
Fixing $\varepsilon\in(0,\fz)$ sufficiently small, we conclude that
\begin{equation}\label{ws-94}
2^{jn (\frac 1u - \frac 1p)} \,
\lf[\int_{2^j \le |x|\le 2^{j+1}} |f_\varepsilon(x)|^p \, dx\r]^{1/p} > \frac{C_{(n,p,u)}}{2}\, , \qquad j \in \nn\, ,
\end{equation}
where $C_{(n,p,u)}$ denotes a positive constant depending on $n$, $p$ and $u$.
But this is in contradiction with $f_\varepsilon \in  \cm^{u_0}_{p_0}(\rn) \cap  \cm^{u_1}_{p_1}(\rn)$.
To explain this  contradiction, by our pointwise assumption $|f_\varepsilon (x)|\le |x|^{-n/u}$
for all $x\in\rn\setminus\{(0,\ldots,0)\}$,
we see that
\begin{equation}\label{ws-93}
\lf[\int_{2^j \le |x|\le 2^{j+1}} |f_\varepsilon (x)|^{p_1} \, dx\r]^{1/p_1}
\lesssim 2^{jn (\frac{1}{p_1}- \frac 1u)} \, , \qquad j \in \nn\, .
\end{equation}
Finally, we employ the H\"older inequality, \eqref{ws-94}
and \eqref{ws-93} to find
that
\begin{eqnarray*}
\frac{C_{(n,p,u)}}{2} &<&  \, 2^{jn (\frac 1u - \frac 1p)} \,
\lf[\int_{2^j \le |x|\le 2^{j+1}} |f_\varepsilon (x)|^p \, dx\r]^{1/p}
\\
& \le & \lf\{2^{jn (\frac 1{u_0} - \frac 1{p_0})} \,
\lf[\int_{2^j \le |x|\le 2^{j+1}} |f_\varepsilon (x)|^{p_0} \, dx\r]^{1/p_0}\r\}^{1-\Theta}
\\
&& \qquad \qquad \times \quad
\, \lf\{2^{jn (\frac 1{u_1} - \frac1{p_1})} \,
\lf[\int_{2^j \le |x|\le 2^{j+1}} |f_\varepsilon (x)|^{p_1} \, dx\r]^{1/p_1}\r\}^\Theta
\\
& \lesssim & \|\, f_\varepsilon \,\|_{\cm_{p_0}^{u_0}(\rn)}^{1-\Theta}
\, 2^{jn (\frac 1{u_1} - \frac1{u})\Theta} \, , \qquad j \in \nn\, .
\end{eqnarray*}
Because of $u < u_1$, the right-hand side
of the above inequalities tends to zero for $j$ tending to infinity.
This is a contradiction. Hence, our assumption, that there is a function
$f_\varepsilon  \in  \cm^{u_0}_{p_0}(\rn) \cap  \cm^{u_1}_{p_1}(\rn)$ in an $\varepsilon$ distance of $h_{n/u}$,
is impossible. This proves $$\lf\laz\cm^{u_0}_{p_0}(\rn),
\cm^{u_1}_{p_1}(\rn)\r\raz_\tz \subsetneqq \accentset{\diamond}{\cm}^{u}_{p}(\rn),$$
which completes the proof of Corollary \ref{morrey42}.


\subsection*{Proof of Theorem \ref{gp01}}


{\em Step 1.} We prove $$
\lf\laz\cm^{u_0}_{p_0}(\rn), \cm^{u_1}_{p_1}(\rn)\r\raz_\tz \hookrightarrow
\cm^{u_0,u_1,\tz}_{p_0,p_1} (\rn).$$
Let $f \in \cm^{u_0}_{p_0}(\rn) \cap \cm^{u_1}_{p_1}(\rn)$. Then $f \in \accentset{\diamond}{\cm}^{u}_{p}(\rn)$ follows from
Lemma \ref{morrey41}, which further
implies the validity of \eqref{ws-95a} and \eqref{ws-95b}  due to Lemma \ref{morrey43}(iii).
The conditions \eqref{ws-96} and \eqref{ws-97} are obviously true since $f \in \cm^{u_0}_{p_0}(\rn) \cap \cm^{u_1}_{p_1}(\rn)$.
Thus, we conclude that $f\in \cm^{u_0,u_1,\tz}_{p_0,p_1} (\rn)$,
which implies that
\[
\cm^{u_0}_{p_0}(\rn) \cap \cm^{u_1}_{p_1}(\rn) \hookrightarrow \cm^{u_0,u_1,\tz}_{p_0,p_1} (\rn) \, .
\]
Taking the closure in $\cm_p^u (\rn)$, by
Proposition \ref{nil} and Corollary \ref{cima}(i), we obtain $$
\lf\laz\cm^{u_0}_{p_0}(\rn), \cm^{u_1}_{p_1}(\rn)\r\raz_\tz \hookrightarrow
\cm^{u_0,u_1,\tz}_{p_0,p_1} (\rn).$$

\noindent{\em Step 2.} It remains to show
$$ \cm^{u_0,u_1,\tz}_{p_0,p_1} (\rn) \hookrightarrow
\lf\laz\cm^{u_0}_{p_0}(\rn), \cm^{u_1}_{p_1}(\rn)\r\raz_\tz.$$
We claim
\begin{equation}\label{ws-98}
 \cm^{u_0,u_1,\tz}_{p_0,p_1} (\rn) = \accentset{\diamond}{\cm}^{u_0,u_1,\tz}_{p_0,p_1} (\rn)\, .
\end{equation}
To prove this claim, we need $p_0 \in[1,\fz)$.
Clearly, we work again with the Sobolev mollification. As in Step 3 of the proof of Lemma \ref{morrey43},
it follows that, for any $f\in \cm^{u_0,u_1,\tz}_{p_0,p_1} (\rn)$,
\[
I_1 (f-f^{(\delta)}) + I_2 (f-f^{(\delta)}) + I_3 (f-f^{(\delta)}) \to 0
\]
if $\delta \downarrow 0$, which implies $f\in
\cm^{u_0}_{p_0}(\rn) \cap \cm^{u_1}_{p_1}(\rn)$
and hence
\begin{equation}\label{4.33x1}
 \cm^{u_0,u_1,\tz}_{p_0,p_1} (\rn) \hookrightarrow \accentset{\diamond}{\cm}^{u_0,u_1,\tz}_{p_0,p_1} (\rn)\, .
\end{equation}

Now, let $f\in  \accentset{\diamond}{\cm}^{u_0,u_1,\tz}_{p_0,p_1} (\rn)$, namely, $f \in \cm^{u_0,u_1,\tz}_{p_0,p_1} (\rn)$ such that all derivatives $D^\alpha f$, $\alpha \in \zz_+^n$, belong to
$\cm^{u_0,u_1,\tz}_{p_0,p_1} (\rn)$ as well.
By the definition of $\cm^{u_0,u_1,\tz}_{p_0,p_1} (\rn)$
and the H\"older inequality, we obtain $D^\alpha f \in \cm_p^u (\rn) $ for all $\alpha \in \zz_+^n$.
Since $W^m (\cm_p^u (\rn)) \hookrightarrow L_\infty (\rn)$ if $m >n/p$ (see \cite{dch}),
we conclude $f \in L_\infty (\rn)$.
By Definition \ref{d2.38},
it is easy to see that $$L_\infty (\rn)\cap {\cm}^{u_0,u_1,\tz}_{p_0,p_1} (\rn)
 \hookrightarrow \cm^{u_0}_{p_0}(\rn) \cap \cm^{u_1}_{p_1}(\rn).$$
 Therefore, by this, Proposition \ref{nil}
 and Step 1, we further conclude that
\begin{eqnarray}\label{4.33x2}
\accentset{\diamond}{\cm}^{u_0,u_1,\tz}_{p_0,p_1}(\rn)  & \hookrightarrow & \overline{\cm^{u_0}_{p_0}(\rn) \cap \cm^{u_1}_{p_1}(\rn)}^
{\|\, \cdot \, \|_{\cm_p^u(\rn)}} = \lf\laz\cm^{u_0}_{p_0}(\rn), \cm^{u_1}_{p_1}(\rn)\r\raz_\tz\nonumber
\\
& \hookrightarrow &
\cm^{u_0,u_1,\tz}_{p_0,p_1} (\rn)\, ,
\end{eqnarray}
since the convergence in $\|\, \cdot \, \|_{{\cm}^{u_0,u_1,\tz}_{p_0,p_1} (\rn)}$ implies the convergence in
$\|\, \cdot \, \|_{{\cm}^{u}_{p} (\rn)}$ by the H\"older inequality.
Combining \eqref{4.33x1} and \eqref{4.33x2},
we obtain \eqref{ws-98},
which, together with \eqref{4.33x2} again, implies that
$$\lf\laz\cm^{u_0}_{p_0}(\rn), \cm^{u_1}_{p_1}(\rn)\r\raz_\tz
=
\cm^{u_0,u_1,\tz}_{p_0,p_1} (\rn).$$
This finishes the proof of Theorem  \ref{gp01}.


\subsection*{Proof of Theorem \ref{gp03}}


First we have to prove
the formula
\begin{equation}\label{ws-104}
\lf\laz\cm^{u_0}_{p_0}((0,1)^n), \cm^{u_1}_{p_1}((0,1)^n)\, , \tz\r\raz = \cm^{u}_{p}((0,1)^n) \, .
\end{equation}
This can be done by the method of the retraction and the coretraction.
Here the coretraction is defined as
the extension from $(0,1)^n$ to $\rn$ by zero.
Hence, \eqref{ws-104} becomes a consequence of Corollary \ref{cima}(i).
Second, we employ  Proposition \ref{nil} to obtain
\[
\lf\laz\cm^{u_0}_{p_0}((0,1)^n), \cm^{u_1}_{p_1}((0,1)^n)\r\raz_\tz =
\overline{\cm^{u_0}_{p_0}((0,1)^n) \cap \cm^{u_1}_{p_1}((0,1)^n)}^{\|\, \cdot \, \|_{\cm^{u}_{p}((0,1)^n)}} \, .
\]
Next we continue, as in Lemma \ref{morrey41}, to conclude that
\[
\cm^{u_0}_{p_0}((0,1)^n) \cap \cm^{u_1}_{p_1}((0,1)^n) \hookrightarrow \accentset{\diamond}{\cm}^{u}_{p}((0,1)^n)\, .
\]
Finally, we observe
\begin{eqnarray*}
\Big\{ f \in C^\infty ((0,1)^n): &&  D^\alpha f \in L_\infty ((0,1)^n) \quad \mbox{for all}\quad \alpha \in \zz_+^n \Big\}
\\
& \subset  & \: \Big(\cm^{u_0}_{p_0}((0,1)^n) \cap \cm^{u_1}_{p_1}((0,1)^n)\Big)
\:  \subset \accentset{\diamond}{\cm}^{u}_{p}((0,1)^n)\, .
 \end{eqnarray*}
Taking the closure in both sides of the above formula
with respect to the norm $\|\, \cdot \, \|_{\cm_p^u (\rn)}$ (see Lemma \ref{gp02}), we then obtain
the desired conclusion of Theorem  \ref{gp03}.


\subsection*{Proof of Theorem \ref{gagl5}}


To prove Theorem  \ref{gagl5}, we need
the following conclusions, which have their own interest. The first one shows that,
if $\Omega$ is a Lipschitz domain, then there exists a universal linear bounded
extension operator
from $\cn^s_{u,p,q}(\Omega)$ into $\cn^s_{u,p,q}(\rn)$
for all $s\in\rr$, $q\in(0,\fz]$ and $0<p\le u<\fz$. In the construction of this operator
we follow Rychkov \cite[Theorem 2.2]{ry99}.

\begin{proposition}\label{exten}
Let $\Omega \subset \rn$ be an interval if $n=1$ or a Lipschitz domain if $n\ge 2$.
Then there exists
a linear bounded operator $\mathcal{E}$ which maps
$\cn^s_{u,p,q}(\Omega)$ into $\cn^s_{u,p,q}(\rn)$
for all $s\in\rr$, $q\in(0,\fz]$ and $0<p\le u<\fz$ such that, for all
$f\in D'(\Omega)$,
$\mathcal{E}f|_\Omega=f$ in $D'(\Omega)$.
\end{proposition}

\begin{proof}
By similarity, we concentrate us on the case $n\ge 2$.
A standard procedure (see, for example, \cite[Subsection 1.2]{ry99})
shows that, to prove Proposition \ref{exten},  we only need to consider
the case when $\Omega$ is a special Lipschitz domain.
In this case, let
$$K:=\{(x',x_n)\in\rn:\ |x'|<A^{-1}x_n\}$$
and $-K:=\{-x:\ x\in K\}$,
where $A$ is the Lipschitz constant of the boundary Lipschitz
function $\omega$ of $\Omega$.
Then $K$ has the property that $x+K\subset \Omega$ for any $x\in \Omega$.

Let $\phi_0\in D(-K)$ and $\phi(\cdot):=\phi_0(\cdot)-\phi_0(\cdot/2)$ be
such that $\int_\rn \phi_0(x)\,dx\neq 0$ and
$L_\phi\ge \lfloor s \rfloor$. Here and hereafter, $L_\phi$ denotes the
maximal number such that $\int_\rn \phi(x)x^\az\,dx=0$ for all
$\az\in\zz_+^n$ with $|\az|\le L_\phi$.
Then, by \cite[Proposition 2.1]{ry99}, there exist functions $\psi_0$
and $\psi$ in $D(-K)$ such that $L_\psi\ge L_\phi$
and, for all $f\in D'(\Omega)$,
$$f=\sum_{j\in\zz_+} \psi_j\ast \phi_j\ast f$$
in $D'(\Omega)$.
For all $f\in D'(\Omega)$, we define
\begin{equation}\label{extope}
\mathcal{E} f:=\sum_{j\in\zz_+} \psi_j\ast (\phi_j\ast f)_\Omega,
\end{equation}
here and hereafter, for any function $g:\ \Omega\to \rr$,
$g_\Omega$ denotes the extension of $g$ from $\Omega$ to
$\rn$ by setting $g_\Omega(x):=g(x)$ if $x\in\Omega$
and $g_\Omega(x):=0$ if $x\in \rn\setminus\Omega$.

For all $s\in\rr$, $q\in(0,\fz]$ and $0<p\le u<\fz$, let $\ell_q^s(\cm^u_p(\rn))$
be the space of all sequences $\{g_j\}_{j\in\zz_+}$ of measurable functions on $\rn$
such that
$$\|\{g_j\}_{j\in\zz_+}\|_{\ell_q^s(\cm^u_p(\rn))}:=\lf\{\sum_{j\in\zz_+}2^{jsq}\lf\|G_j\r\|_{\cm^u_p(\rn)}^q
\r\}^{1/q}<\fz,$$
where $G_j$ denotes the Peetre maximal function of $g_j$, namely,
$$G_j(x):=\sup_{y\in\rn}\frac{|g_j(y)|}{(1+2^j|x-y|)^N}$$ for all $x\in\rn$
and $N\in\nn\cap (\frac{n}{\min\{1,p\}},\fz)$. By \cite[(2.14)]{ry99}, we know that,
if $L_\phi\ge \lfloor s\rfloor$ and $L_\psi\ge N$, then
there exists $\sigma\in(0,\fz)$ such that, for any sequence $\{g_j\}_{j\in\zz_+}$ with $\|\{g_j\}_{j\in\zz_+}\|_{\ell_q^s(\cm^u_p(\rn))}<\fz$,
it holds true that
\begin{equation*}
2^{ls}|\phi_l\ast \psi_j\ast g_j(x)|\ls
2^{-|l-j|\sigma}2^{js} G_j(x),\quad x\in\rn,\ l\in\zz_+,
\end{equation*}
and hence
\begin{eqnarray*}
\|\psi_j\ast g_j\|_{\cn^{s-2\sigma}_{u,p,q}(\rn)}
&&\ls \lf[\sum_{l\in\zz_+}2^{l(-2\sigma+|l-j|\sigma)q}\r]^{1/q}\|2^{js}G_j\|_{\cm^u_p(\rn)}
\ls 2^{-j\sigma}\|\{g_j\}_{j\in\zz_+}\|_{\ell_q^s(\cm^u_p(\rn))}.
\end{eqnarray*}
This implies that $\sum_{j\in\zz_+}\psi_j\ast g_j$ converges in
$\cn^{s-2\sigma}_{u,p,q}(\rn)$ and hence in $\cs'(\rn)$,
since $\cn^{s-2\sigma}_{u,p,q}(\rn)\hookrightarrow \cs'(\rn)$.
Therefore, we further have
\begin{equation*}
2^{ls}\lf|\phi_l\ast \lf(\sum_{j\in\zz_+}\psi_j\ast g_j\r)(x)\r|\ls \sum_{j\in\zz_+}
2^{-|l-j|\sigma}2^{js} G_j(x),\quad x\in\rn,\ l\in\zz_+.
\end{equation*}
Applying this, we then see that
\begin{eqnarray}\label{extinq}
\lf\|\sum_{j\in\zz_+}\psi_j\ast g_j\r\|_{\cn^{s}_{u,p,q}(\rn)}
&&\ls \|\{g_j\}_{j\in\zz_+}\|_{\ell_q^s(\cm^u_p(\rn))}.
\end{eqnarray}

Let $f\in \cn^s_{u,p,q}(\Omega)$. Then, for any $\varepsilon\in(0,\fz)$,
there exists $g\in \cn^s_{u,p,q}(\rn)$
such that $g|_\Omega=f$ in $D'(\Omega)$ and
$$\|g\|_{\cn^s_{u,p,q}(\rn)}\le \|f\|_{\cn^s_{u,p,q}(\Omega)}+\varepsilon.$$
Let $g_j:=(\phi_j\ast f)_\Omega$  for all $j\in\zz_+$. By \eqref{extinq}
and the fact (see  \cite[pp.\,247-248]{ry99}) that
\begin{eqnarray*}
\sup_{y\in \Omega}\frac{|\phi_j\ast f(y)|}{(1+2^j|x-y|)^N}&&\lf\{
\begin{array}
    {l@{}l}
=\dsup_{y\in \Omega}\frac{|\phi_j\ast f(y)|}{(1+2^j|x-y|)^N},\ &\ x\in \Omega,\\
\ls\dsup_{y\in \Omega}\frac{|\phi_j\ast f(y)|}{(1+2^j|\wz{x}-y|)^N},\ &\ x\notin \overline{\Omega},
\end{array}
\right.\\
&&\lf\{
\begin{array}
    {l@{}l}
\le\dsup_{y\in \rn}\frac{|\phi_j\ast g(y)|}{(1+2^j|x-y|)^N},\ &\ x\in \Omega,\\
\ls\dsup_{y\in \rn}\frac{|\phi_j\ast g(y)|}{(1+2^j|\wz{x}-y|)^N},\ &\ x\notin \overline{\Omega},
\end{array}
\right.\\
\end{eqnarray*}
where $\wz{x}:=(x', 2w(x')-x_n)\in \Omega$ is the symmetric point to
$x=(x',x_n)\notin \overline{\Omega}$ with respect to $\partial \Omega$,
we conclude that
\begin{eqnarray*}
\lf\|\mathcal{E}f\r\|_{\cn^{s}_{u,p,q}(\rn)}
&&\ls \|\{(\phi_j\ast f)_\Omega\}_{j\in\zz_+}\|_{\ell_q^s(\cm^u_p(\rn))}\\
&&\ls \lf\{\sum_{j\in\zz_+} 2^{jsq}\lf\|\sup_{y\in \rn}\frac{|\phi_j\ast g(y)|}{(1+2^j|\cdot-y|)^N}\r\|_{\cm^u_p(\rn)}\r\}^{1/q},
\end{eqnarray*}
which, together with the characterization of $\cn^s_{u,p,q}(\rn)$
via the Peetre maximal function (see, for example,
\cite[Subsection 11.2]{lsuyy}) and the choice of $g$,
further implies that
\begin{eqnarray*}
\lf\|\mathcal{E}f\r\|_{\cn^{s}_{u,p,q}(\rn)}
&&\ls \|g\|_{\cn^{s}_{u,p,q}(\rn)}\ls  \|f\|_{\cn^s_{u,p,q}(\Omega)}+\varepsilon.
\end{eqnarray*}
Letting $\varepsilon\to 0$, we then know that
$\mathcal{E}$ is a bounded linear operator from $\cn^s_{u,p,q}(\Omega)$
into $\cn^s_{u,p,q}(\rn)$.

Finally, since the supports of $\psi_0$ and $\psi$ lie in $-K$, it follows that
$$\mathcal{E}f|_\Omega=\sum_{j\in\zz_+} \psi_j\ast \phi_j\ast f=f$$
in $D'(\Omega)$ (see page 249 of \cite{ry99}). Thus, $\mathcal{E}$ is the desired
extension operator from $\cn^s_{u,p,q}(\Omega)$
into $\cn^s_{u,p,q}(\rn)$. This finishes the proof of Proposition \ref{exten}.
\end{proof}

\begin{remark}\label{obs}
 Let $s\in\rr$, $q\in(0,\fz]$,
$0<p\le u<\fz$ and $\Omega \subset \rn$ be an interval if $n=1$ or a Lipschitz domain if $n\ge 2$.
One advantage  of the construction of $\mathcal{E}$ in
\eqref{extope} lies in that, if $f\in \cn^s_{u,p,q}(\Omega)$
and $g:=\mathcal{E}f$ is the extension of $f$ to $\rn$,
then  $\partial^\az g=\mathcal{E}(\partial^\az f)$ for all
$\az\in\zz_+^n$ (see \cite[(4.70)]{t08}).
\end{remark}

By the observation in the above  remark, we have
the following characterization of $\cn^s_{u,p,q}(\Omega)$.

\begin{proposition}\label{lift}
Let  $q\in(0,\fz]$, $0<p\le u<\fz$, $s=\sigma+k$ with $\sigma\in\rr$ and $k\in\nn$,
 and $\Omega \subset \rn$ be an interval if $n=1$ or a Lipschitz domain if $n\ge 2$.
Then
$$\cn^s_{u,p,q}(\Omega)=\lf\{f\in \cn^\sigma_{u,p,q}(\Omega):\
\partial^\az f\in \cn^\sigma_{u,p,q}(\Omega),\ |\az|\le k\r\}$$
and there exists a positive constant $C\in[1,\fz)$ such that, for all $f\in \cn^s_{u,p,q}(\Omega)$,
$$C^{-1}\|f\|_{\cn^s_{u,p,q}(\Omega)}\le \sum_{|\az|\le k} \|\partial^\az f\|_{\cn^\sigma_{u,p,q}(\Omega)}\le C\|f\|_{\cn^s_{u,p,q}(\Omega)}.$$
\end{proposition}

\begin{proof}
It is known from \cite[Theorem 2.15(i)]{tx} that Proposition \ref{lift}
holds true when $\Omega =\rn$. From this and Remark \ref{obs},
we deduce that, for all $f\in \cn^s_{u,p,q}(\Omega)$,
\begin{eqnarray*}
\sum_{|\az|\le k} \|\partial^\az f\|_{\cn^\sigma_{u,p,q}(\Omega)}
\le \sum_{|\az|\le k} \|\partial^\az (\mathcal{E}f)\|_{\cn^\sigma_{u,p,q}(\rn)}
\ls\|\mathcal{E}f\|_{\cn^s_{u,p,q}(\rn)}\ls \|f\|_{\cn^s_{u,p,q}(\Omega)}.
\end{eqnarray*}
Conversely, by \cite[Theorem 2.15(i)]{tx}  and Remark \ref{obs}
again, we  see that
\begin{eqnarray*}
 \|f\|_{\cn^s_{u,p,q}(\Omega)}&&\le \|\mathcal{E} f\|_{\cn^s_{u,p,q}(\rn)}
 \ls \sum_{|\az|\le k} \|\partial^\az (\mathcal{E}f)\|_{\cn^\sigma_{u,p,q}(\rn)}\sim
 \sum_{|\az|\le k} \|\mathcal{E} (\partial^\az f)\|_{\cn^\sigma_{u,p,q}(\rn)}\\
 &&\ls \sum_{|\az|\le k} \|\partial^\az f\|_{\cn^\sigma_{u,p,q}(\Omega)},
\end{eqnarray*}
which completes the proof of Proposition \ref{lift}.
\end{proof}

Now we are ready to prove  Theorem  \ref{gagl5}.

\begin{proof}[Proof of Theorem  \ref{gagl5}]
For brevity, we put
\[
N^s_{u,p,q}(\Omega) := \Big\{f \in \cn^s_{u,p,q} (\Omega): \quad
D^\alpha f \in \cn^s_{u,p,q} (\Omega) \quad \mbox{for all}\quad \alpha \in \zz_+^n\Big\} \, .
\]
We claim that $N^s_{u,p,q}(\Omega) $ is independent of $s,\ u,\ p$ and $q$.
Indeed, by Proposition \ref{lift}, we know that, if $f\in N^s_{u,p,q}(\Omega)$,
then $f\in \cn^\sigma_{u,p,q}(\Omega)$
for any $\sigma\in\rr$, and hence $\mathcal{E}f\in \cn^\sigma_{u,p,q}(\rn)$
for any $\sigma\in\rr$.
In addition, we mention the embedding
\[
\cn^s_{u,p,q} (\rn) \hookrightarrow C_{ub}(\rn) \qquad \mbox{if}\quad s>n/p\,;
\]
see Kozono and Yamazaki \cite{KY} or Sickel \cite{s011a}.
Combining these two arguments, we find that
\[
N^s_{u,p,q}(\Omega)  := \Big\{f \in C^\infty (\Omega): \quad
D^\alpha f \in C (\Omega) \quad \mbox{for all}\quad \alpha \in \zz_+^n\Big\} \,
\]
and this proves the above claim.
Hence, we may write $N(\Omega) := N^s_{u,p,q}(\Omega)$.
This implies  that
\begin{eqnarray*}
\accentset{\diamond}{\cn}^s_{u,p,q} (\Omega) & = &
\overline{N(\Omega) }^{\cn^s_{u,p,q} (\Omega)}
= \overline{N^{s_0}_{u_0,p_0,q_0}(\Omega) \cap N^{s_1}_{u_1,p_1,q_1}(\Omega) }^{\cn^s_{u,p,q} (\Omega)}
\\
& \hookrightarrow &
\overline{\cn^{s_0}_{u_0,p_0,q_0}(\Omega) \cap \cn^{s_1}_{u_1,p_1,q_1}(\Omega)}^{\cn^s_{u,p,q} (\Omega)}  \, .
\end{eqnarray*}
On the other hand, by Lemma \ref{diamond2} and its proof, we know that
\[
\accentset{\diamond}{\cn}^{s_i}_{u_i,p_i,q_i} (\Omega) = {\cn}^{s_i}_{u_i,p_i,q_i} (\Omega)
\]
and  $S_Nf\to f$ as $N\to\fz$ in ${\cn}^{s_i}_{u_i,p_i,q_i}(\Omega)$, $i\in\{0,1\}$, if $q_0,\,q_1\in(0,\infty)$. Thus,
any $f\in \cn^{s_0}_{u_0,p_0,q_0}(\Omega) \cap \cn^{s_1}_{u_1,p_1,q_1}(\Omega)$
can be approximated by $S_Nf\in N(\Omega)$ in $\cn^s_{u,p,q} (\Omega)$, which further implies that
\begin{eqnarray*}
\accentset{\diamond}{\cn}^s_{u,p,q} (\Omega)=
\overline{\cn^{s_0}_{u_0,p_0,q_0}(\Omega) \cap \cn^{s_1}_{u_1,p_1,q_1}((0,1)^n)}^{\cn^s_{u,p,q} (\Omega)}  \, .
\end{eqnarray*}

Now we continue with an application of Theorem \ref{COMI}(ii),
which, together with  the existence of a bounded linear extension operator
\[\ce \in \cl ({\cn}^{s_i}_{u_i,p_i,q_i} (\Omega)),
{\cn}^{s_i}_{u_i,p_i,q_i} (\rn))\, , \qquad i\in\{0,1\},
\]
in Proposition \ref{exten} (see also Sawano \cite{sa10} for the case
of smooth domains),
and the method of the retraction and the coretraction, implies that
\[
\lf\laz \cn_{u_0,p_0,q_0}^{s_0}(\Omega), \cn_{u_1,p_1,q_1}^{s_1}(\Omega),\tz\r\raz=
\cn_{u,p,q}^{s}(\Omega),
\]
if $p_0\, u_1=p_1\, u_0$.
Proposition \ref{nil} makes clear that
\[
 \overline{\cn^{s_0}_{u_0,p_0,q_0}(\Omega) \cap \cn^{s_1}_{u_1,p_1,q_1}(\Omega)}^{\cn^s_{u,p,q} (\Omega)} =
\laz \cn^{s_0}_{u_0,p_0,q_0}(\Omega),  \cn^{s_1}_{u_1,p_1,q_1}(\Omega) \raz_\tz \, .
\]
This finishes the proof of Theorem  \ref{gagl5}.
\end{proof}


\subsection{Proofs of results in Subsection  \ref{inter1c}}
\label{proof6}



\subsubsection{Proofs of results in Subsection  \ref{inter1ca}}


For reader's convenience, we give proofs of  Propositions \ref{complexfinal},
\ref{complexinterpol} and  \ref{complexretract}.
Notice that, in our references \cite{km98} and \cite{kmm}, the additional
assumption that $X_0 \cap X_1$ is dense in $X_j$, $j\in\{0,1\}$, is used.
In the  proofs given below, we avoid this assumption.

\begin{proof}[Proof of Proposition \ref{complexfinal}]
Let $\{f_n\}_{n\in\nn}$ denote a Cauchy sequence in $\ca (X_0,X_1)$.
Since $X_0 + X_1$ is analytically convex, we conclude, from Proposition \ref{ac}, that,
for any $z\in S$,
\[\| \, f_n (z) - f_m (z)\|_{X_0 + X_1} \ls \max_{t\in\rr} \Big\{\| \, f_n (it) - f_m (it)\|_{X_0} \,,\ \
\| \, f_n (1+it) - f_m (1+it)\|_{X_1}\Big\}\, .
\]
Hence, for any $z \in S$, there exists a limit $f(z)= \lim_{n\to\fz} f_n (z) \in X_0 + X_1$,
due to the completeness of $X_0+X_1$.
Because this convergence is uniform on any open set $U \subset S_0$,  we conclude, by
using Proposition \ref{acc}, that $f$ is an analytic function.
On the other hand, since  functions $f_n(it)$ and $f_n (1+it)$ are continuous
and bounded on $t\in\rr$ and the boundedness is uniform in $n\in\nn$,
their limit functions $f(it)$ and $f (1+it)$ are continuous and bounded on $t\in\rr$ as well, i.\,e.,
$f \in \ca (X_0,X_1).$ This proves (i).

Part (ii) is a consequence of the following observation. Let $\cn_\Theta$ be the set of all functions
$f\in \ca (X_0,X_1)$ such that $f(\Theta) =0$. Consequently, $\cn_\Theta$ is a closed linear subspace of
$ \ca (X_0,X_1)$. Since $[X_0,X_1]_\Theta$ is isomorphic to $\ca (X_0,X_1)/\cn_\Theta$, it is a complete space. This finishes the proof of (ii) and hence
Proposition \ref{complexfinal}.
\end{proof}

\begin{proof}[Proof of Proposition \ref{complexinterpol}]
Temporarily we assume $\|\, T \, \|_{X_j \to Y_j}>0$, $j\in\{0,1\}$.
For $\tz\in(0,1)$, we define
\[
 g(z):= \lf(\frac{\|\, T \, \|_{X_0 \to Y_0}}{\|\, T \, \|_{X_1 \to Y_1}}\r)^{z-\Theta} \, Tf (z)\, ,
 \qquad  z \in S, \quad f \in \ca (X_0,X_1)\, .
\]
Hence, for all $t\in\rr$,
\[
\|\, g(it)\, \|_{Y_0}
\le  \lf(\frac{\|\, T \, \|_{X_0 \to Y_0}}{\|\, T \, \|_{X_1 \to Y_1}}\r)^{-\Theta}
\|\, T\, \|_{X_0 \to Y_0} \, \|\, f(it)\, \|_{X_0}
\]
and, similarly,
\[
\|\, g(1+it)\, \|_{Y_1}
\le  \lf(\frac{\|\, T \, \|_{X_0 \to Y_0}}{\|\, T \, \|_{X_1 \to Y_1}}\r)^{1-\Theta}
\|\, T\, \|_{X_1 \to Y_1} \, \|\, f(1+it)\, \|_{X_1}\,.
\]
This implies that $g$ belongs to $\ca (Y_0,Y_1)$. Let
$x:= f(\Theta) \in [X_0,X_1]_\Theta$. Then
$$y:= g(\Theta) = T f (\Theta) \in [Y_0,Y_1]_\Theta$$ and
\[
\| \, y\, \|_{[Y_0,Y_1]_\Theta} \le \| \, g \, \|_{\ca (Y_0,Y_1)}
\le \|\, T \, \|_{X_0 \to Y_0}^{1-\Theta} \,
\|\, T \, \|_{X_1 \to Y_1}^\Theta \| \, f \, \|_{\ca (X_0,X_1)}\, .
\]
Taking the infimum
in both sides of the above inequality
with respect to all $f \in \ca (X_0,X_1)$ such that $f(\Theta) =x$,
we obtain the desired conclusion.
If   $\|\, T \, \|_{X_0 \to Y_0} =0$ or $\|\, T \, \|_{X_1 \to Y_1} =0$,
then one has to replace this quantity by $\varepsilon
\in(0,\fz)$ in the definition of $g$ and
finally, consider $\varepsilon \downarrow 0$, the details being omitted.
This finishes the proof of Proposition \ref{complexfinal}.
\end{proof}

\begin{proof}[Proof of Proposition \ref{complexretract}]
One may use Triebel's arguments in the proof of  \cite[Theorem  1.2.4]{t78},
since the closed graph theorem remains true in the context of quasi-Banach spaces,
the details being omitted. This finishes the proof of Proposition \ref{complexretract}.
\end{proof}


\subsubsection{Proofs of results in Subsection \ref{inter1cd}}



\subsection*{Proof of Theorem \ref{morrey4}}


By Remark \ref{morrey24}(i), we know that all spaces under consideration are
lattice $r$-convex for some $r$.
Proposition \ref{shest2} and Theorem \ref{morrey1}(i) yield
\[
[\cm^{u_0}_{p_0}(\rn), \, \cm^{u_1}_{p_1}(\rn)]_{\tz}^i
\hookrightarrow \lf[\cm^{u_0}_{p_0}(\rn)\r]^{1-\Theta} \,
\lf[ \cm^{u_1}_{p_1}(\rn)\r]^\tz \hookrightarrow  \cm^{u}_{p}(\rn) \, .
\]
This shows (i) of Theorem \ref{morrey4}.

To prove (ii), first observe that
Theorem \ref{morrey1}(iii) implies that
this embedding is proper if $u_0 p_1 \neq u_1 p_0$.
In case $u_0 p_1 = u_1 p_0$, we derive, from  Corollaries \ref{comi}
and \ref{nil2},
that
\[
[\cm^{u_0}_{p_0}(\rn), \, \cm^{u_1}_{p_1}(\rn)]_{\tz}^i = \laz \cm^{u_0}_{p_0}(\rn), \, \cm^{u_1}_{p_1}(\rn) \raz_{\tz}\, .
\]
Now the desired conclusion  (ii) of Theorem \ref{morrey4}
follows from Corollary \ref{morrey42} and Theorem \ref{gp01},
the details being omitted.
This finishes the proof of (ii) and hence Theorem \ref{morrey4}.

\begin{remark}
Theorem \ref{morrey4} shows that \cite[Theorem 3(ii)]{LR} is not correct.
However,  let us mention that \cite{LR} has been our main source for the cases
$u_0 p_1 \neq u_1 p_0$.
\end{remark}


\subsection*{Proofs of Propositions \ref{morreyx} and  \ref{morreyn}}


In both conclusions of Proposition \ref{morreyx},
the first embedding
is a consequence of the definition of the complex method
and of the inner complex method.
The second embedding has been proved in \cite{yyz}. We give a sketch for
the convenience of the reader.
In the proofs of \cite[Propositions 2.6 and 2.7]{yyz}, the condition
$\tau_0 p_0 = \tau_1 p_1$ is not used to establish the embeddings
\[
\lf[a_{p_0,q_0}^{s_0,\tau_0}(\rn)\r]^{1-\tz}
\lf[a_{p_1,q_1}^{s_1,\tau_1}(\rn)\r]^\tz \hookrightarrow a_{p,q}^{s,\tau}(\rn) \, ,
\qquad a\in \{f,b\}\, ;
\]
see also Proposition \ref{morrey2}.
This has to be combined  with  \cite[Proposition 1.10]{yyz}:
\[
[a_{p_0,q_0}^{s_0,\tau_0}(\rn), a_{p_1,q_1}^{s_1,\tau_1}(\rn)]_\tz
\hookrightarrow
\lf[a_{p_0,q_0}^{s_0,\tau_0}(\rn)\r]^{1-\tz}
\lf[a_{p_1,q_1}^{s_1,\tau_1}(\rn)\r]^\tz \hookrightarrow a_{p,q}^{s,\tau}(\rn).
\]
By these and Proposition \ref{wave1},
together with an argument similar to that used in the proof of Theorem \ref{COMI},
we then obtain the second embeddings of  Proposition \ref{morreyx},
the details being omitted.
Hence, Proposition \ref{morreyx} is proved.

Concerning the proof of Proposition \ref{morreyn}, we argue in the same way
as the proof of Proposition \ref{morreyx}.
In case of the embedding
\[
\lf[n_{u_0,p_0,q_0}^{s_0}(\rn)\r]^{1-\tz}\lf[n_{u_1,p_1,q_1}^{s_1}(\rn)\r]^\tz \hookrightarrow
n_{u,p,q}^{s}(\rn),
\]
the restriction $p_0 u_1 = u_0 p_1$ was not used in \cite[Proposition~2.8]{yyz}; see Proposition \ref{morrey3}.
Furthermore, the embedding
\[
[n_{u_0,p_0,q_0}^{s_0}(\rn), n_{u_1,p_1,q_1}^{s_1}(\rn)]_\tz \hookrightarrow
\lf[n_{u_0,p_0,q_0}^{s_0}(\rn)\r]^{1-\tz}\lf[n_{u_1,p_1,q_1}^{s_1}(\rn)\r]^\tz
\]
has been proved in  \cite[Proposition~1.10]{yyz}.
The proof of Proposition  \ref{morreyn} is then finished.


\subsection*{Proof of Theorem \ref{gp03n}}


Because of
\[
\lf\laz\cm^{u_0}_{p_0}((0,1)^n), \cm^{u_1}_{p_1}((0,1)^n)\, , \tz\r\raz = \cm^{u}_{p}((0,1)^n) \,
\]
(see \eqref{ws-104})
and
\[
\lf[\cm^{u_0}_{p_0}((0,1)^n)\r]^{1-\tz}\, \lf[ \cm^{u_1}_{p_1}((0,1)^n)\r]^\tz  = \cm^{u}_{p}((0,1)^n)
\]
(see \cite[(2.3)]{lyy}),
it follows, from Corollary \ref{nil2}, that
\[
[\cm^{u_0}_{p_0}((0,1)^n), \cm^{u_1}_{p_1}((0,1)^n)]_\tz^i = \laz\cm^{u_0}_{p_0}((0,1)^n), \cm^{u_1}_{p_1}((0,1)^n)\raz_\tz\, .
\]
Now we apply Theorem \ref{gp03} to obtain the desired conclusion, which completes the proof
of Theorem \ref{gp03n}.


\subsection{Proofs of results in Subsection  \ref{inter1d}}



\subsection*{Proof of Corollary  \ref{approx10}}


The conditions in Corollary  \ref{approx10} are guaranteeing that
\[
A_{p_i,q_i}^{s_i,\tau_i}(\rn) = B_{\fz,\fz}^{s_i+n(\tau_i-1/p_i)}(\rn) \, ,  \qquad A \in \{B,F\}\,, \quad i\in\{0,1\}\, ;
\]
see Proposition \ref{basic2}(iii).
Now it suffices to recall
\[
(B_{\fz,\fz}^{s_0+n(\tau_0-1/p_0)}(\rn), \, B_{\fz,\fz}^{s_1+n(\tau_1-1/p_1)}(\rn))_{\tz,q} =
B_{\fz,q}^{s+n(\tau-1/p)}(\rn)\, ;
\]
see \cite[Theorem 2.4.2]{t83}. This
finishes the proof of Corollary \ref{approx10}.


\subsection*{Proof of Lemma \ref{morreyreal}}


Part (i) is proved in Lemma \ref{help} and Remark \ref{morrey24}(i).
For (ii), we refer the reader  to Lemari{\'e}-Rieussiet \cite{LR}.
Concerning (iii), we use
$$ L_p (\rn) = (L_{p_0}(\rn), L_{p_1}(\rn))_{\tz,\infty}$$ because of $p_0 = p_1$.
Hence, taking Proposition
\ref{realbasic}(ii) into account, we may employ Lemma \ref{help}
with this functor $(\, \cdot\, , \, \cdot\, )_{\Theta,\infty}$.
Choosing $T=I$, we obtain the desired conclusion in (iii),
which completes the proof of Lemma \ref{morreyreal}.


\subsection*{Proof of Theorem \ref{morreyreal2}}


In the case (a), we see that $p_0=p_1=p$ and $u_0=u_1=u$,
and hence
$$(\cm^{u_0}_{p_0}(\rn), \cm^{u_1}_{p_1}(\rn))_{\tz,q}=(\cm^{u_0}_{p_0}(\rn), \cm^{u_0}_{p_0}(\rn))_{\tz,q}= \cm^{u_0}_{p_0}(\rn)=\cm^u_p(\rn).$$
In the  case (b), we have
$$(\cm^{p_0}_{p_0}(\rn), \cm^{p_1}_{p_1}(\rn))_{\tz,q}=(L_{p_0}(\rn), L_{p_1}(\rn))_{\tz,p}=L_{p}(\rn)=\cm^p_p(\rn).$$

Next we argue by contradiction to show, if neither (a) nor (b) is true, then
$$(\cm^{u_0}_{p_0}(\rn), \cm^{u_1}_{p_1}(\rn))_{\tz,q}\neq \cm^{u}_{p}(\rn).$$
Let us assume
$
(\cm^{u_0}_{p_0}(\rn), \cm^{u_1}_{p_1}(\rn))_{\tz,q} =  \cm^{u}_{p}(\rn)
$
for some $\Theta \in (0,1)$ and some $q \in (0,\infty]$.
We now consider two cases.
\\
{\em Step 1.} Let $ q\in(0,\infty)$.
Then \cite[Theorem 3.4.2]{BL} yields that
$\cm^{u_0}_{p_0}(\rn) \cap  \cm^{u_1}_{p_1}(\rn)$ must be dense in
$\cm^{u}_{p}(\rn)$.
Now, applying Lemma \ref{morrey41},
we see that the above assumption yields a contradiction.
\\
{\em Step 2.} Let $q= \infty$.
Lemari{\'e}-Rieussiet \cite{LR} has proved that
\[
\cm^{u}_{p}(\rn)  \hookrightarrow
 \lf(\cm^{u_0}_{p_0}(\rn), \cm^{u_1}_{p_1}(\rn)\r)_{\tz,\infty}
 \qquad \Longleftrightarrow \qquad p_0 \, u_1 = p_1 \, u_0\, .
\]
Also in the quoted article \cite{LR}, one can find that
\[
 \lf(\cm^{u_0}_{p_0}(\rn), \cm^{u_1}_{p_1}(\rn)\r)_{\tz,\infty} \hookrightarrow
\cm^{u}_{p}(\rn)
\]
implies $p_0 =p_1$.
Hence, our assumption yields $p_0 =p_1$ and $u_0 =u_1$, i.\,e.,
$$\cm^{u_0}_{p_0}(\rn) = L_{p_0}(\rn)\quad{\rm and}\quad \cm^{u_1}_{p_1}(\rn)=L_{p_0}(\rn).$$
Of course,
$\lf(L_{p_0}(\rn), L_{p_0}(\rn)\r)_{\tz,\infty} = L_{p_0}(\rn)$, but this case is
already excluded in Step 2. This finishes the proof of Theorem \ref{morreyreal2}.


\subsection*{Proof of Lemma \ref{lem-real}}


The key tool is \cite[Theorem~1.4.2]{t78}, which holds true also for
quasi-Banach cases; see  \cite[Remark 1.4.2/3]{t78}. Let $\eta:=\tz p/p_1$.
Then $p=(1-\eta)p_0+\eta p_1$ follows. Applying
\cite[Theorem~1.4.2]{t78} with $q=1$,
we obtain, for any sequence  $a:=\{a_j\}_{j\in\zz_+}$, $a_j \in X_0 + X_1$,
\begin{eqnarray*}
\|a\|_{(\ell^{s_0}_{p_0}(X_0), \ell^{s_1}_{p_1}(X_1))_{\tz,p}}^p &\asymp&
\int_0^\fz t^{-\eta} \inf_{a=a^0+a^1}\sum_{j\in\zz_+}\lf[
2^{js_0p_0}\|a^0_j\|^{p_0}_{X_0}+t2^{js_1p_1}\|a^1_j\|^{p_1}_{X_1} \r]\,\frac{dt}t
\\
& \asymp &   \sum_{j\in\zz_+} \int_0^\fz t^{-\eta} \inf_{a_j=a^0_j+a^1_j}\lf[
2^{js_0p_0}\|a^0_j\|^{p_0}_{X_0}+t2^{js_1p_1}\|a^1_j\|^{p_1}_{X_1} \r]\,\frac{dt}t.
\end{eqnarray*}
By a  change of  variable $y:= (k_j)^{-1} t$ with
$k_j:= 2^{j(s_0p_0-s_1p_1)}$, we conclude that
\begin{eqnarray*}
\|a\|_{(\ell^{s_0}_{p_0}(X_0), \ell^{s_1}_{p_1}(X_1))_{\tz,p}}^p
& \asymp &  \sum_{j\in\zz_+} \int_0^\fz y^{-\eta} \inf_{a_j=a^0_j+a^1_j}\lf[
k_j^{-\eta}2^{js_0p_0}\|a^0_j\|^{p_0}_{X_0}+ y k_j^{1-\eta}
2^{js_1p_1}\|a^1_j\|^{p_1}_{X_1} \r]\,\frac{dy}y\, .
\end{eqnarray*}
Observe that $s$ and $p$ satisfy the following identities
$$\eta(s_1p_1-s_0p_0)+s_0p_0=\tz p s_1+ (1-\tz p/p_1)s_0p_0=[\tz s_1+(1-\tz)s_0]p=sp $$
and
\[
(1-\eta)(s_0p_0-s_1p_1)+s_1p_1=sp \, .
\]
Hence
\begin{eqnarray*}
\|a\|_{(\ell^{s_0}_{p_0}(X_0), \ell^{s_1}_{p_1}(X_1))_{\tz,p}}^p
&\asymp &  \sum_{j\in\zz_+} 2^{jsp} \int_0^\fz y^{-\eta} \inf_{a_j=a^0_j+a^1_j}\lf[
\|a^0_j\|^{p_0}_{X_0} + y \|a^1_j\|^{p_1}_{X_1} \r]\,\frac{dy}y
\\
&\asymp & \|a\|^p_{\ell^s_p((X_0,X_1)_{\tz,p})},
\end{eqnarray*}
which proves the desired conclusion and hence
completes the proof of Lemma \ref{lem-real}.


\subsection{Proof of Lemma \ref{help}}


Let $B$ denote any ball in $\rn$.
Since $T \in \cl(X_0,\cm^{u_0}_{p_0}(\rn)) \cap \cl(X_1,\cm^{u_1}_{p_1}(\rn))$,
we know that
\[
\chi_B \,   T \in \cl(X_0, L_{p_0}(B)) \cap \cl(X_1,L_{p_1}(B))\,
\]
and
\[
\| \chi_B \, T \,\|_{X_i \to L_{p_i}(B)}
\le |B|^{\frac{1}{p_i}-\frac{1}{u_i}}\, M_i, \quad i\in\{0,1\}.
\]
Moreover, by \eqref{morrey20}, we see that
\[
\chi_B \,  T \in \cl\lf[F(X_0,X_1), F(L_{p_0}(B), L_{p_1}(B))\r]\: \hookrightarrow \: \cl\lf[F(X_0,X_1), L_{p}(B)\r]\,
\]
and
\[
\|\, \chi_B \,   T \|_{F(X_0,X_1) \to L_{p}(B)} \le C_F \, \lf[|B|^{\frac{1}{p_0}-\frac{1}{u_0}}\, M_0 \r]^{1-\tz}\,
\lf[|B|^{\frac{1}{p_1}-\frac{1}{u_1}}\, M_1\r]^\tz \,
\]
with a positive constant $C_F$ independent of $B$.
Let $f \in F(X_0,X_1)$.
Then we conclude  $Tf \in L_p (B)$
and
\[
\lf[ \int_B |Tf (x)|^p dx\r]^{1/p} \le
C_F \, M_0^{1-\tz}\, M_1^\tz \, |B|^{\frac{1}{p}-\frac{1}{u}}\,  \| \, f \, \|_{F(X_0,X_1)}
\, .
\]
This finishes the proof of Lemma \ref{help}.


\subsection{Proof of Theorem \ref{COMI-u}}
\label{proof2}


Similar to the arguments used in the proof of
Theorem \ref{COMI}, we first consider the sequence spaces related to
$\atu$.

\begin{definition}\label{d2u}
Let $s\in\rr$, $\tau\in[0,\infty)$ and $p,\ q \in(0,\fz]$.

The \emph{sequence space} $\sbtu$ is defined as the collection of all complex-valued sequences
$t:=\{t_Q\}_{Q\in\cq^*}$ such that
$$\|t\|_{\sbtu}:=
\sup_{P\in\mathcal{Q}^*}\frac1{|P|^{\tau}}\left[\sum_{j=j_P}^\fz
2^{j(s+\frac n2)q}\left\{\int_P\lf[\sum_{\ell(Q)=2^{-j}}
|t_Q|\chi_Q(x)\r]^p dx\right\}^{\frac qp}\right]^{\frac 1q}<\fz$$
with the usual modifications made in case $p=\fz$ and/or $q=\fz$.

The \emph{sequence space} $\sftu$ with $p\in(0,\fz)$
is defined as the collection of all  complex-valued sequences
$t:=\{t_Q\}_{Q\in\cq^*}$ such that
$$\|t\|_{\sftu}:=
\sup_{P\in\mathcal{Q}^*}\frac1{|P|^{\tau}}\left\{\int_P\left[
\sum_{j=j_P}^\fz  2^{j(s+\frac n2)q}
\sum_{\ell(Q)=2^{-j}}
|t_Q|^q\chi_Q(x)\right]^{\frac pq}\,dx\right\}^{\frac 1p}<\fz$$
with the usual modification made when $q=\fz$.
\end{definition}

As before, we use $\satu$ to denote either $\sbtu$ or $\sftu$. Notice that the
space $\satu$ coincides with the sequence space
$\cl^{n(\tau-1/p)}a^{s+n/2}_{p,q}(\rn)$ related to
$\cl^{n(\tau-1/p)}A^{s}_{p,q}(\rn)$; see \cite[Definition 1.30]{t12}.

It is easy to see that the space $\satu$ has the following equivalent characterization,
the details being omitted.

\begin{proposition}\label{pu}
Let $s\in\rr$, $\tau\in[0,\infty)$ and $p,\ q \in(0,\fz]$.
A sequence $t:=\{t_Q\}_{Q\in\cq^*}\in \satu$ if and only if
$$\|t\|_{\widetilde{\satu}}:=\sup_{\ell\in\zz^n}
\lf\|\lf\{t_Q \chi_{\ell}(Q)\r\}_{Q\in\cq^*}\r\|_{\sat}<\fz,$$
where $\chi_{\ell}:=\chi_{\{R\in\cq^*:\ R\subset Q_{0,\ell}\}}$. Moreover, $\|\cdot\|_{\widetilde{\satu}}$ is equivalent to $\|\cdot\|_{\satu}.$
\end{proposition}

Applying Proposition \ref{pu}, one can show that the Calder\'on product property of $\sat$
in Proposition \ref{morrey2} is also true for $\satu$.

\begin{proposition}\label{bfcu}
Let all parameters be as in Proposition \ref{morrey2}. Then
\begin{equation*}
\lf[a_{p_0,q_0,{\rm unif}}^{s_0,\tau_0}(\rn)\r]^{1-\tz}
\lf[a_{p_1,q_1,{\rm unif}}^{s_1,\tau_1}(\rn)\r]^\tz=a_{p,q,{\rm unif}}^{s,\tau}(\rn)
\end{equation*}
\end{proposition}

\begin{proof} The embedding $$[a_{p_0,q_0,{\rm unif}}^{s_0,\tau_0}(\rn)]^{1-\tz}
[a_{p_1,q_1,{\rm unif}}^{s_1,\tau_1}(\rn)]^\tz \hookrightarrow a_{p,q,{\rm unif}}^{s,\tau}(\rn)$$
follows directly from the H\"older inequality.

Now we prove the inverse embedding. Let $t:=\{t_Q\}_{Q\in\cq^*}\in \satu$.
By Proposition \ref{pu}, we know that, for all $\ell\in\zz^n$,
$\{t_Q \chi_{\ell}(Q)\}_{Q\in\cq^*}\in\sat$ and
$$\lf\|\lf\{t_Q \chi_{\ell}(Q)\r\}_{Q\in\cq^*}\r\|_{\sat}\ls \|t\|_{\satu}.$$
Then, by Proposition \ref{morrey2}, we conclude that there exist
$t^{0,\ell}:=\{t^{0,\ell}_Q\}_{Q\in\cq^*}\in a^{s_0,\tau_0}_{p_0,q_0}(\rn)$ and
$t^{1,\ell}:=\{t^{1,\ell}_Q\}_{Q\in\cq^*}\in a^{s_1,\tau_1}_{p_1,q_1}(\rn)$
such that $|t_Q\chi_{\ell}(Q)|\le |t^{0,\ell}_Q|^{1-\tz} |t^{1,\ell}_Q|^{\tz}$ for all $Q\in\cq^*$ and
$$\|t^{0,\ell}\|_{a^{s_0,\tau_0}_{p_0,q_0}(\rn)}^{1-\tz}
\|t^{1,\ell}\|_{a^{s_1,\tau_1}_{p_1,q_1}(\rn)}^\tz\ls \lf\|\lf\{t_Q \chi_{\ell}(Q)\r\}_{Q\in\cq^*}\r\|_{\sat}\ls \|t\|_{\satu}.$$

Define $t^{0}$ and $t^1$ by setting, for all $Q\in\cq^*$,
 $t^{0}_Q:=\sum_{\ell\in\zz^n}|t^{0,\ell}_Q|\chi_\ell(Q)$
and  $t^{1}_Q:=\sum_{\ell\in\zz^n}|t^{1,\ell}_Q|\chi_\ell(Q)$.
Then, by the H\"older inequality, we see that
$$|t_Q|=\sum_{\ell\in\zz^n} |t_Q|\chi_\ell(Q)
\le \sum_{\ell\in\zz^n} |t^{0,\ell}_Q|^{1-\tz} |t^{1,\ell}_Q|^{\tz}\chi_\ell(Q)\le |t^0_Q|^{1-\tz}|t^1_Q|^\tz.$$
Moreover, we have
\begin{eqnarray*}
\|t^i\|_{a^{s_i,\tau_i}_{p_i,q_i,{\rm unif}}(\rn)}&&=\sup_{m\in\zz^n}\lf\|\lf\{t^i_Q\chi_m(Q)
\r\}_{Q\in\cq^*}\r\|_{a^{s_i,\tau_i}_{p_i,q_i}(\rn)}\\
&&=\sup_{m\in\zz^n}\|t^{i,m}\|_{a^{s_i,\tau_i}_{p_i,q_i}(\rn)}\ls \|t\|_{\satu},
\end{eqnarray*}
which implies that $a_{p,q,{\rm unif}}^{s,\tau}(\rn)\hookrightarrow[a_{p_0,q_0,{\rm unif}}^{s_0,\tau_0}(\rn)]^{1-\tz}
[a_{p_1,q_1,{\rm unif}}^{s_1,\tau_1}(\rn)]^\tz$, and hence completes the proof
of Proposition \ref{bfcu}.
\end{proof}

Repeating the arguments used in the proofs of Theorem \ref{comi} and Corollary \ref{c-cis},
with Proposition \ref{morrey2} replaced by Proposition \ref{bfcu}, we obtain
the following interpolation formulas,
the details being omitted.

\begin{theorem}\label{comi-u}
Let $\tz\in(0,1)$, $s_i\in\rr$, $\tau_i\in[0,\fz)$,
$p_i, q_i\in(0,\fz]$, $i\in\{1,2\},$
such that $$s=(1-\tz)s_0+\tz s_1,\ \  \tau=(1-\tz)\tau_0+\tz\tau_1,\ \
\frac1p=\frac{1-\tz}{p_0}+\frac\tz{p_1},\ \ \frac1q=\frac{1-\tz}{q_0}+\frac\tz{q_1}\quad{\rm and}\quad
\frac{\tau_0}{p_1}=\frac{\tau_1}{p_0}.$$
Then
\begin{equation*}
\lf\laz a_{p_0,q_0,{\rm unif}}^{s_0,\tau_0}(\rn), a_{p_1,q_1,{\rm unif}}^{s_1,\tau_1}(\rn)\r\raz_\tz
=(a_{p,q,{\rm unif}}^{s,\tau}(\rn))^\#
\end{equation*}
and
$$
\lf\laz a_{p_0,q_0,{\rm unif}}^{s_0,\tau_0}(\rn), a_{p_1,q_1,{\rm unif}}^{s_1,\tau_1}(\rn),\tz\r\raz=a_{p,q,{\rm unif}}^{s,\tau}(\rn).$$

Assume further that $p_i, q_i\in[1,\fz]$. Then
$$\laz a_{p_0,q_0,{\rm unif}}^{s_0,\tau_0}(\rn), a_{p_1,q_1,{\rm unif}}^{s_1,\tau_1}(\rn)\raz_\tz
=(a_{p,q,{\rm unif}}^{s,\tau}(\rn))^\#=[a_{p_0,q_0,{\rm unif}}^{s_0,\tau_0}(\rn), a_{p_1,q_1,{\rm unif}}^{s_1,\tau_1}(\rn)]_\tz.$$
\end{theorem}

Theorem \ref{COMI-u} is then an immediate consequence of Theorem \ref{comi-u}
and the wavelet characterization of
the spaces $\atu=\cl^{n(\tau-1/p)}A^s_{p,q}(\rn)$ in \cite[Theorem 1.32]{t12}, the
details being omitted.


\section{Appendix -- Function spaces}


For the convenience of the reader, we recall definitions and collect some properties
of the function spaces considered in this article.


\subsection{Besov-type and Triebel-Lizorkin-type spaces}


Besov-type and Triebel-Lizorkin-type spaces are generalizations
of Besov and Triebel-Lizorkin spaces. As
 Besov and Triebel-Lizorkin spaces, also these more general scales of spaces
 can be introduced in very different ways.
Here we use  the Fourier-analytical approach.
Let $\psi \in C_c^\infty(\rn)$ be a radial function such that $\psi (x) \le 1$ for all $x$,
\begin{equation}\label{eq-05}
\psi (x):= 1 \qquad \mbox{if} \quad   |x| \le  1
\qquad \mbox{and}\qquad
\psi (x):= 0 \qquad \mbox{if} \quad   |x|\ge \frac 32\,.
\end{equation}
Then, with $\varphi_0 := \psi$,
\begin{equation}\label{eq-06}
\varphi (x):= \varphi_0 (x/2)- \varphi_0 (x) \qquad \mbox{and}\qquad
\varphi_j (x):= \varphi (2^{-j+1}x)\, , \quad x\in\rn, \ \ j \in \nn \, ,
\end{equation}
we have
\[
\sum_{j=0}^\infty \varphi_j (x) = 1 \qquad \mbox{for all} \quad x \in \rn\, .
\]
In what follows, for $f\in\cs'(\rn)$, we use $\cf f$ to denote its \emph{Fourier transform},
and $\cf^{-1}f$ its \emph{inverse Fourier transform}.

\begin{definition}\label{d1}
Let $s\in\rr$, $\tau\in[0,\infty)$, and $q \in(0,\fz]$.

{\rm(i)} Let $p\in(0,\fz]$.
The \emph{Besov-type space} $\bt$ is defined as the
collection of all $f\in \mathcal{S}'(\rn)$ such that
$$\|f\|_{\bt}:=
\sup_{P\in\mathcal{Q}}\frac1{|P|^{\tau}}\left\{\sum_{j=\max\{j_P,0\}}^\fz
2^{js q}\left[\int_P
|\cfi (\vz_j \, \cf f)(x)|^p\,dx\right]^{q/p}\right\}^{1/q}<\fz$$
with the usual modifications made in case $p=\fz$ and/or $q=\fz$.

{\rm(ii)} Let $p\in(0,\fz)$. The
\emph{Triebel-Lizorkin-type space} $\ft$
is defined as the collection of all $f\in \mathcal{S}'(\rn)$ such that
$$\|f\|_{\ft}:=
\sup_{P\in\mathcal{Q}}\frac1{|P|^{\tau}}\left\{\int_P\left[
\sum_{j=\max\{j_P,0\}}^\fz  2^{js q}|\cfi ( \vz_j\, \cf f)(x)|^q\right]^{p/q}\,dx\right\}^{1/p}<\fz$$
with the usual modification made when $q=\fz$.
\end{definition}

The above definition represents a natural extension
 of the Fourier-analytical approach to  Besov
 and Triebel-Lizorkin spaces; see, e.\,g.,
 \cite{fj90}, \cite{Pe76} and \cite{t83,t92}.
The homogenous version of these spaces were introduced in
 \cite{yy1,yy2} in order to clarify the relation between Besov and Triebel-Lizorkin spaces and $Q$ spaces (see \cite{ejpx,dx,x01,x06}).

Several classical spaces can be identified within these scales.

\begin{proposition}\label{basic1}
Let $s\in\rr$, $q\in(0,\fz]$ and $\tau\in[0,\fz)$.

{\rm(i)} It holds true that
\[
F_{p,q}^{s,0}(\rn) = F^s_{p,q}(\rn)\ (\mbox{for}\ p\in(0,\fz)) \qquad \mbox{and}\qquad
B_{p,q}^{s,0}(\rn) = B^s_{p,q}(\rn)\ (\mbox{for}\ p\in(0,\fz]) \, .
\]

{\rm(ii)} For all $p\in(0,\fz)$, $F^{s,1/p}_{p,q}(\rn)=F^s_{\fz,q}(\rn)$
(\cite[Corollary 6.9]{fj90}).

{\rm(iii)} Let $p\in(0,\fz)$.
If either $q\in(0,\fz)$ and $\tau\in(1/p,\fz)$ or
$q=\fz$ and $\tau\in[1/p,\fz)$, then
\[
A_{p,q}^{s,\tau}(\rn) = B_{\fz,\fz}^{s+n(\tau-1/p)}(\rn) \, ,  \qquad A \in \{B,F\}\, .
\]

{\rm(iv)} For all $1<p\le u<\fz$ and $m\in\nn$,
$$\cm^u_p(\rn)=F^{0,1/p-1/u}_{p,2}(\rn)\quad {\rm and}\quad W^m\cm^u_p(\rn)=F^{m,1/p-1/u}_{p,2}(\rn),$$
where $W^m\cm^u_p(\rn)$ denotes the Morrey-Sobolev space of order $m$.
\end{proposition}

\begin{remark}
Proposition \ref{basic1}(i) is obvious,
Proposition \ref{basic1}(ii) is a well-known result of Frazier and Jawerth
\cite[Corollary 6.9]{fj90}. The identity in Proposition \ref{basic1}(iii) was recently proved in
\cite{yy4}.
Finally, the Littlewood-Paley assertion in  Proposition \ref{basic1}(iv)
can be found in Mazzucato
\cite{ma01} and Sawano \cite{sawch}.
\end{remark}


\subsection{Besov-Morrey  and Triebel-Lizorkin-Morrey spaces}


The following spaces were first introduced and investigated in \cite{KY,ma03,tx,st}.

\begin{definition} \label{besov}
Let $s\in\rr$ and  $q\in(0,\fz]$. Let $\{\varphi_j\}_{j\in\zz_+}$ be the smooth dyadic decomposition of unity as defined in
as in \eqref{eq-05} and \eqref{eq-06}.

(i) Let $0<p\le u\le\fz$. The \emph{Besov-Morrey space} $\mathcal{N}_{u,p,q}^s(\rn)$
is defined as the space of all $f\in\cs'(\rn)$ such that
$$\|f\|_{\mathcal{N}_{u,p,q}^s(\rn)}
:=\lf\{\sum_{j\in\zz_+}2^{jsq}\|\cfi (\vz_j \, \cf f)\|_{\cm^u_p(\rn)}^{q}\r\}^{1/q}<\fz.$$

(ii) Let $0<p\le u < \fz$. The \emph{Triebel-Lizorkin-Morrey space}
$\ce_{u,p,q}^s(\rn)$
is defined as the space of all $f\in\cs'(\rn)$ such that
$$\|f\|_{\ce_{u,p,q}^s(\rn)}
:=\lf\|\lf[\sum_{j\in\zz_+}2^{jsq}|\cfi (\vz_j\, \cf f)|^q
\r]^{1/q}\r\|_{\cm^u_p(\rn)}<\fz.$$
\end{definition}

The following relations can be found in \cite{ysy} and \cite{syy}.

\begin{proposition}
 \label{basic2}
Let $s\in\rr$, $q\in(0,\fz]$ and  $0 < p \le u \le \fz$.
Then
\[
\ce^s_{u,p,q}(\rn)=F^{s,1/p-1/u}_{p,q}(\rn)\,  \qquad \mbox{if} \qquad u < \infty
\]
and
\[
\cn^s_{u,p,\fz}(\rn)=B^{s,1/p-1/u}_{p,\fz}(\rn).
\]
In addition, it holds true that
\[
\cn^s_{u,p,q}(\rn) \subsetneqq B^{s,1/p-1/u}_{p,q}(\rn) \qquad \mbox{if} \quad q<\fz
\quad{and}\quad u\neq p
\, .
\]
\end{proposition}


\subsection{The local spaces of Triebel}


We do not recall the original definition of the spaces $\cl^{r}A^s_{p,q}(\rn)$  given in Triebel \cite[1.3.1]{t12}.
We simply state the following identity; see \cite{ysy2}.

\begin{proposition}
Let $s\in\rr$, $\tau\in[0,\fz)$ and $q\in(0,\fz]$.
Let $p\in(0,\fz]$ if $A=B$ and $p\in(0,\fz)$ if $A=F$.
Then
\[
\atu=\cl^{n(\tau-1/p)}A^s_{p,q}(\rn)
 \]
in the sense of  equivalent quasi-norms.
\end{proposition}


\subsection{Associated sequence spaces}
\label{sequence}


As in case of Besov-Triebel-Lizorkin spaces, discretizations play an important role.
Either by the $\varphi$-transform or by the wavelet transform, one can relate these
function spaces  to sequence spaces.
We recall their definitions.

\begin{definition}\label{d2}
Let $s\in\rr$, $\tau\in[0,\infty)$ and $q \in(0,\fz]$.

The \emph{sequence space} $\sbt$ with $p\in(0,\fz]$ is defined as the collection of all sequences
$t:=\{t_Q\}_{Q\in\cq^*}\subset \cc$ such that
$$\|t\|_{\sbt}:=
\sup_{P\in\mathcal{Q}}\frac1{|P|^{\tau}}\left[\sum_{j=\max\{j_P,0\}}^\fz\hspace{-0.2cm}
2^{j(s+\frac n2)q}\left\{\int_P\lf[\sum_{\ell(Q)=2^{-j}}
|t_Q|\chi_Q(x)\r]^p dx\right\}^{\frac qp}\right]^{\frac 1q}<\fz$$
with the usual modifications made in case $p=\fz$ and/or $q=\fz$.

The \emph{sequence space} $\sft$ with $p\in(0,\fz)$
is defined as the collection of all sequences
$t:=\{t_Q\}_{Q\in\cq^*}\subset \cc$ such that
$$\|t\|_{\sft}:=
\sup_{P\in\mathcal{Q}}\frac1{|P|^{\tau}}\left\{\int_P\left[
\sum_{j=\max\{j_P,0\}}^\fz  2^{j(s+\frac n2)q}
\sum_{\ell(Q)=2^{-j}}
|t_Q|^q\chi_Q(x)\right]^{\frac pq}\,dx\right\}^{\frac 1p}<\fz$$
with the usual modification made when $q=\fz$.
\end{definition}

To explain the connection between sequence spaces and function spaces, we employ wavelet decompositions.
Wavelet bases in function spaces are a
well-developed concept. We refer the reader to the monographs of Meyer \cite{me},
Wojtasczyk \cite{woj} and Triebel \cite{t06,t08} for the general
$n$-dimensional case (for the one-dimensional case we refer the reader to the
books of Hernandez and Weiss \cite{hw}  and  Kahane and Lemari{\'e}-Rieuseut
\cite{kl}). Let $\wz{\phi}$ be
an {\it orthonormal scaling function} on $\rr$ with compact support
and of sufficiently high regularity. Let $\wz{\psi}$ be one
{\it corresponding orthonormal wavelet}\index{wavelet}.
Then the {\it tensor product ansatz} yields a scaling function $\phi$ and
associated wavelets\index{wavelet}
$\psi_1,\,\ldots,\,\psi_{2^n-1}$, all defined now on $\rn$; see, e.\,g.,
\cite[Proposition 5.2]{woj}.
We suppose
\begin{equation}\label{4.19}
\phi\in C^{N_1}(\rn)\hs\mathrm{and}\hs
\supp\phi\subset[-N_2,\,N_2]^n
\end{equation}
for some natural numbers $N_1$ and $N_2$. This implies that
\begin{equation*}
\psi_i\in C^{N_1}(\rn)\hs\mathrm{and}\hs
\supp\psi_i\subset[-N_3,\,N_3]^n,\hs i\in\{1,\ldots,2^n-1\}
\end{equation*}
for some $N_3 \in \nn$.
For $k\in\zz^n$, $j\in\zz_+$ and $i\in\{1,\ldots,2^n-1\}$, we shall use the
standard abbreviations in this context:
$$\phi_{j,k}(x):=2^{jn/2}\phi(2^jx-k)\hs\mathrm{and}\hs
\psi_{i,j,k}(x):=2^{jn/2}\psi_i(2^jx-k),\hs x\in\rn.$$
Furthermore, it is well
known that
\[
\int_\rn \psi_{i,j,k}(x)\, x^\gz\,dx = 0 \qquad  \mbox{if}
\qquad  |\gz|\le N_1\,
\]
(see \cite[Proposition 3.1]{woj}) and
\begin{equation*}
\{\phi_{0,k}: \ k\in\zz^n\}\: \cup \: \{\psi_{i,j,k}:\ k\in\zz^n,\
j\in\zz_+,\ i\in\{1,\ldots,2^n-1\}\}
\end{equation*}
yields an {\it orthonormal basis} of $L_2(\rn)$; see \cite[Section
3.9]{me} or \cite[Section 3.1]{t06}. Namely, each function $f \in L_2 (\rn)$
admits a representation
\begin{equation}\label{4.22}
f=\dsum_{k\in\zz^n} \lambda_k\, \phi_{0,k}+\dsum_{i=1}^{2^n-1}\dsum_{j=0}^\infty
\dsum_{k\in\zz^n} \lambda_{i,j,k}\, \psi_{i,j,k}\, ,
\end{equation}
where $\lambda_k:=\laz f,\,\phi_{0,k}\raz$ and
$\lambda_{i,j,k}:=\laz f,\,\psi_{i,j,k}\raz$.
Concerning the mapping
\[
\Phi : \quad f \mapsto \{\lambda_k\}_k \: \cup \:
\{\lambda_{i,j,k}\}_{i,j,k}
\]
the following is known (see \cite{lsuyy}).

\begin{proposition}\label{wave1}
Let $s\in\rr$, $\tau\in[0,\infty)$ and $q \in(0,\fz]$.

{\rm (i)} Let $p\in(0,\fz]$ and $N_1$ be sufficiently large (in dependence on $s,p,\tau$).
 Then $f \in \bt$ if and only if $f$ can be represented as in \eqref{4.22} (convergence in $\cs'(\rn)$) and
\begin{eqnarray*}
\| \, \Phi (f)\, \|^*_{\sbt}&&:=
\sup_{P \in {\mathcal Q}} \frac{1}{|P|^\tau} \lf\{ \sum_{Q_{0,m} \subset P} |\laz f,\,\phi_{0,m}\raz|^p \r\}^{1/p}+ \sum_{i=1}^{2^n-1} \|\, \{ \laz f,\,\psi_{i,j,k}\raz\}_{j,k}\, \|_{\sbt}
\end{eqnarray*}
is finite.
In addition, the quasi-norms $\| \, \Phi (f)\, \|^*_{\sbt}$ and $\| \, f\, \|_{\bt}$
are equivalent.

{\rm (ii)} Let $p\in(0,\fz)$ and $N_1$ be sufficiently large (in dependence on $s,p,q$ and $\tau$).
 Then $f \in \ft$ if and only if $f$ can be represented as in \eqref{4.22} (convergence in $\cs'(\rn)$) and
\begin{equation*}
\| \, \Phi (f)\, \|^*_{\sft}:=
\sup_{P \in {\mathcal Q}} \frac{1}{|P|^\tau} \lf\{ \sum_{Q_{0,m} \subset P} |\laz f,\,\phi_{0,m}\raz|^p\r\}^{1/p}
+ \sum_{i=1}^{2^n-1} \|\, \{ \laz f,\,\psi_{i,j,k}\raz\}_{j,k}\, \|_{\sft}
\end{equation*}
is finite.
In addition, the quasi-norms $\| \, \Phi (f)\, \|^*_{\sft}$ and $\| \, f\, \|_{\ft}$
are equivalent.
 \end{proposition}

\begin{remark}\label{rbinf}
(i) Such isomorphisms in the framework of Morrey spaces and smoothness spaces related to Morrey spaces
are also investigated  in \cite{ysy} ($s>0$); for  $\tau < 1/p$, one may also consult Sawano \cite{sa0} and Rosenthal \cite{ro13}.

(ii) There is a particular case of Proposition \ref{wave1} which plays a role in our investigations.
Let $\tau = 0$ and $p=q= \infty$ (see Proposition \ref{basic1}(i)). Then, for $N_1>|s|$, we have
$f \in B^s_{\infty,\infty} (\rn)$ if and only if $f$ can be represented as in (4.22) (convergence in $\cs'(\rn)$) and
\begin{eqnarray*}
\| \, \Phi (f)\, \|^*_{b^s_{\infty,\infty}(\rn)} & := &
\sup_{P \in {\mathcal Q}} \frac{1}{|P|^\tau} \lf\{ \sum_{Q_{0,m} \subset P} |\laz f,\,\phi_{0,m}\raz|\r\}
\\
& & \qquad + \sup_{i=1,\,  \ldots , \, 2^n-1}\,  \sup_{j \in \zz_+}\,  2^{j(s + n/2)}\,  \sup_{k \in \zz^n}\, |\laz f,\,\psi_{i,j,k}\raz| <\infty\, .\nonumber
\end{eqnarray*}
In addition, the quasi-norms $\| \, \Phi (f)\, \|^*_{b^s_{\infty,\infty}(\rn)}$ and $\| \, f\, \|_{B^s_{\infty,\infty}(\rn)}$
are equivalent.
\end{remark}

Also Besov-Morrey and Triebel-Lizorkin-Morrey spaces allow such characterizations.
However, for  Triebel-Lizorkin-Morrey spaces, this follows immediately from Proposition \ref{basic2}.
Hence, we may concentrate on Besov-Morrey spaces.

\begin{definition} Let $s\in\rr$, $q\in(0,\fz]$ and $0< p\le u\le\fz.$
The \emph{sequence space} $n_{u,p,q}^s(\rn)$
is defined as the space
of all sequences $t:=\{t_Q\}_{Q\in\cq^\ast}\subset\cc$ such that
$$\|t\|_{n_{u,p,q}^s(\rn)}
:=\lf\{\sum_{j\in\zz_+}2^{j(s+\frac n2)q}
\lf\|\sum_{\ell(Q)=2^{-j}}|t_Q|\chi_Q\r\|_{\cm^u_p(\rn)}^q\r\}^{1/q}<\fz.$$
\end{definition}

\begin{proposition}\label{wave2}
Let $s\in\rr$, $p,\ q \in(0,\fz]$ and $p \le u \le \infty$.
Let $N_1$ be sufficiently large (in dependence on $s,p,\tau$).
 Then $f \in \cn^s_{u,p,q} (\rn)$ if and only if $f$ can be represented as in (4.22) (convergence in $\cs'(\rn)$) and
\begin{eqnarray*}
\| \, \Phi (f)\, \|^*_{n_{u,p,q}^s(\rn)}:=
\sup_{P \in {\mathcal Q}} \frac{1}{|P|^\tau} \lf\{ \sum_{Q_{0,m} \subset P} |\laz f,\,\phi_{0,m}\raz|\r\}
+ \sum_{i=1}^{2^n-1} \|\, \{ \laz f,\,\psi_{i,j,k}\raz\}_{j,k}\, \|_{n_{u,p,q}^s(\rn)}
\end{eqnarray*}
is finite.
In addition, the quasi-norms $\| \, \Phi (f)\, \|^*_{n^s_{u,p,q}(\rn)}$ and $\| \, f\, \|_{\cn^s_{u,p,q} (\rn)}$
are equivalent.
\end{proposition}

For a proof, we refer the reader to  \cite{sa0,st} and  \cite{ro13}.


\subsection{Spaces on domains}
\label{domain}

Spaces on domains are defined by restrictions in our article.
For us, this is the most convenient way within this article.
Here, for all domains $\Omega\subset \rn$ and  $g\in \cs'(\rn)$,
$g|_\Omega$ denotes the restriction of $g$ on $\Omega$.

\begin{definition}\label{d5.12}
 Let $X (\rn)$ be a quasi-normed space of tempered distributions such that $X(\rn) \hookrightarrow \cs'(\rn)$.
Let $\Omega $ denote an open, nontrivial subset of $\rn$.
Then  $X(\Omega)$ is defined as the collection of all $f \in \cd' (\Omega)$ such that
there exists a distribution $g \in X(\rn)$ satisfying
\[
f (\varphi) = g (\varphi) \qquad \mbox{for all} \quad \varphi \in \cd (\Omega) \, .
\]
Here $\varphi \in \cd (\Omega)$ is extended by zero on $\rn\setminus \Omega$.
Moreover, let
\[
\| \, f\,\|_{X(\Omega)} := \inf \Big\{\| \, g\,\|_{X(\rn)}: \quad g_{|_\Omega} =f  \Big\} \, .
\]
\end{definition}

Hence, the spaces $F^{s,\tau}_{p,q} (\Omega)$, $B^{s,\tau}_{p,q} (\Omega)$, $\ce_{u,p,q}^s(\Omega)$
and $\cn_{u,p,q}^s(\Omega)$ are now well defined.
In this article,
we also consider Morrey spaces on domains and Campanato spaces on domains.
These spaces are not always spaces of distributions.
Therefore we gave in Section \ref{s1} of this article and in Subsection  \ref{inter1b}
independent definitions of Campanato spaces and Morrey spaces.
They coincide in the sense of equivalent norms in case that both variants are applicable.

\Acknowledgements{Wen Yuan is supported by the National
Natural Science Foundation  of China (Grant No. 11471042) and the
Alexander von Humboldt Foundation. Dachun Yang is supported by the National
Natural Science Foundation of China (Grant Nos. 11171027 and 11361020).  This project is also partially supported
by the Specialized Research Fund for the Doctoral Program of Higher Education
of China (Grant No. 20120003110003)  and the Fundamental Research Funds for Central
Universities of China (Grant Nos. 2013YB60
and 2014KJJCA10).
The authors would like to thank Professor Yoshihiro Sawano and
Doctor Alberto Arenas G\'omez for some helpful discussions on Remarks \ref{Sawano}
and \ref{r2.24}, respectively.}



\end{document}